\documentclass[a4paper,11pt,english]{article}

\usepackage{a4wide,bm,epsfig,booktabs}
\usepackage{amsmath}
\usepackage{tikz}
\usepackage{setspace}
\usepackage{relsize}  
\usepackage{graphicx}  
\usepackage{cancel} 
\usepackage{slashed}
\usepackage{verbatim} 
\usepackage{enumerate} 
\usepackage{stmaryrd} 
\usepackage{cite}
\usepackage[flushmargin,hang]{footmisc}
\usepackage{fancyhdr} 
\usepackage[arrow, matrix,curve]{xy}
\usepackage{hyperref}
\usepackage{amssymb}
\usepackage{amsthm}
\usepackage{eucal}
\usepackage{xcomment}
\usepackage{mathtools}
\usepackage{xspace}

\newcommand{\bg} {\boldsymbol{\gamma}}
\newcommand{\asymb} {\alpha}

\newcommand{\DDelta} {\mathcal{D}}

\newcommand{\MX} {\mathfrak{X}}

\newcommand{\closure} {\mathsf{clos}}
\newcommand{\maxn} {\mathsf{max}}

\newcommand{\clos}[2] {\mathsf{clos}_{#1}(#2)}

\newcommand{\SP} {\mathcal{V}}
\newcommand{\CSP} {\ovl{\mathcal{V}}}

\newcommand{\SPAN}[1] {\boldsymbol{\langle} #1\boldsymbol{\rangle} }

\newcommand{\Gen} {\mathcal{G}}
\newcommand{\CGen} {\ovl{\mathcal{G}}}

\newcommand{\MXP} {\mathfrak{z}}

\newcommand{\ppsi} {\boldsymbol{\psi}}

\newcommand{\inverse}[1] {{\mathfrak{inv}(#1)}}
\newcommand{\inversee} {{\mathfrak{inv}}}

\newcommand{\Rest}  {\boldsymbol{\mathfrak{R}}}
\newcommand{\TPOL}  {\boldsymbol{\lambda}}
\newcommand{\APOL}  {\TPOL}

\newcommand{\BanC}  {\mathfrak{r}}

\newcommand{\Yps}   {\mathcal{Y}}
\newcommand{\Alp}  {\mathcal{T}}

\newcommand{\mU}   {U}

\newcommand{\Banach}  {\mathcal{A}}
\newcommand{\Banmult}  {\ast_\Banach}
\newcommand{\bannorm}[1] 
{\|#1\|_\Banach}
\newcommand{\banel}  {\mathfrak{a}}

\newcommand{\Module}  {\mathcal{N}}
\newcommand{\Modmult}  {\ast_\Module}
\newcommand{\model}  {\mathfrak{n}}

\newcommand{\unitM}  {1_\Module}

\newcommand{\unit}  {1_\Banach}

\newcommand{\bl} {\boldsymbol{[}}
\newcommand{\br} {\boldsymbol{]}}

\newcommand{\bil}[2]  {\bl #1, #2 \br}

\newcommand{\bilbr}[1] {\com{#1}}

\newcommand{\End}  {\mathrm{End}}
\newcommand{\Aut}  {\mathsf{Aut}}

\newcommand{\norm}[1]  {\|#1\|}

\newcommand{\Lip}  {\mathfrak{L}}
\newcommand{\Add} {\mathbf{Ad}}

\newcommand{\REGC}  {\mathfrak{R}}

\newcommand{\Map} {\mathsf{Maps}}

\newcommand{\limin} {\lim^\infty_{n}}

\newcommand{\NN} {\mathbb{N}}

\newcommand{\CC} {\mathbb{C}}
\newcommand{\RR} {\mathbb{R}}
\newcommand{\korp} {\mathbb{K}}

\newcommand{\kk} {\mathrm{k}}

\newcommand{\e} {\mathrm{e}}

\newcommand{\starm} {\star}

\newcommand{\expal} {\mathfrak{i}}
\newcommand{\expalinv} {\mathfrak{i}^{-1}}

\newcommand{\lip}  {\mathrm{lip}}
\newcommand{\const}  {\mathrm{c}}

\newcommand{\TMAP} {\mathfrak{T}}

\newcommand{\AES} {\mathrm{M}}
\newcommand{\NILS} {\mathrm{N}}

\newcommand{\AAE} {$\mathrm{AE}$-set\xspace}
\newcommand{\NIL}[1] {$\mathrm{Nil}_{#1}$-set\xspace}

\newcommand{\HLCV}  {\textsf{hlcVect}}
\newcommand{\Sem}[1]  {\textsf{Sem}(#1)}
\newcommand{\comp}[1]  {\widehat{#1}}

\newcommand{\mm} {\mathsf{m}}

\newcommand{\pp} {\mathsf{p}}
\newcommand{\qq} {\mathsf{q}}

\newcommand{\vv} {\mathsf{v}}
\newcommand{\ww} {\mathsf{w}}

\newcommand{\hh} {\mathsf{h}}

\newcommand{\zzs} {\mathsf{s}}

\newcommand{\SEMG} {\mathfrak{S}}
\newcommand{\SEMMM} {\mathfrak{H}}

\newcommand{\inv} {\mathsf{inv}}

\DeclareMathOperator*{\innt}{\ThisStyle{\vstretch{0.9}{\hstretch{1.5}{\rotatebox{10}{$\SavedStyle\hspace{-0.5pt}\!\int\!\hspace{-0.5pt}$}}}}}

\DeclareMathOperator*{\Der}{\ThisStyle{\hstretch{1.2}{\rotatebox{0}{$\SavedStyle\delta^r$}}}}
\usepackage{scalerel}

\newcommand{\interior} {\mathsf{int}}
\newcommand{\inter}[2] {\mathsf{int}_{#1}(#2)}

\newcommand{\DIDE} {\mathfrak{D}}
\newcommand{\DIDED} {\mathrm{D}}

\newcommand{\EV} {\mathrm{Evol}}
\newcommand{\evol} {{\mathrm{evol}}}

\newcommand{\B} {\mathrm{B}}
\newcommand{\OB} {\ovl{\B}}

\newcommand{\chart} {\Xi}
\newcommand{\chartinv} {\Xi^{-1}}

\newcommand{\RT} {\mathrm{R}}
\newcommand{\LT} {\mathrm{L}}
\newcommand{\ad} {\mathrm{ad}}
\newcommand{\Ad} {\mathrm{Ad}}
\newcommand{\com}[1] {\ad_{#1}}

\newcommand{\CP} {\mathrm{CP}}

\newcommand{\mackarrr} {\rightharpoonup_{\mathfrak{m}}}

\newcommand{\mackarr}[1] {\rightharpoonup_{\mathfrak{m}.#1}}

\newcommand{\etam} {\boldsymbol{\Lambda}}
\newcommand{\etamm} {\boldsymbol{\eta}}

\newcommand{\dind}  {\mathrm{s}}

\newcommand{\mackeyconst} {\mathfrak{c}}

\newcommand{\mackeyindex} {\mathfrak{l}}

\newcommand{\mgc}  {\comp{\mg}}
\newcommand{\mqc}  {\comp{\mq}}

\newcommand{\llleq} {\preceq}

\newcommand{\conj}  {\mathrm{Conj}}

\newcommand{\id} {\mathrm{id}}

\newcommand{\ovl}[1] {\overline{#1}}

\newcommand{\wt}[1] {\widetilde{#1}}
\newcommand{\wh}[1] {\widehat{#1}}

\newcommand{\dd} {\mathrm{d}}
\newcommand{\im} {\mathrm{im}}
\newcommand{\dom} {\mathrm{dom}}

\newcommand{\OO}  {\mathcal{O}}
\newcommand{\U}  {\mathcal{U}}
\newcommand{\V}  {\mathcal{V}}

\newcommand{\ind}  {\mathbf{I}}
\newcommand{\indj}  {\mathbf{J}}

\newcommand{\deff} {if and only if } 
\newcommand{\defff} {if } 

\newcommand{\mg} {\mathfrak{g}}
\newcommand{\mq} {\mathfrak{q}}
\newcommand{\mh} {\mathfrak{h}}
\newcommand{\mt} {\mathfrak{t}}

\newcommand{\cp} {\circ}

\newcommand{\mult}  {\mathsf{m}}

\newcommand{\op} {\mathrm{op}}

\newcommand{\compact} {\mathrm{C}}
\newcommand{\compacto} {\mathrm{K}}

\newcommand{\he} {\hspace{1pt}}

\renewcommand{\theenumi}{\arabic{enumi})} 
\renewcommand{\labelenumi}{\theenumi}

\let\origenumerate\enumerate
\def\enumerate{\origenumerate\itemsep0pt}
\let\origitemize\itemize
\def\itemize{\origitemize\itemsep0pt}

\newtheorem{innercustomthm}{Theorem}
\newenvironment{customthm}[1]
  {\renewcommand\theinnercustomthm{#1}\innercustomthm}
  {\endinnercustomthm}
  
 \newtheorem{innercustompr}{Proposition}
\newenvironment{custompr}[1]
  {\renewcommand\theinnercustompr{#1}\innercustompr}
  {\endinnercustompr}
  
   \newtheorem{innercustomlem}{Lemma}
\newenvironment{customlem}[1]
  {\renewcommand\theinnercustomlem{#1}\innercustomlem}
  {\endinnercustomlem}

\newtheorem*{rem}{Remark}

\newtheorem{statement}{Statement}
  
\newtheorem{theorem}{Theorem}
\newtheorem{proposition}{Proposition}
\newtheorem{lemma}{Lemma}

\newtheorem{corollary}{Corollary}
\newtheorem{remark}{Remark}

\newtheorem{example}{Example}
\newtheorem{definition}{Definition}
\newtheorem{convention}{Convention}

\makeatletter
\def\blfootnote{\gdef\@thefnmark{}\@footnotetext}
\makeatother

\begin{document}
\title{The Lax Equation and Weak Regularity of\\
Asymptotic Estimate Lie Groups}
\author{
  \textbf{Maximilian Hanusch}\thanks{\texttt{mhanusch@math.upb.de}}
  \\[1cm]
  Institut f\"ur Mathematik \\
  Universit\"at Paderborn\\
  Warburger Stra\ss{e} 100 \\
  33098 Paderborn \\
  Germany
}
\date{April 8, 2023}
\maketitle

\begin{abstract}
We investigate the Lax equation in the context of infinite-dimensional Lie algebras. Explicit solutions are discussed in the sequentially complete asymptotic estimate context, and an integral expansion (sums of iterated Riemann integrals over nested commutators with correction term) is derived for the situation that the Lie algebra is inherited by an infinite-dimensional Lie group in Milnor's sense.  
In the context of Banach Lie groups (and Lie groups with suitable regularity properties), we generalize the Baker-Campbell-Dynkin-Hausdorff formula to the product integral (with additional nilpotency assumption in the non-Banach case). We combine 
this formula with the results obtained for the Lax equation     
to derive an explicit representation of the product integral in terms of the exponential map. An important ingredient in the non-Banach case is an integral transformation that we introduce. This transformation maps continuous Lie algebra-valued curves to smooth ones and leaves the product integral invariant. This transformation is also used to prove a regularity statement in the asymptotic estimate context.  
\end{abstract}

\tableofcontents

\section{Introduction}
\label{nmdsnmdsnmsdnmdsnmdsnmdsnmnmds}
We investigate regularity properties of Lie groups in Milnor's sense \cite{MIL,HG}, and extend the Baker-Campbell-Dynkin-Hausdorff formula for the exponential map to the product integral.  The product integral generalizes the Riemann integral (for curves in Hausdorff locally convex vector spaces) to Lie groups, so that Lie algebra-valued curves are ``integrated'' to Lie group elements. 
The results are based on a deeper analysis of the Lax equation in the context of asymptotic estimate and sequentially complete Lie algebras. A further ingredient is an integral transformation that we introduce. This transformation maps continuous Lie algebra-valued curves to smooth ones, and leaves the product integral invariant.

Explicitly, let $G$ be a Lie group in Milnor's sense\footnote{Throughout this paper, we work in the slightly more general setting introduced by Gl\"ockner in \cite{HG}. Specifically, this means that any completeness presumption on the modeling space is dropped. Moreover, Milnor's definition of an infinite-dimensional manifold $M$ involves the requirement that $M$ is a regular topological space, i.e., fulfills the separation axioms $T_2, T_3$. Deviating from that, in \cite{HG} only the $T_2$ property of $M$ is explicitly assumed. This, however, makes no difference in the Lie group case, because topological groups are automatically $T_3$. We also refer to  \cite{KHN2,KHNM} for  introductions into this topic.} \cite{MIL,HG} that is modeled over the Hausdorff locally convex vector space $E$. 
We denote the Lie algebra of $G$ by $(\mg,\bil{\cdot}{\cdot})$,  the  
adjoint action by $\Ad\colon G\times \mg\rightarrow \mg$, and set $\Ad_g:=\Ad(g,\cdot)$ for each $g\in G$. The left and the right translation by some $g\in G$ is denoted by $\LT_g$ and $\RT_g$, respectively. 
For $a<b$, we set
\begin{align*}
	C^1_*([a,b],G):=\{\mu\in C^{1}([a,b],G)\:|\:\mu(a)=e\}. 
\end{align*}
The evolution map is the inverse of the \emph{right logarithmic derivative}\footnote{Injectivity of $\Der$ is easily verified, confer,  e.g., Lemma 9 in \cite{RGM}.}
\hspace*{\fill}($a<b$)
\begin{align*}
 \Der\colon C^1_*([a,b],G)\rightarrow C^0([a,b],\mg),\qquad 
	 \mu\mapsto \dd_\mu\RT_{\mu^{-1}}(\dot\mu),
\end{align*} 
\vspace{-20pt}

\noindent 
and is denoted by
\vspace{-13pt}
\begin{align*}
	\textstyle\EV\colon \underbracket{\textstyle\bigcup_{a<b}\overbracket{\Der(C^1_*([a,b],G))}^{=:\: \DIDE_{[a,b]}}}_{=:\: \DIDE}\rightarrow \bigcup_{a<b}C_*^1([a,b],G).
\end{align*}
\vspace{-12pt}  

\noindent
The \emph{product integral} is defined by 
\begin{align*}
	\textstyle\innt_a^b \phi:=\EV(\phi|_{[a,b]})(b)\quad\text{for all}\quad a<b\quad\text{and}\quad \phi\in \DIDE\quad\text{with}\quad [a,b]\subseteq \dom[\phi];
\end{align*}
and for  $k\in \NN\cup \{\lip,\infty\}$, we set
\begin{align*}
	\textstyle \evol_k \colon \DIDE\cap C^k([0,1],\mg)\ni \phi\mapsto \innt_0^1 \phi\in G.
\end{align*}
We equip $C^k([0,1],\mg)$ for  $k\in \NN\cup \{\lip,\infty\}$ with the $C^k$-topology, as well as $\DIDE\cap C^k([0,1],\mg)$ with the corresponding subspace topology:\he\footnote{The $C^k$-topology is Hausdorff locally convex (the defining seminorms are recalled in Sect.\ \ref{dsdsdssdfdffddfdfdd}). In particular,   differentiability (smoothness) of the evolution map has to be understood w.r.t.\ Bastiani's differential calculus (recalled in Appendix \ref{Diffcalc}) in the following.}  
\vspace{-3pt}
	\begingroup
\setlength{\leftmargini}{11pt}
\begin{itemize}
\item
	We say that $G$ is $C^k$-semiregular if $C^k([0,1],\mg)\subseteq \DIDE$  holds. 
\item
	We say that $G$ is weakly $C^k$-regular if $G$ is $C^k$-semiregular and 
$\evol_k$ is differentiable.
	\item
	We say that $G$ is $C^k$-regular if $G$ is $C^k$-semiregular and $\evol_k$ is smooth. 
\end{itemize}
\endgroup
\noindent
We briefly want to report on the continuity problem (under which circumstances is $\evol_k$ continuous w.r.t.\ the $C^k$-topology): 
	\begingroup
\setlength{\leftmargini}{11pt}
\begin{itemize}
\item
In \cite{RGM}, this problem had been solved for the case $k=0$. Specifically, it was shown that $C^0$-continuity of $\evol_0$ on its domain (also if $C^0$-semiregularity is not assumed) is equivalent to local $\mu$-convexity of $G$ \cite{RGM,HGGG}.  Roughly speaking, local $\mu$-convexity is a generalized triangle inequality for the Lie group multiplication.  
\item
In \cite{AEH}, the continuity problem was completely solved in the asymptotic estimate context that we also consider in this paper (simply put, Theorem 1 in \cite{AEH} states that 
all continuity notions are equivalent in the asymptotic estimate context). 
\end{itemize}
\endgroup 
\noindent
A Lie group is said to be asymptotic estimate if its Lie algebra is asymptotic estimate. More generally,  
we have the following definitions for an 
infinite-dimensional Lie  algebra $(\mq,\bil{\cdot}{\cdot})$:\he\footnote{Specifically, $\mq$ is a Hausdorff locally convex vector space; and $\bil{\cdot}{\cdot}\colon \mq\times \mq\rightarrow \mq$ is bilinear, antisymmetric, continuous, and fulfills the  Jacobi identity 
 $\bil{Z}{\bil{X}{Y}}= \bil{\bil{Z}{X}}{Y} + \bil{X}{\bil{Z}{Y}}$ 
 for all $X,Y,Z\in \mq$.}  
\begingroup
\setlength{\leftmargini}{12pt}
\begin{itemize}
\item
A subset $\AES\subseteq \mq$ is said to be an \AAE \defff for each $\vv\in \Sem{\mq}$, there exists $\vv\leq \ww\in \Sem{\qq}$ with
\begin{align*}
	\vv(\bl X_1,\bl X_2,\bl \dots,\bl X_n,Y\br{\dots}\br\br])\leq \ww(X_1)\cdot{\dots}\cdot \ww(X_n)\cdot \ww(Y)
\end{align*}
for all $X_1,\dots,X_n,Y\in \AES$, and $n\geq 1$. 
\item
The Lie algebra $(\mq,\bil{\cdot}{\cdot})$ is said to be asymptotic estimate \defff $\mq$ is an \AAE.
\end{itemize}
\endgroup
\emph{For instance, abelian Lie groups are asymptotic estimate, and the same is true for Lie groups with nilpotent Lie algebras. Also Banach Lie groups are asymptotic estimate, because their Lie bracket is submultiplicative.    Notably, the class of asymptotic estimate Lie algebras has good permanence properties, as it is closed under passage to subalgebras, Hausdorff quotient Lie algebras, as well as  closed under arbitrary Cartesian products (hence, e.g., under projective limits).} 
\vspace{6pt}

As already indicated above, Theorem 1 in \cite{AEH} states that if $G$ is asymptotic estimate, then $G$ is locally $\mu$-convex if and only if $\evol_k$ is $C^k$-continuous for some (and then each) $k\in \NN\cup \{\lip,\infty\}$. Based on the semiregularity results obtained in \cite{RGM}, this statement was used in \cite{AEH} to prove that $C^\infty$-regularity implies $C^k$-regularity for each $k\in \NN\cup \{\lip,\infty\}$; whereby for $k=0$ additional completeness assumptions (on the Lie group not the Lie algebra) are required. Complementary to that, we prove the following theorem in the weakly regular context (cf.\ Theorem \ref{dkjkjfdkjfdjfdkfdfdfd}):
\begin{customthm}{I}
\label{oisiosdidsoiososdids}
	Assume that $(\mg,\bil{\cdot}{\cdot})$ is asymptotic estimate and sequentially complete. 
	If $G$ is weakly $C^\infty$-regular, then $G$ is weakly $C^k$-regular for each $k\in \NN\cup\{\lip,\infty\}$.
\end{customthm}
\noindent
This theorem is obtained by a comprehensive analysis of the Lax equation in the sequentially complete asymptotic estimate context, as well as by application of an integral transformation that we introduce. As a by-product of these investigations, we obtain a generalization of the Baker-Campbell-Dynkin-Hausdorff formula (BCDH formula) in the Banach case (Proposition \ref{lkjdfkfdjkdfaskkk}) as well as in the nilpotent weakly $C^k$-regular case (Theorem \ref{kfdhjfdhjfdfdfdfd}).  In particular, this yields an explicit formula for the product integral in terms of the exponential map that involves iterated Riemann integrals over nested commutators (see Point \ref{dodsodsoppodsdslkdslkdslkdslkdslkds0ds09ds09ds092221}). 
More explicitly:
\vspace{6pt}

Given an infinite-dimensional Lie algebra $(\mq,\bil{\cdot}{\cdot})$ (not necessarily the Lie algebra of a Lie group), we denote its completion by $\mqc$ and set
\begin{align*}
	\com{Z}\colon \mq\rightarrow \mq,\qquad Y\mapsto \bil{Z}{Y}
\end{align*} 
for each $Z\in \mq$. 
For $\psi\in C^0([a,b],\mq)$ and $X\in \mq$, we define the maps ($\ell\geq 1$)
\begin{align*}
	\textstyle \TPOL_{0,\psi}[X]\colon [a,b]\rightarrow \mqc,\qquad & t\mapsto X\\
	\textstyle \TPOL_{\ell,\psi}[X]\colon [a,b]\rightarrow \mqc,\qquad &\textstyle t\mapsto  \int_a^{t}\dd s_1\int_a^{s_1}\dd s_2 \: {\dots} \int_a^{s_{\ell-1}}\dd s_\ell \: (\bilbr{\psi(s_1)}\cp \dots \cp \bilbr{\psi(s_{\ell})})(X),
\end{align*}
and set
\vspace{-11pt}
\begin{align}
\label{podspodsposdlkdslkdslkdsdsds98ds9898ds9dsdsdsds}
	 \qquad\TPOL_{\ell,\psi}[t]\colon \mq\rightarrow \mqc,\qquad  X\mapsto \TPOL_{\ell,\psi}[X](t)\qquad\qquad\forall\: t\in [a,b],\: \ell\in \NN.
\end{align}
We have the following results:
\begingroup
\setlength{\leftmargini}{19pt}
{
\renewcommand{\theenumi}{{(\alph{enumi})}} 
\renewcommand{\labelenumi}{\theenumi}
\begin{enumerate}
\item
\label{dodsodsoppodsdslkdslkdslkdslkdslkds0ds09ds09ds092221}
	Let $G$ be a Banach Lie group such that the norm on $\mg$ fulfills $\norm{[X,Y]}\leq \norm{X}\cdot \norm{Y}$ for all $X,Y\in \mg$. Then, there exists $\BanC>0$ such that\hspace*{\fill}(Corollary \ref{lkjdfkfdjkdfas} in Sect.\ \ref{sdsssdhsdhcxuzcxuzcxcx})
\begin{align}
\label{jhsdjhsdjhsdq1dsds}
	\textstyle\innt_a^t \phi
	=\textstyle  \exp\big(\sum_{n=1}^\infty \frac{(-1)^{n-1}}{n}\cdot\int_a^t \big(\sum_{\ell=1}^\infty \TPOL_{\ell,\phi}[s]\big)^{n-1} (\phi(s))\: \dd s\he\big) \qquad\quad\forall\: t\in [a,b]
\end{align}
holds for all $a<b$ and $\phi\in C^0([a,b],\mg)$ with $\int_a^b \norm{\phi(s)}\: \dd s< \BanC$.
\vspace{2pt}

The same formula is proven in the nilpotent weakly $C^k$-regular context  (see Corollary \ref{dchdfhdfjhdf} in Sect.\ \ref{kjdkjsjksddxccxocoioicxoiiooicxiocx}), whereby the occurring sums are finite. 
For instance, if $G$ is abelian 
and weakly $C^k$-regular for some $k\in \NN\cup\{\lip,\infty\}$, then \eqref{jhsdjhsdjhsdq1dsds} specializes to the well-known formula \cite{TMICH}
\begin{align*}
	\textstyle \innt_a^t \phi = \exp\big(\int_a^t \phi(s)\: \dd s\big)
	\qquad\quad\forall\: a\leq t\leq b,\: \phi\in C^k([a,b],\mg).
\end{align*} 
Note that \eqref{jhsdjhsdjhsdq1dsds}   provides an explicit formula for holonomies in principal bundles, as such holonomies are locally given by product integrals of curves that are pairings of (smooth)
connections with derivatives of (smooth) curves in the base manifold of the principal bundle.
\item
\label{dodsodsoppodsdslkdslkdslkdslkdslkds0ds09ds09ds092222}
The maps \eqref{podspodsposdlkdslkdslkdsdsds98ds9898ds9dsdsdsds} are the elementary building blocks of solutions to the Lax equation
\begin{align}
\label{ksoidoidsdssddssdsds}
	\dot\alpha=\bil{\psi}{\alpha}\qquad\quad\text{with initial condition}\qquad\quad \alpha(a)=X,
\end{align}
where $X\in \mq$, $\psi\in C^0([a,b],\mq)$ ($a<b$) are fixed parameters, and $\alpha\in C^1([a,b],\mq)$ is the solution. 
\begingroup
\setlength{\leftmarginii}{12pt}
\begin{itemize}
\item
	Assume that $(\mq,\bil{\cdot}{\cdot})=(\mg,\bil{\cdot}{\cdot})$ is the Lie algebra of a Lie group $G$: 
\begingroup
\setlength{\leftmarginiii}{12pt}
\begin{itemize}
\item
For $\psi\in \DIDE_{[a,b]}$, the unique solution to \eqref{ksoidoidsdssddssdsds} is given by $\alpha=\Ad_{\innt_a^\bullet\!\psi}(X)$ (cf.\ Corollary \ref{lksdlkdslkdslk}).
\vspace{3pt} 
\item	
For $\psi\in \DIDE_{[a,b]}$ and $\ell\geq 1$, define
\begin{align*}
	\Rest_{\ell,\psi}[X]\colon [a,b]\ni t \textstyle \mapsto   \int_a^{t}\dd s_1\int_a^{s_1}\dd s_2 \: {\dots} \int_a^{s_{\ell-1}}\dd s_\ell \: (\bilbr{\psi(s_1)}\cp \dots \cp \bilbr{\psi(s_{\ell})})(\Ad_{\innt_a^{s_\ell}\!\psi}(X))\in \mgc.
\end{align*}  
A straightforward induction shows (cf.\ Lemma \ref{hjfdhjfdhjfd})
\begin{align*}
	\textstyle\Ad_{\innt_a^t\!\psi}(X)= \sum_{\ell=0}^n  \TPOL_{\ell,\psi}[X](t) + \Rest_{n+1,\psi}[X](t)\qquad\quad\: \forall\: n\in \NN,\: t\in [a,b]. 	
\end{align*}
In particular, if $\AES$ is an \AAE with $\im[\psi]\subseteq \AES$, then we have (cf.\ Corollary \ref{ljkdskjdslkjsda})
\begin{align}
\label{oisdoidsoioidscxxccxcx}
	\textstyle\Ad_{\innt_a^t\!\psi}(X)= \sum_{\ell=0}^\infty  \TPOL_{\ell,\psi}[X](t) \qquad\quad\: \forall\:  X\in \AES,\: t\in [a,b]. 
\end{align}
Notably, for $\psi\colon [0,1]\ni t\mapsto Y\in \mg$ constant,  \eqref{oisdoidsoioidscxxccxcx} reproduces the well-known formula  
\begin{align*}
\textstyle \Ad_{\exp(t\cdot Y)}(X)\stackrel{\eqref{odaidaooipidadais}}{=}\Ad_{\innt_0^t \psi}(X)=\textstyle\sum_{\ell=0}^\infty \frac{t^\ell}{\ell!}\cdot (\com{Y})^n(X)
\qquad\quad\forall\: X\in \mg,\: t\in [0,1].
\end{align*} 	
\end{itemize}
\endgroup
\item
Assume that $(\mq,\bil{\cdot}{\cdot})$ is sequentially complete and asymptotic estimate:
\vspace{4pt}

Mimicking \eqref{oisdoidsoioidscxxccxcx},  for $\psi\in C^0([a,b],\mq)$ and $X\in \mq$, we define
\begin{align}
\label{oisdpodspodspodspodspodspodspodsopdspodsdsds}
	\textstyle\Add_\psi[X]\colon [a,b]\rightarrow \mq,\qquad t\mapsto  \sum_{\ell=0}^\infty  \TPOL_{\ell,\psi}[X](t). 
\end{align} 
\begingroup
\setlength{\leftmarginiii}{12pt}
\begin{itemize}
\item	
It is straightforward to see that $\alpha:=\Add_\psi[X]$ 
solves the Lax equation \eqref{ksoidoidsdssddssdsds} -- and is of class $C^{k+1}$ if $\psi$ is of class $C^k$ for $k\in \NN\cup\{\infty\}$. 
\vspace{3pt}
\item
It is ad hoc not clear that this solution is unique.  
To prove uniqueness, one basically has to show that for each $t\in [a,b]$  one has
\begin{align*} 
	\Aut(\mq)\ni \Add_\psi[t]\colon \mq\rightarrow \mq,\qquad X\mapsto \Add_\psi[X](t).
\end{align*}  
Indeed, it is not hard to see that $\Add_\psi[t]$ admits a left inverse; but,  to prove the existence of a right inverse some effort is necessary. 
\end{itemize}
\endgroup
\end{itemize}
\endgroup
\end{enumerate}}
\endgroup
\noindent
In this paper, we use the solutions \eqref{oisdpodspodspodspodspodspodspodsopdspodsdsds} to prove 
Theorem \ref{oisiosdidsoiososdids}. 
We combine them with an integral transformation that we introduce, 
and that is also used to prove Formula \eqref{jhsdjhsdjhsdq1dsds} in the nilpotent weakly $C^k$-regular situation. It naturally arises in the context of semiregular Lie groups:
\begingroup
\setlength{\leftmargini}{12pt}
\begin{itemize}
\item
	Let $G$ be $C^k$-semiregular for $k\in \NN\cup \{\lip, \infty\}$. Then, the integral transformation  
\begin{align}
\label{ksoididsooidsoidsdsdsds}
	\textstyle\TMAP\colon C^k([a,b],\mg) \rightarrow C^\infty([0,1],\mg),\qquad \phi \mapsto \big[ [0,1]\ni t \mapsto  \int_a^b \Ad_{[\innt_s^{b}  t\cdot \phi]}(\phi(s)) \:\dd s\he\big]
\end{align}
is defined for all $a<b$ (cf.\ Proposition \ref{fdfdfd}), and has the invariance property 
\begin{align*}
	\textstyle\innt_a^b \phi=\innt_0^1  \TMAP(\phi)\qquad\quad\forall\: \phi\in C^k([a,b],\mg).
\end{align*} 
\item
Let $(\mq,\bil{\cdot}{\cdot})$ be sequentially complete and asymptotic estimate. Motivated by \eqref{ksoididsooidsoidsdsdsds}, we consider the integral transformation \hspace*{\fill}(defined by Lemma  \ref{klsdjklsdkjlsddjklsjklsd})
\begin{align}
\label{ksoididsooidsoidsdsdsdsas}
	\textstyle\TMAP\colon C^0([a,b],\mq) \rightarrow C^\infty([0,1],\mq),\qquad \phi \mapsto \big[ [0,1]\ni t \mapsto  \int_a^b \Add_{t\cdot \phi|_{[s,b]}}[b](\phi(s)) \:\dd s\he\big],
\end{align}
for which we have (Corollary \ref{kjfdkjfdkjfd})
\begin{align}
\label{oifdoifdoifdioifdiofd}
	\textstyle\innt _a^b\TMAP(\psi)=\innt_c^b\TMAP(\psi|_{[c,b]})\cdot  \innt_a^c\TMAP(\psi|_{[a,c]})\qquad\quad\forall\:  a<c<b,\:\psi\in C^0([a,b],\mq).
\end{align}
If $G$ is weakly $C^\infty$-semiregular and $\phi\in C^0([0,1],\mq)$, then 
\begin{align*}
	\textstyle\mu\colon [a,b]\ni z\mapsto \innt_0^1\TMAP(\phi|_{[a,z]})\in G
\end{align*}
is of class $C^1$, whereby     
 \eqref{oifdoifdoifdioifdiofd} implies $\Der(\mu)=\phi$ which    establishes Theorem \ref{oisiosdidsoiososdids}.
\end{itemize}
\endgroup
This paper is organized as follows. In Sect.\ \ref{dsdssd}, we fix the notations and provide some basic facts and definitions  concerning locally convex vector spaces, Lie groups, Lie algebras, and mappings. We furthermore discuss some elementary properties of power series that we need in the main text. Sect.\ \ref{kllkdslkdlklkdslds} contains some preliminary results. We discuss the integral transformation  \eqref{ksoididsooidsoidsdsdsds}, and provide an integral expansion for the adjoint action. We furthermore prove a differentiation result for the product integral in the weakly regular context. In Sect.\ \ref{kjdskjdskjdssd}, we generalize the Baker-Campbell-Dynkin-Hausdorff formula to the product integral, and give some applications to it. In particular, we prove formula \eqref{jhsdjhsdjhsdq1dsds}. In Sect.\ \ref{jkdkjdjkddsdsddds}, we discuss the Lax equation for asymptotic estimate and sequentially complete Lie algebras. We also investigate the elementary properties of the integral transformation  \eqref{ksoididsooidsoidsdsdsdsas} that we finally use to prove Theorem \ref{oisiosdidsoiososdids}.

\section{Preliminaries}
\label{dsdssd}
In this section, we fix the notations, and provide definitions and elementary facts concerning Lie groups, Lie algebras, and locally convex vector spaces that we shall need in the main text.  

\subsection{Conventions}
\label{dssddsdsdsds}
Given sets $X,Y$, the set of all mappings $X\rightarrow Y$ is denoted by $\Map(X,Y)\equiv Y^X$. Let $Z$ be a topological space, and $S\subseteq Z$ a subset:
\begingroup
\setlength{\leftmargini}{12pt}
\begin{itemize}
\item
	We denote the closure of $S$ in $Z$ by $\clos{Z}{S}$, or simply by $\closure(S)$ if it is clear from the context which topological space $Z$ is meant.
\item
	We denote the interior of $S$ in $Z$ by $\inter{Z}{S}$, or simply by $\interior(S)$ if it is clear from the context which topological space $Z$ is meant.
\end{itemize}
\endgroup
\noindent
The class of Hausdorff locally convex vector spaces is denoted by $\HLCV$. Given $F\in \HLCV$,  the system of continuous seminorms on $F$ 
is denoted by $\Sem{F}$. We define
\begin{align*}
	\B_{\qq,\varepsilon}:=\{X\in E\:|\: \qq(X)<\varepsilon\}\qquad\quad\text{and}\qquad\quad\OB_{\qq,\varepsilon}:=\{X\in E\:|\: \qq(X)\leq\varepsilon\}
\end{align*} 
for each $\qq\in \Sem{F}$ and $\varepsilon>0$. 
The completion of $F$ is denoted by $\comp{F}\in \HLCV$. For each $\qq\in \Sem{F}$, we let 
$\comp{\qq}\colon \comp{F}\rightarrow [0,\infty)$ denote the continuous extension of $\qq$ to $\comp{F}$. Given $F_1,\dots,F_n\in \HLCV$, we have $F_1\times{\dots}\times F_n\in \HLCV$ via the Tychonoff topology  generated by the seminorms
\begin{align*}
	\maxn[\qq_1,\dots,\qq_n]\colon F_1\times{\dots}\times F_n\ni (X_1,\dots,X_n)\mapsto \max\{\qq_k(X_k)\:|\: k=1,\dots,n\}\in [0,\infty)
\end{align*}
with $\qq_k\in \Sem{F_k}$ for $k=1,\dots,n$. 
If $\Phi\colon F_1\times {\dots}\times F_n\rightarrow F$ is a continuous $n$-multilinear map, then $\comp{\Phi}\colon \comp{F}_1\times{\dots}\times \comp{F}_n\rightarrow \comp{F}$ denotes its continuous $n$-multilinear extension.

In this paper, manifolds and Lie groups are always understood to be in the sense of \cite{HG}, i.e., smooth, Hausdorff, and modeled over a Hausdorff locally convex vector space.\footnote{We explicitly refer to Definition 3.1 and Definition 3.3 in \cite{HG}. A review of the corresponding differential calculus is provided in Appendix \ref{Diffcalc}.} 
If $f\colon M\rightarrow N$ is a $C^1$-map between the manifolds $M$ and $N$, then $\dd f\colon TM \rightarrow TN$ denotes the corresponding tangent map between their tangent manifolds, and we write $\dd_xf\equiv\dd f(x,\cdot)\colon T_xM\rightarrow T_{f(x)}N$ for each $x\in M$. 
 An interval is a non-empty, non-singleton connected subset $D\subseteq \RR$. A curve is a continuous map $\gamma\colon D\rightarrow M$ for an interval and a manifold $M$. 
If $D\equiv I$ is open, then $\gamma$ is said to be of class $C^k$ for $k\in \NN\cup \{\infty\}$ \defff it is of class $C^k$ when considered as a map between the manifolds $I$ and $M$.
If $D$ is an arbitrary interval, then $\gamma$ is said to be of class $C^k$ for $k\in \NN\cup \{\infty\}$ (we write $\gamma\in C^k(D,M)$) \defff $\gamma=\gamma'|_D$ holds for a $C^k$-curve $\gamma'\colon I\rightarrow M$ that is defined on an open interval $I$ containing $D$.  
If $\gamma\colon D\rightarrow M$ is of class $C^1$, then we denote the corresponding tangent vector at $\gamma(t)\in M$ by $\dot\gamma(t)\in T_{\gamma(t)}M$. 
The above conventions also hold if $M\equiv F\in \HLCV$ is a Hausdorff locally convex vector space. 
       
Throughout this paper, $G$ denotes an infinite-dimensional Lie group in the sense of \cite{HG} that is modeled over $E\in \HLCV$. We fix a chart 
	$\chart\colon G\supseteq \U\rightarrow \V\subseteq E$, 
with $\V$ convex, $e\in \U$, and $\chart(e)=0$. 
We denote the Lie algebra of $G$ by $(\mg, \bil{ \cdot}{\cdot})$,  
 and set\footnote{In other words, we equip $\mg$ with the Hausdorff locally convex topology that is generated by the system of seminorms $\{\pp\cp\dd_e\chart\:|\: \pp\in \Sem{E}\}$. It is then not hard to see that $\Sem{\mg}=\{\pp\cp\dd_e\chart\:|\: \pp\in \Sem{E}\}$ holds, hence $\Sem{\mg}\cong\Sem{E}$.} 
\begin{align*}
	 \pp(X):=(\pp\cp\dd_e\chart)(X)\qquad\quad\forall\:\pp\in \Sem{E},\: X\in \mg. 
\end{align*} 
 We let $\mult\colon G\times G\rightarrow G$ denote the Lie group multiplication, $\RT_g:=\mult(\cdot, g)$ the right translation by $g\in G$, $\inv\colon G\ni g\mapsto g^{-1}\in G$ the inversion, and $\Ad\colon G\times \mg\rightarrow \mg$ the adjoint action. We have
\begin{align*}
	\Ad(g,X)\equiv \Ad_g(X):=\dd_e\conj_g(X)\qquad\quad\text{for}\qquad\quad \conj_g\colon G\ni h\mapsto g\cdot  h\cdot g^{-1}\in G
\end{align*}
for each $g\in G$ and $X\in  \mg$, as well as 
\begin{align}
\label{iufiugfiugfiugfiugfgf}
	\:\he\bil{X}{Y}:=\dd_e\Ad(Y)(X)\qquad\forall\: X\in \mg \qquad\quad\:\:\text{for}\qquad\quad\:\: \Ad(Y)\colon G\ni g\mapsto \Ad_g(Y)\in \mg.\hspace{6pt}
\end{align} 
We recall the Jacobi identity (cf.\ also Example \ref{fdpofsddopdpof})
\begin{align}
\label{jacobi}
	\textstyle \bil{Z}{\bil{X}{Y}}= \bil{\bil{Z}{X}}{Y}+ \bil{X}{\bil{Z}{Y}}\qquad\quad\:\forall\: X,Y,Z\in \mg
\end{align}
as well as the product rule
\begin{align}
\label{LGPR}
	\dd_{(g,h)}\mult(v,w)= \dd_g\RT_h(v) + \dd_h\LT_g(w)\qquad\quad\forall\: g,h\in G,\:\: v\in T_gG,\:\: w\in T_h G.
\end{align}

\subsection{Locally Convex Vector Spaces}
In this section, we collect some elementary facts and definitions  concerning locally convex vector spaces that we shall need in the main text.
\subsubsection{Sets of Curves}
\label{dsdsdssdfdffddfdfdd}
Let $F\in \HLCV$ be given. For each $X\in F$, we let $\mathcal{C}_X\colon \RR\ni t\mapsto X \in F$ denote the constant curve $X$. Let $a<b$ be given. We set $C^\const([a,b],F):=\{\mathcal{C}_X|_{[a,b]}\:|\: X\in F\}$;   
and let  
$C^\lip([a,b],F)$ denote the set of all Lipschitz curves on $[a,b]$, i.e., all curves $\gamma\colon [a,b]\rightarrow F$ with
\begin{align*}
	\textstyle \Lip(\qq,\gamma):=
	\sup\Big\{\frac{\qq(\gamma(t)-\gamma(t'))}{|t-t'|}\:\Big|\: t,t'\in [a,b],\: t\neq t'\Big\}\in [0,\infty)\qquad\quad\forall\: \qq\in \Sem{F}.
\end{align*}
We define  
$\const+1:=\infty$, $\infty +1:=\infty$, $\lip+1:=1$; as well as 
\begin{align*}
\qq^\dind_\infty(\gamma)&:=\sup\big\{\qq\big(\gamma^{(\ell)}(t)\big)\:\big|\: 0\leq \ell\leq \dind,\:\:t\in [a,b]\big\}\qquad\quad\forall\: \gamma\in C^k([a,b],F)\\
\qq_\infty(\gamma)&:=\qq_\infty^0(\gamma)\qquad\quad\hspace{147.2pt}\forall\: \gamma\in C^0([a,b],F)\\
\qq_\infty^\lip(\gamma)&:= \max(\qq_\infty(\gamma), \Lip\big(\qq,\gamma)\big)\qquad\quad \hspace{79.3pt}\forall\: \gamma\in C^\lip([a,b],F)
\end{align*}
for $\qq\in \Sem{F}$, $k\in \NN\cup\{\infty,\const\}$, $0\leq \dind\llleq k$.   
Here, the notation $s\llleq k$ means 
\begingroup
\setlength{\leftmargini}{12pt}
\begin{itemize}
\item
	$\dind= \lip$\hspace{20pt} for\: $k= \lip$,
\item
	$\NN\ni \dind\leq k$\hspace{2.3pt}\: for\: $k\in \NN$, 
\item
	$\dind\in \NN$\hspace{3pt}\:\hspace{20pt} for\: $k=\infty$,
\item
	$\dind=0$\hspace{4.16pt}\:\hspace{20pt} for\: $k= \const$.
\end{itemize}
\endgroup
\noindent
Given $k\in \NN\cup\{\lip,\infty,\const\}$, 
the $C^k$-topology on $C^k([a,b],F)$ is the Hausdorff locally convex topology that is generated by the seminorms $\qq_\infty^\dind$ for $\qq\in \Sem{F}$ and $0\leq \dind\llleq k$. 
\vspace{6pt}

\noindent
By an element in $\CP^0([a,b],F)$ (piecewise continuous curves), we understand a decomposition $a= t_0<{\dots}<t_n=b$ (for $n\geq 1$) together with a collection $\gamma\equiv\{\gamma[p]\}_{0\leq p\leq n-1}$ of curves $\gamma[p]\in C^0([t_p,t_{p+1}],F)$ for $0\leq p\leq n-1$. Linear combinations as well as restrictions to compact  intervals of elements in $\CP^0([a,b],F)$ are defined in the obvious way. We equip $\CP^0([a,b],F)$ with the $C^0$-topology, i.e., the Hausdorff locally convex topology that is generated by the seminorms 
\begin{align*}
	\qq_\infty(\gamma)=\max\{\qq_\infty(\gamma[p])\:|\: 0\leq p\leq n-1 \}
\end{align*}
for $\qq\in \Sem{F}$, and $\gamma\equiv\{\gamma[p]\}_{0\leq p\leq n-1}$ as above.   
\subsubsection{Mackey Convergence}
Let $F\in \HLCV$. A subsystem $\SEMMM\subseteq \Sem{F}$ is said to be a fundamental system \defff $\{\B_{\hh,\:\varepsilon}(0)\}_{\hh\in \SEMMM,\varepsilon>0}$ is a local base of zero in $F$. We recall the following standard result \cite{MV}:
\begin{lemma}
\label{xcxcxccxxcxc}
Let $\SEMMM\subseteq \Sem{F}$ be a fundamental system, and $\SEMG\subseteq \Sem{F}$ a subsystem. Then, $\SEMG$ is a fundamental system \deff to each $\hh\in \SEMMM$ there exist $c>0$ and $\zzs\in  \SEMG$ with $\hh\leq c\cdot \zzs$.
\end{lemma}
\begin{proof}
Confer, e.g., Lemma 20 in \cite{MDM}.
\end{proof}
\noindent
Let $\SEMMM\subseteq \Sem{F}$ be a fundamental system. 
We write $\{X_n\}_{n\in \NN} \mackarrr X$ for $\{X_n\}_{n\in \NN}\subseteq F$, $X\in F$ \defff 
\begin{align}
\label{podsopdspodssdnbbn}
	\hh(X-X_n)\leq \mackeyconst_\hh\cdot \lambda_{n} \qquad\quad\forall\: n\in \NN,\:\:\hh\in \SEMMM
\end{align}
holds for certain $\{\mackeyconst_\hh\}_{\hh\in \SEMMM}\subseteq \RR_{\geq 0}$, and $\RR_{> 0}\supseteq \{\lambda_{n}\}_{n\in \NN}\rightarrow 0$. In this case, $\{X_n\}_{n\in \NN}$ is said to be \emph{Mackey convergent} to $X$.
\begin{remark} 
\begingroup
\setlength{\leftmargini}{12pt}{
\begin{itemize}
\item
It is immediate from Lemma \ref{xcxcxccxxcxc} that the definition made in \eqref{podsopdspodssdnbbn} does not depend on the explicit choice of the fundamental system $\SEMMM\subseteq \Sem{F}$. 
\item
Since we will make use of the differentiability results obtained in \cite{RGM}, we explicitly mention that in \cite{RGM} an equivalent definition of Mackey convergence (more suitable for the technical argumentation there) was used. Specifically, \eqref{podsopdspodssdnbbn} is equivalent to require that 
 \begin{align}
\label{podsopdspodssdnbbnjjj}
	\hh(X-X_n)\leq \mackeyconst_\hh\cdot \lambda_{n} \qquad\quad\forall\: n\geq \mackeyindex_\hh,\:\:\hh\in \SEMMM
\end{align}
holds for certain $\{\mackeyconst_\hh\}_{\hh\in \SEMMM}\subseteq \RR_{\geq 0}$, $\{\mackeyindex_\hh\}_{\hh\in \SEMMM}\subseteq \NN$, and $\RR_{\geq 0}\supseteq \{\lambda_{n}\}_{n\in \NN}\rightarrow 0$.

In other words, the indices $\{\mackeyindex_\hh\}_{\hh\in \SEMMM}$ in \eqref{podsopdspodssdnbbnjjj} can be circumvented (can be set equal to zero) if $\{\lambda_n\}_{n\in \NN}\subseteq \RR_{>0}$ instead of $\{\lambda_n\}_{n\in \NN}\subseteq \RR_{\geq0}$ is required. \hspace*{\fill}$\qed$
\end{itemize}
}
\endgroup
\end{remark}
\noindent
Assume now $F= C^k([a,b],\mg)$ for  $k\in \NN\cup\{\lip,\infty,\const\}$ and $a<b$. We write 
\begin{align}
\label{iudsiudsudspodsisdpoids}
\{\phi_n\}_{n\in \NN}\mackarr{\kk}\phi\qquad\text{for}\qquad \{\phi_n\}_{n\in \NN}\subseteq C^k([a,b],\mg)\qquad\text{and}\qquad \phi\in C^k([a,b],\mg)
\end{align}
\defff $\{\phi_n\}_{n\in \NN}$ is Mackey convergent to $\phi$ w.r.t.\ to the $C^k$-topology.

\subsubsection{The Riemann Integral and Completeness}
\label{opsdpods}
Let $F\in \HLCV$. The Riemann integral of $\gamma\in C^0([a,b],F)$ for $a<b$ is denoted by 
$\int \gamma(s) \:\dd s\in \comp{F}$. 
\begin{remark}
\label{seqcomp}
The Riemann integral in complete Hausdorff locally convex vector spaces can be defined exactly as in the finite-dimensional case; namely, as the limit of Riemann sums (which form a Cauchy sequence in $F$), confer Sect.\ 2 in  \cite{COS}. In particular, the following assertions hold: 
\vspace{-2pt}
\begingroup
\setlength{\leftmargini}{12pt}{
\begin{itemize}
\item
	If $F$ is sequentially complete, then $\int \gamma(s) \:\dd s\in F$ holds for $\gamma\in C^0([a,b],F)$.	
\item
	For $\gamma\in C^0([a,b],F)$ and $\gamma_-\colon [a,b]\ni t\mapsto \gamma(a+b-t)$, we have $\int \gamma_-(s)\:\dd s=\int \gamma(s)\: \dd s$.
\hspace*{\fill}$\qed$	
\end{itemize}}
\endgroup
\end{remark}
\noindent
Given $x<y$ and $\gamma\in C^0([x,y],F)$, then for $x\leq a<b\leq y$ and $x\leq c\leq y$, we set
\begin{align}
\label{fdopfdpo}
	\textstyle\int_a^b \gamma(s)\:\dd s:= \int \gamma|_{[a,b]}(s) \:\dd s,\qquad\:\:\int_b^a \gamma(s) \:\dd s:= - \int_a^b \gamma(s) \:\dd s,\qquad\:\:
	 \int_c^c \gamma(s)\: \dd s:=0.\quad\:\:
\end{align}
Clearly, the Riemann integral is linear with	
\begin{align}
\label{isdsdoisdiosd0}
	\textstyle\int_a^c \gamma(s) \:\dd s&\textstyle=\int_a^b \gamma(s)\:\dd s+ \int_b^c \gamma(s)\:\dd s\qquad\quad\he \forall\: x\leq a< b< c\leq y.
\end{align}
The Riemann integral admits the following additional properties:
\begingroup
\setlength{\leftmargini}{12pt}{
\begin{itemize}
\item
For $a<b$, we have
\begin{align}
	\label{isdsdoisdiosd}
		\textstyle\gamma-\gamma(a)\hspace{4pt}&\textstyle=\int_a^\bullet \dot\gamma(s)\:\dd s\qquad\qquad\hspace{3.25pt}\forall\:  \gamma\in C^1([a,b],F),\\[2pt]
	\label{isdsdoisdiosd1}
		\qq(\gamma-\gamma(a))&\textstyle\leq \int_a^\bullet \qq(\dot\gamma(s))\: \dd s\qquad\quad\forall\:  \gamma\in C^1([a,b],F),\: \qq\in \Sem{F},\\[2pt]
	\label{ffdlkfdlkfd}
		\textstyle\comp{\qq}\big(\int_a^\bullet \gamma(s)\:\dd s\big)&\textstyle\leq \int_a^\bullet \qq(\gamma(s))\:\dd s\qquad\quad\forall\: \gamma\in C^0([a,b],F),\: \qq\in \Sem{F}.
	\end{align}
\item	
For $a<b$ and $\gamma\in C^0([a,b],F)$, we have
\begin{align}
\label{oidsoidoisdoidsoioidsiods}
	\textstyle C^1([a,b],\comp{F})\ni\Gamma\colon [a,b]\ni t\mapsto \int_a^t \gamma(s)\:\dd s \in \comp{F}\qquad\quad\text{with}\qquad\quad	\dot\Gamma=\gamma.
\end{align}
\vspace{-18pt}
\item	  
	For $a<b$, $a'<b'$, $\gamma\in C^1([a,b],F)$, $\varrho\colon [a',b'] \rightarrow [a,b]$ of class $C^1$, and $t\in [a',b']$, we have the substitution formula 
\begin{align}
\label{substitRI}
	\textstyle\int_a^{\varrho(t)} \gamma(s)\: \dd s=\int_{a'}^t \dot\varrho(s)\cdot (\gamma\cp\varrho)(s)\:\dd s + \int_a^{\varrho(a')} \gamma(s)\:\dd s.
\end{align}
\vspace{-18pt}
\item 
Let $\tilde{F}\in \HLCV$, $a<b$, and  
$\mathcal{L}\colon F\rightarrow \tilde{F}$ a continuous and linear map. Then, the following implication holds:
\begin{align}
\label{pofdpofdpofdsddsdsfd}
	\textstyle\int_a^b \gamma(s)\:\dd s\in F \quad\text{for}\quad\gamma\in C^0([a,b],F)\qquad\Longrightarrow\qquad 
	\textstyle \mathcal{L}\big(\int_a^b \gamma(s)\:\dd s\big)=\int_a^b \mathcal{L}(\gamma(s)) \: \dd s.
\end{align}
\vspace{-20pt}
\end{itemize}}
\endgroup
\noindent
Next, let us recall the following definitions: 
\begingroup
\setlength{\leftmargini}{12pt}{
\begin{itemize}
\item
	We say that $F$ is Mackey complete \defff $\int_0^1 \gamma(s)\:\dd s \in F$ exists for each $\gamma\in C^\infty([0,1],F)$.
	  
	 	This is equivalent to require that 	$\int_a^b \gamma(s)\:\dd s \in F$ exists for each $\gamma\in C^\lip([a,b],F)$ and $a<b$, confer Theorem 2.14 in \cite{COS} or Corollary 6 in \cite{RGM}.
\item
	We say that $F$ is integral complete \defff $\int_a^b \gamma(s)\:\dd s \in F$ exists for all $a<b$ and $\gamma\in C^0([a,b],F)$.
\end{itemize}
}
\endgroup
\noindent
Evidently, $F$ is Mackey complete if $F$ is integral complete. Moreover, 
Remark \ref{seqcomp} implies that $F$ is integral complete (hence, Mackey complete) if $F$ is sequentially complete. 
\vspace{6pt}

\noindent
Finally, for $\gamma\equiv\{\gamma[p]\}_{0\leq p\leq n-1}\in \CP^0([a,b],F)$ ($n\geq 1$), we define
\begin{align}
\label{opofdpopfd}
	\textstyle\int_a^{b}\!\gamma(s)\:\dd s:=\sum_{p=0}^{n-1}\int_{t_p}^{t_{p+1}} \gamma[p](s)\:\dd s. 
\end{align} 
Clearly, the so-extended Riemann integral is linear, and $C^0$-continuous -- specifically, \eqref{ffdlkfdlkfd} yields
\begin{align*}
	\textstyle\comp{\qq}\big(\int_a^t \gamma(s)\:\dd s\big)&\textstyle\leq (t-a)\cdot \qq_\infty(\gamma)\qquad\quad\forall\: a<t\leq b,\:\qq\in \Sem{F}.
\end{align*}    
For $x<y$ and $\gamma\in \CP^0([x,y],F)$, we define $\int_a^b\gamma(s)\: \dd s$, $\int_b^a\gamma(s)\: \dd s$, $\int_c^c \gamma(s)\: \dd s$ for $x\leq a<b\leq y$ and $x\leq c\leq y$ as in \eqref{fdopfdpo}. Clearly, then  \eqref{isdsdoisdiosd0} also holds for the extended Riemann integral. We furthermore note that $[a,b]\ni t\mapsto \int_a^t \gamma(s)\:\dd s\in F$ is continuous for each $x\leq a<b\leq y$, as well as of class $C^1$ if $[a,b]\subseteq [t_p,t_{p+1}]$ holds for some $0\leq p\leq n-1$. 

\subsubsection{Extensions and Completeness}
Let $F\in \HLCV$ be given. In this subsection, we collect some statements that we shall need in Sect.\ \ref{jkdkjdjkddsdsddds}. We recall the following result.\footnote{The statement is due to a construction that goes back to Seeley \cite{SEELEY}. We also refer to Proposition 24.10 in \cite{COS} for the case $k=\infty$ in the convenient setting.}
\begin{lemma}
\label{fdfdfdfdsddds}
	Let $D\subseteq \RR$ be an interval, $k\in \NN\cup\{\infty\}$, and      
 $\gamma\in C^k(J,F)$ for $J:=\inter{\RR}{D}$. Let $\{\gamma_\ell\}_{0\leq \ell\llleq k}\subseteq C^0(D,F)$ be given such that $\gamma^{(\ell)}=\gamma_\ell|_{J}$ holds for all $0\leq \ell\llleq k$.  
Then, there exists an open interval $I\subseteq \RR$ with $D\subseteq I$ as well as 
$\tilde{\gamma}\in C^k(I,F)$ such that $\tilde{\gamma}|_D=\gamma$ holds.
\end{lemma}
\begin{proof}
This is immediate from Theorem 1 in \cite{HASEE}.
\end{proof}

\begin{lemma}
\label{kjdskjskjds}
If $F$ is sequentially complete, then $C^k([a,b],F)$ is sequentially complete w.r.t.\ the $C^k$-topology for each $k\in \NN\cup\{\infty\}$ and $a<b$.
\end{lemma}
\begin{proof}
Let $k\in \NN\cup\{\infty\}$ be fixed; and let   
 $\{\gamma_n\}_{n\in \NN}\subseteq C^k([a,b],F)$ be a Cauchy sequence  w.r.t.\ the $C^k$-topology.  
Then, $\{(\gamma_n)^{(\ell)}\}_{n\in \NN}\subseteq C^0([a,b],F)$ is  a Cauchy sequence  w.r.t.\ the $C^0$-topology for $0\leq \ell\llleq k$,  and a standard $\frac{\varepsilon}{3}$-argument shows that it converges w.r.t.\ the $C^0$-topology to some $\gamma_\ell\in C^0([a,b],F)$. Let $\gamma:=\gamma_0|_{(a,b)}$. By Lemma \ref{fdfdfdfdsddds}, it suffices to show 
\begin{align}
\label{jhhjnbnbnbnbcvvccvyxxyxy}
	\gamma\in C^k((a,b),F)\qquad\text{with}\qquad\gamma^{(\ell)}=\gamma_\ell|_{(a,b)}\qquad\text{for}\qquad 0\leq \ell\llleq k.
\end{align}
Since the Riemann integral is $C^0$-continuous by \eqref{ffdlkfdlkfd}, we have for $0\leq \ell < k$ and $t\in [a,b]$ that
\begin{align*}
	\textstyle\gamma_{\ell}(t)
	\textstyle= \lim_n(\gamma_n)^{(\ell)}(t)
	\textstyle\stackrel{\eqref{isdsdoisdiosd}}{=} \lim_n\big((\gamma_n)^{(\ell)}(a)+ \int_a^t (\gamma_n)^{(\ell+1)}(s) \:\dd s \big)
	=\textstyle\gamma_\ell(a) + \int_{a}^t \gamma_{\ell+1}(s)\: \dd s. 
\end{align*}
Then, \eqref{oidsoidoisdoidsoioidsiods} shows $\gamma_\ell|_{(a,b)}\in C^1((a,b),F)$ with $(\gamma_\ell|_{(a,b)})^{(1)}=\gamma_{\ell+1}|_{(a,b)}$ for $0\leq \ell < k$.
\vspace{6pt}

\noindent
Now, since $\gamma=\gamma_0|_{(a,b)}$ holds by definition, we can assume that there exists some $0\leq q< k$ such that \eqref{jhhjnbnbnbnbcvvccvyxxyxy} holds for $k\equiv q$. We obtain 
\begin{align*}
	\gamma_{q+1}|_{(a,b)}=(\gamma_q|_{(a,b)})^{(1)}=(\gamma^{(q)})^{(1)}=\gamma^{(q+1)},
\end{align*}
hence $\gamma\in C^{q+1}((a,b),F)$.
\end{proof}

\begin{corollary}
\label{dsdsdsdsdsdsdsdsasas}
Assume that $F$ is sequentially complete, and let $\SEMMM\subseteq \Sem{F}$ be a fundamental system. Let $k\in \NN\cup\{\infty\}$, $a<b$, and $\{\gamma_p\}_{p\in \NN}\subseteq C^k([a,b],F)$ be given with 
\begin{align*}
	\textstyle\sum_{p=0}^\infty \qq^\dind_\infty(\gamma_p)< \infty\qquad\quad\forall\: \qq\in \SEMMM,\: 0\leq \dind\llleq k.
\end{align*}
Then, $\textstyle  \gamma:= \sum_{p=0}^\infty \gamma_p\in C^k([a,b],F)$ converges w.r.t.\ the $C^k$-topology, i.e., $\gamma^{(\ell)}=\sum_{p=0}^\infty (\gamma_p)^{(\ell)}$ converges w.r.t.\ the $C^0$-topology for each $0\leq \ell\llleq k$. 
\end{corollary}
\begin{proof}
The assumptions imply that $\{\sum_{p=0}^n\gamma_p\}_{n\in \NN}\subseteq C^k([a,b],F)$ is Cauchy w.r.t.\ the $C^k$-topology. The rest is clear from Lemma \ref{kjdskjskjds}. 
\end{proof}

\begin{corollary}
\label{dsdsdsdsdsdsdsds}
Assume that $F$ is sequentially complete, and let $\SEMMM\subseteq \Sem{F}$ be a fundamental system.  
Let $\{X_p\}_{p\in \NN}\subseteq F$ and $\delta>0$ be given, such that 
\begin{align*}
	\textstyle \gamma[\qq]\colon (-\delta,\delta)\ni t\mapsto \sum_{p=0}^\infty t^p\cdot \qq(X_p)\in [0,\infty)
\end{align*}
is defined for each $\qq\in \SEMMM$. Define $\gamma_p\colon (-\delta,\delta)\ni t\mapsto  t^p\cdot X_p\in F$ for each $p\in \NN$. 
\begingroup
\setlength{\leftmargini}{17pt}{
\renewcommand{\theenumi}{{\alph{enumi}})} 
\renewcommand{\labelenumi}{\theenumi}
\begin{enumerate}
\item
\label{lksdkltzastzsatzsatzsatzsatzastz1}
	Let $-\delta<a<b<\delta$. Then, 
		$\textstyle\gamma[a,b]:= \sum_{p=0}^\infty\gamma_p|_{[a,b]}\in C^\infty([a,b],F)$ 		 
 converges  w.r.t.\ the $C^\infty$-topology, with 
\begin{align*}
	 \textstyle\gamma[a,b]^{(\ell)}=\sum_{p=0}^\infty(\gamma_p|_{[a,b]})^{(\ell)}\colon [a,b]\ni t\mapsto \sum_{p=\ell}^\infty \frac{p !}{(p-\ell)!}\cdot t^{p-\ell}\cdot X_p\in F\qquad\quad\forall\: \ell\in \NN.
\end{align*} 
\vspace{-17pt}
\item
\label{lksdkltzastzsatzsatzsatzsatzastz2}
	We have $\gamma :=\sum_{p=0}^\infty\gamma_p\in C^\infty((-\delta,\delta),F)$ with
	\begin{align*}
		\textstyle\gamma^{(\ell)}=\sum_{p=0}^\infty(\gamma_p)^{(\ell)}\colon (-\delta,\delta)\ni t\mapsto \sum_{p=\ell}^\infty \frac{p !}{(p-\ell)!}\cdot t^{p-\ell}\cdot X_p\in F \qquad\quad\forall\: \ell\in \NN. 
\end{align*}	 
\vspace{-22pt}
\end{enumerate}}
\endgroup
\end{corollary}
\begin{proof}
Point \ref{lksdkltzastzsatzsatzsatzsatzastz2} is clear from Point \ref{lksdkltzastzsatzsatzsatzsatzastz1}. To prove Point \ref{lksdkltzastzsatzsatzsatzsatzastz1}, we observe that $\gamma[\qq]$ is smooth for each $\qq\in \SEMMM$, with 
\begin{align*}
	\textstyle\gamma[\qq]^{(\ell)}(t)= \sum_{p=\ell}^\infty \frac{p !}{(p-\ell)!}\cdot t^{p-\ell}\cdot \qq(X_p)\in [0,\infty) \qquad\quad\forall\: t\in (-\delta,\delta),\: \ell\in \NN. 
\end{align*}
This implies $\sum_{p=0}^\infty \qq_\infty^\dind(\gamma_p|_{[a,b]})<\infty$ for all $\qq\in \SEMMM$, $\dind\in \NN$, $a<b$, so that the claim is clear from Corollary \ref{dsdsdsdsdsdsdsdsasas}. 
\end{proof} 
\subsubsection{Lie Algebras}
\label{dlkdslklkds} 
Let $(\mq,\bil{\cdot}{\cdot})$ be a fixed Lie algebra (not necessarily the Lie algebra of a Lie group), i.e.,
\begingroup
\setlength{\leftmargini}{12pt}
\begin{itemize}
\item
	$\mq\in \HLCV$,
\item
	$\bil{\cdot}{\cdot}\colon \mq\times \mq\rightarrow \mq$ is bilinear, antisymmetric, and continuous, with (Jacobi identity)
\begin{align}
\label{nmvcnmvcnmnkjsakjsakjsaa}
 \bil{Z}{\bil{X}{Y}}= \bil{\bil{Z}{X}}{Y} + \bil{X}{\bil{Z}{Y}}\qquad\quad\forall\: X,Y,Z\in \mq.
\end{align}   
\end{itemize}
\endgroup
\noindent
We shall need the following definitions:
\begingroup
\setlength{\leftmargini}{12pt}
\begin{itemize}
\item
Given $X\in \mq$, we denote $\com{X}^1\equiv\bilbr{X}\colon \mq\ni Y\mapsto \bil{X}{Y}\in \mq$,\footnote{We note that if $(\mq,\bil{\cdot}{\cdot})$ is the Lie algebra of a Lie group $Q$, then one usually  defines $\com{X}\colon \mq\ni Y\mapsto \dd_e\Ad(Y)(X)\in \mq$ for $X\in \mg$. According to our definition \eqref{iufiugfiugfiugfiugfgf} of the commutator, this is consistent with our notation in the general case (i.e., where $(\mq,\bil{\cdot}{\cdot})$ is not necessarily the Lie algebra of a Lie group).}
and define inductively  
\begin{align*}
	\com{X}^0:=\id_\mq \qquad\quad\text{as well as}\qquad\quad \com{X}^n:=\com{X}\cp\com{X}^{n-1}\qquad\forall\: n\geq 1.\hspace{17pt}
\end{align*} 
\vspace{-18pt}	
\item
Given a subset $S\subseteq \mq$, then $\SPAN{S}\subseteq \mq$ denotes the linear subspace that is generated by $S$. We define inductively
\begin{align*}
	\SP_1(S):=\SPAN{S}\qquad\quad\text{as well as}\qquad\quad \SP_{n+1}(S):=\SPAN{\bil{\SP(S)}{\SP_n(S)}}\qquad\forall\: n\geq 1;
\end{align*}
and set $\textstyle\Gen_n(S):= \SPAN{\bigcup_{\ell= n}^{\infty} \SP_\ell(S)}$ for each $n\geq 1$.
\end{itemize}
\endgroup
\noindent  
A straightforward induction involving \eqref{nmvcnmvcnmnkjsakjsakjsaa} shows (confer Appendix \ref{asassadsdsdsdsdsdsdsaaa} or Exercise 5.2.7 in \cite{HIGN}) 
\begin{align}
\label{kjsdkjdsjdsa}
	\SPAN{\bil{\mathcal{V}_m(S)}{\mathcal{V}_n(S)}}\subseteq \mathcal{V}_{m+n}(S)\qquad\quad\forall\:m,n\geq 1.
\end{align}
For $n\geq 1$, we set $\CSP_n:=\clos{\mq}{\SP_n}$ as well as $\CGen_n:=\clos{\mq}{\Gen_n}$. Since $\bil{\cdot}{\cdot}$ is continuous,  \eqref{kjsdkjdsjdsa} implies
\begin{align}
	\label{kjfdkjkjfdkjfdkjfd}
	\SPAN{\bil{\CSP_m(S)}{\CSP_n(S)}}\subseteq \CSP_{m+n}(S)\qquad\quad\forall\:m,n\geq 1. 
\end{align}
We observe that $\CGen_{n+\ell}(S)\subseteq \CGen_{n}(S)$ holds for all $n,\ell\in \NN$. We shall need the following definitions:\he\footnote{Note that the first definition is recalled from Sect.\ \ref{nmdsnmdsnmsdnmdsnmdsnmdsnmnmds}.} 
\begingroup
\setlength{\leftmargini}{12pt}
\begin{itemize}
\item
A subset $\AES\subseteq \mq$ is said to be an \AAE \defff to each $\vv\in \Sem{\mq}$, there exist $\vv\leq \ww\in \Sem{\qq}$ with
\begin{align}
\label{assaaas1}
	\vv(\bl X_1,\bl X_2,\bl \dots,\bl X_n,Y\br{\dots}\br\br])\leq \ww(X_1)\cdot{\dots}\cdot \ww(X_n)\cdot \ww(Y)
\end{align}
for all $X_1,\dots,X_n,Y\in \AES$, and $n\geq 1$. 
We say that $(\mq,\bil{\cdot}{\cdot})$ is asymptotic estimate \defff $\mq$ is an \AAE.
\item
A subset $\NILS\subseteq \mq$ is said to be a \NIL{q} for $q\geq 2$ \defff  
\begin{align}
\label{assaaas2}
	\bl X_1,\bl X_2,\bl \dots,\bl X_{q-1},X_q\br{\dots}\br\br\br=0
	\qquad\quad\forall\: X_1,\dots,X_q\in \NILS.
\end{align}
Evidently, each \NIL{q} (for $q\geq 2$) is an \AAE.   
We say that $(\mq,\bil{\cdot}{\cdot})$ is nilpotent \defff $\mq$ is a \NIL{q} for some $q\geq 2$.
\end{itemize}
\endgroup
\noindent
In view of Sect.\ \ref{dsdsdsddssddsds} (the proof of Lemma \ref{mnsdmnsdmnsdmndsmnsd}), we observe the following.
\begin{remark}
\label{dsdsdsdsdsv}
Let $\NILS\subseteq \mq$ be a \NIL{q} for some $q\geq 2$. Then, 
\vspace{-4pt}
\begingroup
\setlength{\leftmargini}{17pt}
{
\renewcommand{\theenumi}{\emph{\arabic{enumi})}} 
\renewcommand{\labelenumi}{\theenumi}
\begin{enumerate}
\item
\label{dsdsdsdsdsv1}
$\CGen_{q+n}(\NILS)=\{0\}$ holds for each $n\in \NN$.
\item
\label{dsdsdsdsdsv4}
$\bil{\CGen_m(\NILS)}{\CGen_n(\NILS)}\subseteq \CGen_{m+n}(\NILS)$ holds for $m,n\geq 1$, by \eqref{kjfdkjkjfdkjfdkjfd}. 
\item
\label{dsdsdsdsdsv2}
	$\CGen_n(\NILS)$ is a \NIL{q} for each $n\geq 1$, by the previous  points  \ref{dsdsdsdsdsv1} and \ref{dsdsdsdsdsv4}.
\item
\label{dsdsdsdsdsv3}
Let $k\in \NN\cup\{\lip,\infty\}$ be given; and assume that $\mq$ is Mackey complete for $k\in \NN_{\geq 1}\cup\{\lip,\infty\}$, as well as integral complete for $k\equiv 0$. Then, for $n\in \{1,\dots,q\}$ and $\phi\in C^k([a,b],\mq)$ with $\im[\phi]\subseteq \CGen_n(\NILS)$, we have 
\begin{align*}
	\textstyle\int_s^t \phi(s)\:\dd s \in \CGen_n(\NILS)\qquad\quad\forall\: a\leq s<t\leq b.
\end{align*}
\vspace{-40pt}

\noindent	
\hspace*{\fill}\qed
\end{enumerate}}
\endgroup		  
\end{remark}

\subsubsection{Some Properties of Maps}
Let $F_1,\dots,F_n,F,E\in \HLCV$ be given.
\begin{lemma}
\label{alalskkskaskaskas}
Let $X$ be a topological space; and let $\Phi\colon X\times F_1\times{\dots}\times F_n\rightarrow F$ be continuous, such that $\Phi(x,\cdot)$ is 
	$n$-multilinear for each $x\in X$. Then, to each compact $\compacto\subseteq X$ and each $\qq\in \Sem{F}$, there exist $\qq_1\in \Sem{F_1},\dots,\qq_n\in \Sem{F_n}$ as well as $O\subseteq X$ open with $\compacto\subseteq O$, such that
	\begin{align*}
		(\qq\cp\Phi)(y,X_1,\dots,X_n) \leq \qq_1(X_1)\cdot {\dots}\cdot \qq_n(X_n)\qquad\quad\forall\: y\in O,\: X_1\in F_1,\dots,X_n\in F_n.
	\end{align*}  
\end{lemma}
\begin{proof}
Confer, e.g., Corollary 1 in \cite{RGM}.
\end{proof}

\begin{lemma}
\label{hjfdkjfdkjfd}
Let $V\subseteq F$ be open with $0\in V$. Let furthermore  
$\Psi\colon V\times E\rightarrow E$ 
be smooth with $\Psi(0,\cdot)=\id_{E}$, such that $\Psi(x,\cdot)$ is linear for each $x \in V$. Then, to each $\pp\in \Sem{E}$, there exist  $\qq\in \Sem{F}$ and $\ww\in \Sem{E}$ with
\begin{align*}
	\pp(\Psi(x,Y)-Y)\leq \qq(x)\cdot \ww(Y)\qquad\quad\forall\: x\in \B_{\qq,1}\subseteq V,\: Y\in E.
\end{align*}
\end{lemma}
\begin{proof}
Confer Appendix \ref{asassadsdsdsdsdsdsds}.	
\end{proof}
\begin{lemma}
\label{sdsdds}
Let $F_1,F_2\in \HLCV$, and $f\colon F_1\supseteq U\rightarrow F_2$ of class $C^{2}$. Assume that 
$\gamma\colon D\rightarrow F_1$ is continuous at $t\in D$, such that $\lim_{h\rightarrow 0} 1/h \cdot (\gamma(t+h)-\gamma(t))=:X\in \comp{F}_1$ exists. Then, we have
\begin{align*}
	\textstyle\lim_{h\rightarrow 0} 1/h\cdot (f(\gamma(t+h))-f(\gamma(t)))=\comp{\dd_{\gamma(t)}f}\he(X).
\end{align*}  
\end{lemma}
\begin{proof}
Confer, e.g., Lemma 7 in \cite{RGM}.
\end{proof}
\noindent
We close this subsection with the following convention concerning differentiable maps with values in Lie groups.
\begin{convention}
\label{lkcxlkcxlkxlkcxpoidsoids}
Let $F\in \HLCV$, $U\subseteq F$, $G$ a Lie group modeled over $E\in \HLCV$. A map $f\colon U\rightarrow G$ is said to be  
\begingroup
\setlength{\leftmargini}{12pt}
\begin{itemize}
\item
differentiable at $x\in U$ \defff there exists a chart $(\chart',\U')$ of $G$ with $f(x)\in \U'$, such that 
\begin{align}
\label{pofdpofdpofdfd}
	\textstyle (D^{\chart'}_v f)(x):=\lim_{h\rightarrow 0} 1/h\cdot ((\chart'\cp f)(x+h\cdot v)-(\chart'\cp f)(x))\in E\qquad\quad\forall\: v\in F
\end{align}
exists. Lemma \ref{sdsdds} applied to coordinate changes shows that \eqref{pofdpofdpofdfd} holds for one chart around $f(x)$ \deff it holds for each chart around $f(x)$, and that
\begin{align*}
	\dd_x f(v):= \big(\dd_{\chart'(f(x))}\chart'^{-1} \cp (D^{\chart'}_v f)\big)(x)\in T_{f(x)}G \qquad\quad\forall\: v\in F	
\end{align*}  
is independent of the explicit choice of $(\chart',U')$.
\item
differentiable \defff $f$ is differentiable at each $x\in U$. \hspace*{\fill}\qed
\end{itemize}
\endgroup
\end{convention}

\subsection{Power Series}
\label{ouwuiewuiewuoew}
In this subsection, we collect some statements concerning power series in (Banach) algebras that we shall need to work informally in Sect.\ \ref{sdsssdhsdhcxuzcxuzcxcx}. We set
\begin{align*}
	\mU_{\varepsilon}(z):=\{w\in \CC \: |\:  |w-z|<\varepsilon\}\qquad\quad\forall\: \varepsilon>0,\: z\in \CC, 
\end{align*}
and let $\korp\in \{\RR,\CC\}$ be fixed. 
\begingroup
\setlength{\leftmargini}{12pt}{
\begin{itemize}
\item
Let $(\Module,\Modmult,\unitM)$ be a unital $\korp$-algebra, and let $\Module_q:=\{\model\in \Module\:|\: \model^{q+1}=0\}$ for $q\geq 1$. Set furthermore $\model^0:=\unitM$ and $\model^1:=\model$ for each $\model\in \Module$.
\item
Let $(\Banach,\Banmult,\unit,\bannorm{\cdot})$ be a unital submultiplicative Banach algebra over $\korp$, and set $\banel^0:=\unit$ as well as $\banel^1:=\banel$ for each $\banel\in\Banach$.
\end{itemize}
}
\endgroup
\noindent 
Let $f\colon  \mU_{R}(0)\ni z\mapsto \sum_{n=0}^\infty a_n\cdot z^n\in \CC$ for $\{a_n\}_{n\in \NN}\subseteq \korp$ be a power series with radius of convergence $R >0$. We define
\begin{align*}
	\textstyle|f|_r:=\sum_{n=0}^\infty |a_n|\cdot r^n\in [0,\infty)\qquad\quad\forall\: 0\leq r<R,
\end{align*}
and set $\textstyle f_p\colon \mU_{R}(0)\ni z\mapsto \sum_{n=0}^p a_n\cdot z^n\in \CC$ for each $p\in \NN$. 
We furthermore let
\begin{align*}
	f(\model)&\textstyle:= \sum_{n=0}^q a_n\cdot \model^n\in \Module\qquad\quad\forall\: \model\in \Module_q\:\: \text{with}\:\: \hspace{0.8pt} q\geq 1\\[1pt]
	f(\banel)&\textstyle:=\sum_{n=0}^\infty a_n\cdot \banel^n \in \Banach\qquad\quad\hspace{0.5pt}\forall\: \banel\in \Banach\hspace{9.8pt} \text{with}\:\: \bannorm{\banel}<R.
\end{align*}
Let $g\colon  \mU_{S}(0)\ni z\mapsto \sum_{n=0}^\infty b_n\cdot z^n\in \CC$ be a power series with radius of convergence $S>0$. 
\begingroup
\setlength{\leftmargini}{12pt}{
\begin{itemize}
\item
Assume $S=R$. The Cauchy product formula yields
\begin{align*}
		\textstyle f \ast g\colon \mU_R\ni z\mapsto \sum_{n=0}^\infty \big(\sum_{\ell=0}^n a_\ell\cdot b_{n-\ell}\big) \cdot z^n=f(z)\cdot g(z)\in \CC.
\end{align*}
\begingroup
\setlength{\leftmarginii}{12pt}{
\begin{itemize}
\item[$\circ$]
For $\model\in \Module_q$ with $q\geq 1$, we evidently have
\begin{align*}
	\textstyle f(\model)\Modmult g(\model) = \sum_{n=0}^q \big(\sum_{\ell=0}^n a_\ell\cdot b_{n-\ell}\big) \cdot \model^n = (f\ast g)(\model)\in \Module. 
\end{align*}  
\item[$\circ$]
For $\banel\in \Banach$ with $\bannorm{\banel}< R$, we obtain
\begin{align*}
	\textstyle f(\banel)\Banmult g(\banel) = \sum_{n=0}^\infty \big(\sum_{\ell=0}^n a_\ell\cdot b_{n-\ell}\big) \cdot \banel^n = (f\ast g)(\banel)\in \Banach.
\end{align*}
(Apply, e.g., Exercise 3.1.3 in \cite{HIGN}, with $X,Y,Z\equiv \Banach$ and $\beta \equiv \Banmult$ there.)
\end{itemize}
}
\endgroup
\item
Assume $|g|_s=\sum_{n=0}^\infty |b_n|\cdot s^n < R$ for all $0\leq s< S$, thus $g(\mU_{S})\subseteq \mU_{R}(0)$. 
By analyticity, we have
\begin{align*}
	\textstyle f\diamond g\colon \mU_{S}\ni z\mapsto  \sum_{n=0}^\infty \underbracket{\textstyle\frac{(f\cp g)^{(n)}(0)}{n!}}_{=:\: c_n} \:\cdot\: z^n = f(g(z))\in \CC.
\end{align*}
The following statements are verified in Appendix \ref{asassadsdsdsdfdfdsdsdsds}:
\begingroup
\setlength{\leftmarginii}{12pt}{
\begin{itemize}
\item[$\circ$]
Assume $b_0=0$, and let $\model\in \Module_q$ with $q\geq 1$. We have $g(\model)\in \Module_q$ with
\begin{align}
\label{ljksdkljsdklsdnil}
	\textstyle f(g(\model))= \sum_{n=0}^q c_n \cdot \model^n = (f\diamond g)(\model)\in \Module.
\end{align}
\vspace{-15pt}
\item[$\circ$]
For $\banel\in \Banach$ with $\bannorm{\banel}< S$, we have 
\begin{align}
\label{ljksdkljsdklsd}
	\textstyle f(g(\banel))= \sum_{n=0}^\infty c_n \cdot \banel^n = (f\diamond g)(\banel)\in \Banach.
\end{align}
\end{itemize}
}
\endgroup
\end{itemize}
}
\endgroup
\noindent
\begin{example}
\label{iureiuriureiuiure}
Let $\korp=\RR$, and define the power series
\begin{align*}
	\qquad\quad f&\colon  \hspace{36pt}(-1,1)\ni \textstyle t\mapsto              	 \sum_{n=1}^\infty \frac{(-1)^{n-1}}{n}\cdot t^{n-1}\in \RR\hspace{49.5pt} \Big(\frac{\ln(1+t)}{t}\Big)\\
	 g&\colon  (-\ln(2),\ln(2))\ni    \textstyle t\mapsto    			   \sum_{n=1}^\infty \frac{1}{n!}\cdot t^n\in (-1,1)\hspace{59pt}\big(\e^t-1\big) \\
	 h&\colon  \hspace{38.2pt}(-1,1)\ni    \textstyle t\mapsto    \sum_{n=0}^\infty\frac{1}{(n+1)!}\cdot t^n \in \RR.\hspace{62.2pt}\Big(\frac{\e^t-1}{t}\Big)
\end{align*}
In the context of the above notations, we have $b_0=0$, $R=1$, $S=\ln(2)$, with $|g|_s<\e^{S}-1=R$ for $0\leq s<S$ as well as
\begin{align}
\label{uisiusuisuisuis}
	\textstyle ((f\diamond g) \ast h)(t) =f(g(t))\cdot h(t)=1\qquad\quad\forall\: t\in (-\ln(2),\ln(2)).
\end{align}
We consider the power series\footnote{More specifically, $\tilde{f}=p\ast f$ for $p\colon (-1,1)\ni t\mapsto (1+t)\in \RR$, and $\tilde{g}\colon (-\ln(2),\ln(2))\ni t\mapsto \sum_{n=1}^\infty \frac{1}{n!}\cdot (-t)^n \in (-1,1)$.}
\begin{align*}
	&\tilde{f}\colon\hspace{36pt} (-1,1)\ni t\mapsto (1+t)\cdot f(t) \in \RR\\
	&\tilde{g}\colon (-\ln(2),\ln(2))\ni t\mapsto g(-t)\hspace{31.5pt} \in (-1,1),
\end{align*}
and obtain for $t\in (-\ln(2),\ln(2))$ that
\begin{align}
\label{ljkdlkjfdlkjfd}
	\textstyle(\tilde{f}\diamond\tilde{g})(t)=\tilde{f}(\tilde{g}(t))=\tilde{f}(e^{-t}-1)= e^{-t}\cdot \frac{-t}{e^{-t}-1}=\frac{t}{e^t-1}=f(g(t))=(f\diamond g)(t).
\end{align}
Let now  $Z\in \mg$ be fixed, and assume that one of the following  situations hold:
\begingroup
\setlength{\leftmargini}{19pt}
{
\renewcommand{\theenumi}{\rm\Alph{enumi})} 
\renewcommand{\labelenumi}{\theenumi}
\begin{enumerate}
\item
\label{oidiudiudiudiu2}
Let $\NILS\subseteq \mg$ be a \NIL{q} for $q\geq 1$. We define $V:=\CGen_1(\NILS)$, and let $\End(V)$ denote the set of all linear maps $V\rightarrow V$. Then, $(\Module,\Modmult,\unitM)\equiv(\End(V),\cp,\id_{V})$ is a unital 
$\RR$-algebra, with $\com{Z}\in \Module_{q-1}$ for each $Z\in V$ by Remark \ref{dsdsdsdsdsv}.\ref{dsdsdsdsdsv2}. 
\item
\label{oidiudiudiudiu1}
$G$ is a Banach Lie group with 
\begin{align*}
	\norm{\bil{X}{Y}}\leq \norm{X}\cdot \norm{Y}
	\qquad\quad\forall\: X,Y\in \mg,
\end{align*}
and we have $\norm{Z}<\ln(2)$. We define $V:=\mg$, and let 
	$\End^{\mathrm{c}}(V)$ denote the set of all continuous linear maps $V\rightarrow V$. Then, 
	 $(\Banach,\Banmult,\unit,\bannorm{\cdot})\equiv (\End^{\mathrm{c}}(V),\cp,\id_{V},\norm{\cdot}_\op)$ is a unital submultiplicative Banach algebra over $\RR$, and we have $\com{Z}\in \End^{\mathrm{c}}(V)$ with $\norm{\com{Z}}_\op< \ln(2)$.   
\end{enumerate}
}
\endgroup
\noindent
In both situations \ref{oidiudiudiudiu2} and \ref{oidiudiudiudiu1}, the above discussions together with \eqref{uisiusuisuisuis} and \eqref{ljkdlkjfdlkjfd} show 
\begin{align*}
	f(g(\com{Z}))\cp h(\com{Z})&=((f\diamond g) \ast h)(\com{Z})=\id_V\\[2pt]
	\tilde{f}(\tilde{g}(\com{Z}))&=f(g(\com{Z}))
\end{align*}
\vspace{-20pt}

\noindent
for $Z\in V$. This can be rewritten as
\begin{align}
\label{lkjfdjlfdkjfdlka}
\begin{split}
	\textstyle\Psi\big(\sum_{n=0}^\infty \frac{1}{n!}\cdot\com{Z}^n\big)(\Phi(\com{Z})(Y))&\textstyle=Y
	\\[2pt]
	\textstyle\wt{\Psi}\big(\sum_{n=0}^\infty \frac{1}{n!}\cdot\com{-Z}^n\big)(Y)&\textstyle=\Psi\big(\sum_{n=0}^\infty \frac{1}{n!}\cdot\com{Z}^n\big)(Y) 
\end{split}
\end{align}
for each $Y\in V$, 
whereby for maps $\xi,\zeta\colon \mg\rightarrow \mg$ and $X\in \mg$ we set (convergence presumed)
\begin{align*}
	\textstyle \Psi(\xi)(X)&\textstyle:= \sum_{n=1}^\infty \frac{(-1)^{n-1}}{n}\cdot(\xi-\id_\mg)^{n-1}(X)\\
	\textstyle \wt{\Psi}(\xi)(X)&\textstyle:= \sum_{n=1}^\infty \frac{(-1)^{n-1}}{n}\cdot\big(\xi\cp (\xi-\id_\mg)^{n-1}\big)(X)\\[3pt]
	\textstyle \Phi(\zeta)(X)&\textstyle:= \sum_{n=0}^\infty\frac{1}{(n+1)!}\cdot \zeta^{n}(X).
\end{align*}
Equation \eqref{lkjfdjlfdkjfdlka} will be relevant for our discussions in Sect.\ \ref{sdsssdhsdhcxuzcxuzcxcx}.  
\hspace*{\fill}\qed
\end{example}

\subsection{The Evolution Map}
\label{dsdsdsdsdspopopo}
The subject of this section is the evolution map. We recall its elementary properties (confer also \cite{RGM}), as well as the differentiability results obtained in \cite{MDM}. We furthermore introduce the notion of weak $C^k$-regularity (confer Definition \ref{tztztggfgf}).

\subsubsection{Elementary Definitions}
The right logarithmic derivative is given by
\begin{align*}
	\Der\colon C^1(D,G)\rightarrow C^0(D,\mg),\qquad \mu\mapsto \dd_\mu\RT_{\mu^{-1}}(\dot \mu)
\end{align*}
for each interval $D\subseteq \RR$. Notably, for $\mu\in C^1(D,G)$, $g\in G$, an interval $D'\subseteq D$, and $\varrho\colon  D''\rightarrow D$ of class $C^1$ for $D''\subseteq \RR$ an interval, we have
\begin{align}
\label{fgfggf}
	\Der(\mu\cdot g)=\Der(\mu)\qquad\quad\: \Der(\mu|_{D'})=\Der(\mu)|_{D'}\qquad\quad\: \Der( \mu\cp\varrho)=\dot\varrho\cdot(\Der(\mu)\cp\varrho).
\end{align}
Moreover, for $\mu,\nu\in C^1(D,G)$, it follows from the product rule \eqref{LGPR} that
	\begin{align*}
		\Der(\mu\cdot \nu)= \Der(\mu)+\Ad_\mu(\Der(\nu))
	\end{align*}
holds.  
For $a<b$ and $k\in \NN\cup\{\lip,\infty,\const\}$, we define
\begin{align*}
\textstyle \DIDE_{[a,b]}:=\Der(C^1([a,b],G))\qquad\quad\text{as well as}\qquad\quad
\textstyle\DIDE_{[a,b]}^k:=\DIDE_{[a,b]}\cap C^k([a,b],\mg).
\end{align*}
Now, $\Der$ restricted to the set
\begin{align*}
	C_*^1([a,b],G):=\{\mu\in C^1([a,b],G)\:|\: \mu(a)=e\}
\end{align*}
is injective for $a<b$ (confer, e.g., Lemma 9 in \cite{RGM}). We thus obtain a map
\begin{align*}
	\textstyle\EV\colon \DIDE:=\bigcup_{a<b}\DIDE_{[a,b]}\rightarrow \bigcup_{a<b}C_*^1([a,b],G),
\end{align*}
if for $a<b$ we define
\begin{align*}
	\EV\colon \DIDE_{[a,b]}\rightarrow C_*^{1}([a,b],G),\qquad\Der(\mu)\mapsto \mu\cdot \mu(a)^{-1}.
\end{align*}
Notably, for $a<b$ and $k\in \NN\cup\{\lip,\infty,\const\}$, we have (confer, e.g., Lemma 10 in \cite{RGM})
\begin{align}
\label{kjdskjdsjkdskjdsjk}
	\EV|_{\DIDE_{[a,b]}^k}\colon \DIDE_{[a,b]}^k\rightarrow C^{k+1}([a,b],G).
\end{align}
Moreover, for $a<b$ and $\phi\in \DIDE_{[a,b]}^k$, we have 
\begin{align*}
	\phi|_{[a',b']}\in \DIDE_{[a',b']}^k\qquad\quad\forall\: a\leq a'<b'\leq b
\end{align*}
by the second equality in \eqref{fgfggf} as well as Lemma 10 in \cite{RGM}.
\begin{remark}
\label{iudsouidsuiodsiuduoisd}
Given $X\in \mg$, there exist $\varepsilon>0$ and $\phi_X\in \DIDE^\infty_{[0,\varepsilon]}$ with $\phi_X(0)=X$. 
In fact, fix $\varepsilon>0$ with $(-2\varepsilon,2\varepsilon)\cdot \dd_e\chart(X)\subseteq \V$, and define  
\begin{align*}
	\mu_X\colon (-\varepsilon,\varepsilon)\ni t\mapsto \chartinv(t\cdot \dd_e\chart(X))\in \U.
\end{align*}
Then, $\mu_X$ is of class $C^\infty$, and $\phi_X:=\Der(\mu_X|_{[0,\varepsilon]})$ has the desired properties. \hspace*{\fill}\qed
\end{remark}
\begin{example}[The Riemann Integral]
\label{hjfdhjhjfd}
Assume $(G,\cdot)\equiv (F,+)$ equals the additive group of some $F\in \HLCV$. We have $\Der\colon C^1([a,b],F)\ni\gamma\mapsto \dot\gamma\in C^0([a,b],F)$ for each $a<b$, hence 
\begin{align*}
	\textstyle\DIDE_{[a,b]}=\big\{\gamma\in C^0([a,b],\mg)\:\big|\: \int_a^t \gamma(s)\:\dd s \in \mg\:\text{ for each }\: t\in [a,b]\big\}
\end{align*}
as well as $\EV(\gamma)\colon [a,b]\ni t\mapsto \int_a^t \gamma(s)\:\dd s$ for each $\gamma \in \DIDE_{[a,b]}$.\hspace*{\fill}\qed
\end{example}

\subsubsection{The Product Integral}
\label{wqwqztwzuwqzuiusxy}
The product integral is defined by
\begin{align*}
	\textstyle\innt_s^t\phi:= \EV\big(\phi|_{[s,t]}\big)(t)\in G\qquad\quad \forall \:[s,t]\subseteq \dom[\phi],\: \phi\in \DIDE.
\end{align*}
We set $\innt\phi\equiv\innt_a^{b}\phi$ as well as $\innt_c^c\phi:= e$ for $a<b$, $\phi\in \DIDE_{[a,b]}$, $c\in [a,b]$, and define
\begin{align*}
\evol_{[a,b]}^k&\textstyle\equiv \innt\big|_{\DIDE_{[a,b]}^k}\qquad\quad\forall\:  k\in \NN\cup\{\lip,\infty,\const\}. 
\end{align*} 
We let $\evol_\kk\equiv \evol_{[0,1]}^k$ as well as $\DIDED_\kk\equiv \DIDE^k_{[0,1]}$ for each $k\in \NN\cup\{\lip,\infty,\const\}$.
We furthermore let
\begin{align*}
	\evol\equiv \evol_{\mathrm{0}}\colon \DIDED\equiv \DIDED_0\rightarrow G.
\end{align*}
The following elementary identities hold for $a<b$, confer   \cite{HGGG,MK} or Sect.\ 3.5.2 in \cite{RGM}: 
\vspace{2pt}
\begingroup
\setlength{\leftmargini}{19pt}
{
\renewcommand{\theenumi}{{(\alph{enumi})}} 
\renewcommand{\labelenumi}{\theenumi}
\begin{enumerate}
\item
\label{kdsasaasassaas}
For each $\phi,\psi\in \DIDE_{[a,b]}$, we have $\phi+\Ad_{\innt_a^\bullet\phi}(\psi)\in \DIDE_{[a,b]}$ with
\begin{align*}
	\textstyle\innt_a^t \phi \cdot \innt_a^t\psi=\innt_a^t \phi+\Ad_{\innt_a^\bullet\phi}(\psi)\qquad\quad  
	t\in [a,b].
\end{align*}
	\vspace{-18pt}
\item
\label{kdskdsdkdslkds}
For each $\phi,\psi\in \DIDE_{[a,b]}$, we have $\Ad_{[\innt_a^\bullet\phi]^{-1}}(\psi-\phi)\in \DIDE_{[a,b]}$ with
\begin{align*}
	\textstyle[\innt_a^t \phi]^{-1} [\innt_a^t\psi]=\innt_a^t\Ad_{[\innt_a^\bullet\phi]^{-1}}(\psi-\phi)\qquad\quad  
	t\in [a,b].
\end{align*}
	\vspace{-18pt}
\item
\label{oitoioiztoiztoiztoizt}
For each $\phi\in \DIDE_{[a,b]}$, we have $-\Ad_{[\innt_a^\bullet\phi]^{-1}}(\phi)\in \DIDE_{[a,b]}$ with
\begin{align*}
	\textstyle [\innt_a^t\phi]^{-1}=\innt_a^t -\Ad_{[\innt_a^\bullet\phi]^{-1}}(\phi)\qquad\quad  
	t\in [a,b].
\end{align*}
	\vspace{-18pt}
\item
\label{pogfpogf}
For $a=t_0<{\dots}<t_n=b$ and $\phi\in \DIDE_{[a,b]}$, we have 
	\begin{align*}
		\textstyle\innt_a^t\phi=\innt_{t_{p}}^t\! \phi\cdot \innt_{t_{p-1}}^{t_{p}} \!\phi \cdot {\dots} \cdot \innt_{a}^{t_1}\!\phi\qquad\quad\forall\:t\in (t_p,t_{p+1}],\: p=0,\dots,n-1.
	\end{align*}
		\vspace{-15pt}
\item
\label{subst}
For $\varrho\colon [a',b']\rightarrow [a,b]$ 
of class $C^1$ and $\phi\in \DIDE_{[a,b]}$, we have $\dot\varrho\cdot (\phi\cp\varrho)\in \DIDE_{[a',b']}$ with
\begin{align*}
	 \textstyle\innt_a^{\varrho(\bullet)}\phi=[\innt_{a'}^\bullet\dot\varrho\cdot (\phi\cp\varrho)]\cdot [\innt_a^{\varrho(a')}\phi].
\end{align*} 
\item
\label{homtausch}
For each homomorphism $\Psi\colon G\rightarrow H$ between Lie groups $G$ and $H$ that is of class $C^1$, we have
\begin{align*}
	\textstyle\Psi\cp \innt_a^\bullet \phi = \innt_a^\bullet \dd_e\Psi\cp\phi\qquad\quad\forall\:\phi\in \DIDE_{[a,b]}.
\end{align*}
\end{enumerate}}
\endgroup
\begin{remark}
\label{dsdsdsdsjkjkkjkjkjjkkjkj}
Let $k\in \NN\cup\{\lip,\infty\}$ and $a<b$ be given. We have by  \ref{kdsasaasassaas}, \ref{oitoioiztoiztoiztoizt}, and Lemma 13 in \cite{RGM} (confer also Lemma \ref{Adlip}) that 
\begin{align*}
	\psi^{-1}&:= -\Ad_{[\innt_a^\bullet\psi]^{-1}}(\psi)\in \DIDE^k_{[a,b]}\qquad\quad		
	\text{with}\qquad\quad \textstyle\innt_a^\bullet \psi^{-1}\hspace{5.5pt}=[\innt_a^\bullet\psi]^{-1}\\ 
	\phi\star\psi&:= \phi + \Ad_{\innt_a^\bullet\phi}(\psi)\hspace{4.1pt}\in \DIDE^k_{[a,b]}\qquad\quad		
	\text{with}\qquad\quad \textstyle \innt_a^\bullet \phi\star\psi=\innt_a^\bullet\phi\cdot \innt_a^\bullet\psi
\end{align*}  
holds for each $\phi,\psi\in \DIDE^k_{[a,b]}$. 
It is then not hard to see that $(\DIDE^k_{[a,b]},\star,\cdot^{-1},0)$ is a group:
\begingroup
\setlength{\leftmargini}{12pt}{
\begin{itemize}
\item
	We have $(\psi^{-1})^{-1}=\psi$ for each $\psi\in \DIDE^k_{[a,b]}$. In fact, applying \ref{oitoioiztoiztoiztoizt} twice, we obtain 
	\begin{align*}
		\textstyle\innt_a^\bullet (\psi^{-1})^{-1}=[\innt_a^\bullet \psi^{-1}]^{-1}=[[\innt_a^\bullet \psi]^{-1}]^{-1}=\innt_a^\bullet \psi. 
	\end{align*}			
	The claim is thus clear from injectivity of $\Der|_{C_*^1([a,b],G)}$.   
\item 
	We have $\psi\star \psi^{-1} =0=\psi^{-1}\star \psi$ for each $\psi\in \DIDE^k_{[a,b]}$. In fact, it is clear that
	\begin{align*}
		\psi\star \psi^{-1} =0\qquad\quad\forall\: \psi\in \DIDE^k_{[a,b]}.
	\end{align*}		
	Then, we obtain from the previous point that
\begin{align*}
	\psi^{-1}\star \psi=\psi^{-1}\star (\psi^{-1})^{-1}=0\qquad\quad\forall\: \psi\in \DIDE^k_{[a,b]}.
\end{align*}	
\vspace{-20pt}
\item
	We have $\phi\star(\psi\star\chi) = (\phi\star\psi)\star \chi$ for all $\phi,\psi,\chi\in \DIDE^k_{[a,b]}$. In fact, we obtain from \ref{kdsasaasassaas} that
	\begin{align*}
		\textstyle\innt_a^\bullet \phi\star(\psi\star\chi)= \innt_a^\bullet \phi \cdot \innt_a^\bullet (\psi\star\chi) = \innt_a^\bullet \phi \cdot \innt_a^\bullet \psi\cdot  \innt_a^\bullet\chi= \innt_a^\bullet (\phi\star \psi) \cdot \innt_a^\bullet \chi = \innt_a^\bullet (\phi\star\psi)\star\chi.
	\end{align*}
	The claim is thus clear from injectivity of $\Der|_{C_*^1([a,b],G)}$.
\end{itemize}}
\endgroup
\noindent 
We will reconsider this group structure in Sect.\ \ref{kjdkjsjksddxccxocoioicxoiiooicxiocx} and Sect.\  \ref{jkdkjdjkddsdsddds}. 
\hspace*{\fill}\qed
\end{remark}
\begin{example}[The Riemann Integral]
\label{mmbvcvccv}
Assume we are in the situation of Example \ref{hjfdhjhjfd}. Then, 
\begin{align*}
	\textstyle\innt_a^t \phi=\int_a^t \phi(s)\: \dd s\qquad\quad\forall\:  
\phi\in \DIDE_{[a,b]},\:t\in [a,b]
\end{align*} 
holds; and, the identities \ref{kdsasaasassaas}, \ref{pogfpogf}, \ref{subst}, and \ref{homtausch} encode (in the given order) the additivity of the Riemann integral,  \eqref{isdsdoisdiosd0}, \eqref{substitRI}, and \eqref{pofdpofdpofdsddsdsfd} (for if in \ref{homtausch}, $(H,\cdot)\equiv (\tilde{F},+)$ is the additive group of some further $\tilde{F}\in \HLCV$), respectively. 
\hspace*{\fill}\qed
\end{example}
\begin{example}
\label{fdpofdopdpof}
For $a<b$, we define
\begin{align}
\label{iudsoiudsiudsiudsods}
\begin{split}
	\inversee\colon C^0([a,b],\mg)&\rightarrow C^0([a,b],\mg)\\
	\phi&\mapsto [t\mapsto -\phi(a+b-t)].
\end{split}
\end{align}
Let $a<b$ be fixed. Then, $\inversee|_{C^0([a,b],\mg)}$ is linear. 
Moreover, for $\phi\in C^0([a,b],\mg)$ and  
$\varrho\colon [a,b]\ni t\mapsto a +b -t\in [a,b]$, we have $\inverse{\phi}=\dot\varrho\cdot \phi\cp\varrho$. Then, \ref{subst} shows that $\inverse{\phi}\in \DIDE^k_{[a,b]}$ holds for each $\phi\in 
\DIDE^k_{[a,b]}$ and $k\in \NN\cup\{\lip,\infty\}$, with
\begin{align}
\label{pfifpfpofdpofd}
	\textstyle e=\innt_a^{\varrho(b)}\phi\stackrel{\emph{\ref{subst}}}{=}[\innt_{a}^{b} \inverse{\phi}] \cdot [\innt_a^{b}\phi ]\qquad\quad\text{hence}\qquad\quad [\innt_a^b \phi]^{-1}=\innt_a^b \inverse{\phi}.
\end{align}
For instance, in the situation of Example \ref{mmbvcvccv}, the right side of \eqref{pfifpfpofdpofd} reads 
\begin{align*}
	\textstyle -\int_a^{b} \phi(s)\:\dd s = \int_a^{b} -\phi(a+b-s)\:\dd s,
\end{align*}
which is in line with the second point in Remark \ref{seqcomp}. 
The relation $[\innt_a^{b} \phi]^{-1}=\innt_a^{b} \inverse{\phi}$  
will be useful for our argumentation in Sect.\ \ref{dlkjlkjfdlkjfddofdoifdoi}.
\hspace*{\fill}\qed
\end{example}
\begin{example}[The Lie bracket and Homomorphisms]
\label{fdpofsddopdpof}
Assume that we are in the situation of \ref{homtausch}, and let $(\mh,\bil{\cdot}{\cdot}_\mh)$ denote the Lie algebra of $H$. Then, we have
\begin{align}
\label{dsdssddsdsdsdxcxcxgfggfgfs}
	\textstyle\dd_e\Psi(\bil{X}{Y})=\bil{\dd_e\Psi(X)}{\dd_e\Psi(Y)}_\mh\qquad\quad\forall\: X,Y\in \mg.
\end{align}
\begin{proof}[Proof of Equation \eqref{dsdssddsdsdsdxcxcxgfggfgfs}]
Let $X,Y\in \mg$ be fixed, and choose $\phi_{X}\colon [0,\varepsilon_1]\rightarrow\mg$ as well as $\phi_{Y}\colon [0,\varepsilon_2]\rightarrow\mg$ as in Remark \ref{iudsouidsuiodsiuduoisd}. We obtain
\begin{align*}
	\textstyle\dd_e\Psi(\bil{X}{Y})&\textstyle=\frac{\dd}{\dd x}\big|_{x=0}\: \dd_e\Psi(\Ad_{\innt_0^x \phi_{X}}(Y))\\
	&\textstyle=\frac{\dd}{\dd x}\big|_{x=0}\:\frac{\dd}{\dd y}\big|_{y=0}\: \Psi(\conj_{\innt_0^x \phi_{X}}(\innt_0^y \phi_{Y})) \\
	&\textstyle=\frac{\dd}{\dd x}\big|_{x=0}\:\frac{\dd}{\dd y}\big|_{y=0}\: \conj_{\Psi(\innt_0^x \phi_{X})}(\Psi(\innt_0^y \phi_{Y})) 
	\\
	&\textstyle=\frac{\dd}{\dd x}\big|_{x=0}\:\frac{\dd}{\dd y}\big|_{y=0}\: \conj_{\innt_0^x \dd_e\Psi(\phi_{X})}(\innt_0^y \dd_e\Psi(\phi_{Y})) 
	\\
	&\textstyle=\frac{\dd}{\dd x}\big|_{x=0}\:\Ad_{\innt_0^x \dd_e\Psi(\phi_{X})}(\dd_e\Psi(Y)) 
	\\
	&\textstyle=\bil{\dd_e\Psi(X)}{\dd_e\Psi(Y)}_\mh,
\end{align*}
which shows \eqref{dsdssddsdsdsdxcxcxgfggfgfs}.
\end{proof}
\noindent
For instance, let $\Psi\equiv \conj_g\colon G\rightarrow H= G$ with $g\in G$. We obtain
\begin{align}
\label{dsdsdsdsdsa}
	\Ad_g(\bil{X}{Y})=\bil{\Ad_g(X)}{\Ad_g(Y)}\qquad\quad\forall\: g\in G,\: X,Y\in \mg.
\end{align}
Then, given $X,Y,Z\in \mg$, fix $\mu\colon (-\varepsilon,\varepsilon)\rightarrow G$ ($\varepsilon>0$) of class $C^1$, with $\mu(0)=e$ and $\dot\mu(0)=Z$. We obtain from \eqref{dsdsdsdsdsa} and the parts \ref{linear}, \ref{chainrule}, \ref{productrule} of Proposition \ref{iuiuiuiuuzuzuztztttrtrtr} that
\begin{align*}
	\textstyle \bil{Z}{\bil{X}{Y}}&\textstyle= \frac{\dd}{\dd h}\big|_{h=0} \Ad_{\mu(h)}(\bil{X}{Y})
	\textstyle= \frac{\dd}{\dd h}\big|_{h=0} [\Ad_{\mu(h)}(X),\Ad_{\mu(h)}(Y)]
	\textstyle = \bil{\bil{Z}{X}}{Y}+ \bil{X}{\bil{Z}{Y}}
\end{align*}
holds, which is the Jacobi identity \eqref{jacobi}. 
\hspace*{\fill}\qed
\end{example}
 
\subsubsection{Weak Regularity}
\label{dsddsdsdsdsds}
In this section, we recall certain differentiation results from \cite{MDM}, and introduce the notion of weak $C^k$-regularity for $k\in \NN\cup\{\lip,\infty\}$ (cf.\ Definition \ref{tztztggfgf}). 
We say that $G$ is $C^k$-semiregular \cite{HGGG,RGM,MDM} for $k\in \NN\cup\{\lip,\infty,\const\}$ \defff $\DIDED_\kk=C^k([0,1],\mg)$ holds.  
\begin{remark}
\label{sdffdffdfd}
	It follows from \ref{subst} when applied to affine transformations that $G$ is $C^k$-semiregular \deff $\DIDE^k_{[a,b]}=C^k([a,b],\mg)$ holds for all $a<b$ (confer, e.g., Lemma 12 in \cite{RGM}).\hspace*{\fill}\qed
\end{remark}
\noindent
We write $\limin\mu_n=\mu$ for $\{\mu_n\}_{n\in \NN}\subseteq C^0([a,b],G)$ and $\mu\in C^0([a,b],G)$ \defff $\{\mu_n\}_{n\in \NN}$ converges uniformly to $\mu$, i.e., \defff to each open $U\subseteq G$ with $e\in U$, there exists some $N_U\in \NN$ with 
\begin{align*}
	\mu_n(t)\in \mu(t)\cdot U\qquad\quad\forall\: t\in [a,b]. 
\end{align*}
We recall \eqref{iudsiudsudspodsisdpoids}, and say that $G$ is Mackey k-continuous for $k\in \NN\cup\{\lip,\infty\}$ \defff 
\begin{align*}
\textstyle	\DIDED_\kk\supseteq\{\phi_n\}_{n\in \NN}\mackarr{\kk}\phi\in \DIDED_\kk\qquad\quad\Longrightarrow\qquad\quad \limin \innt_0^\bullet \phi_n = \innt_0^\bullet \phi.
\end{align*}
Lemma 13 in \cite{MDM} shows:
\begin{customlem}{A}\label{fdhjfdkjj}
$G$ is Mackey k-continuous for $k\in \NN\cup\{\lip,\infty\}$ if and only if for each $a<b$ the following implication holds:
\begin{align*}
	\textstyle\DIDE^k_{[a,b]}\supseteq \{\phi_n\}_{n\in \NN}\mackarr{\kk} \phi\in \DIDE^k_{[a,b]} \qquad\quad\Longrightarrow\qquad\quad \limin\innt_a^\bullet \phi_n=\innt_a^\bullet \phi.\hspace{22pt}
\end{align*}
\end{customlem}
\noindent
Theorem 1 in \cite{MDM} states: 
\begin{customthm}{B}\label{weakdiffdf}
If $G$ is $C^k$-semiregular for $k\in \NN\cup\{\lip,\infty\}$, then $G$ is Mackey k-continuous. 
\end{customthm}
\noindent
Theorem 3 in \cite{MDM} (in particular) states:
\begin{customthm}{C}
\label{ofdpofdpofdpofdpofd}
Assume that $G$ is Mackey k-continuous for $k\in \NN\cup\{\lip,\infty\}$. Let $\Phi\colon I\times [a,b]\rightarrow \mg$ ($I\subseteq \RR$ open) be given with $\Phi(z,\cdot)\in \DIDE^k_{[a,b]}$ for each $z\in I$.  
Then, for $x\in I$ we have
\begin{align*}
	\textstyle\frac{\dd}{\dd h}\big|_{h=0}\: \chart\big([\innt_a^{b} \Phi(x,\cdot)]^{-1}[\innt_a^{b}\Phi(x+h,\cdot)]\big)=\textstyle\int_a^{b} \big(\dd_e\chart\cp \Ad_{[\innt_a^s\Phi(x,\cdot)]^{-1}}\big)(\partial_z\Phi(x,s))\:\dd s \in \comp{E},
\end{align*}
provided that the following conditions hold:
\begingroup
\setlength{\leftmargini}{17pt}{
\renewcommand{\theenumi}{{\roman{enumi}})} 
\renewcommand{\labelenumi}{\theenumi}
\begin{enumerate}
\item
\label{saasaassasasa2}
We have $(\partial_1 \Phi)(x,\cdot)\in C^k([a,b],\mg)$.
\item
\label{saasaassasasa1}
To $\pp\in \Sem{E}$ and $\dind\llleq k$, there exists $L_{\pp,\dind}\geq 0$ as well as $I_{\pp,\dind}\subseteq I$ open with $x\in I_{\pp,\dind}$, such that
\begin{align*}
	\textstyle\pp^\dind_\infty(\Phi(x+h,\cdot)-\Phi(x,\cdot))\leq |h|\cdot L_{\pp,\dind}\qquad\quad \forall\: h\in \RR_{\neq 0}\:\text{ with }\: x+h\in I_{\pp,\dind}.
\end{align*}
\vspace{-22pt}
\end{enumerate}}
\endgroup
\noindent
In particular, we have (recall Convention \ref{lkcxlkcxlkxlkcxpoidsoids})
\begin{align*}
	\textstyle\frac{\dd}{\dd h}\big|_{h=0} \he \innt_a^{b}\Phi(x+h,\cdot)=\textstyle \dd_e\LT_{\innt_a^{b} \Phi(x,\cdot)}\big(\int_a^{b} \Ad_{[\innt_a^s\Phi(x,\cdot)]^{-1}}(\partial_1\Phi(x,s))\:\dd s\he\big)
\end{align*}
\deff the Riemann integral on the right side exists in $\mg$.  
\end{customthm}
\noindent
Recall from the end of Sect.\ \ref{opsdpods} that the last condition in Theorem \ref{ofdpofdpofdpofdpofd} concerning the Riemann integral is always fulfilled 
\begingroup
\setlength{\leftmargini}{12pt}
\begin{itemize}
\item
for $k\in \NN_{\geq 1}\cup\{\lip,\infty\}$ if $\mg$ is Mackey complete,
\item
for $k=0$ if $\mg$ integral complete for $k=0$. 
\end{itemize}
\endgroup
\noindent 
In particular, this implies the following statement (cf.\ Theorem 2 in \cite{MDM} and Corollary 3 in \cite{MDM}):
\begin{custompr}{D}\label{ableiti}
Assume that $G$ is $C^k$-semiregular for $k\in \NN\cup\{\lip, \infty\}$. Then, $\evol_\kk$ is differentiable \deff $\mg$ is Mackey complete for $k\in \NN_{\geq 1}\cup\{\lip,\infty\}$ as well as  integral complete for $k=0$.  In this case, $\evol^k_{[a,b]}$ is differentiable for each $a<b$, with 
\begin{align*}
	\textstyle\dd_\phi\he \evol^k_{[a,b]}(\psi)=\dd_e\LT_{\innt \phi}\big(\int \Ad_{[\innt_a^s \phi]^{-1}}(\psi(s))\:\dd s \big)\qquad\quad \forall\: \phi,\psi\in C^k([a,b],\mg).
\end{align*} 
Moreover, for  $a<b$ and $\phi,\psi\in C^k([a,b],\mg)$, we have	$\textstyle C^1(\RR,G)\ni \mu\colon \RR\ni t\mapsto \innt \phi+ t\cdot \psi\in G$.
\end{custompr}
\begin{proof}
Clear from Theorem 2 in \cite{MDM} and Corollary 3 in \cite{MDM}.
\end{proof}
\noindent
Proposition \ref{ableiti} motivates the following definition:\footnote{According to Proposition \ref{ableiti}, Definition \ref{tztztggfgf} is equivalent to the definition given in Sect.\ \ref{nmdsnmdsnmsdnmdsnmdsnmdsnmnmds}.}
\begin{definition}
\label{tztztggfgf}
$G$ is said to be weakly $C^k$-regular for $k\in \NN\cup\{\lip,\infty\}$ \defff $G$ is $C^k$-semiregular, and $\mg$ is Mackey complete for $k\in \NN_{\geq 1}\cup\{\lip,\infty\}$ as well as integral complete for $k=0$. \hspace*{\fill}\qed
\end{definition}
\begin{remark}
	Apart from $C^k$-semiregularity, the (standard) notion of 
 $C^k$-regularity involves smoothness (and continuity) of the evolution map w.r.t.\ the $C^k$-topology. 
These additional assumptions, however, are unnecessarily strong for our purposes. This is because, due to the results stated above, the usual differentiability properties of the evolution map are already available in the weakly $C^k$-regular context.
\hspace*{\fill}\qed
\end{remark}
\begin{remark}
\label{kjfdkjdkjfdkj}
	Theorem \ref{ofdpofdpofdpofdpofd} will in particular be applied to the situation in Example \ref{hjfdhjhjfd}, i.e., where $(G,\cdot)\equiv (F,+)$ equals the additive group of some given $F\in \HLCV$. For this  observe that $\innt=\int$ is $C^0$-continuous, hence Mackey {\rm 0}-continuous. Moreover, $G\equiv F$ is
\vspace{-3pt}
\begingroup
\setlength{\leftmargini}{12pt}{
\begin{itemize}
\item
	\hspace{5.2pt}$C^0$-semiregular if $F$ is integral complete,
\item
	$C^\lip$-semiregular if $F$ is Mackey complete.\hspace*{\fill}\qed
\end{itemize}
}
\endgroup
\end{remark}
\noindent
For instance, we obtain the following statements that we shall need in Sect.\ \ref{kjdskjsdkjdskjs}.
\begin{corollary}
\label{lkjjslkjdlkjdskjldskjs}
Let $F\in \HLCV$, $k\in \NN$, $a<b$, $I\subseteq \RR$ an open interval, and $\Theta^0\colon I\times [a,b] \rightarrow F$ a map with $\asymb_x:=\Theta^0(\cdot,x)\in C^k(I,F)$ for each $x\in [a,b]$. Assume that for $0\leq \ell\leq k$ the map
\begin{align*}
	\textstyle\Theta^\ell\colon I\times [a,b]\ni (t,x)\mapsto \asymb_x^{(\ell)}(t)\in F
\end{align*}
is continuous, such that
\begin{align*}
	\textstyle\bg^\ell\colon I\ni t \mapsto \int_a^b \Theta^\ell(t,x) \:\dd x\in F
\end{align*} 
is defined. Then, $\bg:=\bg^0\in C^k(I,F)$ holds with 
$\bg^{(\ell)}=\bg^\ell$ for $0\leq \ell\leq k$.
\end{corollary}
\begin{proof}
For each $0\leq \ell\leq k$, we have by \eqref{ffdlkfdlkfd}, compactness of $[a,b]$, and continuity of $\Theta^\ell$ that $\bg^\ell\in C^0(I,F)$ holds. Let now $0\leq \ell < k$ be given, and set $\Phi\equiv \Theta^\ell$. We observe the following:
\vspace{-2pt}  
\begingroup
\setlength{\leftmargini}{12pt}{
\begin{itemize}
\item
	For  $t\in I$, we have\hspace{10.5pt} $\Phi(t,\cdot)=\Theta^\ell(t,\cdot)\in C^0([a,b],F)$.
\item
	For $t\in I$, we have $\partial_1\Phi(t,\cdot)\in C^0([a,b],F)$, as 
	\begin{align*}
		\partial_1\Phi(t,x)=\partial_1\Theta^\ell(t,x)=\asymb_x^{(\ell+1)}(t)=\Theta^{\ell+1}(t,x)\qquad\quad\forall\: t\in I,\: x\in [a,b].
	\end{align*}
	\vspace{-17pt}
\item
	For $\qq\in \Sem{F}$, $x\in [a,b]$, $t\in I$, and $\tau>0$ with $t+[-\tau,\tau]\subseteq I$, we have by \eqref{isdsdoisdiosd1}
	\begin{align*}
		\textstyle 1/|h|\cdot \qq(\Phi(t+h,x)&-\Phi(t,x))\\
		&\textstyle= 1/|h|\cdot \qq\big(\alpha^{(\ell)}_x(t+h)-\asymb_x^{(\ell)}(t)\big)\\
		& \leq \sup\big\{-\tau\leq s\leq \tau\:|\: \qq\big(\asymb_x^{(\ell+1)}(t+s)\big)\big\}\\
		& \leq \sup\big\{-\tau\leq s\leq \tau,\:a\leq y\leq b\:|\: \qq\big(\Theta^{\ell+1}(t+s,y)\big)\big\}
	\end{align*}
	for each $-\tau\leq h\leq \tau$.		 
\end{itemize}
}
\endgroup
\noindent
Theorem \ref{ofdpofdpofdpofdpofd} (Remark \ref{kjfdkjdkjfdkj}) shows that for each $t\in I$, we have 
\begin{align*}
	\textstyle \frac{\dd}{\dd h}\big|_{h=0} \he\bg^\ell(t+h)
	&\textstyle=\frac{\dd}{\dd h}\big|_{h=0} \he \int_a^b \Theta^\ell(t+h,x)\:\dd x\\
	&\textstyle=\frac{\dd}{\dd h}\big|_{h=0} \he \int_a^b \Phi(t+h,x)\:\dd x\\
	&\textstyle=\int_a^b\partial_1 \Phi(t,x)\:\dd x\\
	&\textstyle=\int_a^b\partial_1 \Theta^\ell(t,x)\:\dd x\\
	&\textstyle= \int_a^b \asymb_x^{(\ell+1)}(t)\:\dd x\\
	&\textstyle= \int_a^b \Theta^{\ell+1}(t,x)\:\dd x\\
	&\textstyle= \bg^{\ell+1}(t).
\end{align*}
Since this holds for each $0\leq \ell<k$, the claim follows from $\bg^{(0)}=\bg=\bg^{0}$ by induction.
\end{proof}

\begin{lemma}
\label{lkjfdlkjfdkjf}
Let $F\in \HLCV$ be Mackey complete, $a<b$, $I\subseteq \RR$ an open interval,  $\Theta^0\colon I\times [a,b] \rightarrow F$ continuous, and $\Omega\colon F\times F\rightarrow F$ smooth. Assume that the following two conditions are fulfilled:
\begingroup
\setlength{\leftmargini}{20pt}
{
\renewcommand{\theenumi}{{(\alph{enumi})}} 
\renewcommand{\labelenumi}{\theenumi}
\begin{enumerate}
\item
\label{kjdkjkjdskjdskj1}
	\hspace{29.9pt}$\Theta^0(t,\cdot)\hspace{2.5pt}\in C^1([a,b],F)$ holds for each $t\in I$.
\item  
\label{kjdkjkjdskjdskj2}
$\asymb_x:=\Theta^0(\cdot,x)\in C^1(I,F)$\hspace{16pt} holds for each $x\in [a,b]$, with
\begin{align*}
	\textstyle\dot\asymb_x(t)=\Omega\big(\int_x^b\asymb_y(t)\:\dd y,\asymb_x(t)\big)\qquad\quad\forall\: t\in I.
\end{align*}
\vspace{-20pt}
\end{enumerate}
}
\endgroup
\noindent
Then, $\bg[z]\colon I\ni t\mapsto \int_z^b \Theta^0(t,x)\:\dd x\in F$ is smooth for each $z\in [a,b]$. 
\end{lemma}
\begin{proof}
It suffices to prove the following statement.
\begin{statement}
\label{iufdiufdiufdfd}
Let $k\in \NN$ be given. Then, $\asymb_x\in C^k(I,F)$ holds for each $x\in [a,b]$. Moreover,  
\begin{align*}
	\textstyle\Theta^\ell\colon I\times [a,b]\ni (t,x)\mapsto \asymb_x^{(\ell)}(t)\in F
\end{align*}
is continuous for $0\leq \ell\leq k$, with $\Theta^\ell(t,\cdot)\in C^1([a,b],F)$ for each $t\in I$.
\end{statement}
\noindent
In fact, let $k\in \NN$ be given.  
Since $F$ is Mackey complete, it follows from Statement \ref{iufdiufdiufdfd} that 
\begin{align*}
	\textstyle\bg[z]^\ell\colon I\ni t \mapsto \int_z^b \Theta^\ell(t,x)\:\dd x\in F
\end{align*}
 exists for each $z\in [a,b]$ and $0\leq \ell\leq k$. Moreover, Corollary \ref{lkjjslkjdlkjdskjldskjs} shows that $\bg[z]\in C^k(I,F)$ holds for each $z\in [a,b]$, with
\begin{align}
\label{hdsskjdsjkjdskjskjdsds}
	\textstyle\bg[z]^{(\ell)}= \bg[z]^\ell=\int_z^b \Theta^\ell(\cdot,x) \:\dd x\qquad\quad\forall\: 0\leq \ell\leq k.
\end{align}
Since $k\in \NN$ was arbitrary, Lemma \ref{lkjfdlkjfdkjf} follows.
 
It thus remains to prove Statement \ref{iufdiufdiufdfd}: 
\begin{proof}[Proof of Statement \ref{iufdiufdiufdfd}]
\renewcommand\qedsymbol{\scalebox{0.9}{$\blacksquare$}}
We first discuss the cases $k\in \{0,1\}$, and then argue by induction: 
\begingroup
\setlength{\leftmargini}{12pt}{
\begin{itemize}
\item
	$k=0$: Clear from the assumptions.
\item
	$k=1$: By \ref{kjdkjkjdskjdskj2}, we have $\asymb_x\in C^1(I,F)$ for each $x\in [a,b]$. Since $\Theta^0$ is continuous, the map 
\begin{align*}
	\textstyle\Gamma^0\colon I\times [a,b]\ni (t,z) \mapsto \int_z^b \Theta^0(t,x)\:\dd x\in F
\end{align*}
is continuous. It is clear that $\Gamma^0(t,\cdot)\in C^1([a,b],F)$ holds for each $t\in I$; and we have $\Theta^0(t,\cdot)\in C^1([a,b],F)$ for each $t\in I$ by \ref{kjdkjkjdskjdskj1}. Now,  \ref{kjdkjkjdskjdskj2} yields
\begin{align*}
	\textstyle\Theta^1(t,x)=\dot\asymb_x(t)=\Omega(\Gamma^0(t,x),\Theta^0(t,x))\qquad\quad\forall\: t\in I,\: x\in [a,b].
\end{align*}  
Since $\Omega$ is smooth, $\Theta^1$ is continuous with 
 $\Theta^1(t,\cdot)\in C^1([a,b],F)$ for each $t\in I$. 
\end{itemize}
}
\endgroup
\noindent
Assume now that Statement \ref{iufdiufdiufdfd} holds for some $k\geq 1$. We observe the following:
\begingroup
\setlength{\leftmargini}{19pt}{
\renewcommand{\theenumi}{{\roman{enumi}})} 
\renewcommand{\labelenumi}{\theenumi}
\begin{enumerate}
\item
\label{uidsiudsiu1}
	$\Theta^\ell$ is continuous for $0\leq \ell\leq k$, hence  
\begin{align*}
	\textstyle\Gamma^\ell\colon I\times [a,b]\ni (t,z) \mapsto \int_z^b \Theta^\ell(t,x)\:\dd x\in F
\end{align*}  
is continuous for $0\leq \ell\leq k$.
\item
\label{uidsiudsiu2}
	For $z\in [a,b]$, we have $\asymb_z,\bg[z]\in C^k(I,F)$, with
\begin{align}
\label{lkdslkdsklklds}
	\textstyle\asymb_z^{(\ell)}=\Theta^\ell(\cdot,z)\qquad\text{and}\qquad
	\bg[z]^{(\ell)}\textstyle\stackrel{\eqref{hdsskjdsjkjdskjskjdsds}}{=} \Gamma^\ell(\cdot,z)\qquad\text{for}\qquad 0\leq \ell\leq k.
\end{align}
Then, \ref{uidsiudsiu1} implies that the following maps are continuous:
\begin{align}
\label{ouidfoudfifdfdfdfdfd}
\begin{split}
	I\times [a,b]\ni (t,z)&\mapsto \hspace{5.2pt}\asymb^{(\ell)}_z(t)\in F\\
	 I\times [a,b]\ni (t,z)&\mapsto \bg[z]^\ell(t)\in F
\end{split}
\end{align}
for $0\leq \ell\leq k$
\item
\label{uidsiudsiu3}
	For  $t\in I$ and $0\leq \ell\leq k$, the maps 
	\begin{align}
	\label{iudiudsiudsoiosidds}
	\begin{split}
		[a,b]\ni z&\mapsto \hspace{11.8pt}\alpha_z^{(\ell)}(t)\in F\\
		[a,b]\ni z&\mapsto \bg[z]^{(\ell)}(t)\in F
	\end{split}
	\end{align} 
	are of class $C^1$ (use \eqref{lkdslkdsklklds}).
\end{enumerate}}
\endgroup
\noindent
Then, by \ref{kjdkjkjdskjdskj2} we have
\begin{align}
\label{oisdoisdoidsxccxcnnnbb}
	\dot\asymb_x(t)=\Omega(\bg[x](t),\asymb_x(t))\qquad\quad\forall\: t\in I,\: x\in [a,b].
\end{align}
Together with \ref{uidsiudsiu2}, this shows $\asymb_x\in C^{k+1}(I,F)$ for all $x\in [a,b]$. 
Moreover, set $\gamma[x]_1:=\alpha_x$ and $\gamma[x]_2:=\bg[x]$ for each $x\in [a,b]$. It follows from \eqref{oisdoisdoidsxccxcnnnbb} that\footnote{The statement is obtained by a straightforward induction involving the differentiation rules in Proposition \ref{iuiuiuiuuzuzuztztttrtrtr}, confer also the proof of Lemma 4 in \cite{RGM}.} 
\begin{align*}
	\Theta^{k+1}\colon I\times [a,b]\ni (t,x)=\asymb_x^{(k+1)}(t)\in F 
\end{align*} 
is a sum of maps of the form 
\begin{align*}
	I\times [a,b]\ni (t,x)\mapsto \Psi\big(\gamma[x]^{(\ell_1)}_{i_1}(t),\dots,\gamma[x]_{i_m}^{(\ell_m)}(t)\big)\in F
\end{align*}
for certain $0\leq \ell_1,\dots,\ell_m\leq k$, $1\leq i_1,\dots,i_m \leq 2$, $m\geq 2$, 
where $\Psi\colon F^{m}\rightarrow F$ is smooth. Then, \eqref{ouidfoudfifdfdfdfdfd} implies that $\Theta^{k+1}$ is continuous, and \eqref{iudiudsiudsoiosidds} implies that $\Theta^{k+1}(t,\cdot)\in C^1([a,b],F)$ holds for each $t\in I$. This establishes the claim for $k+1$, so that Statement \ref{iufdiufdiufdfd} follows by induction. 
\end{proof}
\noindent
This proves the Lemma \ref{lkjfdlkjfdkjf}.
\end{proof}

\subsubsection{The Exponential Map}
\label{kfdlkfdlkfdscvpdfpofdofd}
We let $\expal\colon \mg\ni X\rightarrow \mathcal{C}_X|_{[0,1]}\in C^\const([0,1],\mg)$, hence 
\begin{align}
\label{uisdiusduiuisduisd}
	\expalinv\colon C^\const([0,1],\mg) \rightarrow \mg,\qquad \chi\mapsto \chi(0). 
\end{align}
The exponential map is given by 
\begin{align*}
	\textstyle\exp\colon \dom[\exp]= \expal^{-1}(\DIDED_\const)\rightarrow G,\qquad X\mapsto \innt \mathcal{C}_X|_{[0,1]}= (\evol_\const\cp \expal)(X).
\end{align*}
\begingroup
\setlength{\leftmargini}{12pt}{
\begin{itemize}
\item
Instead of saying that $G$ is $C^\const$-semiregular, in the following we rather say  that $G$ admits an exponential map.
\item
The relation \textrm{\ref{subst}} implies $\RR\cdot \dom[\exp]\subseteq \dom[\exp]$, as well as that $t\mapsto \exp(t\cdot X)$ is a $1$-parameter group for each $X\in \dom[\exp]$ with
\begin{align}
\label{odaidaooipidadais}
	\textstyle\exp(t\cdot X)= \innt t\cdot \mathcal{C}_X|_{[0,1]}\stackrel{\textrm{\ref{subst}}}{=} \innt_0^t\mathcal{C}_X|_{[0,1]}\qquad\quad\forall\:  t\geq 0, 
\end{align}
confer, e.g., Remark 2.1) in \cite{RGM}. 
\end{itemize}}
\endgroup
\noindent
Finally, Theorem 1 in \cite{MDM} (cf.\ Theorem \ref{weakdiffdf}), Corollary 6 in \cite{MDM}, and Remark 9.1 in \cite{MDM} provide the following statement.
\begin{corollary}
\label{fdjkjfdjkjfjkfd}
Let $\MX\colon I\rightarrow \dom[\exp]\subseteq \mg$ ($I\subseteq \RR$ an open interval) be of class $C^1$, and set $\alpha:=\exp\cp\he\MX$. Assume that $G$ is weakly $C^\infty$-regular, or that $\exp\colon \mg\rightarrow G$ is defined and of class $C^1$. 
  Then, $\alpha$ is of class $C^1$, with
\begin{align*}
	\textstyle\dot\alpha(t)=\dd_e\LT_{\exp(\MX(t))}\big(\int_0^1 \Ad_{\exp(-s\cdot \MX(t))}(\dot\MX(t)) \:\dd s \he \big)\qquad\quad\forall\: t\in I.
\end{align*}   
\end{corollary}
\begin{proof}
If $G$ is weakly $C^\infty$-regular, then the claim is clear from Theorem 1 in \cite{MDM} and Corollary 6 in \cite{MDM}. Assume now that $\exp$ is of class $C^1$. Then, Remark 9.1 in \cite{MDM} shows that 
\begin{align*}
	\textstyle\dd_X\exp(Y)= \dd_e\LT_{\exp(X)}\big(\int_0^1 \Ad_{\exp(-s\cdot X)}(Y) \:\dd s\big)
\end{align*}
holds, so that the claim is clear from Part \ref{chainrule} of Proposition \ref{iuiuiuiuuzuzuztztttrtrtr}.  
\end{proof}

\section{Preliminary Results}
\label{kllkdslkdlklkdslds}
In this section, we derive some elementary results from Theorem \ref{ofdpofdpofdpofdpofd} (Proposition \ref{fdfdfd} in Sect.\  \ref{kjdskjsdkjdskjs} and Proposition \ref{fdfddffddfsdsdssdsd} in Sect.\ \ref{dlkjlkjfdlkjfddofdoifdoi}), and provide an integral expansion for the adjoint action (Lemma \ref{hjfdhjfdhjfd} in Sect.\ \ref{kfdkjfdkjfdfd}). Proposition \ref{fdfdfd} and Lemma \ref{hjfdhjfdhjfd} will be used in Sect.\ \ref{dsdsdsddssddsds} to investigate the product integral of nilpotent curves. Also, Proposition \ref{fdfddffddfsdsdssdsd} was supposed to be applied in this paper, but eventually turned out not to be necessary for our argumentation. We kept this result for academic reasons, and because it   certainly will play a role in future applications. Our considerations in Sect.\ \ref{kjdskjsdkjdskjs} and Sect.\ \ref{dlkjlkjfdlkjfddofdoifdoi} furthermore serve as motivations for the constructions made in Sect.\ \ref{jkdkjdjkddsdsddds}.

\subsection{An Integral Transformation}
\label{kjdskjsdkjdskjs}
Let $\REGC:=\{\phi\in \DIDE\:|\: [0,1]\cdot \phi \subseteq \DIDE\}$. 
We define $\TMAP\colon \REGC\rightarrow \Map([0,1],\comp{\mg})$ by
\begin{align}
\label{dfdfdfdfdfd}
\begin{split}
	\textstyle\TMAP\colon \REGC\cap \DIDE_{[a,b]}&\rightarrow \Map([0,1],\comp{\mg})\\
	 \phi&\textstyle\mapsto \big[ [0,1]\ni t \mapsto  \int_a^b \Ad_{[\innt_s^{b}  t\cdot \phi]}(\phi(s)) \:\dd s\he\big]
\end{split}
\end{align}
for each $a<b$.  
In view of Sect.\ \ref{dsdsdsddssddsds} (the proof of Lemma \ref{mnsdmnsdmnsdmndsmnsd}), for $a<b$ and $\phi\in \REGC\cap \DIDE_{[a,b]}$, we set
\begin{align}
\label{dsdsdsdcxcxcxcxcxcxcxsds}
	\TMAP(\phi|_{[a,a]})\colon [0,1]\ni t \mapsto 0\in \mg.
\end{align}
In this section, we proof the following proposition.
\begin{proposition}
\label{fdfdfd}
Assume that $G$ is weakly $C^k$-regular for $k\in \NN\cup\{\lip,\infty\}$. Then, for $a<b$ and $\phi\in C^k([a,b],\mg)$, we have $\TMAP(\phi)\in C^\infty([0,1],\mg)$ with
\begin{align*}
	\textstyle\innt_a^b t\cdot \phi=\innt_0^t  \TMAP(\phi)\qquad\quad\forall\: t\in [0,1].
\end{align*}
In particular, $\innt_a^b\phi=\innt_0^1 \TMAP(\phi)$ holds, and $\mu\colon [0,1]\ni t\mapsto \innt_a^b t\cdot \phi$ is smooth by \eqref{kjdskjdsjkdskjdsjk}. 
\end{proposition}
\begin{example}
\label{kfdkjfdkjfkjda}
Assume that $G$ is abelian, as well as weakly $C^k$-regular for $k\in \NN\cup\{\lip,\infty\}$. Then, Proposition \ref{fdfdfd} recovers the well-known  formula \cite{TMICH}
\begin{align}
\label{ddsnmvcnmvcnmvcnm}
	\textstyle \innt_a^x \phi = \exp\big(\int_a^x \phi(s)\: \dd s\big)
	\qquad\quad\forall\: x\in [a,b],
\end{align}
for each $a<b$ and $\phi\in C^k([a,b],\mg)$. 
\hspace*{\fill}\qed
\end{example}
\begin{remark}
\label{oidsoidsoidsiods}
Assume that $G$ is weakly $C^k$-regular for $k\in \NN\cup\{\lip,\infty\}$; and recall the group structure discussed in Remark \ref{dsdsdsdsjkjkkjkjkjjkkjkj}, as well as the map \eqref{iudsoiudsiudsiudsods} introduced in Example \ref{fdpofdopdpof}.    It is straightforward from Proposition \ref{fdfdfd} and the properties of the product integral that 
\begin{align}
\label{pdpofdpofdpofd}
\begin{split}
\TMAP(\psi)^{-1}&=\TMAP(\inverse{\psi})\\[2pt]
	\TMAP(\psi)&=\TMAP(\psi|_{[c,b]})\star \TMAP(\psi|_{[a,c]})\\
	\TMAP(\psi)&=\TMAP(\dot\varrho\cdot (\psi\cp\varrho))
\end{split}
\end{align}
holds for $a<c<b$, $\psi\in C^k([a,b],\mg)$, as well as $\varrho\colon [a',b']\rightarrow [a,b]$ ($a'<b'$) of class $C^1$ with $\dot\varrho|_{(a',b')}>0$, $\varrho(a')=a$, $\varrho(b')=b$:
\begingroup
\setlength{\leftmargini}{12pt}
\begin{itemize}
\item
	By \ref{kdsasaasassaas} (first step), Proposition \ref{fdfdfd} (second step), and \eqref{pfifpfpofdpofd} (fourth step), we have
	\begin{align*}
		\textstyle\innt_0^t\TMAP(\psi)\star\TMAP(\inverse{\psi})&\textstyle=\hspace{4.7pt}\innt_0^t\TMAP(\psi)\cdot \innt_0^t\TMAP(\inverse{\psi})\\
		&\textstyle=(\innt_a^b t\cdot \psi) \cdot  (\innt_a^b t\cdot \inverse{\psi})\\
		&\textstyle = (\innt_a^b t\cdot \psi) \cdot  (\innt_a^b \inverse{t\cdot \psi})\\
		&\textstyle=(\innt_a^b t\cdot \psi) \cdot  (\innt_a^b t\cdot \psi)^{-1}\\
		&=e
	\end{align*}
	for each $t\in [0,1]$. Then, injectivity of $\Der|_{C^1_*([0,1],G)}$ implies  $\TMAP(\psi)\star\TMAP(\inverse{\psi})=0$, which proves the first line in \eqref{pdpofdpofdpofd}.
\item
	By Proposition \ref{fdfdfd} (first and third step),  \ref{pogfpogf} (second step), and \ref{kdsasaasassaas} (fourth step), we have
\begin{align*}
	\textstyle\innt_0^t\TMAP(\psi)=\innt_a^b t\cdot \psi = (\innt_c^b t\cdot \psi)\cdot (\innt_a^c t\cdot \psi) = \innt_0^t \TMAP(\psi|_{[c,b]})\cdot \innt_0^t \TMAP(\psi|_{[a,c]})=\innt_0^t \TMAP(\psi|_{[c,b]})\star \TMAP(\psi|_{[a,c]})
\end{align*}
for each $t\in [0,1]$. The second line in \eqref{pdpofdpofdpofd} now follows from injectivity of $\Der$.
\item
Since $\dot\varrho|_{(a',b')}>0$ holds, the substitution formula \ref{subst} applies to $\varrho|_{[x',y']}$ and $t\cdot \psi|_{[\varrho(x'),\varrho(y')]}$ for all $a'\leq x'< y'\leq b'$ and $t\in [0,1]$.  
	We obtain 
	\begin{align*}
		\TMAP(\dot\varrho\cdot (\psi\cp\varrho))&\textstyle=\int_{a'}^{b'} \Ad_{\innt_s^{b'} (\dot\varrho\he\cdot\he ((t\cdot \psi)\he\cp\he\varrho))|_{[s,b']}}(\dot\varrho(s)\cdot (\psi\cp\varrho)(s)) \:\dd s\\
		&\textstyle=\int_{a'}^{b'}\dot\varrho(s)\cdot\big( \Ad_{\innt_{\varrho(s)}^{b} t\cdot \psi|_{[\varrho(s),b]}}( \psi(\varrho(s))\big) \:\dd s\\
		&\textstyle=\int_{a}^{b}\Ad_{\innt_{s}^{b} t\cdot \psi|_{[s,b]}}(\psi(s)) \:\dd s\\
		&=\TMAP(\psi),
	\end{align*}
	where we have applied \eqref{substitRI} in the third step. 
	This proves the third line in \eqref{pdpofdpofdpofd}. 
\hspace*{\fill}\qed 
\end{itemize}
\endgroup
\end{remark}
\begin{remark}
\label{fdjfdklkjfdlkjfd}
The smoothness statement in Proposition \ref{fdfdfd}, in particular,  ensures that $\TMAP$ can be applied iteratively. We will use this fact in Sect.\ \ref{dsdsdsddssddsds} to prove an identity for the product integral in the nilpotent context.\hspace*{\fill}\qed
\end{remark}
\noindent
For the proof of Proposition \ref{fdfdfd}, we shall need the following  statements. 
\begin{lemma}
\label{Adlip}
Let $k\in \NN\cup\{\lip,\infty\}$, $a<b$, $\mu\in C^{k+1}([a,b],G)$, and $\phi\in C^k([a,b],\mg)$. Then, we have $\Ad_\mu(\phi)\in C^k([a,b],\mg)$.  
\end{lemma}
\begin{proof}
Confer, e.g., Lemma 13 in \cite{RGM}.
\end{proof}
\begin{lemma}
\label{lkjflkjfdjkdflkjfd}
Assume that $G$ is weakly $C^k$-regular for $k\in \NN\cup\{\lip,\infty\}$, and let $\phi\in C^k([a,b])$ ($a<b$) be given. Then, 
$\kappa\colon \RR\times [a,b]\ni (t,x)\mapsto \innt_x^bt\cdot \phi\in G$ is continuous. 
\end{lemma}
\begin{proof}
Confer Appendix \ref{asassadsdsdsdsdsdcxcxsds}.
\end{proof}
\begin{lemma}
\label{dkjfkjdjfdlkjlkjfdkjdfkdf}
Assume that $G$ is weakly $C^k$-regular for $k\in \NN\cup\{\lip,\infty\}$, and let $a<b$ as well as $\phi\in C^k([a,b])$ be given. Then, $\TMAP(\phi)\colon \RR\ni t \mapsto  \int_a^b \Ad_{[\innt_s^{b}  t\cdot \phi]}(\phi(s)) \:\dd s\in \mg$ is smooth.
\end{lemma}
\begin{proof}
We consider the map
\begin{align*}
	\Theta^0\colon \RR\times [a,b]\rightarrow \mg,\qquad (t,x)\mapsto  \Ad_{[\innt_x^{b}  t\cdot \phi]}(\phi(x)),	
\end{align*}
and observe the following:
\begingroup
\setlength{\leftmargini}{12pt}{
\begin{itemize}
\item
	Lemma \ref{lkjflkjfdjkdflkjfd} shows that $\Theta^0$ is continuous, as we have (with $\kappa$ as in Lemma \ref{lkjflkjfdjkdflkjfd})
\begin{align*}
	\Theta^0(t,x)= \Ad_{\kappa(t,x)}(\phi(x))\qquad\quad\forall\: t\in \RR, \: x\in [a,b].
\end{align*}	
\vspace{-18pt}
\item
Given $t\in \RR$, we have $\Theta^0(t,\cdot)\in C^k([a,b],\mg)\subseteq C^1([a,b],\mg)$ by Lemma \ref{Adlip}, \eqref{kjdskjdsjkdskjdsjk}, and
\begin{align*}
	\Theta^0(t,x)=\Ad_{[\innt_a^{b}  t\cdot \phi]}(\Ad_{[\innt_a^{x}  t\cdot \phi]^{-1}}(\phi(x)))\qquad\quad\forall\: t\in \RR, \: x\in [a,b].
\end{align*}
\vspace{-12pt}
\item
Given $x\in [a,b]$, Proposition \ref{ableiti} yields
\begin{align*}
	\textstyle C^1(\RR,G)\ni\mu_x\colon \RR\ni t\mapsto \innt_x^b t\cdot \phi\in G.
\end{align*} 
\vspace{-22pt}

\noindent
Then, Lemma \ref{Adlip} shows
\begin{align}
\label{jfdkjfdkjfd}
	\textstyle C^1(\RR,\mg)\ni \alpha_x:= \Theta^0(\cdot,x) \colon \RR\ni t \mapsto \Ad_{\mu_x(t)}(\phi(x))\in \mg.
\end{align}
Moreover, for $t,h\in \RR$ we have
\begin{align}
\label{kjdkjdkjfdkjfdkjfd}
\begin{split}
	\textstyle\mu_x(t+h)&\textstyle\stackrel{\ref{kdsasaasassaas}}{=}\mu_x(t)\cdot \innt_x^b  h\cdot\Ad_{[\innt_x^\bullet t\cdot \phi]^{-1}}(\phi)\\
	&\textstyle\stackrel{\phantom{\ref{kdsasaasassaas}}}{=}  \conj_{\mu_x(t)}\big(\innt_x^b  h\cdot\Ad_{[\innt_x^\bullet t\cdot \phi]^{-1}}(\phi)\big)\cdot \mu_x(t)\\  
	&\textstyle\stackrel{\ref{homtausch}}{=} [\innt_x^b  h\cdot\Ad_{\innt_\bullet^b t\cdot \phi}(\phi)]\cdot\mu_x(t)
	\\  
	&\textstyle\stackrel{\phantom{\ref{kdsasaasassaas}}}{=}  [\innt_x^b  h\cdot\Theta^0(t,\cdot)]\cdot\mu_x(t).
\end{split}
\end{align}
By Proposition \ref{ableiti}, the map $\RR\ni h\mapsto \innt _x^b h\cdot \Theta^0(t,\cdot)\in G$ is of class $C^1$ for $t\in \RR$ (recall that $\Theta^0(t,\cdot)\in C^k([a,b],\mg)$ holds by the second point above), with
\begin{align}
\label{hjdshjdshjds}
	\textstyle\frac{\dd}{\dd h}\big|_{h=0}\: \innt _x^b h\cdot \Theta^0(t,\cdot) = \dd_0\he\evol^k_{[x,b]}(\Theta^0(t,\cdot))= \int_x^b \Theta^0(t,y)\: \dd y\in \mg.
\end{align}
We obtain for $t\in \RR$ that
\begin{align}
\label{kjfdjkkjfdkjfd}
\begin{split}
	\textstyle\dot\alpha_{x}(t)&\textstyle
	\stackrel{\phantom{\eqref{iufdiufiudiufd}}}{=} \frac{\dd}{\dd h}\big|_{h=0}\: \Ad_{\mu_{x}(t+h)}(\phi(x))\\
&	\textstyle\stackrel{\eqref{kjdkjdkjfdkjfdkjfd}}{=} \frac{\dd}{\dd h}\big|_{h=0}\: \Ad_{\innt_{x}^b  h\cdot\Theta^0(t,\cdot)}\big(\Ad_{\mu_{x}(t)}(\phi(x))\big)
	\\
	&\textstyle\stackrel{\eqref{hjdshjdshjds}}{=}\bil{\int_{x}^b \Theta^0(t,y)\:\dd y}{\alpha_{x}(t)}\\
&\textstyle \stackrel{\eqref{jfdkjfdkjfd}}{=}	
	\textstyle\bil{\int_{x}^b\alpha_{y}(t)\:\dd y}{\alpha_{x}(t)}.
\end{split}
\end{align}
\end{itemize}
}
\endgroup
\noindent
The claim is now clear from Lemma \ref{lkjfdlkjfdkjf} (with $I=\RR$, $\Omega=\bil{\cdot}{\cdot}$, and $z=b$ there).
\end{proof}
\noindent
We are ready for the proof of Proposition \ref{fdfdfd}.
\begin{proof}[Proof of Proposition \ref{fdfdfd}]
By Lemma \ref{dkjfkjdjfdlkjlkjfdkjdfkdf}, we have $\TMAP(\phi)\in C^\infty([0,1],\mg)$.  
By Proposition \ref{ableiti}, we have $C^1([0,1],\mg)\ni \mu\colon [0,1]\ni t\mapsto \innt_a^b t\cdot \phi$ with
\begin{align*}
	\textstyle\dot\mu(t)&= \dd_{t\cdot\phi}\he \evol^k_{[a,b]}(\phi)
	=\textstyle \dd_e\LT_{\mu(t)}\big(\int_a^b\Ad_{[\innt_a^s t\cdot \phi]^{-1}}(\phi(s))\:\dd s\big)
	\textstyle=\dd_e\RT_{\mu(t)}\big(\int_a^b \Ad_{\innt_s^b t\cdot \phi}(\phi(s))\:\dd s\big).
\end{align*}
In the third step, we have applied \eqref{pofdpofdpofdsddsdsfd} with $\mathcal{L}\equiv \Ad_{\mu(t)}$. This shows that $\Der(\mu)=\TMAP(\phi)$ holds, with $\mu(0)=e$. We obtain
\begin{align*}
	\textstyle\innt_a^b t\cdot \phi = \mu(t)=\innt_0^t \Der(\mu)=\innt_0^t \TMAP(\phi),
\end{align*}
which proves the claim.
\end{proof}
\noindent
The following corollary verifies the guess made in Remark 8.2) in \cite{MDM}.
\begin{corollary}
\label{hjfdhjfdhjfdhjdsaiuoiurexcccx}
Assume that $G$ is weakly $C^k$-regular for $k\in \NN\cup\{\lip,\infty\}$, and let $\phi,\psi \in C^k([a,b],\mg)$ be given. Then, $\mu\colon \RR\ni t\mapsto \innt_a^b \phi + t\cdot \psi\in G$ is smooth.
\end{corollary}
\begin{proof}
We have  
\begin{align*}
	\textstyle\innt_a^b (\phi + t\cdot \psi)\stackrel{\ref{kdsasaasassaas}}{=} [\innt_a^b \phi] \cdot[\innt_a^b t\cdot \Ad_{[\innt_a^\bullet\phi]^{-1}}(\psi)]\qquad\quad\forall\:\phi,\psi \in C^k([a,b],\mg).
\end{align*} 
By  Lemma \ref{Adlip}, it thus suffices to prove the claim for $\phi=0$  and arbitrary $\psi\in C^k([a,b],\mg)$. Let thus $\psi \in C^k([a,b],\mg)$ be given. Then, we have
\begin{align*}
	 \textstyle\innt_a^b (t+h)\cdot \psi \stackrel{\ref{kdsasaasassaas}}{=} [\innt_a^b t\cdot \psi] \cdot[\innt_a^b h\cdot \Ad_{[\innt_a^\bullet t\cdot \psi]^{-1}}(\psi)]\qquad\quad\forall\: t,h\in \RR.
\end{align*} 
By  Lemma \ref{Adlip}, it thus suffices to prove smoothness of the map $[0,1]\ni h\mapsto \innt_a^b h\cdot \chi\in G$ for each $\chi\in C^k([a,b],\mg)$, so that the claim is clear from (the last statement in) Proposition \ref{fdfdfd}.
\end{proof}

\subsection{Differentiation of Parameter-dependent Integrals}
\label{dlkjlkjfdlkjfddofdoifdoi}
In this section, we prove the following differentiation result.
\begin{proposition}
\label{fdfddffddfsdsdssdsd}
Assume that $G$ is Mackey k-continuous 
for $k\in \NN\cup\{\lip,\infty\}$. Let $a<b$, $\sigma>0$, $I\subseteq \RR$ an open interval with $[a-\sigma,b+\sigma]\subseteq I$, as well as $\Phi\colon I\times [a-\sigma,b+\sigma]\rightarrow \mg$ a map with $\Phi(z,\cdot)\in \DIDE^k_{[a-\sigma,b+\sigma]}$ for each $z\in I$. Let $x\in [a,b]$ be given, such the following conditions hold: 
\begingroup
\setlength{\leftmargini}{18pt}{
\renewcommand{\theenumi}{{\roman{enumi}})} 
\renewcommand{\labelenumi}{\theenumi}
\begin{enumerate}
\item
\label{saasaassasasas2}
We have $(\partial_1 \Phi)(x,\cdot)\in C^k([a,b],\mg)$ with $\int_a^x \Ad_{[\innt_a^s\Phi(x,\cdot)]^{-1}}(\partial_1\Phi(x,s))\:\dd s\in \mg$.
\item
\label{saasaassasasas1}
To $\pp\in \Sem{E}$ and $\dind\llleq k$, there exists $L_{\pp,\dind}\geq 0$ as well as $I_{\pp,\dind}\subseteq I$ open with $x\in I_{\pp,\dind}$, such that
\begin{align*}
	\pp^\dind_\infty(\Phi(x+h,\cdot)-\Phi(x,\cdot))\leq |h|\cdot L_{\pp,\dind}\qquad\quad \forall\: h\in \RR_{\neq 0}\:\text{ with }\: x+h\in I_{\pp,\dind}.
\end{align*}  
\end{enumerate}}
\endgroup
\noindent
Then, we have
\begin{align*}
	\textstyle\frac{\dd}{\dd h}\big|_{h=0}\he\innt_a^{x+h} \Phi(x+h,\cdot)=\dd_e\RT_{\innt_a^x\Phi(x,\cdot)}(\Phi(x,x)) + \dd_e\LT_{\innt_a^x\Phi(x,\cdot)}\big(\int_a^x \Ad_{[\innt_a^s\Phi(x,\cdot)]^{-1}}(\partial_1\Phi(x,s))\:\dd s\he\big).
\end{align*}  
\end{proposition} 
\vspace{6pt}

\noindent
For the proof of Proposition \ref{fdfddffddfsdsdssdsd}, we shall need the following statements.
\begin{lemma}
\label{lkdslklkdslkdsldsds}
To $\pp\in \Sem{E}$, there exist $\pp\leq \qq\in \Sem{E}$ and $U\subseteq G$ open with $e\in U$, such that for each $a<b$ and $\chi\in \DIDE_{[a,b]}$ with $\innt_a^\bullet \chi \in U$ we have
\begin{align*}
	\textstyle(\pp\cp\chart)(\innt_a^\bullet \chi)\leq \int_s^\bullet \qq(\chi(s))\: \dd s.
\end{align*} 
\end{lemma}
\begin{proof}
	Confer Lemma 4 in \cite{MDM}.
\end{proof}
\begin{corollary}
\label{opopsopsdopdsas}
	For each compact $\compact\subseteq G$ and each $\qq\in \Sem{E}$, there exists some $\qq\leq \mm\in \Sem{E}$, such that  
	$\qq\cp \Ad_g\leq \mm$ holds for each  
	$g\in \compact$.
\end{corollary}
\begin{proof}
	This is clear from Lemma \ref{alalskkskaskaskas}, because 
 $\Ad\colon G\times \mg\ni (g,X)\mapsto \Ad_g(X)\in \mg$ is smooth  as well as linear in the second argument. 
\end{proof}
\begin{proof}[Proof of Proposition \ref{fdfddffddfsdsdssdsd}]
We consider the maps
\begin{align*}
	\mu\colon [-\sigma,\sigma]&\textstyle\ni h\mapsto \hspace{3pt}\innt_a^{x} \Phi(x+h,\cdot)
	\\[2pt]
	\nu_+\colon\hspace{8.15pt} [0,\sigma]&\textstyle\ni h\mapsto [\innt_{x}^{x+h} \Phi(x,\cdot)]^{-1}[\innt_x^{x+h}\Phi(x+h,\cdot) ]
	\\	
	\eta_+\colon\hspace{8.15pt} [0,\sigma]&\textstyle\ni h\mapsto \hspace{3pt}\innt_x^{x+h} \Phi(x,\cdot)
	\\[4pt]
	\nu_-\colon [-\sigma,0]&\textstyle\ni h\mapsto [\innt_{x+h}^{x} \Phi(x,\cdot)]^{-1}	\\
	\eta_-\colon [-\sigma,0]&\textstyle\ni h\mapsto [\innt_{x+h}^x \Phi(x+h,\cdot)]^{-1}
	[\innt_{x+h}^x \Phi(x,\cdot)],
\end{align*}
and obtain
\begin{align}
\label{kjsdjkkjdsdstzztuz}
\begin{split}
	\textstyle\innt_a^{x+h} \Phi(x+h,\cdot)&\textstyle= \eta_+(h)\cdot\nu_+(h)\cdot \mu(h)\qquad\quad\forall\: h\in [0,\sigma]\\
	\textstyle\innt_a^{x+h} \Phi(x+h,\cdot)&\textstyle= \eta_-(h)\cdot\nu_-(h)\cdot \mu(h)\qquad\quad\forall\: h\in [-\sigma,0].
\end{split}
\end{align}
We furthermore observe the following:
\begingroup
\setlength{\leftmargini}{15pt}{
\renewcommand{\theenumi}{{\bf\small\arabic{enumi}})} 
\renewcommand{\labelenumi}{\theenumi}
\begin{enumerate}
\item
\label{iudskjdsiuiudsiuds2}
The maps $\nu_+,\eta_-$ are continuous at $h=0$, with $\nu_+(0)=e=\eta_-(0)$. To show this, it suffices to prove the same property for the maps
\begin{align*}
	\alpha_+\colon\hspace{7.8pt} 
	[0,\sigma]\ni h&\textstyle\mapsto \innt_x^{x+h}\Phi(x,\cdot)
\\
	\beta_+\colon\hspace{7.8pt} 
	[0,\sigma]\ni h&\textstyle\mapsto \innt_x^{x+h}\Phi(x+h,\cdot)\\
	\alpha_-\colon 
	\textstyle[-\sigma,0]\ni h&\textstyle\mapsto\innt_{x+h}^x\Phi(x,\cdot) 
\\
	\beta_-\colon 
	\textstyle[-\sigma,0]\ni h&\textstyle\mapsto\innt_{x+h}^x\Phi(x+h,\cdot).
\end{align*}
First, it is clear that $\alpha_\pm$ is continuous (in particular at $h=0$) with $\alpha_\pm(0)=e$, as we have
\begin{align*}
	\alpha_-(h)\textstyle= [\innt_{a-\sigma}^x\Phi(x,\cdot)]\cdot [\innt_{a-\sigma}^{x+h}\Phi(x,\cdot)]^{-1}\qquad\quad\forall\: h\in [-\sigma,0].
\end{align*}
Second, Condition \ref{saasaassasasas1} implies that for each sequence $[-\sigma,0)\cup (0,\sigma]\supseteq \{h_n\}_{n\in \NN}\rightarrow 0$,   we have 
	$\{\Phi(x+h_n,\cdot)\}_{n\in \NN}\mackarr{\kk}\Phi(x,\cdot)$.  
Since $G$ is Mackey k-continuous, Lemma \ref{fdhjfdkjj} yields the following:
\begingroup
\setlength{\leftmarginii}{12pt}{
\begin{itemize}
\item
	For each sequence $(0,\sigma]\supseteq \{h_n\}_{n\in \NN}\rightarrow 0$, we have 
	\begin{align*}
	\textstyle\limin \innt_x^{\bullet} \Phi(x+h_n,\cdot)=\innt_x^\bullet \Phi(x,\cdot)\qquad\Longrightarrow\qquad \lim_n \innt_x^{x+h_n}\Phi(x+h_n,\cdot)=e.\hspace{28.2pt}
\end{align*} 
This implies $\lim_{h\rightarrow 0}\beta_+(h) = e$.
\item
	For each sequence $[-\sigma,0)\supseteq \{h_n\}_{n\in \NN}\rightarrow 0$, we have 
	\begin{align*}
	\textstyle\limin \innt_{a-\sigma}^{\bullet} \Phi(x+h_n,\cdot)=\innt_{a-\sigma}^\bullet \Phi(x,\cdot)\qquad\Longrightarrow\qquad \lim_n \innt_{a-\sigma}^{x+h_n}\Phi(x+h_n,\cdot)=\innt_{a-\sigma}^{x}\Phi(x,\cdot).
\end{align*} 
This implies $\lim_{h\rightarrow 0}\beta_-(h) = e$, as we have
\begin{align*}
	\beta_-(h)\textstyle=[\innt_{a-\sigma}^x\Phi(x+h,\cdot)]\cdot[\innt_{a-\sigma}^{x+h}\Phi(x+h,\cdot)]^{-1}\qquad\quad\forall\: h\in [-\sigma,0].
\end{align*}
\end{itemize}}
\endgroup
\item
\label{iudskjdsiuiudsiuds3}
We have
\begin{align}
\label{fdfdfd2}
	\dot\mu(0)&\textstyle=\dd_e\LT_{\innt_a^x\Phi(x,\cdot)}\big(\int_a^x \Ad_{[\innt_a^s\Phi(x,\cdot)]^{-1}}(\partial_1\Phi(x,s))\:\dd s\he\big)\\
\label{fdfdfd3}
	\dot\nu_-(0)&=\Phi(x,x)\\
\label{fdfdfd1}
	\dot\eta_+(0)&\textstyle=\Phi(x,x).
\end{align}
In fact, \eqref{fdfdfd1} is evident, \eqref{fdfdfd2} is clear from 
Theorem \ref{ofdpofdpofdpofdpofd}, and \eqref{fdfdfd3} is obtained as follows:
\vspace{6pt}
 
\noindent   
\emph{Proof of Equation \eqref{fdfdfd3}.}
Given $-\sigma\leq h<0$, we define
	$\varrho_h\colon [x,x-h]\ni t\mapsto t+h \in [x + h,x]$, 
and obtain
\begin{align}
\label{iudsiudsdsooidsoidsoidsods}
	 \textstyle\innt_{x+h}^x\psi=\innt_{x+h}^{\varrho_h(x-h)}\psi \stackrel{\ref{subst}}{=}\innt_{x}^{x-h}\psi(\cdot+h)\qquad\quad\forall\:\psi\in \DIDE_{[x-\sigma,x]}.	
\end{align} 
According to Example \ref{fdpofdopdpof} (with 
$\varrho\colon [x+h,x]\ni t\mapsto 2x+h -t\in [x+h,x]$ there), we have
\vspace{-4pt}
\begin{align*}
	 \textstyle\nu_-(h)=[\innt_{x+h}^x\Phi(x,\cdot)]^{-1} 
	 =\innt_{x+h}^{x}-\Phi(x, 2x+h-\cdot)	\stackrel{\eqref{iudsiudsdsooidsoidsoidsods}}{=}\innt_x^{x-h}-\Phi(x,2x-\cdot).
\end{align*} 
We obtain
\begin{align*}
\textstyle\lim_{0>h\rightarrow 0}\frac{1}{h}\cdot (\chart(\nu_-(h))\textstyle-\chart(\nu_-(0)))
&=\textstyle -\lim_{0<h\rightarrow 0}\frac{1}{h}\cdot \chart(\nu_-(-h)^{-1})\\
&=\textstyle -\lim_{0<h\rightarrow 0}\frac{1}{h}\cdot \chart(\innt_x^{x+h}-\Phi(x,2x-\cdot))\\[1pt]
&=\dd_e\chart(\Phi(x,x)),
\end{align*}
which proves \eqref{fdfdfd3}.\hspace*{\fill}$\small\qed$
\item
\label{iudskjdsiuiudsiuds4}
By \ref{iudskjdsiuiudsiuds2}, there exists $0<\delta\leq \sigma$, such that  the following maps are defined: 
\begin{align*}
	\Delta_+\colon (0,\delta]\ni h\mapsto \textstyle \frac{1}{h}\cdot \chart(\nu_+(h))\qquad\quad\text{and}\qquad\quad \Delta_-\colon [-\delta,0)\ni h\mapsto \textstyle \frac{1}{h}\cdot \chart(\eta_-(h)).
\end{align*}
Then, to establish the proof, it suffices to show  
\begin{align}
	\label{vldsdlkskdvls}
	\textstyle\lim_{0<h\rightarrow 0}\Delta_+(h)=0\qquad\quad \text{as well as}\qquad\quad \lim_{0>h\rightarrow 0}\Delta_-(h)=0.
\end{align}
In fact, let $(\chart',\U')$ be a chart around $\mu(0)$. Let furthermore $O\subseteq \U$ and $O'\subseteq \U'$ be open with $e\in O$ and $\mu(0)\in O'$, such that $O\cdot O\cdot O'\subseteq \U'$ holds. We set $W:=\chart(O)$, $W':=\chart'(O')$, $\V':=\chart'(\U')$. We furthermore define  
\begin{align*}
	f\colon W\times W\times W'&\rightarrow \V'\\
	(X,Y,Z)&\mapsto \chart'(\chartinv(X)\cdot \chartinv(Y)\cdot \chart'^{-1}(Z)), 
\end{align*}
as well as (shrink $\delta>0$ if necessary, to ensure $\im[\eta_\pm]\subseteq O$, $\im[\nu_\pm]\subseteq O$, $\im[\mu]\subseteq O'$)
\begin{align*}
	&\gamma_+\colon \hspace{8pt}[0,\delta]\ni h\mapsto (\chart(\eta_+(h)),\chart(\nu_+(h)) ,\chart'(\mu(h)))\in E\times E\times E\\ 
	&\gamma_-\colon [-\delta,0]\ni h\mapsto (\chart(\eta_-(h)),\chart(\nu_-(h)) ,\chart'(\mu(h)))\in E\times E\times E.
\end{align*}
Then, Lemma \ref{sdsdds} (with $F_1\equiv E\times E\times E$, $F_2\equiv E$, and $\gamma\equiv \gamma_\pm$ there) together with Part \ref{productrule} of Proposition \ref{iuiuiuiuuzuzuztztttrtrtr} (cf.\ also \eqref{LGPR}), \eqref{kjsdjkkjdsdstzztuz},  \eqref{fdfdfd2}, \eqref{fdfdfd3}, \eqref{fdfdfd1},  \eqref{vldsdlkskdvls} implies the claim. 
\end{enumerate}}
\endgroup
\noindent
Finally, to prove \eqref{vldsdlkskdvls}, let $\pp\in \Sem{E}$ be fixed, and choose $\pp\leq \qq\in \Sem{E}$ as well as $U\subseteq G$ as in Lemma \ref{lkdslklkdslkdsldsds}. Since both $\nu_+,\eta_-$ are continuous by \ref{iudskjdsiuiudsiuds2}, there exists $0<\delta_U\leq\delta$ ($\delta$ as in \ref{iudskjdsiuiudsiuds4}) with $\nu_+([0,\delta_U])\subseteq U$ and $\eta_-([-\delta_U,0]))\subseteq U$. We observe that
\begin{align*}
	\nu_+(h)
	&=\textstyle\innt_x^{x+h}\Ad_{\alpha_+(\cdot)^{-1}}(\Phi(x+h,\cdot)-\Phi(x,\cdot))\qquad\quad\forall\: h\in (0,\sigma]\\
	\eta_-(h)&=\textstyle\innt_{x+h}^{x}\Ad_{\beta_-(\cdot)^{-1}}(\Phi(x,\cdot)-\Phi(x+h,\cdot))\qquad\quad\forall\: h\in [-\sigma,0)
\end{align*}
holds by \ref{kdskdsdkdslkds}, and obtain from Lemma \ref{lkdslklkdslkdsldsds} that
\begin{align*}
	\pp(\Delta_+(h))&\textstyle\leq \frac{1}{|h|}\cdot \int_x^{x+h} \qq\big(\Ad_{\alpha_+(\cdot)^{-1}}(\Phi(x+h,\cdot)-\Phi(x,\cdot)\big)\:\dd s\qquad\quad\forall\: h\in (0,\delta_U]\\
	\pp(\Delta_-(h))&\textstyle\leq \frac{1}{|h|}\cdot \int_x^{x+h} \qq\big(\Ad_{\beta_-(\cdot)^{-1}}(\Phi(x,\cdot)-\Phi(x+h,\cdot)\big)\:\dd s\qquad\quad\forall\: h\in [-\delta_U,0).
\end{align*} 
Let $\qq\leq \mm\in \Sem{E}$ be as in Corollary \ref{opopsopsdopdsas}, for $\compact\equiv \inv(\im[\alpha_+])\cup\inv(\im[\beta_-])$ there. We shrink $0<\delta_U\leq \delta$ such that $x+ (-\delta_U,\delta_U)\subseteq I_{\mm,0}$ holds, for $I_{\mm,0}$ as in Condition \ref{saasaassasasas1}  (with $\dind=0$ there). Then, Condition \ref{saasaassasasas1} yields the following: 
\begingroup
\setlength{\leftmargini}{12pt}{
\begin{itemize}
\item
For $0< h<\delta_U$, we have
\begin{align*}
	\textstyle\pp(\Delta_+(h))&\textstyle\leq \frac{1}{|h|}\cdot \int_x^{x+h}\qq\big(\Ad_{\alpha_+(\cdot)^{-1}}(\Phi(x+h,s)-\Phi(x,s)\big)\:\dd s
	\\
	&\textstyle\leq  \sup\{x\leq s\leq x+h\:|\:\mm(\Phi(x+h,s)-\Phi(x,s))\}
	\\
	&\textstyle\leq |h|\cdot L_{\mm,0}.
\end{align*}
\item
For $-\delta_U< h<0$, we have
\begin{align*}
	\textstyle\pp(\Delta_-(h))&\textstyle\leq \frac{1}{|h|}\cdot \int_{x+h}^{x}\qq\big(\Ad_{\beta_-(\cdot)^{-1}}(\Phi(x,s)-\Phi(x+h,s)\big)\:\dd s
	\\
	&\textstyle\leq  \sup\{x+h\leq s\leq x\:|\: \mm(\Phi(x,s)-\Phi(x+h,s))\}
	\\
	&\textstyle\leq |h|\cdot L_{\mm,0}.
\end{align*}
\end{itemize}
}
\endgroup
\noindent
This proves \eqref{vldsdlkskdvls}, hence the claim.
\end{proof}

\subsection{An Integral Expansion for the Adjoint Action}
\label{kfdkjfdkjfdfd}
Let $a<b$ and  
 $\psi \in \DIDE_{[a,b]}$ be given. For $X\in \mg$, we define
\begin{align}
\label{kldslkdslkdslcxcxcxcx}
	\Add^\pm_{\psi}[X]\colon [a,b]\rightarrow \mg,\qquad  t \mapsto \Ad_{[\innt_a^t\psi]^{\pm 1}}(X).
\end{align}
We furthermore define
\begin{align*}
	\Add_{\psi}^\pm[t]\colon \mg\ni X&\mapsto \Add_\psi^\pm[X](t)\in \mg
	\qquad\quad\forall\: t\in [a,b]\\
	\Add_\psi^\pm&:= \Add_{\psi}^\pm[b].
\end{align*}
For $\chi\in C^0([a,b],\mg)$, we set
\begin{align*}
	\etamm^\pm(\psi,\chi)\colon [a,b]\ni t \mapsto \Add_\psi^\pm[\chi(t)](t)\in \mg.
\end{align*}
The following assertions are immediate from the definitions:
\begingroup
\setlength{\leftmargini}{12pt}{
\begin{itemize}
\item
	We have $\etamm^\pm(\psi,\chi)\in C^{k+1}([a,b],\mg)$ for $k\in \NN\cup \{\infty\}$, for each $\psi\in \DIDE^k_{[a,b]}$ and $\chi\in C^{k+1}([a,b],\mg)$ (by \eqref{kjdskjdsjkdskjdsjk} and smoothness of the group operations).  
\item
	We have $\etamm^\pm(\psi,\etamm^\mp(\psi,\chi))=\chi$ for each $\psi\in \DIDE_{[a,b]}$ and $\chi\in C^0([a,b],\mg)$.
\end{itemize}
}
\endgroup
\noindent
We furthermore observe the following:
\begin{lemma}
\label{hjhjfdhjfdhjfd}
Let $a<b$, $\psi\in \DIDE_{[a,b]}$, and $X\in \mg$ be given. Then, $\Add^\pm_\psi[X]\in C^1([a,b],\mg)$ holds, with $\Add^\pm_\psi[X](a)=X$  as well as
\begin{align}
\label{iureiureiueriureiure}
	\partial_t \Add^+_\psi[X](t) \textstyle=\bil{\psi(t)}{\Add^+_\psi[X](t)}\qquad\:\:\text{and}\qquad\:\:
	\partial_t \Add^-_\psi[X](t) \textstyle=-\Add^-_\psi[\bil{\psi(t)}{X}](t)
\end{align}
for each $t\in [a,b]$. 
In particular, for $\chi\in C^1([a,b],\mg)$, we have
\begin{align}
\label{nbbndsjhdsuhdsjhsdksd}
\begin{split}
	\dot\etamm^+(\psi,\chi)&= \etamm^+(\psi,\dot\chi)+\bil{\psi}{\etamm^+(\psi,\chi)}\\
	\dot\etamm^-(\psi,\chi)&= \etamm^-(\psi,\dot\chi)-\etamm^-(\psi,\bil{\psi}{\chi}).
\end{split} 
\end{align}
\end{lemma}
\begin{proof}
Equation \eqref{iureiureiueriureiure} is verified  
in Appendix \ref{asassadsdsdsdsdsdcxcxsdscvvcvcvcvcvcc}; and  
\eqref{nbbndsjhdsuhdsjhsdksd} is clear from \eqref{iureiureiueriureiure} as well as the parts \ref{linear}, \ref{chainrule}, \ref{productrule} of Proposition \ref{iuiuiuiuuzuzuztztttrtrtr}.
\end{proof}
\noindent
For the sake of completeness, we want to mention the following well-known result:
\begin{corollary}
\label{lksdlkdslkdslk}
Let $a<b$, $\psi\in \DIDE_{[a,b]}$, $X\in \mg$. Then, $\Add^+_\psi[X]$ is the unique solution $\alpha\in C^1([a,b],\mg)$ to the differential equation (Lax equation) $\dot\alpha=\bil{\psi}{\alpha}$, with the initial condition $\alpha(a)=X$.
\end{corollary} 
\begin{proof}
By Lemma \ref{hjhjfdhjfdhjfd}, it remains to show uniqueness. Let thus $\alpha\in C^1([a,b],\mg)$ be given, with $\dot\alpha=\bil{\psi}{\alpha}$  and $\alpha(a)=X$.  
Then, $\dot\etamm^-(\psi,\alpha)=0$ holds by \eqref{nbbndsjhdsuhdsjhsdksd}, with $\etamm^-(\psi,\alpha)(a)=X$. Then, \eqref{isdsdoisdiosd} yields $\etamm^-(\psi,\alpha)=X$, hence
\begin{align*}
		\Add^+_\psi[X]= \etamm^+(\psi,\mathcal{C}_X|_{[a,b]})=\etamm^+(\psi,\etamm^-(\psi,\alpha))=\alpha. 
\end{align*}
This proves the claim.
\end{proof}
\noindent
Let $\comp{\mg}$ denote the completion of $\mg$. Let $\psi \in C^0([a,b],\mg)$, $X\in \mg$, $n\in \NN$ be given. We set
\begin{align*}
 \TPOL^\pm_{0,\psi}[X]\colon [a,b]\rightarrow \mgc,\qquad t\mapsto X,	
\end{align*}
\vspace{-18pt}

\noindent
and define for $\ell\geq 1$
\begin{align*}
	\textstyle \TPOL^+_{\ell,\psi}[X]\colon [a,b]\rightarrow \mgc,\qquad &\textstyle t\mapsto \hspace{34.45pt} \int_a^{t}\dd s_1\int_a^{s_1}\dd s_2 \: {\dots} \int_a^{s_{\ell-1}}\dd s_\ell \: (\bilbr{\psi(s_1)}\cp \dots \cp \bilbr{\psi(s_{\ell})})(X)\\
	\textstyle \TPOL^-_{\ell,\psi}[X]\colon [a,b]\rightarrow \mgc,\qquad &\textstyle t\mapsto  (-1)^\ell\cdot\int_a^{t}\dd s_1\int_a^{s_1}\dd s_2 \: {\dots} \int_a^{s_{\ell-1}}\dd s_\ell \: (\bilbr{\psi(s_\ell)}\cp \dots \cp \bilbr{\psi(s_{1})})(X).
\end{align*}
For $\ell\geq 1$, we set
\begin{align*}
	\textstyle \Rest^+_{\ell,\psi}[X]\colon [a,b]\rightarrow \mgc,\quad\: &\textstyle t\mapsto \hspace{34.45pt} \int_a^{t}\dd s_1\int_a^{s_1}\dd s_2 \: {\dots} \int_a^{s_{\ell-1}}\dd s_\ell \: (\bilbr{\psi(s_1)}\cp \dots \cp \bilbr{\psi(s_{\ell})})(\Add^+_\psi[X](s_\ell))\\
	\textstyle \Rest^-_{\ell,\psi}[X]\colon [a,b]\rightarrow \mgc,\quad\: &\textstyle t\mapsto  (-1)^\ell\cdot\int_a^{t}\dd s_1\int_a^{s_1}\dd s_2 \: {\dots} \int_a^{s_{\ell-1}}\dd s_\ell \: \Add_\psi^-[(\bilbr{\psi(s_\ell)}\cp \dots \cp \bilbr{\psi(s_{1})})(X)](s_\ell).
\end{align*}
To simplify the notations,   
we define
\begin{align*}
	\APOL_{\ell,\psi}^\pm[t]\colon \mg\ni X&\mapsto \APOL_{\ell,\psi}^\pm[X](t)\in \mgc
	\qquad\quad\hspace{4pt}\forall\: t\in [a,b],\: \ell\in \NN\\
	\APOL_{\ell,\psi}^\pm&:=\APOL_{\ell,\psi}^\pm[b]\qquad\quad\hspace{40pt}\forall\: \ell\in \NN.
\end{align*}
We obtain the following statement.
\begin{lemma}
\label{hjfdhjfdhjfd}
	Let $\psi\in \DIDE_{[a,b]}$, $X\in \mg$, and $n\in \NN$ be given. Then, for $t\in [a,b]$ we have
\begin{align*}
	\textstyle\Add^+_\psi[X](t)&\textstyle= \sum_{\ell=0}^n  \TPOL^+_{\ell,\psi}[X](t) + \Rest^{+}_{n+1,\psi}[X](t)\\
	\textstyle\Add^-_\psi[X](t)&\textstyle= \sum_{\ell=0}^n  \TPOL^-_{\ell,\psi}[X](t) + \Rest^{-}_{n+1,\psi}[X](t).	
\end{align*}	
\end{lemma}
\begin{proof}
	By Lemma \ref{hjhjfdhjfdhjfd} and \eqref{isdsdoisdiosd}, we have
	\begin{align}
	\label{oidsoisoidsoidsoidsoi}
		\Add^+_\psi[X](t)\textstyle&\textstyle=X+\int_a^t\he \bil{\psi(s)}{\Add^+_\psi[X](s)} \:\dd s\textstyle= \TPOL^+_{0,\psi}[X](t) +   \Rest^+_{1,\psi}[X](t)\\
	\label{oidsoisoidsoidsoidsoii}
		\Add^-_\psi[X](t)\textstyle&\textstyle=X-\int_a^t\he \Add^-_\psi[\bil{\psi(s)}{X}](s) \:\dd s\textstyle= \TPOL^-_{0,\psi}[X](t) +   \Rest^-_{1,\psi}[X](t).
	\end{align}
	We thus can assume that the claim holds for some $n\in \NN$. We obtain from \eqref{pofdpofdpofdsddsdsfd} (third steps) that
	\begin{align*}
		\Add^+_\psi[X](t)&\textstyle\stackrel{\phantom{\eqref{oidsoisoidsoidsoidsoi}}}{=} \sum_{\ell=0}^n  \TPOL^+_{\ell,\psi}[X](t) + \Rest^+_{n+1,\psi}[X](t)\\[-3pt]
		&\textstyle\stackrel{\phantom{\eqref{oidsoisoidsoidsoidsoi}}}{=} \sum_{\ell=0}^n  \TPOL^+_{\ell,\psi}[X](t)\\
		&\quad\:\:\:\he\textstyle + \int_a^{t}\dd s_1\int_a^{s_1}\dd s_2 \: {\dots} \int_a^{s_{n}}\dd s_{n+1} \: (\bilbr{\psi(s_1)}\cp \dots \cp \bilbr{\psi(s_{n+1})})(\Add^+_\psi[X](s_{n+1}))\\[-3pt]
		&\textstyle\stackrel{\eqref{oidsoisoidsoidsoidsoi}}{=}
		\sum_{\ell=0}^n  \TPOL^+_{\ell,\psi}[X](t) + \big(\TPOL^{+}_{n+1,\psi}[X](t) + \Rest^+_{n+2,\psi}[X](t)\big)
		\\[-3pt]
		&\textstyle\stackrel{\phantom{\eqref{oidsoisoidsoidsoidsoi}}}{=}
		\sum_{\ell=0}^{n+1}  \TPOL^+_{\ell,\psi}[X](t) + \Rest^+_{n+2,\psi}[X](t)
	\end{align*}
	as well as
	\begin{align*}
		\Add^-_\psi[X](t)&\textstyle\stackrel{\phantom{\eqref{oidsoisoidsoidsoidsoi}}}{=} \sum_{\ell=0}^n  \TPOL^-_{\ell,\psi}[X](t) + \Rest^-_{n+1,\psi}[X](t)\\[-3pt]
		&\textstyle\stackrel{\phantom{\eqref{oidsoisoidsoidsoidsoi}}}{=} \sum_{\ell=0}^n  \TPOL^-_{\ell,\psi}[X](t)\\
		&\quad\:\:\:\he\textstyle + (-1)^{n+1}\cdot \int_a^{t}\dd s_1\int_a^{s_1}\dd s_2 \: {\dots} \int_a^{s_{n}}\dd s_{n+1} \: \Add^-_\psi[(\bilbr{\psi(s_{n+1})}\cp \dots \cp \bilbr{\psi(s_{1})})(X)](s_{n+1})\\[-3pt]
		&\textstyle\stackrel{\eqref{oidsoisoidsoidsoidsoii}}{=}
		\sum_{\ell=0}^n  \TPOL^-_{\ell,\psi}[X](t) + \big( \TPOL^{-}_{n+1,\psi}[X](t) + \Rest^-_{n+2,\psi}[X](t)\big)
		\\[-3pt]
		&\textstyle\stackrel{\phantom{\eqref{oidsoisoidsoidsoidsoi}}}{=}
		\sum_{\ell=0}^{n+1}  \TPOL^-_{\ell,\psi}[X](t) + \Rest^-_{n+2,\psi}[X](t).
	\end{align*}
The claim now follows by induction.
\end{proof}
\noindent
We recall the definition of an \AAE in \eqref{assaaas1}, and obtain the following corollary.
\begin{corollary}
\label{ljkdskjdslkjsda}
Let $\AES\subseteq\mg$ be an \AAE, $a<b$, and $\psi\in \DIDE_{[a,b]}$ with $\im[\psi]\subseteq \AES$ be given. 
Then, 
\begin{align}
\label{sdssddsds}
	\textstyle\Add^\pm_\psi[X]= \lim_n \sum_{\ell=0}^n \TPOL^\pm_{\ell,\psi}[X]= \sum_{\ell=0}^\infty \TPOL^\pm_{\ell,\psi}[X]
\end{align}
converges uniformly for each fixed $X\in \AES$. In particular, the following assertions hold:
\begingroup
\setlength{\leftmargini}{17pt}
{
\renewcommand{\theenumi}{\emph{\arabic{enumi})}} 
\renewcommand{\labelenumi}{\theenumi}
\begin{enumerate}
\item
\label{ljkdskjdslkjsda1}
	For $\vv\leq \ww$ as in \eqref{assaaas1}, we have
	\begin{align*}
		\vv\cp\Add^\pm_\psi[X]\leq \ww(X)\cdot \e^{\int_a^\bullet \ww(\psi(s))\:\dd s}\qquad\quad\forall\: X\in \AES.
	\end{align*}
\vspace{-18pt}
\item
\label{ljkdskjdslkjsda2}
	By \eqref{odaidaooipidadais}, for $Z\in \dom[\exp]\cap \AES$ we have 
 \begin{align*}
 	\textstyle \Ad_{\exp(t\cdot Z)^{\pm 1}}(X)=\Add^\pm_{\mathcal{C}_Z|_{[0,1]}}[X](t)= \sum_{\ell=0}^\infty \frac{t^\ell}{\ell!}\cdot \com{\pm Z}^\ell(X)\qquad\quad\forall\: X\in \AES,\: t\in [0,1].
 \end{align*}	
 \vspace{-22pt}
\end{enumerate}}
\endgroup	
\end{corollary}
\begin{proof}
	Equation \eqref{sdssddsds} is clear from the following two estimates:
\begingroup
\setlength{\leftmargini}{12pt}
\begin{itemize}
\item
	Let $\vv\leq \ww$  be as in \eqref{assaaas1}, as well as $\ww\leq \mm$ be as in Corollary \ref{opopsopsdopdsas} for $\qq\equiv \ww$ and $\compact\equiv \im[\innt_a^\bullet \psi]$ there. We obtain from Lemma \ref{hjfdhjfdhjfd} that
\begin{align*}
	\textstyle\vv\big(\Add^+_\psi[X](t)\textstyle-\sum_{\ell=0}^n \TPOL^+_{\ell,\psi}[X](t)\big)
	&\textstyle=\vv\big(\Rest^+_{n+1,\psi}[X](t)\big)\\
	&\textstyle\leq 
	\frac{(b-a)^{n+1}}{(n+1)!}\cdot \ww_\infty(\psi)^{n+1}\cdot \ww_\infty(\Add^+_\psi[X])\\
	&\textstyle\leq \frac{(b-a)^{n+1}}{(n+1)!}\cdot \ww_\infty(\psi)^{n+1}\cdot \mm(X)
\end{align*}
holds for each $t\in [a,b]$, $X\in \mg$, and $n\in \NN$.
\item
 	Let $\qq\leq \mm$ be as in Corollary \ref{opopsopsdopdsas} for  $\compact\equiv \im[\inv\cp \innt_a^\bullet \psi]$ there. Let furthermore $\mm\leq \ww$ be as in \eqref{assaaas1} for $\vv\equiv \mm$ there. We obtain from Lemma \ref{hjfdhjfdhjfd} that
\begin{align*}
	\textstyle\qq\big(\Add^-_\psi[X](t)\textstyle-\sum_{\ell=0}^n \TPOL^-_{\ell,\psi}[X](t)\big)
	\textstyle=\qq\big( \Rest^-_{n+1,\psi}[X](t) \big)\leq 
	\frac{(b-a)^{n+1}}{(n+1)!}\cdot \ww_\infty(\psi)^{n+1}\cdot \ww(X)
\end{align*}
holds for each $t\in [a,b]$, $X\in \mg$, and $n\in \NN$.
\end{itemize}
\endgroup
\noindent 
Now, Part \ref{ljkdskjdslkjsda2} is just clear from \eqref{sdssddsds}, and  
Part \ref{ljkdskjdslkjsda1} is  clear from \eqref{sdssddsds}  as well as
\begin{align*}
	\textstyle\vv\big(\TPOL^\pm_{\ell,\psi}[X](t)\big) &\textstyle\leq \int_a^{t}\dd s_1\int_a^{s_1}\dd s_2 \: {\dots} \int_a^{s_{\ell-1}}\dd s_\ell \: \ww(\psi(s_1))\cdot {\dots} \cdot \ww(\psi(s_{\ell}))\cdot \ww(X)
	\\
	&
	\textstyle =  \frac{1}{\ell !}\cdot (\int_a^t \ww(\psi(s))\:\dd s)^\ell\cdot \ww(X)
\end{align*}
for each $\ell\geq 1$. (The equality in the second step just follows by induction over $\ell\geq 1$, via taking the derivative of both expressions.)
\end{proof}
\begin{example}
\label{dsddsdsds}
Let $\NILS\subseteq \mg$ be a \NIL{q} for some $q\geq 2$.  
 Then, $\NILS$ is an \AAE,   
and Corollary \ref{ljkdskjdslkjsda} 
shows  
\begin{align*}
	\textstyle\Add^\pm_\psi[X]=\sum_{\ell=0}^{q-2}  \TPOL^\pm_{\ell,\psi}[X]
\end{align*} 
for each $X\in \NILS$ and $\psi\in \DIDE$ with $\im[\psi]\subseteq \NILS$. In particular, we have  
\begin{align*}
 	\textstyle\Ad_{\exp(t\cdot Z)^{\pm}}(X)= \sum_{\ell=0}^{q-2} \frac{t^\ell}{\ell!}\cdot \com{\pm Z}^\ell(X)\qquad\quad\forall\:  t\in [0,1]
 \end{align*}
 for each $X\in \NILS$ and $Z\in \dom[\exp]\cap \NILS$. 
\hspace*{\fill}\qed
\end{example}

\section{A Generalized BCDH Formula}
\label{kjdskjdskjdssd}
In the first part of this section, we generalize the Baker-Campbell-Dynkin-Hausdorff formula (for the exponential map) to the product integral (cf.\ Proposition \ref{kjfdkjfdkjfdjkfd}). In the second part, we apply this formula (together with the integral transformation introduced in Sect.\ \ref{kjdskjsdkjdskjs}) to express the product integral of nilpotent curves\footnote{Specifically, such $\mg$-valued curves whose image is a \NIL{q} for some $q\geq 2$, recall \eqref{assaaas2}.} (cf.\ Theorem \ref{kfdhjfdhjfdfdfdfd}) via  the exponential map. Various applications of the derived formula are discussed in Sect.\ \ref{kjdkjsjksddxccxocoioicxoiiooicxiocx}.

\subsection{A BCDH Formula for the Product Integral}
\label{sdsssdhsdhcxuzcxuzcxcx}
Let $\MX\in C^1([a,b],\mg)$ for $a<b$ be given, with 
$\im[\MX]\subseteq \dom[\exp]$. 
In this section, we consider the following two situations: 
\begingroup
\setlength{\leftmargini}{18pt}
{
\renewcommand{\theenumi}{{\Alph{enumi})}} 
\renewcommand{\labelenumi}{\theenumi}
\begin{enumerate}   
\item
\label{jhsdjhsdjhsd1}
	 $G$ is a Banach Lie group with $\norm{[X,Y]}\leq \norm{X}\cdot \norm{Y}$ for all $X,Y\in \mg$, and we have $\norm{\MX}_\infty<\ln(2)$. We set $V:=\mg$ as well as $\ind,\indj:=\infty$, and observe the following:
\vspace{-6pt}
\begingroup
\setlength{\leftmarginii}{14pt}{
\begin{itemize}
\item[$-$]
	$(\mg,\bil{\cdot}{\cdot})$ is asymptotic estimate.
\item[$-$]
	 $\exp\colon \mg\rightarrow \mg$ is defined and smooth, and a local diffeomorphism. In particular, there exists an identity neighbourhood $\OO\subseteq G$ on which $\exp$ is a diffeomorphism, such that $\norm{\exp^{-1}(g)}<\ln(2)$ holds for each $g\in \OO$.
\item[$-$]
	$G$ is $C^0$-regular (even $L^1$-regular by Theorem C in \cite{HGM}). 
In particular, there exists some $\BanC>0$, such that for all $\phi\in C^0([a,b],\mg)$ ($a<b$) and $\psi\in C^0([a',b'],\mg)$ ($a'<b'$) with $\int_a^b \norm{\phi(s)}\: \dd s + \int_{a'}^{b'} \norm{\psi(s)}\: \dd s< \BanC$, we have\footnote{Alternatively, combine Proposition 14.6 in \cite{HGGG} (or Example 2 in \cite{RGM}) with Proposition 2 in \cite{RGM}.}
\begin{align}
\label{aassasasasa}
	 \textstyle \innt_a^t\phi\cdot \innt_{a'}^\tau\psi \in \OO\qquad\quad\forall\: t \in[a,b],\: \tau\in [a',b']. 	
\end{align}	  
\end{itemize}
}
\endgroup 
\item
\label{jhsdjhsdjhsd2}
	$G$ is weakly $C^\infty$-regular, or $\exp\colon \mg\rightarrow G$ is defined ($\dom[\exp]=\mg$) and of class $C^1$. Moreover, $\NILS:=\im[\MX]\cup \im[\dot\MX]$ is a \NIL{q} for some $q\geq 2$. We set $V:=\CGen_1(\NILS)$ as well as $\ind:=q-1$ and $\indj:=q-2$, and recall that $V$ is a \NIL{q} by Remark \ref{dsdsdsdsdsv}.\ref{dsdsdsdsdsv2}.
\end{enumerate}}
\endgroup
\noindent
In both situations, the presumptions in Corollary \ref{fdjkjfdjkjfjkfd} are fulfilled. Moreover, by Example \ref{iureiuriureiuiure}, for $t\in [a,b]$ and $Y\in V$, we have
\begin{align}
\label{lkjfdjlfdkjfdlk1}
	\textstyle\Psi\big(\sum_{n=0}^\infty \frac{1}{n!}\cdot\com{\MX(t)}^n\big)(\Phi(\com{\MX(t)})(Y))&\textstyle=Y
	\\[2pt]
\label{lkjfdjlfdkjfdlk2}
	\textstyle\wt{\Psi}\big(\sum_{n=0}^\infty \frac{1}{n!}\cdot\com{-\MX(t)}^n\big)(Y)&\textstyle=\Psi\big(\sum_{n=0}^\infty \frac{1}{n!}\cdot\com{\MX(t)}^n\big)(Y)
\end{align}
with $\Phi(\com{\MX(t)})(Y)\in V$ (by Remark \ref{dsdsdsdsdsv} in the situation of \ref{jhsdjhsdjhsd2}). Here, for maps $\xi,\zeta\colon \mg\rightarrow \mg$ as well as $X\in \mg$, we set (convergence presumed)
\begin{align*}
	\textstyle \Psi(\xi)(X)&\textstyle:= \sum_{n=1}^\infty \frac{(-1)^{n-1}}{n}\cdot(\xi-\id_\mg)^{n-1}(X)\\
	\textstyle \wt{\Psi}(\xi)(X)&\textstyle:= \sum_{n=1}^\infty \frac{(-1)^{n-1}}{n}\cdot\big(\xi\cp (\xi-\id_\mg)^{n-1}\big)(X)\\[3pt]
	\textstyle \Phi(\zeta)(X)&\textstyle:= \sum_{n=0}^\infty\frac{1}{(n+1)!}\cdot \zeta^{n}(X).
\end{align*}
We obtain the following proposition:
\begin{proposition}
\label{kjfdkjfdkjfdjkfd}
Let $\phi\in \DIDE_{[a,b]}$, $g\in G$, $\MX\in C^1([a,b],\mg)$ 
be given, with 
$\im[\MX]\subseteq \dom[\exp]$ as well as 
\begin{align}
\label{vnbvcnbvcbnvcvcvcvc}
	\textstyle(\innt_a^t \phi)\cdot g =\exp(\MX(t))\qquad\quad\forall\: t\in [a,b].
\end{align}
Assume furthermore that we are in the situation of \ref{jhsdjhsdjhsd1} or \ref{jhsdjhsdjhsd2}. 
Then, we have $\im[\phi]\subseteq V$, and the following identities hold for each $t\in [a,b]$:
\begin{align}
\label{kfdkjfjkfdkj1}
	\textstyle \MX(t)-\MX(a)&\textstyle=  \int_a^t  \Psi\big(\Ad_{\innt_a^s\phi}\cp \Ad_{g}\big)(\phi(s))\:\dd s\\[2pt]
\label{kfdkjfjkfdkj3}
		\MX(t)-\MX(a)&\textstyle=  \int_a^t  \wt{\Psi}\big(\Ad_{g^{-1}}\cp \Ad_{[\innt_a^s\phi]^{-1}} \big)(\phi(s))\:\dd s\\
\label{kfdkjfjkfdkj2}
	\textstyle \MX(t)-\MX(a)&=\textstyle   \sum_{n=1}^{\ind} \frac{(-1)^{n-1}}{n}\cdot \int_a^t \big(\Ad_{\innt_a^s\phi}\cp \Ad_{g}-\id_\mg \big)^{n-1}(\phi(s))\: \dd s.
\end{align}
\end{proposition}
\begin{proof}
We observe the following: 
\begingroup
\setlength{\leftmargini}{12pt}{
\begin{itemize}
\item
	For $t\in [a,b]$, we have by Corollary \ref{ljkdskjdslkjsda}.\ref{ljkdskjdslkjsda2} (and Example \ref{dsddsdsds}) 
\begin{align}
\label{mncxmqwopopqwwopqwopqwopq}
	\textstyle \Ad_{\exp(\lambda\cdot \MX(t))^{\pm1}}|_V
\textstyle= \sum_{n=0}^\indj \frac{\lambda^n}{n!}\cdot\com{\pm\MX(t)}^n|_V.
\end{align}
\item
For $t\in [a,b]$, we have 
\begin{align}
\label{jdsjdskjdskjkjdskjds}
	\Phi(\com{\MX(t)})(\dot{\MX}(t))&\textstyle 
	= \sum_{n=0}^{\indj} \frac{1}{(n+1)!} \cdot \com{\MX(t)}^n(\dot{\MX}(t))\in V.
\end{align}
\item
For $t\in [a,b]$, we have by \eqref{lkjfdjlfdkjfdlk1},  \eqref{lkjfdjlfdkjfdlk2}, and \eqref{jdsjdskjdskjkjdskjds} that
\begin{align}
\label{lkdslkdslklkdsdsdsds}
\begin{split}
	\dot{\MX}(t)&= \textstyle\Psi\big(\sum_{n=0}^\infty \frac{1}{n!}\cdot\com{\MX(t)}^n\big)(\Phi(\com{\MX(t)})(\dot{\MX}(t)))\\
	&= \textstyle\wt{\Psi}\big(\sum_{n=0}^\infty\frac{1}{n!}\cdot\com{-\MX(t)}^n\big)(\Phi(\com{\MX(t)})(\dot{\MX}(t))).
\end{split}
\end{align} 
\end{itemize}}
\endgroup
\noindent
Set $\alpha:=\exp\cp \MX \in C^1([a,b],\mg)$. Then, given $t\in [a,b]$, we obtain from \eqref{vnbvcnbvcbnvcvcvcvc} (first step), Corollary \ref{fdjkjfdjkjfjkfd} (second step), \eqref{pofdpofdpofdsddsdsfd} (third step),  
\eqref{mncxmqwopopqwwopqwopqwopq} (fourth step), $C^0$-continuity of the Riemann integral in the situation of \ref{jhsdjhsdjhsd1} (fifth step), as well as \eqref{jdsjdskjdskjkjdskjds} (fifth step) that 
\begin{align*}
	\textstyle\dd_e\RT_{\alpha(t)}(\phi(t))&=\dot\alpha(t)\\
	&\textstyle= \dd_e\LT_{\alpha(t)}\big(\int_0^1 \Ad_{\exp(-s\cdot \MX(t))}(\dot{\MX}(t)) \:\dd s\big)\\
						&\textstyle= \dd_e\RT_{\alpha(t)}\big(\int_0^1 \Ad_{\exp((1-s)\cdot \MX(t))}(\dot{\MX}(t)) \:\dd s\big)\\
						&\textstyle= \dd_e\RT_{\alpha(t)}\big(\int_0^1 \sum_{n=0}^\indj \frac{(1-s)^n}{n!} \cdot \com{\MX(t)}^n(\dot{\MX}(t))\:\dd s\big)\\
						&\textstyle= \dd_e\RT_{\alpha(t)}\big(\sum_{n=0}^\indj \frac{1}{(n+1)!} \cdot \com{\MX(t)}^n(\dot{\MX}(t))\big).
\end{align*}
Together with \eqref{jdsjdskjdskjkjdskjds}, this shows 
\begin{align}
\label{oioidsoidsds}
	\phi(t)=\Phi(\com{\MX(t)})(\dot{\MX}(t))\in V\qquad\quad \forall\: t\in [a,b].
\end{align} 
We obtain from \eqref{isdsdoisdiosd}, \eqref{lkdslkdslklkdsdsdsds}, \eqref{oioidsoidsds}, and \eqref{mncxmqwopopqwwopqwopqwopq} the following identities for each $t\in [a,b]$:
\begin{align}
\label{oidsoidsoidsisxccxc}
\begin{split}
	\MX(t)-\MX(a)&\textstyle=\int_a^t  \Psi(\Ad_{\exp(\MX(s))})(\phi(s))\:\dd s\\[3pt]
	\MX(t)-\MX(a)&\textstyle= \int_a^t  \wt{\Psi}(\Ad_{\exp(-\MX(s))})(\phi(s))\:\dd s\\[1pt]
	&\textstyle =	\int_a^t  \wt{\Psi}(\Ad_{\exp(\MX(s))^{-1}})(\phi(s))\:\dd s.
\end{split}
\end{align} 
Together with \eqref{vnbvcnbvcbnvcvcvcvc}, this implies \eqref{kfdkjfjkfdkj1} and \eqref{kfdkjfjkfdkj3}. We now finally have to prove \eqref{kfdkjfjkfdkj2}: 
\begingroup
\setlength{\leftmargini}{12pt}{
\begin{itemize}
\item
	Assume that we are in the situation of \ref{jhsdjhsdjhsd2}. Since $\im[\phi]\subseteq V$ holds by \eqref{oioidsoidsds}, we obtain from \eqref{isdsdoisdiosd}, \eqref{lkdslkdslklkdsdsdsds}, \eqref{oioidsoidsds} (first step), $\im[\phi]\subseteq V$ (second step),  linearity of the Riemann integral (third step), \eqref{mncxmqwopopqwwopqwopqwopq} (fourth step), and \eqref{vnbvcnbvcbnvcvcvcvc} (fifth step) that
	\begin{align*}
		\MX(t)-\MX(a)&\textstyle=\int_a^t  \Psi\big(\sum_{\ell=0}^\infty \frac{1}{\ell!}\cdot\com{\MX(s)}^\ell\big)(\phi(s))\:\dd s\\[2pt]
		&\textstyle =\int_a^t  \sum_{n=1}^\ind \frac{(-1)^{n-1}}{n}\cdot\big(\sum_{\ell=1}^\indj \frac{1}{\ell!}\cdot\com{\MX(s)}^\ell\big)^{n-1}(\phi(s))\:\dd s\\[2pt]
		&\textstyle =\sum_{n=1}^\ind \frac{(-1)^{n-1}}{n}\cdot\int_a^t  \big(\sum_{\ell=1}^\indj \frac{1}{\ell!}\cdot\com{\MX(s)}^\ell\big)^{n-1}(\phi(s))\:\dd s\\[2pt]
		&\textstyle =\sum_{n=1}^\ind \frac{(-1)^{n-1}}{n}\cdot\int_a^t  \big(\Ad_{\exp(\MX(s))}-\id_\mg\big)^{n-1}(\phi(s))\:\dd s\\[2pt]
		&=\textstyle   \sum_{n=1}^{\ind} \frac{(-1)^{n-1}}{n}\cdot \int_a^t \big(\Ad_{\innt_a^s\phi}\cp \Ad_{g}-\id_\mg \big)^{n-1}(\phi(s))\: \dd s.
	\end{align*}
\item
 Assume that we are in the situation of \ref{jhsdjhsdjhsd1}.   
Then, \eqref{mncxmqwopopqwwopqwopqwopq} (first step) together with $\norm{\MX}_\infty<\ln(2)$ (last step) implies
\begin{align}
\label{kdsoidsoioidsdsdssd}
\begin{split}
	\textstyle\norm{(\Ad_{\exp(\MX(t))}-\id_\mg)^k(\phi(t))}&\textstyle\leq (\norm{\sum_{\ell=1}^\infty \frac{1}{\ell!}\cdot\com{\MX(t)}^\ell}_\op)^k\cdot \norm{\phi}_\infty\\
&\textstyle\leq (\sum_{\ell=1}^\infty \frac{1}{\ell!}\cdot\norm{\com{\MX(t)}^\ell}_\op)^k \cdot \norm{\phi}_\infty\\
&	\textstyle\leq \big(\e^{\norm{\MX}_\infty}-1\big)^k \cdot\norm{\phi}_\infty\\
	&\textstyle<\norm{\phi}_\infty
\end{split}
\end{align}
for each $t\in [a,b]$ and $k \geq 1$. We obtain from \eqref{oidsoidsoidsisxccxc} (first step), $C^0$-continuity of the Riemann integral and \eqref{kdsoidsoioidsdsdssd} (third step), as well as \eqref{vnbvcnbvcbnvcvcvcvc} (fourth step) that
\begin{align*}
	\MX(t)-\MX(a)&\textstyle=\int_a^t  \Psi(\Ad_{\exp(\MX(s))})(\phi(s))\:\dd s\\
	&\textstyle=\int_a^t  \sum_{n=1}^\infty \frac{(-1)^{n-1}}{n}\cdot(\Ad_{\exp(\MX(s))})(\phi(s)-\id_\mg)^{n-1}(\phi(s)) \:\dd s\\
	&\textstyle= \sum_{n=1}^\infty \frac{(-1)^{n-1}}{n}\cdot \int_a^t \big(\Ad_{\exp(\MX(s))}-\id_\mg \big)^{n-1}(\phi(s))\: \dd s\\
	&\textstyle= \sum_{n=1}^\infty \frac{(-1)^{n-1}}{n}\cdot \int_a^t \big(\big(\Ad_{\innt_a^s\phi}\cp \Ad_{g}\big)-\id_\mg \big)^{n-1}(\phi(s))\: \dd s
\end{align*}
holds for each $t\in [a,b]$.\qedhere
\end{itemize}}
\endgroup
\end{proof}
\begin{proposition}
\label{lkjdfkfdjkdfaskkk}
Assume that $G$ is a Banach Lie group, and let $\BanC>0$ be as in \ref{jhsdjhsdjhsd1}. Let $\phi\in C^0([a,b],\mg)$ ($a<b$) and $\psi\in C^0([a',b'])$ ($a'<b'$) be given, with $\int_a^b \norm{\phi(s)}\: \dd s+ \int_{a'}^{b'} \norm{\psi(s)}\: \dd s< \BanC$. Then, we have 
\begin{align*}
	\textstyle\innt_a^t \phi\cdot \innt_{a'}^{b'}\psi
	=\textstyle \exp\big(\MX_\psi(b') + \MX_{\phi,\psi}(t)\big) \qquad\quad\forall\: t\in [a,b],
\end{align*} 
provided that for $t'\in [a',b']$ and $t\in [a,b]$ we set
\begin{align*}
	\MX_\psi(t')&:=
	\textstyle\sum_{n=1}^{\infty} \frac{(-1)^{n-1}}{n}\cdot\int_{a'}^{t'} \big(\sum_{\ell=1}^{\infty} \TPOL^+_{\ell,\psi}[s]\big)^{n-1} (\psi(s))\: \dd s\\
	\MX_{\phi,\psi}(t)
	&:=\textstyle\sum_{n=1}^{\infty} \frac{(-1)^{n-1}}{n}\cdot \int_a^t \big(\big(\sum_{\ell=0}^{\infty} \TPOL^+_{\ell,\phi}[s]\big)\cp \big(\sum_{\ell=0}^{\infty} \TPOL^+_{\ell,\psi}[b']\big)-\id_\mg\big)^{n-1}(\phi(s))\: \dd s.
\end{align*}	
\end{proposition}
\begin{proof}
By definition, there exist open intervals $I,I'\subseteq \RR$ with $[a,b]\subseteq I$ and $[a',b']\subseteq I'$, as well as $\mu\in C^1(I,G)$ and  $\nu\in C^1(I',G)$, with $\Der(\mu)|_{[a,b]}=\phi$ and $\Der(\nu)|_{[a',b']}=\psi$. We make the following modifications to these curves:
\begingroup
\setlength{\leftmargini}{12pt}{
\begin{itemize}
\item
	We replace $\mu$ by $\mu\cdot \mu(a)^{-1}$ to ensure $\mu|_{[a,b]}=\innt_{a}^\bullet\phi$.
\item
	We replace $\nu$ by $\nu\cdot \nu(a')^{-1}$ to ensure $\nu|_{[a',b']}=\innt_{a'}^\bullet\psi$.
\item
	We define $\alpha:=\mu\cdot \nu(b')$, and observe $\alpha|_{[a,b]}=\innt_a^\bullet\phi\cdot \innt_{a'}^{b'}\psi$.
		\vspace{2pt} 
\item
	We shrink $I$ around $[a,b]$ and $I'$ around $[a',b']$, to ensure $\mu(I)\cdot \nu(I')\subseteq \OO$ (recall \eqref{aassasasasa}).
\end{itemize}}
\endgroup
\noindent
Then, the maps 
\begin{align*}
	\MX\colon [a,b]\ni t\mapsto \exp^{-1}(\alpha(t))\in \mg \qquad\quad\text{and}\qquad\quad \MX'\colon [a',b']\ni t'\mapsto \exp^{-1}(\nu(t'))\in \mg
\end{align*}
are defined and of class $C^1$, with 
\begin{align*}
	\textstyle\MX(a)=\MX'(b'),\qquad \MX'(a')=0,\qquad\innt_{a'}^{\bullet}\psi=\exp\cp\he\MX',\qquad\innt_a^\bullet\phi\cdot \innt_{a'}^{b'}\psi=\exp\cp\he \MX.
\end{align*}
We conclude the following:
\begingroup
\setlength{\leftmargini}{12pt}{
\begin{itemize}
\item
	Equation \eqref{kfdkjfjkfdkj2} in 
Proposition \ref{kjfdkjfdkjfdjkfd} for $g\equiv e$ there, together with Corollary \ref{ljkdskjdslkjsda} shows $\MX'=\MX_\psi$. 
\item
	Equation \eqref{kfdkjfjkfdkj2} in 
Proposition \ref{kjfdkjfdkjfdjkfd} for $g\equiv \innt_{a'}^{b'}\psi$ there, together with Corollary \ref{ljkdskjdslkjsda} and the previous point shows 
\begin{align*}
	\MX-\MX_\psi(b')=\MX-\MX'(b')=\MX-\MX(a)=\MX_{\phi,\psi},
\end{align*}
which proves the claim.\qedhere
\end{itemize}}
\endgroup
\end{proof}
\begin{corollary}
\label{lkjdfkfdjkdfas}
Assume that $G$ is a Banach Lie group, and let $\BanC>0$ be as in \ref{jhsdjhsdjhsd1}. Then, for  $\phi\in C^0([a,b],\mg)$ ($a<b$) with $\int_a^b \norm{\phi(s)}\: \dd s< \BanC$, we have 
\begin{align*}
	\textstyle\exp^{-1}\!\big(\innt_a^t \phi\big)
	=\textstyle  \sum_{n=1}^\infty \frac{(-1)^{n-1}}{n}\cdot\int_a^t \big(\sum_{\ell=1}^\infty \TPOL^+_{\ell,\phi}[s]\big)^{n-1} (\phi(s))\: \dd s \qquad\quad\forall\: t\in [a,b].
\end{align*}
\end{corollary}
\begin{proof}
	Apply Proposition \ref{lkjdfkfdjkdfaskkk} with $\psi=0$.
\end{proof}
\begin{example}[The BCDH Formula]
\label{lkfdjhdfkdffdhdf}
Assume that we are in the situation of Proposition \ref{kjfdkjfdkjfdjkfd}, with 
\begin{align}
\label{fdhjfdhjfdhjfdfdfd}
	[a,b]=[0,1],\qquad\phi=-\mathcal{C}_Y|_{[0,1]},\qquad
	g=\exp(-X),\qquad
	\MX(0)=-X
\end{align}
for certain $X,Y\in \mg$ and $\MX\in C^1([0,1],\mg)$. We obtain the well-known\footnote{For the finite-dimensional case, confer, e.g., Proposition 3.4.4 in \cite{HIGN},  with $\Psi\equiv \wt{\Psi}$ there.}  integral representation of the Baker-Campbell-Dynkin-Hausdorff formula:
\begin{align}
\label{fdhjfdkjfdkjfdkjfdkjfdjkfd}
	\textstyle\exp(X)\cdot \exp(t\cdot Y)= \exp\big(X + \int_0^t  \wt{\Psi}(\Ad_{\exp(X)}\cp \Ad_{\exp(s\cdot Y)})(Y)\:\dd s\big)\qquad\quad\forall\: t\in [0,1].
\end{align}
\begin{proof}[Proof of Equation \eqref{fdhjfdkjfdkjfdkjfdkjfdjkfd}]
By \eqref{odaidaooipidadais}, we have $\textstyle\innt_0^t \phi=\exp(-t\cdot Y)$ for each $t\in [0,1]$, hence 
\begin{align*}
	\textstyle\exp(X)\cdot \exp(t\cdot Y)=([\innt_0^t\phi]\cdot g)^{-1}\stackrel{\eqref{vnbvcnbvcbnvcvcvcvc}}{=}\exp(\MX(t))^{-1}=\exp(-\MX(t)).
\end{align*}
Then, \eqref{kfdkjfjkfdkj3} in Proposition \ref{kjfdkjfdkjfdjkfd} shows 
\begin{align*}
	-(\MX(t)- \MX(0))\textstyle=  - \int_0^t  \wt{\Psi}\big(\Ad_{g^{-1}}\cp \Ad_{[\innt_0^s\phi]^{-1}}\big)(-Y)\:\dd s
	= \int_0^t  \wt{\Psi}\big(\Ad_{\exp(X)}\cp \Ad_{\exp(s\cdot Y)}\big)(Y)\:\dd s
\end{align*}
for each $t\in [0,1]$, from which the claim follows. 
\end{proof}
\noindent
For instance, the assumptions in \eqref{fdhjfdhjfdhjfdfdfd} are fulfilled in the following situations:
\begingroup
\setlength{\leftmargini}{12pt}
\begin{itemize}
\item
$G$ is a Banach Lie group, with $\norm{X}+\norm{Y}<\BanC$. Indeed, by \ref{jhsdjhsdjhsd1}, then there exists $\MX\in C^1([0,1],\mg)$ with $\MX(0)=-X$, such that \eqref{vnbvcnbvcbnvcvcvcvc} holds for $\phi$ and $g$ as in \eqref{fdhjfdhjfdhjfdfdfd}.
\item
$G$ is weakly $C^k$-regular for $k\in \NN\cup\{\lip,\infty\}$, and $\{X,Y\}$ is a \NIL{q} for some $q\geq 2$. Indeed, we will reconsider this situation in 
Corollary \ref{kjdskjdsjkjdskjdsiuewiuiuweiuewuiew} in Sect.\ \ref{dsdsdsddssddsds}. There, we construct some $\MX\in C^1([0,1],\mg)$  with $\MX(0)=-X$, such that \eqref{vnbvcnbvcbnvcvcvcvc} holds for $\phi$ and $g$ as in \eqref{fdhjfdhjfdhjfdfdfd}. \hspace*{\fill}\qed
\end{itemize}
\endgroup
\end{example}

\subsection{The Product Integral of Nilpotent Curves}
\label{dsdsdsddssddsds}
Assume that $G$ is weakly $C^k$-regular for $k\in \NN\cup \{\lip,\infty\}$, and let $\phi\in C^k([a,b],\mg)$ for $a<b$ as well as $\psi\in C^k([a',b'],\mg)$ for $a'<b'$ be given.
\begingroup
\setlength{\leftmargini}{12pt}
\begin{itemize}
\item
	If $\im[\psi]$ is a \NIL{q} for some $q\geq 2$, then for each $t\in [a',b']$ we set
\begin{align}
\label{hjfdhjfdhjfdd1}
\begin{split}
	\MX_\psi(t):=&\: \textstyle\sum_{n=1}^{q-1} \frac{(-1)^{n-1}}{n}\cdot\int_{a'}^t \big(\Ad_{\innt_{a'}^{s}\psi}-\id_\mg\big)^{n-1} (\psi(s))\: \dd s\\
	=&\:\textstyle\sum_{n=1}^{q-1} \frac{(-1)^{n-1}}{n}\cdot\int_{a'}^t \big(\sum_{\ell=1}^{q-2} \TPOL^+_{\ell,\psi}[s]\big)^{n-1} (\psi(s))\: \dd s.\hspace{102pt}
\end{split}
\end{align}	
The second line is due to \eqref{pofdpofdpofdsddsdsfd} and Example \ref{dsddsdsds}.
\item
	If $\im[\phi]\cup \im[\psi]$ is a \NIL{q} for some $q\geq 2$, then for each $t\in [a,b]$ we set
\begin{align}
\label{hjfdhjfdhjfdd2}
\begin{split}
	\textstyle\MX_{\phi,\psi}(t)
	:=&\:\textstyle\sum_{n=1}^{q-1} \frac{(-1)^{n-1}}{n}\cdot \int_a^t \big(\Ad_{\innt_a^s\phi}\cp \Ad_{\innt_{a'}^{b'}\psi}-\id_\mg \big)^{n-1}(\phi(s))\: \dd s\\
	=&\:\textstyle\sum_{n=1}^{q-1} \frac{(-1)^{n-1}}{n}\cdot \int_a^t \big(\big(\sum_{\ell=0}^{q-2} \TPOL^+_{\ell,\phi}[s]\big)\cp \big(\sum_{\ell=0}^{q-2} \TPOL^+_{\ell,\psi}[b']\big)-\id_\mg\big)^{n-1}(\phi(s))\: \dd s.
\end{split}
\end{align}	
The second line is due to \eqref{pofdpofdpofdsddsdsfd} and Example \ref{dsddsdsds}.
\end{itemize}
\endgroup
\noindent
Notably, both maps take values in $\mg$ by Lemma \ref{Adlip}, and are thus of class $C^1$ by \eqref{oidsoidoisdoidsoioidsiods}. 
In this section, we combine Proposition \ref{kjfdkjfdkjfdjkfd} with  the integral transformation introduced in Sect.\ \ref{kjdskjsdkjdskjs} to prove the following statement.
\begin{theorem}
\label{kfdhjfdhjfdfdfdfd}
	Assume that $G$ is weakly $C^k$-regular for $k\in \NN\cup\{\lip,\infty\}$, and let $\phi\in C^k([a,b],\mg)$, $\psi\in C^k([a',b'],\mg)$ be given, such that $\NILS:=\im[\phi]\cup\im[\psi]$ is a \NIL{q} for some $q\geq 2$. Then, 
\begin{align*}
	\textstyle\innt_a^t \phi\cdot \innt_{a'}^{b'}\psi
	=\textstyle \exp\big(\MX_\psi(b') + \MX_{\phi,\psi}(t)\big) \qquad\quad\forall\: t\in [a,b].
\end{align*} 
\end{theorem}
\noindent
The proof of Theorem \ref{kfdhjfdhjfdfdfdfd} is given in Sect.\ \ref{dszudszidsudsiuiudsidsd}. We now first want to discuss some consequences.

\subsubsection{Some Consequences}
\label{kjdkjsjksddxccxocoioicxoiiooicxiocx}
Setting $\psi=0$ in Theorem \ref{kfdhjfdhjfdfdfdfd}, we immediately obtain the following statement.
\begin{corollary}
\label{dchdfhdfjhdf}
	Assume that $G$ is weakly $C^k$-regular for $k\in \NN\cup\{\lip,\infty\}$, and let $\phi\in C^k([a,b],\mg)$ be given such that $\im[\phi]$ is a \NIL{q} for some $q\geq 2$. Then, we have 
\begin{align*}
	\textstyle\innt_a^t \phi
	=\textstyle \exp\big(\sum_{n=1}^{q-1} \frac{(-1)^{n-1}}{n}\cdot\int_a^t \big(\sum_{\ell=1}^{q-2} \TPOL^+_{\ell,\phi}[s]\big)^{n-1} (\phi(s))\: \dd s\big) \qquad\quad\forall\: t\in [a,b].
\end{align*}
\end{corollary}
\begin{example}
\label{kfdkjfdkjfkjd}
Let $G$ be abelian, as well as weakly $C^k$-regular for some $k\in \NN\cup\{\lip,\infty\}$. Then, $\mg$ is a \NIL{2}, and Corollary \ref{dchdfhdfjhdf} recovers formula \eqref{ddsnmvcnmvcnmvcnm} in Example \ref{kfdkjfdkjfkjda}.  
\hspace*{\fill}\qed
\end{example}
\noindent
Moreover, combining Theorem \ref{kfdhjfdhjfdfdfdfd} with Example \ref{lkfdjhdfkdffdhdf}, we obtain the BCDH formula:
\begin{corollary}
\label{kjdskjdsjkjdskjdsiuewiuiuweiuewuiew}
	Assume that $G$ is weakly $C^k$-regular for $k\in \NN\cup\{\lip,\infty\}$. Let furthermore $X,Y\in\mg$ be given, such that $\{X,Y\}$ is a \NIL{q} for some $q\geq 2$. Then, we have
\begin{align*}
	\textstyle\exp(X)\cdot \exp(t\cdot Y)= \exp\big(X + \int_0^t  \wt{\Psi}\big(\Ad_{\exp(X)}\cp \Ad_{\exp(s\cdot Y)}\big)(Y)\:\dd s\big)\qquad\quad\forall\: t\in [0,1]. 
\end{align*}
\end{corollary}
\begin{proof}
	Set $\phi:=-\mathcal{C}_Y|_{[0,1]}$ as well as $\psi:= -\mathcal{C}_X|_{[0,1]}$, and let $\MX_\psi,\MX_{\phi,\psi}$ be as in \eqref{hjfdhjfdhjfdd1},\eqref{hjfdhjfdhjfdd2}. Then, $\MX\colon [0,1]\ni t\mapsto  \MX_\psi(1) + \MX_{\phi,\psi}(t) \in \mg$ is of class $C^1$, with $\MX(0)=\MX_\psi(1)=-X$, as well as
	\begin{align*}
		\textstyle\innt_0^t \phi \cdot \innt_0^1\psi = \exp(\MX(t))\qquad\quad \forall\: t\in [0,1]
\end{align*}
by Theorem \ref{kfdhjfdhjfdfdfdfd}. 
The claim is thus clear from Example \ref{lkfdjhdfkdffdhdf}.	 
\end{proof}
\noindent 
Next, we define 
\begin{align*}
	\Gamma^\infty_{\mathrm{Nil}}:=\{\phi\in \DIDE^\infty_{[0,1]}\:|\: \im[\phi] \:\text{is a \NIL{q} for some}\: q\geq 2\}\qquad\text{and set}\qquad G^\infty_{\mathrm{Nil}}:=\evol_\infty(\Gamma^\infty_{\mathrm{Nil}}).
\end{align*}
We observe the following:
\begin{lemma}
\label{kjkjdskjdsjkdsdssd}
If $G$ is weakly $C^\infty$-regular, then $G^\infty_{\mathrm{Nil}}=\im[\exp]$ holds.
\end{lemma}
\begin{proof}
	Obviously, have $\expal(\dom[\exp])\subseteq \Gamma^\infty_{\mathrm{Nil}}$, hence $\im[\exp]\subseteq G^\infty_{\mathrm{Nil}}$. Moreover, Corollary \ref{dchdfhdfjhdf} shows $G^\infty_{\mathrm{Nil}}\subseteq\im[\exp]$, which proves the claim.	
\end{proof}
\begin{lemma}
\label{kjfdkkjfdkjfjfdjk}
Assume that $G$ is connected. Then, $\evol_\infty$ is surjective.  
\end{lemma}
\begin{proof}
Let $\mathcal{A}\subseteq G$ denote subgroup of all $g\in G$, such that there exist $\mu_1,\dots,\mu_n\in C^\infty([0,1],G)$ ($n\in \NN$), with $g=\mu_n(1)\cdot {\dots}\cdot \mu_1(1)$ as well as $\mu_\ell(0)=e$ for $\ell=1,\dots,n$.  
Set furthermore $\OO:= \U\cap\inv(\U)$. Since $G$ is connected, we have $G=\bigcup_{n\geq 1} \OO^n$. Since $\OO\subseteq \U\subseteq \mathcal{A}$ holds,\footnote{For $g \in \U$, consider $C^\infty([0,1],G)\ni\mu_1\colon [0,1]\ni t\mapsto \chartinv(t\cdot \chart(g))\in G$.} we obtain $G=\mathcal{A}$. Let now $g\in G$ be fixed, and choose $\mu_1,\dots,\mu_n\in C^\infty([0,1],\mg)$ with $\mu_1(0),\dots,\mu_n(0)=e$ and $g=\mu_n(1)\cdot {\dots}\cdot \mu_1(1)$. We set $\phi_\ell:=\Der(\mu_\ell)\in \DIDE^\infty_{[0,1]}$ for $\ell=1,\dots,n$, and observe
\begin{align*}
	\textstyle g=\mu_n(1)\cdot {\dots}\cdot \mu_1(1)= \innt_0^1 \phi_n \cdot {\dots}\cdot \innt_0^1 \phi_1. 
\end{align*}	
Applying \ref{kdsasaasassaas} inductively, we obtain (recall Remark \ref{dsdsdsdsjkjkkjkjkjjkkjkj}) 
\begin{align*}
	\textstyle g\textstyle\stackrel{\phantom{\ref{kdsasaasassaas}}}{=}\innt_0^1 \phi_n\cdot{\dots}\cdot \innt_0^1 \phi_1 = \innt_0^1\he \underbrace{\textstyle\phi_n\star(\phi_{n-1}\star(\dots (\phi_2\star \phi_1))\dots)}_{\in \: \DIDE^\infty_{[0,1]}}
\end{align*} 
\vspace{-17pt}

\noindent 
which proves the claim.
\end{proof}
\noindent
Combining Lemma \ref{kjkjdskjdsjkdsdssd} with Lemma \ref{kjfdkkjfdkjfjfdjk}, we obtain the following proposition.
\begin{proposition}
\label{kjkjdskjdsjkdsds}
Assume that $G$ is weakly $C^\infty$-regular and connected. If $\Gamma^\infty_{\mathrm{Nil}}= C^\infty([0,1],\mg)$ holds,  
then $\exp\colon \mg\rightarrow G$ is surjective.
\end{proposition}
\begin{proof}
Lemma \ref{kjkjdskjdsjkdsdssd} shows $\im[\exp]=G^\infty_{\mathrm{Nil}}=\evol_\infty(\Gamma^\infty_{\mathrm{Nil}})$, with $\evol_\infty(\Gamma^\infty_{\mathrm{Nil}})=\evol_\infty(C^\infty([0,1],\mg))$ by assumption. 
Since  $\evol_\infty$ is surjective by Lemma \ref{kjfdkkjfdkjfjfdjk}, the claim follows.
\end{proof}
\noindent
Proposition \ref{kjkjdskjdsjkdsds} immediately implies the following statement.
\begin{corollary}
\label{iufdiufdiuiufd}
Assume that $G$ is weakly $C^\infty$-regular and connected. If $(\mg,\bil{\cdot}{\cdot})$ is nilpotent, then $\exp$ is surjective.
\end{corollary}

\begin{remark}
\label{dfkljdfljkdfkljdfjkl}
\begingroup
\setlength{\leftmargini}{17pt}
{
\renewcommand{\theenumi}{\emph{\arabic{enumi})}} 
\renewcommand{\labelenumi}{\theenumi}
\begin{enumerate}
\item
\label{dfkljdfljkdfkljdfjkl1}
	The statement in Corollary \ref{iufdiufdiuiufd} follows from Theorem IV.2.6 in \cite{KHN2} provided we replace ``weakly $C^\infty$-regular'' by ``$\mg$ is Mackey complete and $\exp$ is smooth''.
\item
\label{dfkljdfljkdfkljdfjkl2}
Let $\mg$ be finite-dimensional. Then, the hypothesis $\Gamma^\infty_{\mathrm{Nil}}= C^\infty([0,1],\mg)$  
in Proposition \ref{kjkjdskjdsjkdsds} 
is equivalent to nilpotency of $\mg$.

In fact, the one implication is evident. Assume thus that $\Gamma^\infty_{\mathrm{Nil}}= C^\infty([0,1],\mg)$ holds, and let $\{X_0,\dots,X_{n-1}\}\subseteq  \mg$ be a base of $\mg$. Let furthermore $\rho\colon [0,1/n]\rightarrow [0,\infty)$ be a bump function, i.e., $\rho$ is smooth with 
\begin{align*}
	\rho|_{(0,1/n)}>0\qquad\quad\text{as well as}\qquad\quad \rho^{(\ell)}(0)=0=\rho^{(\ell)}(1/n)\qquad \forall\:\ell\in \NN. 
\end{align*}
We define $\phi\colon [0,1]\rightarrow \mg$ by $\phi(0):=0$ and
\begin{align*}
	\phi|_{(\ell/n,(\ell+1)/n]}\colon (\ell/n,(\ell+1)/n]\ni t \mapsto \rho(t-\ell/n)\cdot X_\ell\in \mg\qquad\quad\forall\:  \ell=0,\dots,n-1.
\end{align*}
Then, $\phi$ is easily seen to be smooth (for details confer, e.g., Lemma 24 and Appendix B.2 in \cite{RGM}). By assumption, there exists some $q\geq 2$, such that $\NILS:=\im[\phi]$ is a \NIL{q}. Then, $\SP_1(\NILS)$ is a \NIL{q} by Remark \ref{dsdsdsdsdsv}, with $\{X_0,\dots,X_{n-1}\}\subseteq \SP_1(\NILS)$, hence $\mg=\SP_1(\NILS)$.
	\hspace*{\fill}\qed
\end{enumerate}}
\endgroup		
\end{remark}
\noindent
We finally want to mention that Proposition \ref{kjkjdskjdsjkdsds} also provides an easy   
proof of the following well-known result in the finite-dimensional context.
\begin{corollary}
\label{fdfddfdfdfd}
If $G$ is compact and connected, then $\exp$ is surjective.   
\end{corollary}
\begin{proof}
	Let $T\subseteq G$ be a maximal torus with Lie algebra $\mt$. Since $T$ is abelian,  Corollary \ref{iufdiufdiuiufd} shows that $\exp|_\mt\colon \mt\rightarrow T$ is surjective. Since $G$ is covered by maximal tori (confer, e.g., Theorem 12.2.2.(iii) in \cite{HIGN}), the claim follows.
\end{proof}
\begin{remark}
The proof of Corollary \ref{fdfddfdfdfd} can be simplified. Specifically, surjectivity of 
$\exp|_\mt\colon \mt\rightarrow T$ also follows from Lemma \ref{kjfdkkjfdkjfjfdjk} and formula \eqref{ddsnmvcnmvcnmvcnm} (as $T$ is abelian). Notably, in finite dimensions, \eqref{ddsnmvcnmvcnmvcnm} follows easily from smoothness of the exponential map: 
\begin{proof}[Proof of \eqref{ddsnmvcnmvcnmvcnm} in the finite-dimensional case]
Let $G$ be abelian and finite-dimensional, and let $\phi\in C^0([a,b],\mg)$ be given. Let $\psi\in C^0([a,b+\varepsilon),\mg)$ for $\varepsilon>0$ be a continuous extension of $\phi$. Then, 
\begin{align*}
	\textstyle\mu\colon [a,b+\varepsilon)\ni t \mapsto \exp\big(\!\int_a^t\psi(s)\:\dd s\big)\in G
\end{align*}
is of class $C^1$ as $\exp$ is smooth. Since $G$ is abelian, we have 
\begin{align*}
	\textstyle(\mu(t+h)\cdot \mu(t)^{-1})=\exp\big(\int_t^{t+h}\psi(s) \:\dd s\big)\qquad\quad\forall\: t\in[a,b], \: 0<h<\varepsilon.
\end{align*} 
Since $\dd_e\exp=\id_\mg$ holds, we obtain $\Der(\mu|_{[a,b]})=\phi$, hence 
\begin{align*}
	\textstyle\innt_a^\bullet \phi=\mu|_{[a,b]}=\exp(\int_a^\bullet\phi(s)\:\dd s)
\end{align*}
by injectivity of $\Der|_{C^1_*([a,b],G)}$.
\end{proof}
\end{remark}

\subsubsection{The Proof of Theorem \ref{kfdhjfdhjfdfdfdfd}}
\label{dszudszidsudsiuiudsidsd}
In this subsection, we prove Theorem \ref{kfdhjfdhjfdfdfdfd}. Our argumentation is based on the following (straightforward but technical) lemma that is proven in Sect.\ \ref{llkdsljdsdsuidsiudsoiudsods}.  
\begin{lemma}
\label{mnsdmnsdmnsdmndsmnsd}
Let $G$ be weakly $C^k$-regular for $k\in \NN\cup\{\lip,\infty\}$, and let $\phi\in C^k([a,b],\mg)$ be given such that $\im[\phi]$ is a \NIL{q} for some $q\geq 2$. Then, the following assertions hold:
\begingroup
\setlength{\leftmargini}{12pt}{
\begin{itemize}
\item
	We have $\TMAP^{q-1}(\phi|_{[a,\tau]})\in C^\const([0,1],\mg)$ for each $\tau\in [a,b]$, hence the map (recall \eqref{uisdiusduiuisduisd}) 
	\begin{align*}
		\MX\colon [a,b]\ni \tau\mapsto  \expalinv(\TMAP^q(\phi|_{[a,\tau]}))\in \mg	
	\end{align*}
	is defined.
\item
	We have $\MX\in C^1([a,b],\mg)$, and $\im[\MX]\cup \im[\dot\MX]$ is a \NIL{q}. 
\end{itemize}}
\endgroup
\noindent
In particular, iterated application of Proposition \ref{fdfdfd} yields
\begin{align*}
 	\textstyle\innt_a^\tau \phi = \innt_0^1 \TMAP^{q-1}(\phi|_{[a,\tau]})=\exp(\MX(\tau))\qquad\quad\forall\: \tau\in [a,b].
\end{align*}
\end{lemma}
\begin{proof}
Confer Sect.\ \ref{llkdsljdsdsuidsiudsoiudsods}.
\end{proof}
\begin{remark}
\label{iudsiudsiudsiuiudsnmewnmew}
	Let $\psi\in C^0([x,y],\mg)$ with $x<y$ be given. Let $\varrho\colon [x',y']\rightarrow [x,y]$ ($x'<y'$) be of class $C^1$, with $\varrho(x')=x$ and $\varrho(y')=y$. We set $\phi:= \dot\varrho\cdot (\psi\cp\varrho)$, and obtain inductively from \eqref{substitRI} that
\begin{align*}
	 \textstyle\int_{x'}^{y'}\TPOL^+_{\ell,\phi}[s](\phi(s))\: \dd s=\int_{x}^{y}\TPOL^+_{\ell,\psi}[s](\psi(s))\: \dd s 
\end{align*}	
	  holds for each $\ell\in \NN$.
\hspace*{\fill}\qed
\end{remark}
\noindent
We are ready for the proof of Theorem \ref{kfdhjfdhjfdfdfdfd}.
\begin{proof}[Proof of Theorem \ref{kfdhjfdhjfdfdfdfd}]
Let $\varrho_1\colon [0,1]\rightarrow [a',b']$ and $\varrho_2\colon [1,2]\rightarrow [a,b]$ be smooth, such that $\rho_{j}:=\dot\varrho_j$ is a bump function for $j=1,2$, i.e., we have\footnote{Confer, e.g., Sect.\ 4.3 in \cite{RGM} for explicit constructions of such bump functions.} 
\begin{align*}
	\rho_{1}|_{(0,1)},\rho_2|_{(1,2)}>0\qquad\quad \text{as well as}\qquad\quad  
	\rho_1^{(\ell)}(0),\rho_1^{(\ell)}(1),\rho_2^{(\ell)}(1),\rho_2^{(\ell)}(2)=0\qquad\forall\: \ell\in \NN.
\end{align*}	
	 Define $\chi\colon [0,1]\rightarrow \mg$ by
	\begin{align*}
		\chi|_{[0,1]}:= \rho_1\cdot (\psi\cp \varrho_1)\qquad\quad\text{as well as}\qquad\quad \chi|_{(1,2]}:= (\rho_2\cdot (\phi\cp \varrho_2))|_{(1,2]}.  	
	\end{align*}
	Then, $\chi$ is easily seen to be of class $C^k$ (confer, e.g., Lemma 24 and Appendix B.2 in \cite{RGM}), with $\im[\chi]\subseteq \SP_1(\NILS)\subseteq \CGen_1(\NILS)$ (a \NIL{q} by Remark \ref{dsdsdsdsdsv}). Then, for $\tau\in [0,1]$ and $t\in [1,2]$, we have  
\begin{align}
\label{nmcnmnmnmvcvcllklkdlkdskls2}
\begin{split}
\textstyle \innt_0^\tau\chi&\textstyle \stackrel{\phantom{\ref{pogfpogf}}}{=} \innt_0^\tau \rho_1\cdot (\psi\cp \varrho_1)\stackrel{\ref{subst}}{=}\innt_{a'}^{\varrho_1(\tau)}\psi,\\
\textstyle \innt_1^t\chi&\textstyle \stackrel{\phantom{\ref{pogfpogf}}}{=} \innt_1^t \hspace{0.4pt}\rho_2\cdot (\hspace{1pt}\phi\cp \varrho_2)\stackrel{\ref{subst}}{=}\innt_{a}^{\varrho_2(t)}\hspace{1pt}\phi,\\
	\textstyle\innt_0^t\chi&\textstyle\stackrel{\ref{pogfpogf}}{=} [\innt_1^t\chi|_{[1,2]}]\cdot [\innt_0^1 \chi|_{[0,1]}]
	=[\innt_1^t\rho_2\cdot (\phi\cp \varrho_2)]\cdot [\innt_0^1 \rho_1\cdot (\psi\cp \varrho_1)]
	=[\innt_a^{\varrho_2(t)} \phi]\cdot [\innt_{a'}^{b'}\psi].\hspace{15pt}
\end{split}
\end{align}	
Now, Lemma \ref{mnsdmnsdmnsdmndsmnsd} provides some $\tilde{\MX}\in C^1([0,2],\mg)$ with $\innt_0^\bullet \chi =\exp\cp\he\tilde{\MX}$, such that $\im[\tilde{\MX}]\cup \im[\dot{\tilde{\MX}}]$ is a \NIL{q}. In particular, we have
\begin{align}
\label{dspodsodspodspodsopdspodsds}
	\textstyle\innt_0^t\chi= \exp(\tilde{\MX}|_{[0,1]}(t))\qquad\quad\forall\: t\in [0,1];
\end{align}  
and obtain from Proposition \ref{kjfdkjfdkjfdjkfd} (specifically from  \eqref{kfdkjfjkfdkj2} with $g\equiv e$ and $\MX\equiv \tilde{\MX}|_{[0,1]}$ there) and \eqref{sdssddsds} in Corollary \ref{ljkdskjdslkjsda} (first step), as well as Remark \ref{iudsiudsiudsiuiudsnmewnmew} (second step) that
\begin{align}
\label{oifdoifdoifdoifdoifdoifdoif}
\begin{split}
	\textstyle\tilde{\MX}|_{[0,1]}(1)&\textstyle\stackrel{\phantom{\eqref{hjfdhjfdhjfdd1}}}{=} \sum_{n=1}^{q-1} \frac{(-1)^{n-1}}{n}\cdot\int_{0}^1 \big(\sum_{\ell=1}^{q-2} \TPOL^+_{\ell,\chi}[s]\big)^{n-1} (\chi(s))\: \dd s\\
	&\stackrel{\phantom{\eqref{hjfdhjfdhjfdd1}}}{=}\textstyle\sum_{n=1}^{q-1} \frac{(-1)^{n-1}}{n}\cdot\int_{a'}^{b'} \big(\sum_{\ell=1}^{q-2} \TPOL^+_{\ell,\psi}[s]\big)^{n-1} (\psi(s))\: \dd s\\
	&\stackrel{\eqref{hjfdhjfdhjfdd1}}{=}\MX_\psi(b') 
\end{split}
\end{align}
holds. Moreover, \eqref{kfdkjfjkfdkj2} in Proposition \ref{kjfdkjfdkjfdjkfd} applied with $g\equiv \innt_0^1 \chi$, $\phi\equiv \chi|_{[1,2]}$, $\MX\equiv \tilde{\MX}|_{[1,2]}$ there (first step) yields for $t\in [1,2]$ that 
\begin{align*}
	\textstyle \tilde{\MX}(t)&\stackrel{\phantom{\eqref{oifdoifdoifdoifdoifdoifdoif}}}{=}\hspace{8.5pt}\textstyle \tilde{\MX}(1)+   \sum_{n=1}^{q-1} \frac{(-1)^{n-1}}{n}\cdot \int_1^t \big(\Ad_{\innt_1^s\chi}\cp \Ad_{\innt_0^1 \chi}-\id_\mg \big)^{n-1}(\chi(s))\: \dd s\\
	&\stackrel{\eqref{oifdoifdoifdoifdoifdoifdoif}}{=}\textstyle \MX_\psi(b')+   \sum_{n=1}^{q-1} \frac{(-1)^{n-1}}{n}\cdot \int_1^t \big(\Ad_{\innt_1^s\chi}\cp \Ad_{\innt_0^1 \chi}-\id_\mg \big)^{n-1}(\chi(s))\: \dd s\\
	&\stackrel{\eqref{nmcnmnmnmvcvcllklkdlkdskls2}}{=}\textstyle \MX_\psi(b')+   \sum_{n=1}^{q-1} \frac{(-1)^{n-1}}{n}\cdot \int_1^t \rho_2(s)\cdot\big(\Ad_{\innt_a^{\varrho_2(s)}\phi}\cp \Ad_{\innt_{a'}^{b'} \psi}-\id_\mg \big)^{n-1}(\phi(\varrho_2(s)))\: \dd s\\
	&\stackrel{\eqref{substitRI}}{=}\textstyle \MX_\psi(b')+   \sum_{n=1}^{q-1} \frac{(-1)^{n-1}}{n}\cdot \int_1^{\varrho_2(t)} \big(\Ad_{\innt_a^{s}\phi}\cp \Ad_{\innt_{a'}^{b'} \psi}-\id_\mg \big)^{n-1}(\phi(s))\: \dd s\\
	&\textstyle\stackrel{\eqref{hjfdhjfdhjfdd2}}{=}\MX_\psi(b')+  \MX_{\phi,\psi}(\varrho_2(t)).
\end{align*}
We obtain 
\begin{align*}
	\textstyle\innt_a^{\varrho_2(t)} \phi\cdot \innt_{a'}^{b'}\psi\stackrel{\eqref{nmcnmnmnmvcvcllklkdlkdskls2}}{=}\innt_0^t\chi \stackrel{\eqref{dspodsodspodspodsopdspodsds}}{=} \exp(\tilde{\MX}(t))=\exp(\MX_\psi(b')+  \MX_{\phi,\psi}(\varrho_2(t)))\qquad\quad\forall\: t\in [1,2].
\end{align*}
From this, the claim follows, because 
 $\varrho_2$ is strictly increasing on $(1,2)$, thus strictly increasing on $[1,2]$, hence bijective.
\end{proof}

\subsubsection{The Proof of Lemma \ref{mnsdmnsdmnsdmndsmnsd}}
\label{llkdsljdsdsuidsiudsoiudsods}
Let $\NILS\subseteq \mg$ and $\ell \geq 1$ be given: 
\vspace{6pt}

\noindent
We let   
$\mathcal{H}_\ell(\NILS)$ denote the set of all maps $\alpha\colon [a,b]\times[0,1]\rightarrow \mg$ that admit the following properties:
\begingroup
\setlength{\leftmargini}{17pt}
{
\renewcommand{\theenumi}{\textbf{\small\arabic{enumi})}} 
\renewcommand{\labelenumi}{\theenumi}
\begin{enumerate}
\item
\label{cond1}
	We have $\im[\alpha]\subseteq \CGen_{\ell+1}(\NILS)$.  
\item
\label{cond2}
	There exists a continuous map $L\colon [a,b]\times [0,1]\rightarrow  \CGen_1(\NILS)\subseteq \mg$, as well as a map $\tilde{\alpha}\colon I\times [0,1]\rightarrow \mg$ for $I\subseteq \RR$ an open interval with $[a,b]\subseteq I$,  
	 such that the following conditions hold:  
\vspace{-3pt}	
\begingroup
\setlength{\leftmarginii}{16pt}
{
\renewcommand{\theenumii}{\alph{enumii})} 
\renewcommand{\labelenumii}{\theenumii}
\begin{enumerate}
\item
\label{dsdsdds0}
$\tilde{\alpha}$ is continuous, with   
$\tilde{\alpha}|_{[a,b]\times[0,1]}=\alpha$. 
\item
\label{dsdsdds1}
To each $\tau\in [a,b]$, there exist $\delta_{\tau}>0$ as well as $\epsilon_\tau\colon (-\delta_\tau,\delta_\tau)\times [0,1]\rightarrow \mg$ continuous with $\lim_{h\rightarrow 0}\pp_\infty(\epsilon_\tau(h,\cdot))=0$ for each $\pp\in \Sem{E}$, such that
	\begin{align*}
		\tilde{\alpha}(\tau+h,t)= \tilde{\alpha}(\tau,t) + h\cdot L(\tau,t) + h\cdot \epsilon_\tau(h,t)\qquad\quad (h,t)\in (-\delta_\tau,\delta_\tau)\times [0,1].
	\end{align*}
 In particular, $\Phi\equiv \tilde{\alpha}$ fulfills the assumptions in Theorem \ref{ofdpofdpofdpofdpofd} for $(G,\cdot)\equiv (\mg,+)$ there (recall Remark \ref{kjfdkjdkjfdkj}), i.e., we have
\begin{align}
\label{kjfdkjfdkjfdkjfdfd}
 	\textstyle\frac{\dd}{\dd h}\big|_{h=0}\int_0^1\alpha(\tau+h,s)\:\dd s= \int_0^1 L(\tau,s)\:\dd s\qquad\quad\forall\: \tau\in [a,b].
\end{align}  
\end{enumerate}}	
\endgroup
\end{enumerate}}
\endgroup
\noindent
We let $\mathcal{K}_{\ell}(\NILS)$ denote the set of all (continuous) maps of the form 
\begin{align*}
	\textstyle \eta\colon [a,b]\times[0,1]\ni (\tau,t)\mapsto \MXP(\tau) + \alpha(\tau,t)\in \mg, 
\end{align*}
for $\MXP\in C^1([a,b],\mg)$ with $\im[\MXP]\cup\im[\dot\MXP]\subseteq \CGen_1(\NILS)$, and $\alpha\in \mathcal{H}_\ell(\NILS)$ (in particular, $\im[\eta]\subseteq \CGen_1(\NILS)$). 
\vspace{6pt}

\noindent
We recall Convention \eqref{dsdsdsdcxcxcxcxcxcxcxsds}, and are ready for the proof of Lemma \ref{mnsdmnsdmnsdmndsmnsd}.

\begin{proof}[Proof of Lemma \ref{mnsdmnsdmnsdmndsmnsd}]
Set $\NILS:=\im[\phi]$. 
We prove by induction that 
\begin{align}
\label{lkjfdjlkfdlkjfdfdfd}
	\mathcal{K}_{\ell}(\NILS)\ni \eta_\ell\colon [a,b]\times[0,1]\ni(\tau,t)\mapsto \TMAP^\ell\big(\phi|_{[a,\tau]}\big)(t)\in \mg\qquad\quad\forall\:   \ell=1,\dots,q-1
\end{align} 
holds. Evaluating this for $\ell=q-1$, the claim follows because: 
\begingroup
\setlength{\leftmargini}{12pt}{
\begin{itemize}
\item
	 For $\alpha\in \mathcal{H}_{q-1}(\NILS)$, we have $\im[\alpha]\subseteq \CGen_q(\NILS)=\{0\}$ by Condition \ref{cond1} and Remark \ref{dsdsdsdsdsv}.\ref{dsdsdsdsdsv1}. 
\item
	$\CGen_1(\NILS)$ is a \NIL{q} by Remark \ref{dsdsdsdsdsv}.\ref{dsdsdsdsdsv2}.
\end{itemize}}
\endgroup
\noindent
To perform the induction, we fix $\varepsilon>0$ and $\tilde{\phi}\in C^k([a-\varepsilon,b+\varepsilon],\mg)$ with $\tilde{\phi}|_{[a,b]}=\phi$. We observe that for each $\psi\in \DIDE_{[a,b]}$ with $\im[\psi]$ a \NIL{q}, and $\lambda,t\in [a,b]$ we have (confer Example \ref{dsddsdsds})
\begin{align}
\label{kjfdfdkjfdkjlkjfdkjfdfd}
\begin{split}
	\textstyle\TMAP\big(\psi|_{[a,\lambda]}\big)(t)
	&=\textstyle\int_a^\lambda \Ad_{\innt_s^{\lambda}  t\cdot \psi}(\psi(s)) \:\dd s\\
	&\textstyle = \int_a^\lambda \Add^+_{t\cdot\psi|_{[s,\lambda]}}[\psi(s)](\lambda) \:\dd s 
	\\
	&\textstyle  = \int_a^\lambda \psi(s) \:\dd s + 
	\sum_{p=1}^{q-2}\int_a^\lambda \TPOL^+_{p,t\cdot\psi|_{[s,\lambda]}}[\psi(s)](\lambda) \:\dd s
		\\
	&\textstyle  = \int_a^\lambda \psi(s) \:\dd s\\
	&\quad\he\textstyle + \sum_{p=1}^{q-2}\he t^p\cdot \int_a^\lambda \dd s 	
	\int_s^{\lambda}\dd s_1\int_s^{s_1}\dd s_2 \: {\dots} \int_s^{s_{p-1}}\dd s_p \: (\com{\psi(s_1)}\cp \dots \cp \com{\psi(s_{p})})(\psi(s)).
	\end{split}
\end{align}
{\bf Induction Basis:} 
Let $\ell=1$. By \eqref{kjfdfdkjfdkjlkjfdkjfdfd}, for $\tau\in [a,b]$ and $t\in [0,1]$ we have
\begin{align*}
	\textstyle\eta_1(\tau,t)= \underbracket{\textstyle\int_a^\tau \phi(s)\:\dd s}_{=:\:\MXP(\tau)}\he +\he \underbracket{\textstyle\sum_{p=1}^{q-2}\he t^p\cdot \int_a^\tau \dd s 	
	\int_s^{\tau}\dd s_1\int_s^{s_1}\dd s_2 \: {\dots} \int_s^{s_{p-1}}\dd s_p \: (\com{\phi(s_1)}\cp \dots \cp \com{\phi(s_{p})})(\phi(s))}_{=:\: \alpha(\tau,t)},
\end{align*}
with 
\begingroup
\setlength{\leftmargini}{12pt}{
\begin{itemize}
\item
	$\MXP\in C^1([a,b],\mg)$ by  \eqref{oidsoidoisdoidsoioidsiods}, and $\im[\MXP]\cup\im[\dot\MXP]\subseteq \CGen_1(\NILS)$\hspace{7.2pt} by Remark \ref{dsdsdsdsdsv}.\ref{dsdsdsdsdsv3}. 
\item
	$\im[\alpha]\subseteq \CGen_2(\NILS)$ (hence Condition \ref{cond1} for $\ell\equiv 1$ there) by Remark \ref{dsdsdsdsdsv}.\ref{dsdsdsdsdsv3}.
\end{itemize}}
\endgroup
\noindent 
To verify Condition \ref{cond2}, for $\tau_1,\tau_2\in I:= (a-\varepsilon,b+\varepsilon)$ and $t\in [0,1]$, we set
\begin{align*}
	\tilde{\beta}(\tau_1,\tau_2,t):= \textstyle\sum_{p=1}^{q-2}\he t^p\cdot \int_{a-\varepsilon}^{\tau_1} \dd s 	
	\int_s^{\tau_2}\dd s_1\int_s^{s_1}\dd s_2 \: {\dots} \int_s^{s_{p-1}}\dd s_p \: (\com{\tilde{\phi}(s_1)}\cp \dots \cp \com{\tilde{\phi}(s_{p})})(\tilde{\phi}(s)),
\end{align*}
and define $\tilde{\alpha}(\tau,t):= \tilde{\beta}(\tau,\tau,t) - \tilde{\beta}(a,\tau,t)$ for $\tau\in I$ and $t\in [0,1]$. Then, $\tilde{\alpha}$ is continuous, with $\tilde{\alpha}|_{[a,b]\times [0,1]}=\alpha$, hence fulfills Condition \ref{cond2}.a). 
Moreover, for $\tau\in I$, $t\in [0,1]$, and $|h|>0$ suitably small, we have
	\begin{align*}
		\tilde{\alpha}(\tau+h,t)&\textstyle-\tilde{\alpha}(\tau,t)\\[1pt]
		&\textstyle=\sum_{p=1}^{q-2}\he t^p\cdot \int_{a}^{\tau} \dd s 	
	\int_\tau^{\tau+h}\dd s_1\int_s^{s_1}\dd s_2 \: {\dots} \int_s^{s_{p-1}}\dd s_p \: (\com{\tilde{\phi}(s_1)}\cp \dots \cp \com{\tilde{\phi}(s_{p})})(\tilde{\phi}(s))\\
	&\quad\he\textstyle+\sum_{p=1}^{q-2}\he t^p\cdot \int_{\tau}^{\tau+h} \dd s 	
	\int_s^{\tau+h}\dd s_1\int_s^{s_1}\dd s_2 \: {\dots} \int_s^{s_{p-1}}\dd s_p \: (\com{\tilde{\phi}(s_1)}\cp \dots \cp \com{\tilde{\phi}(s_{p})})(\tilde{\phi}(s)),\\[6pt]
		\tilde{\alpha}(\tau,t)-\tilde{\alpha}&\textstyle(\tau-h,t)\\[1pt]
		&\textstyle=\sum_{p=1}^{q-2}\he t^p\cdot \int_{a}^{\tau-h} \dd s 	
	\int_{\tau-h}^{\tau}\dd s_1\int_s^{s_1}\dd s_2 \: {\dots} \int_s^{s_{p-1}}\dd s_p \: (\com{\tilde{\phi}(s_1)}\cp \dots \cp \com{\tilde{\phi}(s_{p})})(\tilde{\phi}(s))\\
	&\quad\he\textstyle+\sum_{p=1}^{q-2}\he t^p\cdot \int_{\tau-h}^{\tau} \dd s 	
	\int_s^{\tau}\dd s_1\int_s^{s_1}\dd s_2 \: {\dots} \int_s^{s_{p-1}}\dd s_p \: (\com{\tilde{\phi}(s_1)}\cp \dots \cp \com{\tilde{\phi}(s_{p})})(\tilde{\phi}(s)).
	\end{align*}
	We define $L\colon [a,b]\times [0,1]\rightarrow \CGen_1(\NILS)\subseteq \mg$ (recall Remark \ref{dsdsdsdsdsv}.\ref{dsdsdsdsdsv3}) by 
	\begin{align*}
	L(\tau, t):= \textstyle\sum_{p=1}^{q-2}\he t^p\cdot \int_{a}^{\tau} \dd s 	
	\int_s^{\tau}\dd s_2 \: {\dots} \int_s^{s_{p-1}}\dd s_p \: (\com{\phi(\tau)}\cp \com{\phi(s_2)} \cp \dots \cp \com{\phi(s_{p})})(\phi(s))
\end{align*}
for all $\tau\in [a,b]$ and $t\in [0,1]$. Clearly, $L$ is continuous; and it follows from \eqref{ffdlkfdlkfd} as well as continuity of the involved maps that 
\begin{align*}
	\textstyle\lim _{h\rightarrow 0} 1/|h|\cdot \pp_\infty\textstyle(\tilde{\alpha}(\tau+h,\cdot)-\tilde{\alpha}(\tau,\cdot)- h\cdot L(\tau,\cdot))=0\qquad\quad\forall\: \pp\in \Sem{E},\: \tau\in [a,b]
\end{align*}
holds, which verifies Condition \ref{cond2}.b). 	
We thus have shown \eqref{lkjfdjlkfdlkjfdfdfd} for $\ell= 1$. In particular, this proves the claim for $q=2$.
\vspace{6pt}

\noindent
{\bf Induction Step:} 
 Assume that \eqref{lkjfdjlkfdlkjfdfdfd} holds 
for some $1\leq \ell\leq q-1$ (with $q\geq 3$), hence we have
\begin{align}
\label{nmvcnmvcxnmvcvcvc}
	\textstyle \mathcal{K}_{\ell}(\NILS)\ni \eta_\ell\colon [a,b]\times[0,1]\ni (\tau,t)\mapsto \MXP_\ell(\tau) + \alpha_\ell(\tau,t)\in \CGen_1(\NILS)\subseteq \mg 
\end{align}
for some $\MXP_\ell\in C^1([a,b],\mg)$ with $\im[\MXP_\ell]\cup\im[\dot\MXP_\ell]\subseteq \CGen_1(\NILS)$, as well as $\alpha_\ell\in \mathcal{H}_{\ell}(\NILS)$. Let $I\subseteq \RR$ be as in Condition \ref{cond2} for $\alpha\equiv \alpha_\ell$ there. 
We can shrink $I$ around $[a,b]$ in such a way that $\MXP_\ell$ extends to a $C^1$-curve $\tilde{\MXP}_\ell\in C^1(I,\mg)$, and set
\begin{align}
\label{jjksakjskjsakjkjsa}
	\tilde{\eta}_\ell\colon I\times [0,1]\ni (\tau,t)\mapsto \tilde{\MXP}_\ell(\tau) + \tilde{\alpha}_\ell(\tau,t)\in \mg. 
\end{align}
Now, \eqref{kjfdfdkjfdkjlkjfdkjfdfd} for $\psi\equiv \eta_\ell(\tau,\cdot)$ and $\lambda\equiv 1$ yields for $\tau\in [a,b]$ and $t\in [0,1]$ that
\begin{align*}
	\eta^{\ell+1}(\tau,t)&\textstyle=\TMAP^{\ell+1}\big(\phi|_{[a,\tau]}\big)(t)\\[2pt]
	&\textstyle =\TMAP(\eta_\ell(\tau,\cdot))(t)\\
	&\textstyle = \underbracket{\textstyle\int_0^1 \eta_\ell(\tau,s) \:\dd s}_{=:\: \MXP_{\ell+1}(\tau)} \\[2pt]
	&\quad\he\textstyle +\underbracket{\textstyle\sum_{p=1}^{q-2}\he t^p\cdot \int_0^1 \dd s 	
	\int_s^{1}\dd s_1\int_s^{s_1}\dd s_2 \: {\dots} \int_s^{s_{p-1}}\dd s_p \: (\com{\eta_\ell(\tau,s_1)}\cp \dots \cp \com{\eta_\ell(\tau,s_{p})})(\eta_\ell(\tau,s))}_{=:\: \alpha_{\ell+1}(\tau,t)}.
\end{align*}
We obtain from \eqref{nmvcnmvcxnmvcvcvc} that $\MXP_{\ell+1}=\MXP_{\ell}+ \int_0^1 \alpha_\ell(\cdot,s) \: \dd s$ holds, and conclude the following:
\begingroup
\setlength{\leftmargini}{12pt}{
\begin{itemize}
\item
Let $L\colon [a,b]\times [0,1]\rightarrow  \CGen_1(\NILS)\subseteq \mg$  be as in Condition \ref{cond2}.b) for $\alpha\equiv \alpha_\ell$ there (induction hypothesis). Then, we have
\vspace{-5pt}
\begin{align*}
	\textstyle\dot\MXP_{\ell+1}(\tau)\stackrel{\eqref{kjfdkjfdkjfdkjfdfd}}{=}\dot\MXP_{\ell}(\tau) + \int_0^1 L(\tau,s)\: \dd s\in \CGen_1(\NILS) \qquad\quad\forall\: \tau\in [a,b].
\end{align*}
Since $L$ is continuous, this implies $\MXP_{\ell+1}\in C^1([a,b],\mg)$ with $\im[\MXP_{\ell+1}]\cup\im[\dot\MXP_{\ell+1}]\subseteq \CGen_1(\NILS)$. 
\item
We have $\im[\alpha_{\ell+1}]\subseteq \CGen_{\ell+2}(\NILS)$ by Remark \ref{dsdsdsdsdsv}.\ref{dsdsdsdsdsv4},  Remark \ref{dsdsdsdsdsv}.\ref{dsdsdsdsdsv3}, bilinearity of $\bil{\cdot}{\cdot}$, as well as
\vspace{-3pt}
\begingroup
\setlength{\leftmarginii}{13pt}{
\begin{itemize}
\item
	$(\com{\MXP_\ell(\tau)}\cp \dots \cp \com{\MXP_\ell(\tau)})(\MXP_\ell(\tau))=0$ for all $\tau\in [a,b]$,
	\vspace{4pt}
\item
\hspace{2.8pt}$\im[\MXP_\ell]\subseteq \CGen_1(\NILS)$, 
	\vspace{4pt}
\item
$\im[\alpha_\ell]\subseteq \CGen_{\ell+1}(\NILS)$,
\end{itemize}}
\endgroup
\noindent 
which verifies Condition \ref{cond1} (for $\alpha\equiv \alpha_{\ell+1}$ there).
\end{itemize}}
\endgroup
\noindent
It remains to verify Condition \ref{cond2} (for $\alpha\equiv \alpha_{\ell+1}$ there). For this, we set (recall \eqref{jjksakjskjsakjkjsa})
\begin{align*}
	\tilde{\alpha}_{\ell+1}(\tau,t):= \textstyle\sum_{p=1}^{q-2}\he t^p\cdot \int_0^1 \dd s 	
	\int_s^{1}\dd s_1\int_s^{s_1}\dd s_2 \: {\dots} \int_s^{s_{p-1}}\dd s_p \: (\com{\tilde{\eta}_\ell(\tau,s_1)}\cp \dots \cp \com{\tilde{\eta}_\ell(\tau,s_{p})})(\tilde{\eta}_\ell(\tau,s))
\end{align*}
for $\tau\in I$ and $t\in [0,1]$. Clearly, $\tilde{\alpha}_{\ell+1}$ fulfills Condition \ref{cond2}.a) (for $\alpha\equiv \alpha_{\ell+1}$ there). Moreover, since $\MXP_\ell$ is of class $C^1$ with $\im[\MXP]\cup\im[\dot\MXP]\subseteq \CGen_1(\NILS)$, by  \eqref{jjksakjskjsakjkjsa}, and 
Condition \ref{cond2}.b) for $\alpha\equiv \alpha_\ell$ (induction hypothesis), there exists a continuous map $L\colon [a,b]\times [0,1]\rightarrow  \CGen_1(\NILS)\subseteq \mg$, as well as to each fixed $\tau\in [a,b]$ some  
$\delta_{\tau}>0$ and $\epsilon_\tau\colon (-\delta_\tau,\delta_\tau)\times [0,1]\rightarrow \mg$ continuous with $\lim_{h\rightarrow 0}\pp_\infty(\epsilon_\tau(h,\cdot))=0$ for each $\pp\in \Sem{E}$, such that 
\begin{align}
\label{lkjfdlkjdffdlknbcxnbnbcxnbcxxc}
		\tilde{\eta}_\ell(\tau+h,t)= \tilde{\eta}_\ell(\tau,t) + h\cdot L(\tau,t) + h\cdot \epsilon_\tau(h,t)\qquad\quad (h,t)\in (-\delta_\tau,\delta_\tau)\times [0,1]
\end{align} 
holds. 
We define for $\tau\in I$ and $t\in [0,1]$
\begin{align*}
	\tilde{L}(\tau,t):= &\textstyle\sum_{p=1}^{q-2}\he t^p\cdot \int_0^1 \dd s 	
	\int_s^{1}\dd s_1\int_s^{s_1}\dd s_2 \: {\dots} \int_s^{s_{p-1}}\dd s_p \sum_{k=1}^{p+1}\:   \gamma_k(s_1,\dots,s_p,s),
\end{align*}
where for $1\leq p\leq q-2$ and $0\leq s_1,\dots,s_p,s\leq 1$, we set 
\begin{align*}
	\textstyle\gamma_{p+1}(s_1,\dots,s_p,s):= (\com{\tilde{\eta}_\ell(\tau,s_1)}\cp \dots\cp \com{\tilde{\eta}_\ell(\tau,s_p)})(L(\tau,s))
\end{align*}
 as well as for $1\leq k\leq p$
\begin{align*}
	\textstyle\gamma_k(s_1,\dots,s_p,s):= (\com{\tilde{\eta}_\ell(\tau,s_1)}\cp \dots&\textstyle\cp \com{\tilde{\eta}_\ell(\tau,s_{k-1})} \cp \com{L_\tau(s_k)}\\
	&\textstyle\cp \com{\tilde{\eta}_\ell(\tau,s_{k+1})} \cp \dots \cp \com{\tilde{\eta}_\ell(\tau,s_p)})( \tilde{\eta}_\ell(\tau,s)).
\end{align*}
 It follows from continuity and bilinearity of $\bil{\cdot}{\cdot}$, \eqref{lkjfdlkjdffdlknbcxnbnbcxnbcxxc}, and \eqref{ffdlkfdlkfd} that
\begin{align*}
	\textstyle\lim_{h\rightarrow 0} 1/|h|\cdot \pp_\infty(\tilde{\alpha}_{\ell+1}(\tau+h,\cdot) - \tilde{\alpha}_{\ell+1}(\tau,\cdot) - h \cdot \tilde{L}(\tau,\cdot))=0\qquad\quad\forall\: \pp\in \Sem{E}
\end{align*} 
holds, which verifies Condition \ref{cond2}.b) for $\alpha\equiv \alpha_{\ell+1}$ there. 
\end{proof}

\section{Asymptotic Estimate Lie Algebras} 
\label{jkdkjdjkddsdsddds}
Throughout this section, $(\mq,\bil{\cdot}{\cdot})$ denotes a fixed sequentially complete and asymptotic estimate Lie algebra, i.e., 
 the following conditions are fulfilled: 
\begingroup
\setlength{\leftmargini}{12pt}{
\begin{itemize}
\item
$\mq\in \HLCV$ is sequentially complete.
\item
$\bil{\cdot}{\cdot}\colon \mq\times \mq\rightarrow \mq$ is bilinear, antisymmetric, continuous, and fulfills the Jacobi identity  \eqref{nmvcnmvcnmnkjsakjsakjsaa}. Moreover, to each $\vv\in \Sem{\mq}$, there exist $\vv\leq \ww\in \Sem{\mq}$ with 
\begin{align}
\label{assaaass}
	\vv(\bl X_1,\bl X_2,\bl \dots,\bl X_n,Y\br{\dots}\br\br\br)\leq \ww(X_1)\cdot{\dots}\cdot \ww(X_n)\cdot \ww(Y)
\end{align}
for all $X_1,\dots,X_n,Y\in \mq$ and $n\geq 1$. 
\end{itemize}}
\endgroup
\noindent
The first part of this section (Sect.\ \ref{kdsjlkdslkjsdjkds}) is dedicated to a comprehensive analysis of the Lax equation 
\begin{align*}
	\dot\alpha=\bil{\psi}{\alpha}\qquad\text{with}\qquad \alpha(a)=X,
\end{align*}
for $X\in \mq$ and $\psi\in C^0([a,b],\mq)$ ($a<b$) fixed, and $\alpha\in C^1([a,b],\mq)$. We first show the existence and uniqueness of solutions, and discuss their elementary properties (Sect.\ \ref{lkdslkdlkdslk} and Sect.\ \ref{ofdoifdoifdojkdnmcvnmc}). More specifically, let $\Add^\pm_\psi[X]$ be defined by the right hand side of \eqref{sdssddsds}. We show that the unique solution to the above equation is given by $\alpha=\Add_\psi^+[X]$. We also show\footnote{Although it might be obvious, we explicitly mention at this point that in contrast to the finite-dimensional case, in the context considered it not suffices to show the existence of a left- or right inverse of $\Add_\psi^\pm[t]$ but both. Basically, this is the reason for the necessity of the investigations made in Sect.\ \ref{ofdoifdoifdojkdnmcvnmc}.}
\begin{align*}	
	\Aut(\mq)\ni \Add_\psi^\pm[t]\colon \mq\in X\mapsto \Add^+_\psi[X](t)\in \mq\qquad\quad\forall\: t\in [a,b],
\end{align*}
 with $\Add^\pm_\psi[t]^{-1}=\Add^\mp_\psi[t]$ (both statements are proven in Proposition \ref{lkjfdldlkkjfdYYX}). This allows to define a group structure on $C^k([a,b],\mq)$ for $k\in \NN\cup \{\infty\}$ (in analogy to Remark \ref{dsdsdsdsjkjkkjkjkjjkkjkj} and Sect.\ 1.14 in \cite{DUIS}) by  
\begin{align*}
	\psi^{-1}&\colon [a,b]\ni t \mapsto \hspace{22.3pt}-\Add_\psi^-[t](\psi(t))\in\mq\\
	\phi\starm\psi&\colon [a,b]\ni t\mapsto \phi(t) + \Add^+_\phi[t](\psi(t))\in \mq
\end{align*}
for $\phi,\psi\in C^k([a,b],\mq)$.
The group axioms are verified in Sect.\ \ref{sdklsdklsdklsd}. 
\begin{remark}[A Lie Group Construction]
\label{nmdnmfnmfdnmfdnmfdfd}
Let $k\in \NN\cup\{\infty\}$ and $a<b$ be given; and set 
\begin{align*}
	G_k=E_k=\mg_k:=C^k([a,b],\mq),\qquad e_k:=0\in C^k([a,b],\mq),\qquad \chart_k:=\id_{C^k([a,b],\mq)}.
\end{align*}
As already mentioned, we will show that the maps
\begin{align}
\label{oioidsoidsdsds}
\begin{split}
	\inv_k\colon& \hspace{27pt}G_k\rightarrow G_k,\qquad\hspace{20.1pt} \psi\mapsto \psi^{-1}\\
	\mult_k\colon&  G_k\times G_k\rightarrow G_k,\qquad (\phi,\psi)\mapsto \phi\star\psi
\end{split}
\end{align}
are defined (confer Lemma \ref{knkjfdsjkfasdkjdsjkdsa}), and that $(G_k,\mult_k,\inv_k,e_k)$ is a group (confer Lemma  \ref{lkjkjsjlksdlkjdskj}). Expectably, the group operations \eqref{oioidsoidsdsds} are smooth w.r.t.\ the $C^k$-topology; hence give rise to a Lie group $G_k$ with global chart $\chart_k$, and sequentially complete Lie algebra $\mg_k$ (recall Lemma \ref{kjdskjskjds}). It is furthermore to be expected that the corresponding Lie bracket $\bil{\cdot}{\cdot}_k\colon \mg_k\times \mg_k\rightarrow \mg_k$ is given by 
\begin{align*}
	 \textstyle\bil{\phi}{\psi}_k(t) = \bil{\int_a^t\phi(s)\:\dd s}{\psi(t)} + \bil{\phi(t)}{\int_a^t \psi(s)\:\dd s}\qquad\quad\forall\: t\in [a,b]
\end{align*}
for $\phi,\psi\in \mg_k$, 
just as in Sect.\ 1.14 in \cite{DUIS} (confer Proposition 1.14.1 in \cite{DUIS}). Presumably, $(\mg_k,\bil{\cdot}{\cdot}_k)$ is asymptotic estimate, and $G_k$ is $C^0$-regular.  
The technical details will be worked out in a separate paper. This serves as a preparation for a possible extension of Lie's third theorem to the infinite-dimensional asymptotic estimate case -- just by performing the same (a similar) construction made in Sect.\ 1.14 in \cite{DUIS} to prove this theorem (Theorem 1.14.3 in \cite{DUIS}) in the finite-dimensional context. Notably, this construction had already been used in \cite{MCRLF}, to prove Lie's third theorem in the context of Lie algebroids.  
\hspace*{\fill}\qed
\end{remark}
\noindent
In the last part of Sect.\ \ref{kdsjlkdslkjsdjkds}, we use the proven statements to investigate the properties of the following integral transformation that mimics \eqref{dfdfdfdfdfd} in Sect.\ \ref{kjdskjsdkjdskjs}:
\begin{align*}
	\textstyle\TMAP\colon C^0([a,b],\mq)&\rightarrow C^\infty([0,1],\mq)\\
	 \phi&\textstyle\mapsto \big[ [0,1]\ni t \mapsto  \int_a^b \Add^+_{ t\cdot \phi|_{[s,b]}}[b](\phi(s)) \:\dd s\he\big].
\end{align*} 
The investigations serve as a preparation for the discussions in Sect.\ \ref{mncxnmnxcmnxmnmnxciooioiweioewwe}. There, the above transformation is applied to the situation where $(\mq,\bil{\cdot}{\cdot})\equiv (\mg,[\cdot,\cdot])$ equals the Lie algebra of a given Lie group $G$. Specifically, we prove the following regularity result (cf.\ Theorem \ref{oisiosdidsoiososdids}):
\begin{theorem}
	\label{dkjkjfdkjfdjfdkfdfdfd}
	Assume that $(\mg,\bil{\cdot}{\cdot})$ is asymptotic estimate and sequentially complete. 
	If $G$ is weakly $C^\infty$-regular, then $G$ is weakly $C^k$-regular for each $k\in \NN\cup\{\lip,\infty\}$.
\end{theorem}
\begin{remark}
\label{lksdlkdslkdslkds98ds09ds09ds09ds09ds09dsdsdsdsdsdscx}
Note that in the situation of Theorem \ref{dkjkjfdkjfdjfdkfdfdfd}, $G$ is  weakly $C^k$-regular for $k\in \NN\cup\{\lip,\infty\}$ if and only if $G$ is $C^k$-semiregular, just because $\mg$ is assumed to be sequentially complete.
\hspace*{\fill}\qed
\end{remark}
\noindent
This result complements Theorem 2 in \cite{AEH} that essentially states that $C^\infty$-regularity\footnote{We recall that in contrast to the definition of weak $C^k$-regularity for $k\in \NN\cup\{\lip,\infty\}$, the definition of $C^k$-regularity additionally involves continuity of the product integral w.r.t.\ the $C^k$-topology.} implies $C^k$-regularity for each $k\in \NN\cup \{\lip,\infty\}$, where for $k=0$ additionally sequentially completeness of the Lie group has to be assumed.\footnote{Sequentially completeness of a Lie group, and sequentially completeness of its Lie algebra are ad hoc different properties, details can be found in \cite{RGM}.} Hence, in a certain sense, the assumption in Theorem 2 in \cite{AEH} that $G$ is sequentially complete, had been replaced in Theorem \ref{dkjkjfdkjfdjfdkfdfdfd} by the assumption that $\mg$ is sequentially complete.  The key result proven in \cite{AEH} (Theorem 1 in \cite{AEH}) states that in the asymptotic estimate context, $C^\infty$-continuity of the product integral is equivalent to its $C^0$-continuity (local $\mu$-convexity of $G$), which makes the semiregularity results obtained in \cite{RGM} applicable -- Basically, in analogy to the definition of the Riemann integral as a limit over Riemann sums, in \cite{RGM} the product integral is obtained as a limit over product integrals of piecewise integrable curves (under certain assumptions on the Lie group that are automatically fulfilled in the asymptotic estimate context). The proof of Theorem \ref{dkjkjfdkjfdjfdkfdfdfd} works differently: For some given $\phi\in C^0([a,b],\mg)$, we show that (this is defined by $C^\infty$-semiregularity of $G$, as well as $\im[\TMAP]\subseteq C^\infty([0,1],\mg)$) 
\begin{align*}
	\textstyle\mu\colon [a,b]\ni z\mapsto \innt_0^1 \TMAP(\phi|_{[a,z]})\in G	
\end{align*}	
	is of class $C^1$, with $\Der(\mu)=\phi$.

\begin{remark}
We will tacitly use throughout this section that sequentially completeness of $\mq$ implies (recall Remark \ref{seqcomp})
\vspace{-6pt}
\begin{align*}
	\textstyle\int\phi(s)\:\dd s\in \mq\qquad\quad\forall\: \phi\in C^0([a,b],\mq),\: a<b.
\end{align*}
We will furthermore use that $\bil{\cdot}{\cdot}$ is smooth (apply, e.g., the parts \ref{linear}, \ref{productrule} of Proposition \ref{iuiuiuiuuzuzuztztttrtrtr}).
\hspace*{\fill}\qed
\end{remark}
\noindent
Theorem \ref{dkjkjfdkjfdjfdkfdfdfd}, together with Theorem 4 in \cite{RGM} and  Theorem 1 in  \cite{AEH}, yields the following statement: 
\begin{corollary}
\label{kjdskjdsiudsiudsiudsds98798ds98ds98dsdsdsdsdsds}
If $(\mg,\bil{\cdot}{\cdot})$ is asymptotic estimate and sequentially complete, then the following assertions are equivalent:
\begingroup
\setlength{\leftmargini}{17pt}
{
\renewcommand{\theenumi}{\emph{\alph{enumi})}} 
\renewcommand{\labelenumi}{\theenumi}
\begin{enumerate}
\item
\label{kjdskjdsiudsiudsiudsds98798ds98ds98dsdsdsdsdsdsa1}
$G$ is $C^k$-semiregular for some $k\in \NN\cup \{\infty\}$, and $\evol_\kk$ is of class $C^\ell$ for some $\ell\in \NN\cup \{\infty\}$.
\item
\label{kjdskjdsiudsiudsiudsds98798ds98ds98dsdsdsdsdsdsa2} 
$G$ is $C^0$-semiregular, and $\evol_0$ is $C^0$-continuous (\he$G$ is locally $\mu$-convex\he).
\item
\label{kjdskjdsiudsiudsiudsds98798ds98ds98dsdsdsdsdsdsa3}
$G$ is $C^k$-regular for all $k\in \NN\cup \{\infty\}$.
\end{enumerate}}
\endgroup
\end{corollary}
\begin{proof}
The implication \ref{kjdskjdsiudsiudsiudsds98798ds98ds98dsdsdsdsdsdsa3}\hspace{1pt} $\Rightarrow$\hspace{1pt} \ref{kjdskjdsiudsiudsiudsds98798ds98ds98dsdsdsdsdsdsa1} is evident. Moreover:
\vspace{6pt}

\noindent
\ref{kjdskjdsiudsiudsiudsds98798ds98ds98dsdsdsdsdsdsa1}\hspace{1pt} $\Rightarrow$\hspace{1pt} \ref{kjdskjdsiudsiudsiudsds98798ds98ds98dsdsdsdsdsdsa2}:
If the assertion  \ref{kjdskjdsiudsiudsiudsds98798ds98ds98dsdsdsdsdsdsa1} holds, then 
\begingroup
\setlength{\leftmargini}{12pt}{
\begin{itemize}
\item
$G$ is $C^\infty$-semiregular, as $C^\infty([0,1],\mg)\subseteq C^k([0,1],\mg)$ holds. Hence, $G$ is weakly $C^\infty$-regular by Remark \ref{lksdlkdslkdslkds98ds09ds09ds09ds09ds09dsdsdsdsdsdscx}; thus, $G$ is $C^0$-semiregular by Theorem \ref{dkjkjfdkjfdjfdkfdfdfd}.
\item
$\evol_\kk$ is $C^k$-continuous, as of class $C^\ell$ (with $\ell\geq 0$). Hence, $\evol_0$ is $C^0$-continuous ($G$ is locally $\mu$-convex) by Theorem 1 in \cite{AEH}. 
\end{itemize}}
\endgroup 
\noindent
\ref{kjdskjdsiudsiudsiudsds98798ds98ds98dsdsdsdsdsdsa2}\hspace{1pt} $\Rightarrow$\hspace{1pt}  \ref{kjdskjdsiudsiudsiudsds98798ds98ds98dsdsdsdsdsdsa3}:
Assume that the assertion  \ref{kjdskjdsiudsiudsiudsds98798ds98ds98dsdsdsdsdsdsa2} holds, and let $k\in \NN\cup \{\infty\}$ be given. Then,   
\begingroup
\setlength{\leftmargini}{12pt}{
\begin{itemize}
\item
$G$ is $C^k$-semiregular, as $C^k([0,1],\mg)\subseteq C^0([0,1],\mg)$ holds.
\item
$\evol_\kk=\evol_0|_{C^k([0,1],\mg)}$ is $C^k$-continuous, as even $C^0$-continuous.
 \end{itemize}}
\endgroup
\noindent
Since $\mg$ is sequentially complete (hence, integral complete and Mackey complete), Theorem 4 in \cite{RGM} shows that $\evol_\kk$ is smooth.
\end{proof}

\subsection{The Lax Equation}
\label{kdsjlkdslkjsdjkds}
Let $a<b$ and $\psi \in \CP^0([a,b],\mq)$ be given.\footnote{We consider $C^0([a,b],\mq)\subseteq \CP^0([a,b],\mq)$ as a subset in the obvious way.} Motivated by Sect.\ \ref{kfdkjfdkjfdfd}, 
for $X\in \mq$ we set
\begin{align*}
	\APOL^\pm_{0,\psi}[X]\colon [a,b]\rightarrow \mq,\qquad  t\mapsto X,	
\end{align*}
\vspace{-26pt}

\noindent
as well as (recall \eqref{opofdpopfd})
\begin{align*}
	\textstyle \APOL^+_{\ell,\psi}[X]\colon [a,b]\rightarrow \mq,\qquad &\textstyle t\mapsto \hspace{34.45pt} \int_a^{t}\dd s_1\int_a^{s_1}\dd s_2 \: {\dots} \int_a^{s_{\ell-1}}\dd s_\ell \: (\bilbr{\psi(s_1)}\cp \dots \cp \bilbr{\psi(s_{\ell})})(X)\\
	\textstyle \APOL^-_{\ell,\psi}[X]\colon [a,b]\rightarrow \mq,\qquad &\textstyle t\mapsto  (-1)^\ell\cdot\int_a^{t}\dd s_1\int_a^{s_1}\dd s_2 \: {\dots} \int_a^{s_{\ell-1}}\dd s_\ell \: (\bilbr{\psi(s_\ell)}\cp \dots \cp \bilbr{\psi(s_{1})})(X)
\end{align*}
for each $\ell\geq 1$.  
We furthermore define (definedness is covered by Lemma  \ref{knkjfdsjkfasdkjdsjkdsa})
\begin{align*}
	\textstyle \Add^\pm_\psi[X](t):= \sum_{\ell=0}^\infty  \APOL^\pm_{\ell,\psi}[X](t)\qquad\quad\forall\:  X\in \mq,\: t\in [a,b],
\end{align*} 
and let
\vspace{-8pt}
\begin{align*}
	\etam_\psi^\pm\colon [a,b]\times \mq\rightarrow  \mq,\qquad (t,X)\mapsto \Add_\psi^\pm[X](t). 
\end{align*}
To simplify the notations,   
we define
\begin{align*}
	\Add_{\psi}^\pm[t]\colon \mq\ni X&\mapsto \Add_\psi^\pm[X](t)\in \mq
	\qquad\quad\forall\: t\in [a,b]\\
	\Add_\psi^\pm&:= \Add_{\psi}^\pm[b]\\[1pt]
	\APOL_{\ell,\psi}^\pm[t]\colon \mq\ni X&\mapsto \APOL_{\ell,\psi}^\pm[X](t)\in \mq
	\qquad\quad\hspace{4pt}\forall\: t\in [a,b],\:\: \ell\in \NN\\
	\APOL_{\ell,\psi}^\pm&:= \APOL_{\ell,\psi}^\pm[b]\qquad\quad\hspace{40pt}\forall\: \ell\in \NN.
\end{align*}
\subsubsection{Elementary Properties}
\label{lkdslkdlkdslk} 
In this section, we discuss the elementary properties of the objects defined in the beginning of Sect.\ \ref{kdsjlkdslkjsdjkds}. We start with the following lemma.
\begin{lemma}
\label{knkjfdsjkfasdkjdsjkdsa}
	Let $\psi \in C^0([a,b],\mq)$ be given. Then, the following assertions hold:
\begingroup
\setlength{\leftmargini}{17pt}
{
\renewcommand{\theenumi}{\emph{\arabic{enumi})}} 
\renewcommand{\labelenumi}{\theenumi}
\begin{enumerate}
\item
\label{aknkjfdsjkfasdkjdsjkdsa1}
	The maps $\etam_\psi^\pm$ 	
	 are defined and of class $C^1$, and we have 
\begin{align*}
	\partial_1 \etam_\psi^+(t,X)&\textstyle=\partial_t\Add_\psi^+[X](t)=\bil{\psi(t)}{\Add^+_\psi[X](t)}\\
	\partial_1 \etam^-_\psi(t,X)&\textstyle=\partial_t\Add_\psi^-[X](t)=-\Add^-_\psi[\bil{\psi(t)}{X}](t)
\end{align*}
for all $t\in [a,b]$ and $X\in \mq$.
\item
\label{aknkjfdsjkfasdkjdsjkdsa2}
	$\etam^\pm_\psi(t,\cdot)=\Add_\psi^\pm[t]$ is linear and continuous for each $t\in [a,b]$.
\item
\label{aknkjfdsjkfasdkjdsjkdsa3}
	If $\psi\in C^k([a,b],\mq)$ holds for $k\in \NN\cup\{\infty\}$, then $\etam_\psi^\pm$ is of class $C^{k+1}$.    	
\end{enumerate}}
\endgroup   	
\end{lemma}
\begin{proof}
	The proof is straightforward but technical. It is provided in Appendix \ref{asassadsdsdsdsdsdcxcxsdssdssddssd}.
\end{proof}
\begin{remark}
\label{iufdicxcxcxcxufdiufdiufddd}
The substitution formula \eqref{substitRI} yields the following:
\begingroup
\setlength{\leftmargini}{12pt}{
\begin{itemize}
\item
	Let $a<b$, $a'<b'$, $\psi\in C^0([a,b],\mq)$, as well as $\varrho\colon [a',b']\rightarrow [a,b]$ be of class $C^1$ with $\varrho(a')=a$ and $\varrho(b')=b$. We set $\phi:= \dot\varrho\cdot (\psi\cp\varrho)$, and obtain inductively from \eqref{substitRI} that
\begin{align*}
	 \textstyle\APOL^\pm_{\ell,\phi}[t]=\APOL^\pm_{\ell,\psi}[\varrho(t)]\qquad\quad\forall\: t\in[a',b'] 
\end{align*}	
	  holds for all $\ell\in \NN$. Hence, we have $
		\Add^\pm_{\dot\varrho\he\cdot\he (\psi\he\cp\he\varrho)}[\cdot]=\Add^\pm_{\psi}[\varrho(\cdot)]$. In particular, if additionally $\dot\varrho|_{(a,b)}>0$ holds, then we have
		\begin{align}
		\label{kjfdkjfdjkfdkjfdkjkjfdkjfdkjfdkjfd}
			\Add^\pm_{(\dot\varrho\he\cdot\he (\psi\he\cp\he\varrho))|_{[x',y']}}[\cdot]=\Add^\pm_{\psi|_{[\varrho(x'),\varrho(y')]}}[\varrho(\cdot)]\qquad\quad\forall\: a'\leq x'<y\leq b'.
		\end{align}
\item
	The previous point applied to affine transformations yields\footnote{Specifically, in the previous point set $\psi:=\chi(\cdot -x)\colon [a+x,b+x]\rightarrow 
	 \mq$ as well as $\varrho\colon [a,b]\ni t\mapsto t+x\in [a+x,b+x]$.}
\begin{align*}
	\Add^\pm_\chi[b]=\Add^\pm_{\chi(\cdot-x)}[b+x]
\end{align*} 
for $a<b$, $\chi\in C^0([a,b],\mq)$, and $x\in \RR$.\hspace*{\fill}\qed
\end{itemize}}
\endgroup
\end{remark} 
\begin{corollary}
\label{lkjfdkjfdkjkjldfjldf}
	For $\psi \in C^0([a,b],\mq)$ and $t\in [a,b]$, we have $\Add_\psi^-[t]\cp \Add_\psi^+[t]=\id_\mq$. 
\end{corollary}
\begin{proof}
	Let $Z\in \mq$ be given, and define 
\begin{align*}
	\alpha_Z\colon [a,b]\ni t\mapsto \etam_\psi^-(t,\etam_\psi^+(t,Z))=(\Add^-_\psi[t]\cp\Add^+_\psi[t])(Z)\in \mq.
\end{align*}	
	 Then, Lemma \ref{knkjfdsjkfasdkjdsjkdsa} shows $\alpha_Z\in C^1([a,b],\mq)$; 
and the parts \ref{linear}, \ref{chainrule}, \ref{productrule} of Proposition \ref{iuiuiuiuuzuzuztztttrtrtr} yield
\begin{align*}
	\dot\alpha_Z(t)&\textstyle = 
	-\etam_\psi^-(t,\bil{\psi(t)}{\etam_\psi^+(t,Z)}) + \etam_\psi^-(t,\bil{\psi(t)}{\etam_\psi^+(t,Z)})=0\qquad\quad\forall\: t\in [a,b],
\end{align*}
hence $\alpha_Z(t)=Z+ \int_a^t \dot\alpha_Z(s) \:\dd s=Z$ for each $t\in [a,b]$ by \eqref{isdsdoisdiosd}. 
\end{proof}

\begin{lemma}
\label{jdjdfkjfdkjfd}
For $\psi\in \CP^0([a,b],\mq)$ and $a< c< b$, we have $\Add^+_\psi=\Add^+_{\psi|_{[c,b]}}\cp\Add^+_{\psi|_{[a,c]}}$.
\end{lemma}
\begin{proof}
	Let $X\in \mq$ be fixed. We first prove by induction that 
	\begin{align}
	\label{uidfuidfuiduifdlkdfdklklfd}
	\APOL^+_{\ell,\psi}[X](s)\textstyle=\sum_{m=0}^\ell \APOL^+_{m,\psi|_{[c,b]}}[\APOL^+_{\ell-m,\psi|_{[a,c]}}[X](c)](s)\qquad\quad\forall\: c<s\leq b
	\end{align}
	holds for each $\ell\in \NN$. 
	It is clear from the definitions that 
		\begin{align*}
		\APOL^+_{0,\psi}[X](s)=X= \APOL^+_{0,\psi|_{[c,b]}}[\APOL^+_{0,\psi|_{[a,c]}}[X](c)](s) \qquad\quad\forall\: c<s\leq b. 
	\end{align*}
	We thus can assume that \eqref{uidfuidfuiduifdlkdfdklklfd} holds for some $\ell\in \NN$; and obtain for $c<s\leq b$ that
\begin{align*}
	\APOL^+_{\ell+1,\psi}[X](s)
	&\textstyle=\int_a^s \bil{\psi(s_0)}{\APOL^+_{\ell,\psi}[X](s_0)} \:\dd s_0\\[4pt]
	&\textstyle=\int_a^c \bil{\psi(s_0)}{\APOL^+_{\ell,\psi|_{[a,c]}}[X](s_0)} \:\dd s_0 + \int_c^s \bil{\psi(s_0)}{\APOL^+_{\ell,\psi}[X](s_0)} \:\dd s_0 \\[1pt]
	&\textstyle= \APOL^+_{\ell+1,\psi|_{[a,c]}}[X](c) + \sum_{m=0}^\ell \int_c^s \bil{\psi(s_0)}{ \APOL^+_{m,\psi|_{[c,b]}}[\APOL^+_{\ell-m,\psi|_{[a,c]}}[X](c)](s_0)} \:\dd s_0\\
	&\textstyle=\APOL^+_{\ell+1,\psi|_{[a,c]}}[X](c) + \underbracket{\textstyle\sum_{m=0}^\ell \APOL^+_{m+1,\psi|_{[c,b]}}[\APOL^+_{(\ell+1)-(m+1),\psi|_{[a,c]}}[X](c)](s)}_{\sum_{m=1}^{\ell+1} \APOL^+_{m,\psi|_{[c,b]}}[\APOL^+_{(\ell+1)-m,\psi|_{[a,c]}}[X](c)](s)}\\
	&\textstyle=\sum_{m=0}^{\ell+1} \APOL^+_{m,\psi|_{[c,b]}}[\APOL^+_{\ell+1-m,\psi|_{[a,c]}}[X](c)](s) 
\end{align*}	 
holds, so that \eqref{uidfuidfuiduifdlkdfdklklfd} follows by induction. 
Then, for $n\geq 1$ we have
\begin{align}
\label{dsdsdsdsds1}
	\Delta_n:=&\he\textstyle\sum_{\ell=0}^{2 n}\APOL^+_{\ell,\psi}(X) - \sum_{p=0}^{n}\APOL^+_{p,\psi|_{[c,b]}}( \sum_{q=0}^{n}\APOL^+_{q,\psi|_{[a,c]}}(X))\\
\label{dsdsdsdsds2}
   =&\he\textstyle\sum_{\ell=0}^{2 n}\sum_{m=0}^\ell \APOL^+_{m,\psi|_{[c,b]}}(\APOL^+_{\ell-m,\psi|_{[a,c]}}(X)) - \sum_{p,q=0}^{n}\APOL^+_{p,\psi|_{[c,b]}}(\APOL^+_{q,\psi|_{[a,c]}}(X)).
\end{align}
For $\vv\leq \ww$ as in \eqref{assaaass}, we obtain from \eqref{dsdsdsdsds2} and the binomial theorem that
\begin{align*}
	\textstyle\vv(\Delta_n)&\textstyle\leq \sum_{\ell=n+1}^{2 n}\sum_{m=0}^\ell \vv(\APOL^+_{m,\psi|_{[c,b]}}(\APOL^+_{\ell-m,\psi|_{[a,c]}}(X)))\\
	&\textstyle\leq \ww(X)\cdot \sum_{\ell=n+1}^{\infty} (\sum_{m=0}^\ell \frac{1}{m!\cdot(\ell-m)!}\cdot (b-c)^m\cdot (c-a)^{\ell-m})\cdot\ww_\infty(\psi)^\ell\\
	&\textstyle= \ww(X)\cdot \sum_{\ell=n+1}^{\infty} \frac{1}{\ell!}\cdot (b-a)^\ell\cdot  \ww_\infty(\psi)^\ell 
\end{align*}
holds, which yields
\vspace{-15pt}
\begin{align*}
	\textstyle 0=\lim_n \Delta_n\stackrel{\eqref{dsdsdsdsds1}}{=}\Add^+_\psi(X) - \lim_n\he \overbracket{\textstyle\sum_{p=0}^{n}\APOL^+_{p,\psi|_{[c,b]}}(\sum_{q=0}^{n}\APOL^+_{q,\psi|_{[a,c]}}(X))}^{=:\he \beta_n}.
\end{align*}	
It thus remains to show $\lim_n \beta_n=\Add^+_{\psi|_{[c,b]}}(\Add^+_{\psi|_{[a,c]}}(X))$. For this, we observe that
\begin{align*}
	\textstyle\wt{\Delta}_n:=
	&\he\textstyle\Add^+_{\psi|_{[c,b]}}(\Add^+_{\psi|_{[a,c]}}(X))\textstyle- \beta_n\\
	=&\he\textstyle \underbracket{\textstyle\sum_{p=0}^{n}\APOL^+_{p,\psi|_{[c,b]}}(\sum_{q=n+1}^{\infty}\APOL^+_{q,\psi|_{[a,c]}}(X))}_{\mathrm{T}_1(n)} +  
					\underbracket{\textstyle\textstyle\sum_{p=n+1}^{\infty}\APOL^+_{p,\psi|_{[c,b]}}( \Add^+_{\psi|_{[a,c]}}(X))}_{\mathrm{T}_2(n)}
\end{align*}
holds for each $n\in \NN$; and choose $\ww\leq \mm$ as in \eqref{assaaass} for $\vv\equiv \ww$ and $\ww\equiv \mm$ there. We obtain 
\begin{align*}
	\textstyle\vv(\mathrm{T}_1(n))
	&\textstyle \leq \e^{(b-c)\cdot \ww_\infty(\psi)}\cdot \ww(\sum_{q=n+1}^{\infty}\APOL^+_{q,\psi|_{[a,c]}}(X))\\
		&\textstyle \leq \e^{(b-c)\cdot \ww_\infty(\psi)}\cdot \sum_{q=n+1}^{\infty} \frac{1}{q!}\cdot (c-a)^q\cdot \mm_\infty(\psi)^p,
\end{align*}
hence $\lim_n \mathrm{T}_1(n)=0$. 
We furthermore obtain
\begin{align*}
	\textstyle\vv(\mathrm{T}_2(n))
	&\textstyle \leq (\sum_{p=n+1}^{\infty} \frac{1}{p!}\cdot (b-c)^q\cdot \ww_\infty(\psi)^p) \cdot \ww(\Add^+_{\psi|_{[a,c]}}(X))\\
	&\textstyle \leq  (\sum_{p=n+1}^{\infty}\frac{1}{p!}\cdot (b-c)^p\cdot \ww_\infty(\psi)^p )\cdot  \e^{(c-a)\cdot \mm_\infty(\psi)} \cdot \mm(X),
\end{align*}
hence $\lim_n \mathrm{T}_2(n)=0$. It follows that $\lim_n\wt{\Delta}_n=0$ holds, which proves the claim.
\end{proof}
\begin{remark}
\label{pofdpofdpofdpofdfdfdcvcvvvc}
	Lemma \ref{jdjdfkjfdkjfd}, together with the second point in Remark \ref{iufdicxcxcxcxufdiufdiufddd} shows
\begin{align*}
	\Add^+_\psi[b]=\Add^+_{\psi(\cdot-x)}[b+x]
\end{align*} 
for $a<b$, $\psi\in \CP^0([a,b],\mq)$, and $x\in \RR$.
\hspace*{\fill}\qed
\end{remark}
\begin{lemma}
\label{kjfdkjdfljllkjfd}
Let $\psi\in \CP^0([a,b],\mq)$ and $\{\psi_n\}_{n\in \NN}\subseteq \CP^0([a,b],\mq)$ be given, with $\{\psi_n\}_{n\in \NN}\rightarrow\psi$ w.r.t.\ the $C^0$-topology. Then, for each $X\in \mq$, we have  $\{\Add^\pm_{\psi_n}[X]\}_{n\in \NN} \rightarrow \Add_\psi^\pm[X]$ 
w.r.t.\ the $C^0$-topology.  
\end{lemma}
\begin{proof}
For each $\ell\geq 1$ and $X_1,\dots,X_\ell,Y_1,\dots,Y_\ell,X\in \mq$, we have
\begin{align}
\begin{split}
\label{oioigfoifdd}
	(\com{X_1}\cp{\dots}\cp\com{X_\ell})(X)&\textstyle-(\com{Y_1}\cp{\dots}\cp\com{Y_\ell})(X)\\[2pt]
	&\textstyle= \sum_{p=1}^\ell\: (\com{Y_1}\cp\dots\cp \com{Y_{p-1}}\cp \com{X_p-Y_p}\cp \com{X_{p+1}}\cp\dots\cp \com{X_\ell})(X).
\end{split} 
\end{align}	
Let $\vv\leq \ww$ be as in \eqref{assaaass}. For $N\in \NN$ suitably large, we have $\ww_\infty(\psi_n)\leq 2\cdot \ww_\infty(\psi)$ for each $n\geq N$. 
We obtain for $n\geq N$ from \eqref{oioigfoifdd} (second step) that
\begin{align*}
	\vv(\Add^\pm_{\psi_n}[X](t)-\Add^\pm_\psi[X](t))&\textstyle\leq \sum_{\ell=0}^\infty \vv(\APOL^\pm_{\ell,\psi_n}[X](t)-\APOL^\pm_{\ell,\psi}[X](t))\\
	&\textstyle\leq (b-a)\cdot \ww_\infty(\psi_n-\psi) \cdot \ww(X)\cdot\sum_{\ell=1}^\infty \frac{(b-a)^{\ell-1}}{(\ell-1)!}\cdot 2^{\ell-1}\cdot \ww_\infty(\psi)^{\ell-1}\\
	&\textstyle= (b-a)\cdot \ww_\infty(\psi_n-\psi) \cdot \ww(X)\cdot\e ^{2\cdot (b-a)\cdot \ww_\infty(\psi)}
\end{align*}
holds for each $t\in [a,b]$, which proves the claim.
\end{proof}

\subsubsection{Uniqueness of the Solution}
\label{ofdoifdoifdojkdnmcvnmc}
For $\psi\in C^0([a,b],\mq)$ with $a<b$, we define $\inverse{\psi}\in C^0([a,b],\mq)$ as in Example \ref{fdpofdopdpof} by 
\begin{align*}
	\inverse{\psi}\colon [a,b]\ni t\mapsto -\psi(a+b-t)\in \mq. 
\end{align*}
Notably, we have $\inverse{\inverse{\psi}}=\psi$. 
In this section, we prove the following proposition.\footnote{For $a<b$ and $\chi\in C^0([a,b],\mq)$, set $\Add^\pm_{\chi|_{[a,a]}}:=\id_\mq$. } 
\begin{proposition}
\label{lkjfdldlkkjfdYYX}
	Let $\psi\in C^0([a,b],\mq)$ be given. 
	Then, the following assertions hold:
\begingroup
\setlength{\leftmargini}{17pt}
{
\renewcommand{\theenumi}{\emph{\arabic{enumi})}} 
\renewcommand{\labelenumi}{\theenumi}
\begin{enumerate}
\item
\label{lkjfdldlkkjfdYYX1}
	We have $\Add^\pm_\psi[t]\in \Aut(\mq)$ for each $t\in [a,b]$,   with $\Add^\pm_\psi[t]^{-1}=\Add^\mp_\psi[t]=\Add^\pm_{\inverse{\psi|_{[a,t]}}}$.
\item
\label{lkjfdldlkkjfdYYX2}
Let $X\in \mq$ be given. Then, $\Add^+_\psi[X]$ is the unique solution $\alpha\in C^1([a,b],\mq)$ to the differential equation (Lax equation) 
\begin{align*}
	\dot\alpha=\bil{\psi}{\alpha}\qquad\:\:\text{with initial condition}\qquad\:\: \alpha(a)=X.
\end{align*}
\end{enumerate}}
\endgroup		
\end{proposition}
\noindent
We obtain the following corollaries.
\begin{corollary}
\label{fdjkfdlkjfdkjfd}
For each $\psi\in C^0([a,b],\mq)$, we have
\begin{align*}
	\Add^\pm_\psi[\bil{X}{Y}]=\bil{\Add^\pm_\psi[X]}{\Add^\pm_\psi[Y]}\qquad\quad\forall\: X,Y\in \mq.
\end{align*}
\end{corollary}
\begin{proof}
	Let $X,Y\in \mg$ and $\psi\in C^0([a,b],\mq)$ be given.  
	Lemma \ref{knkjfdsjkfasdkjdsjkdsa} shows 
	\begin{align*}
		C^1([a,b],\mq)\ni\alpha\colon [a,b]\ni t\mapsto \bil{\Add^+_\psi[X](t)}{\Add^+_\psi[Y](t)}\in \mq,
	\end{align*}		
	with $\alpha(a)=\bil{X}{Y}$. The parts \ref{linear},   \ref{productrule} of Proposition \ref{iuiuiuiuuzuzuztztttrtrtr} (first step), and the Jacobi identity \eqref{nmvcnmvcnmnkjsakjsakjsaa} (second step) show that
	\begin{align*}
		\dot\alpha(t)&\textstyle=\bil{\bil{\psi(t)}{\Add^+_\psi[X](t)}}{\Add^+_\psi[Y](t)} + \bil{\Add^+_\psi[X](t)}{\bil{\psi(t)}{\Add^+_\psi[Y](t)}}\\
				  &\textstyle= \bil{\psi(t)}{\bil{\Add^+_\psi[X](t)}{\Add^+_\psi[Y](t)}}\\
				  &\textstyle=\bil{\psi(t)}{\alpha(t)}		
	\end{align*}
	holds for each $t\in [a,b]$. Proposition \ref{lkjfdldlkkjfdYYX}.\ref{lkjfdldlkkjfdYYX2} yields $\alpha=\Add^+_\psi[\bil{X}{Y}]$, which proves the one part of the statement.  Then,  Proposition \ref{lkjfdldlkkjfdYYX}.\ref{lkjfdldlkkjfdYYX1} shows
\begin{align*}
	\Add^-_\psi[\bil{X}{Y}](t)&= \Add^-_\psi[t](\bil{(\Add^+_\psi[t] \cp\Add^-_\psi[t])(X)}{(\Add^+_\psi[t]\cp \Add^-_\psi[t])(Y)})\\
	&= (\Add^-_\psi[t]\cp \Add^+_\psi[t])(\bil{\Add^-_\psi[X](t)}{\Add^-_\psi[Y](t)})\\
	&=\bil{\Add^-_\psi[X](t)}{\Add^-_\psi[Y](t)}
\end{align*}	
for each $t\in [a,b]$ and $X,Y\in \mq$, 
which establishes the claim.
\end{proof}
\begin{corollary}
\label{oidspodspodspod}
	For $\psi\in C^0([a,b],\mq)$ and $\varrho\colon [a',b']\rightarrow [a,b]$ of class $C^1$, we have
\begin{align*}
	 \textstyle\Add^+_{\psi}[\varrho(t)]=\Add^+_{\dot\varrho\he\cdot \he(\psi\he\cp\he\varrho)}[t]\cp \Add_{\psi}^+[\varrho(a')]\qquad\quad\forall\: t\in [a',b'].
\end{align*}  
\end{corollary}
\begin{proof}
	For $Z\in \mq$, define
\begin{align*}
	\alpha_Z(t):=\Add^+_{\psi}[\varrho(t)](\Add_\psi^-[\varrho(a')](Z))\qquad\quad\forall\: t\in [a',b'].
\end{align*}	
	 Proposition \ref{lkjfdldlkkjfdYYX} shows $\alpha_Z\in C^1([a',b'],\mq)$, with $\alpha_Z(a')=Z$ and (also apply \ref{chainrule} in Proposition \ref{iuiuiuiuuzuzuztztttrtrtr})
	 \begin{align*}
		\dot\alpha_Z(t)= \bil{\dot\varrho(t)\cdot (\psi\cp\varrho)(t)}{ \alpha_Z(t)} \qquad\quad\forall\: t\in [a',b'].	
	\end{align*}	 
	The uniqueness statement in Proposition \ref{lkjfdldlkkjfdYYX} (second identity) yields 
	\begin{align*}
		\Add^+_{\psi}[\varrho(t)](\Add_\psi^-[\varrho(a')](Z))=\alpha_Z(t)= \Add^+_{\dot\varrho\he\cdot \he(\psi\he\cp\he\varrho)}[t](Z)\qquad\quad\forall\: t\in [a,b],\: Z\in \mq. 
	\end{align*}
	Given $X\in \mq$, we set $Z:=\Add_\psi^+[\varrho(a')](X)$, and obtain from Proposition \ref{lkjfdldlkkjfdYYX}.\ref{lkjfdldlkkjfdYYX1} that 
\begin{align*}
	 \textstyle\Add^+_{\psi}[\varrho(t)](X)=\Add^+_{\dot\varrho\he\cdot \he(\psi\he\cp\he\varrho)}[t]( \Add_{\psi}^+[\varrho(a')](X))\qquad\quad\forall\: t\in [a',b']
\end{align*} 
holds, from which the claim is clear.
\end{proof}

\begin{remark}
\label{kfdkjfdjkfdfd}
	Let $\mathcal{L}\colon \mq\rightarrow \mq$ be linear and  continuous, with 
	\begin{align}
	\label{podsposdposdposd}
		\mathcal{L}(\bil{X}{Y})=\bil{\mathcal{L}(X)}{\mathcal{L}(Y)}\qquad\quad\forall\: X,Y\in \mq. 
	\end{align}	 
	It follows by iterated application of \eqref{podsposdposdposd} and  \eqref{pofdpofdpofdsddsdsfd} that
		\begin{align}
		\label{oidsoidsoidsds}
		\mathcal{L}\cp \Add^\pm_\psi[t] = \Add^\pm_{\mathcal{L}\he\cp\he \psi}[t]\cp \mathcal{L}\qquad\quad\forall\: a<b,\: \psi\in C^0([a,b],\mq),\:t\in [a,b]
	\end{align} 	
	holds. 
	This equality alternatively follows from Proposition \ref{lkjfdldlkkjfdYYX}: 
\begin{proof}[Proof of Equation \eqref{oidsoidsoidsds}]
	Given $\psi\in C^0([a,b],\mq)$ and $Z\in \mq$, define $\alpha:=\mathcal{L}\cp \Add^+_{\psi}[Z]\in C^1([a,b],\mq)$. Then, we have $\alpha(a)=\mathcal{L}(Z)$, with
	\begin{align*}
		\dot\alpha=\mathcal{L}(\bil{\psi}{\Add^+_\psi[Z]})=\bil{\mathcal{L}\cp\psi}{\alpha}.
	\end{align*}
	Proposition \ref{lkjfdldlkkjfdYYX}.\ref{lkjfdldlkkjfdYYX2} shows $\alpha=\Add^+_{\mathcal{L}\he\cp\he \psi}[\mathcal{L}(Z)]$, which establishes the ``$+$''-case in \eqref{oidsoidsoidsds}. 
	Together with Proposition \ref{lkjfdldlkkjfdYYX}.\ref{lkjfdldlkkjfdYYX1}, we obtain 	
\begin{align*}
		\mathcal{L}\cp \Add^-_{\psi}[t](Z)&=\mathcal{L}\cp \Add^+_{\inverse{\psi}|_{[a,t]}}(Z)\\
		&=\Add^+_{\mathcal{L}\he \cp\he\inverse{\psi|_{[a,t]}}}(\mathcal{L}(Z))\\
		&=\Add^+_{\inverse{\mathcal{L}\he\cp\he\psi|_{[a,t]}}}(\mathcal{L}(Z))\\
		&=\Add^-_{\inverse{\inverse{\mathcal{L}\he\cp\he\psi|_{[a,t]}}}}(\mathcal{L}(Z))\\
		&=\Add^-_{\mathcal{L}\he\cp\he\psi|_{[a,t]}}(\mathcal{L}(Z))
	\end{align*}
	for $a<t\leq b$ and $\psi\in C^0([a,b],\mq)$, which establishes the ``$-$''-case in \eqref{oidsoidsoidsds}.
\end{proof}
\end{remark}
\noindent
For the proof of Proposition \ref{lkjfdldlkkjfdYYX}, we shall need the following facts and definitions:
\begingroup
\setlength{\leftmargini}{12pt}{
\begin{itemize}
\item
	For $\phi\equiv \{\phi[p]\}_{0\leq p\leq n-1}\in \CP^0([a,b],\mq)$, we define $\inverse{\phi}\in \CP^0([a,b],\mq)$ by 
	\begin{align*}
		\inverse{\phi}:= \{\inverse{\phi[(n-1)-p]}\}_{0\leq p\leq n-1}.
	\end{align*}
	\vspace{-18pt}
\item
	Let $\phi\equiv \{\phi[p]\}_{0\leq p\leq n-1}\in \CP^0([a,b],\mq)$ ($a<b$, $n\geq 1$) and $\phi'\equiv \{\phi'[q]\}_{0\leq q\leq n'-1}\in \CP^0([a',b'],\mq)$ ($a'<b'$, $n'\geq 1$) be given. We set
$$
\psi[\ell]:= 
\begin{cases}
\phi[\ell] &\text{for}\quad 0\leq \ell\leq n-1\\
\phi'[\ell-n]  &\text{for}\quad n\leq \ell\leq n+n'-1, 
\end{cases}
$$ 	
and define $\phi\triangleright \phi'\in \CP^0([a,b+(b'-a')],\mq)$ by
\begin{align*}
	\phi\triangleright \phi':= \{\psi[\ell]\}_{0\leq \ell\leq n+n'-1}. 
\end{align*}
Lemma \ref{jdjdfkjfdkjfd} together with Remark \ref{pofdpofdpofdpofdfdfdcvcvvvc} shows
\begin{align}
\label{ckjkjdskjkjdds}
	\Add^+_{\phi\:\triangleright\: \phi'}\textstyle=\Add^+_{\phi\:\triangleright\: \phi'}|_{[b,b + (b'-a')]}\cp \Add^+_{\phi\:\triangleright\: \phi'}|_{[a,b]}
	= \Add^+_{\phi'}\cp \Add^+_{\phi}.
\end{align}
\item
We have $\inverse{\inverse{\psi}}=\psi$ for each $\psi\in C^0([a,b],\mq)$.
\item
	For $a<b$ and $Z\in \mq$, we have (observe $\inverse{\mathcal{C}_Z|_{[a,b]}}=-\mathcal{C}_Z|_{[a,b]}$)
	\begin{align}
	\label{dslkklldslkdowqpowqpowq}
	\textstyle\Add^+_{\mathcal{C}_Z|_{[a,b]}}[b]\textstyle=\sum_{\ell=0}^\infty\:  \frac{(b-a)^\ell}{\ell!}\cdot \com{Z}^\ell=\Add^-_{\inverse{\mathcal{C}_Z|_{[a,b]}}}[b]. 
	\end{align}
	Corollary \ref{lkjfdkjfdkjkjldfjldf} (second step) yields 
\begin{align}
\label{iudiufdiufdiufdiufd}
	\Add^+_{\mathcal{C}_Z|_{[a,b]}}[b]\cp \Add^+_{\inverse{\mathcal{C}_Z|_{[a,b]}}}[b]\stackrel{\eqref{dslkklldslkdowqpowqpowq}}{=}\Add^-_{\inverse{\mathcal{C}_Z|_{[a,b]}}}[b]\cp \Add^+_{\inverse{\mathcal{C}_Z|_{[a,b]}}}[b]=\id_\mq.
\end{align}
\end{itemize}}
\endgroup  
\noindent
We obtain the following statement. 
\begin{lemma}
\label{kjdskjsdkjdskjkjds}
	For $\psi \in C^0([a,b],\mq)$, we have $\Add^+_\psi\cp \Add^+_{\inverse{\psi}}=\id_\mq$. 
\end{lemma}
\begin{proof}
For $n\geq 1$ and $p=0,\dots,n$, we define $t_{n,p}:=a+ \frac{p}{n}\cdot (b-a)$ and  $Z_{n,p}:=\psi(t_{n,p})$. We set
\begin{align*}
	\chi_{n,p}:=\mathcal{C}_{Z_{n,p}}|_{[t_{n,p},t_{n,p+1}]}
	\qquad\quad\forall\: n\geq 1,\: p=0,\dots,n-1.
\end{align*}
For $n\geq 1$, we define $\psi^\pm_n\equiv\{\psi_n^\pm[p]\}_{0\leq p\leq n-1}\in \CP^0([a,b],\mq)$ by  
\begin{align*}
	\textstyle\psi_n^+[p] := \chi_{n,p}\qquad\text{as well as}\qquad \psi_n^-[p] := \inverse{\chi_{n,(n-1) -p}}\qquad\text{for}\qquad p=0,\dots,n-1. 
\end{align*}
(hence $\psi_n^-=\inverse{\psi_n^+}$. 
Obviously, we have $\{\psi^+_n\}_{n\geq 1}\rightarrow \psi$ 
as well as $\{\psi^-_n\}_{n\geq 1}\rightarrow \inverse{\psi}$ 
w.r.t.\ the $C^0$-topology. Moreover, for $n\geq 1$, we have by Lemma \ref{jdjdfkjfdkjfd} and \eqref{iudiufdiufdiufdiufd} that
\begin{align}
\label{jdskkjdsjds}
	\Add^+_{\psi^+_n}\cp\Add^+_{\psi^-_n}=\Add^+_{\chi_{n,0}}\cp\dots\cp (\Add^+_{\chi_{n,n-1}} \cp \Add^+_{\inverse{\chi_{n,n-1}}} )\cp{\dots}\cp \Add^+_{\inverse{\chi_{n,0}}}=\id_\mq
\end{align}
holds. 
We define $\phi:=\inverse{\psi}\triangleright \psi$, as well as $\phi_n:=\psi^-_n\triangleright \psi^+_n$ for each $n\geq 1$. Then, $\{\phi_n\}_{n\geq 1}\rightarrow \phi$ converges w.r.t.\ the $C^0$-topology; and we obtain from Lemma \ref{kjfdkjdfljllkjfd} (second step) that
\begin{align*}
	\textstyle
	 (\Add^+_{\psi}\cp\Add^+_{\inverse{\psi}})(X)\stackrel{\eqref{ckjkjdskjkjdds}}{=} \Add^+_{\phi}(X)=\lim_n \Add^+_{\phi_n}(X)\stackrel{\eqref{ckjkjdskjkjdds}}{=} \lim_n (\Add^+_{\psi_n^+}\cp \Add^+_{\psi_n^-})(X)\stackrel{\eqref{jdskkjdsjds}}{=}X
\end{align*}
holds for each $X\in \mq$, which proves the claim. 
\end{proof}
\noindent
We are ready for the proof of Proposition \ref{lkjfdldlkkjfdYYX}.
\begin{proof}[Proof of Proposition \ref{lkjfdldlkkjfdYYX}]
\begingroup
\setlength{\leftmargini}{17pt}
{
\renewcommand{\theenumi}{\emph{\arabic{enumi})}} 
\renewcommand{\labelenumi}{\theenumi}
\begin{enumerate}
\item[\ref{lkjfdldlkkjfdYYX1}]
The claim is clear for $t=a$. Then, it suffices to prove the statement for $t=b$, as then we can apply it to the restriction $\psi|_{[a,t]}$ for each $a<t<b$.   
By Lemma \ref{kjdskjsdkjdskjkjds} (first step), and Corollary \ref{lkjfdkjfdkjkjldfjldf} (third step),  
we have 
\begin{align*}
	\Add_\psi^-[b]
	&=\Add_\psi^-[b]\cp (\Add_\psi^+[b]\cp \Add^+_{\inverse{\psi}}[b])\\
	&=(\Add_\psi^-[b]\cp \Add_\psi^+[b])\cp \Add^+_{\inverse{\psi}}[b]
	=\Add^+_{\inverse{\psi}}[b].
\end{align*}
Applying $\Add_\psi^+$ from the left, Lemma \ref{kjdskjsdkjdskjkjds} yields (the second equality is due to Corollary \ref{lkjfdkjfdkjkjldfjldf})
\begin{align*}
	\Add_\psi^+[b]\cp \Add_\psi^-[b]=\id_\mq=\Add_\psi^-[b]\cp \Add_\psi^+[b].
\end{align*}
Together, we have shown the following properties:
\begingroup
\setlength{\leftmarginii}{12pt}{
\begin{itemize}
\item
$\Add^\pm_{\psi}[b]\in \Aut(\mq)$\quad with\quad 
$\Add^\pm_{\psi}[b]^{-1}=\Add^\mp_{\psi}[b]$,
\item
$\Add_\psi^-[b]=\Add^+_{\inverse{\psi}}[b]$.
\end{itemize}}
\endgroup
\noindent 
Applying the third property to 
$\wt{\psi}:=\inverse{\psi}$ (second step), we obtain from $\inverse{\inverse{\psi}}=\psi$ (third step) that
\begin{align*}
	\Add^-_{\inverse{\psi}}[b]=\Add^-_{\wt{\psi}}[b]=\Add_{\inverse{\wt{\psi}}}^+[b]=\Add_\psi^+[b]	 	
\end{align*}
holds, which completes the proof.
\item[\ref{lkjfdldlkkjfdYYX2}]
By Lemma \ref{knkjfdsjkfasdkjdsjkdsa}, for $\alpha\equiv\Add^+_\psi[X]$,  we have $\alpha\in C^1([a,b],\mq)$ with $\alpha(a)=X$ and $\dot\alpha=\bil{\psi}{\alpha}$. This shows the solution property. For uniqueness, let $\alpha\in C^1([a,b],\mq)$ be given, such that $\alpha(a)=X$ and $\dot\alpha=\bil{\psi}{\alpha}$ holds. By Lemma \ref{knkjfdsjkfasdkjdsjkdsa}, we have $\beta:=\etam_\psi^-(\cdot,\alpha(\cdot))\in C^1([a,b],\mq)$ with (use Part \ref{productrule} of Proposition \ref{iuiuiuiuuzuzuztztttrtrtr} on the left side, as well as \eqref{oidsoidoisdoidsoioidsiods} on the right side)
\begin{align*}
	\dot\beta=-\etam_\psi^-(\cdot,\bil{\psi(\cdot)}{\alpha(\cdot)})+ \etam_\psi^-(\cdot,\bil{\psi(\cdot)}{\alpha(\cdot)})=0\qquad\:\:\:\textstyle\Longrightarrow\qquad\:\:\: \beta=\beta(a)+ \int_a^\bullet \dot\beta(s) \:\dd s=X.
\end{align*}
Applying Part \ref{lkjfdldlkkjfdYYX1}, we obtain
\begin{align*}
	\Add^+_\psi[X](t)=\Add_\psi^+[t](\beta(t))=(\Add_\psi^+[t]\cp \Add_\psi^-[t])(\alpha(t))=\alpha(t)\qquad\quad\forall\: t\in [a,b],
\end{align*}
which proves the claim.\qedhere
\end{enumerate}}
\endgroup	
\end{proof}

\subsubsection{A Group Structure}
\label{sdklsdklsdklsd}
For $\phi,\psi\in C^0([a,b],\mq)$ with $a<b$, we define
\begin{align*}
	 C^0([a,b],\mq)\ni\hspace{5.9pt} \psi^{-1}:=\phantom{\phi(\cdot)} - \etam_\psi^-(\cdot,\psi(\cdot))&\colon [a,b]\ni t \mapsto \phantom{\phi(t)} - \Add_\psi^-[t](\psi(t))\in\mq\\
	 C^0([a,b],\mq)\ni \phi\starm\psi:= \phi(\cdot) + \etam^+_\phi(\cdot,\psi(\cdot))&\colon [a,b]\ni t\mapsto \phi(t) + \Add^+_\phi[t](\psi(t))\in \mq.
\end{align*}
We observe the following.
\begin{lemma}
\label{hdshjdskjdsjkds}
Let $\phi,\psi\in C^0([a,b],\mq)$ be fixed. Then, the following assertions hold:
\begingroup
\setlength{\leftmargini}{17pt}
{
\renewcommand{\theenumi}{\emph{\arabic{enumi})}} 
\renewcommand{\labelenumi}{\theenumi}
\begin{enumerate}
\item
\label{hdshjdskjdsjkds1}
We have $\Add^+_{\psi^{-1}}[t]=\Add_\psi^-[t]$\hspace{45.3pt} for each $t\in [a,b]$.
\item
\label{hdshjdskjdsjkds2}
We have $\Add^+_{\phi\he\starm\he\psi}[t]=\Add^+_\phi[t]\cp\Add^+_\psi[t]$\: for each $t\in [a,b]$.
\end{enumerate}}
\endgroup		
\end{lemma}
\begin{proof}
\begingroup
\setlength{\leftmargini}{17pt}
{
\renewcommand{\theenumi}{\emph{\arabic{enumi})}} 
\renewcommand{\labelenumi}{\theenumi}
\begin{enumerate}
\item
Let $X\in \mq$ be given, and set $\alpha_X:= \Add_\psi^-[X]\in C^1([a,b],\mq)$. We have $\alpha_X(a)=X$; and Lemma \ref{knkjfdsjkfasdkjdsjkdsa} shows $\alpha_X:= \Add_\psi^-[X]\in C^1([a,b],\mq)$ with (for the second step use Corollary \ref{fdjkfdlkjfdkjfd})     
\begin{align*}
	\dot\alpha_X(t)\textstyle= -\Add_\psi^-[t](\bil{\psi(t)}{X})=\bil{-\Add_\psi^-[t](\psi(t))}{\Add_\psi^-[X](t)}=\bil{\psi^{-1}(t)}{\alpha_X(t)}
\end{align*}
for each $t\in [a,b]$. 
Then, Proposition \ref{lkjfdldlkkjfdYYX}.\ref{lkjfdldlkkjfdYYX2} gives  $\alpha_X=\Add^+_{\psi^{-1}}[X]$.
\item
For $X\in \mq$, we define 
\begin{align*}
	\alpha_X\colon [a,b]\ni t\mapsto (\Add^+_\phi[t]\cp\Add^+_\psi[t])(X) = \etam^+_\phi(t,\Add^+_\psi[X](t))\in \mq.
\end{align*}
We have $\alpha_X(a)=X$; and Lemma \ref{knkjfdsjkfasdkjdsjkdsa} shows $\alpha_X\in C^1([a,b],\mq)$ with (for the first step additionally apply the parts \ref{chainrule} and \ref{productrule} 
of Proposition \ref{iuiuiuiuuzuzuztztttrtrtr}; and for the second step apply Corollary \ref{fdjkfdlkjfdkjfd})
\begin{align*} 
	\dot\alpha_X(t)	&\textstyle = \bil{\phi(t)}{\alpha_X(t)}+ \Add^+_\phi[t](\bil{\psi(t)}{\Add^+_\psi[t](X)})\\
		&\textstyle = \bil{\phi(t)}{\alpha_X(t)} + \bil{\Add^+_{\phi}[t](\psi(t))}{\alpha_X(t)} \\
		&\textstyle = \bil{(\phi\star\psi)(t)}{\alpha_X(t)}.
\end{align*}
Proposition \ref{lkjfdldlkkjfdYYX}.\ref{lkjfdldlkkjfdYYX2} yields $\alpha_X=\Add^+_{\phi\he\starm\he\psi}[X]$, which proves the claim.\qedhere
\end{enumerate}}
\endgroup	
\end{proof}

\begin{lemma}
\label{lkjkjsjlksdlkjdskj}
Given $\phi,\psi,\chi \in C^0([a,b],\mq)$, then we have
\begin{align*}
&\hspace{15pt}0\starm\psi=\psi=\psi\starm 0,
\\[1pt]
&\hspace{3pt}\psi^{-1}\starm\psi=0=\psi\starm\psi^{-1},
\\[1pt]
&(\phi\starm \psi)\starm \chi=\phi\starm (\psi\starm \chi).
\end{align*}
\end{lemma}
\begin{proof}
We have $\Add^+_0[t]=\id_\mq$ and $\Add^+_\psi[t](0)=0$ for each $t\in [a,b]$; hence, 
\begin{align*}
	(0\starm \psi)(t)= 0 + \Add^+_0[t](\psi(t))=  \psi(t)=
\psi(t) + \Add^+_\psi[t](0)=	(\psi\starm 0)(t)\qquad\quad \forall\: t\in [a,b].
\end{align*}
Next, Proposition \ref{lkjfdldlkkjfdYYX}.\ref{lkjfdldlkkjfdYYX1} yields
\begin{align*}
	(\psi\starm \psi^{-1})(t)&= \psi(t) - (\Add^+_\psi[t]\cp\Add^{-}_\psi[t])(\psi(t))=  0\qquad\quad\forall\: t\in [a,b].
\end{align*}
Moreover, Lemma \ref{hdshjdskjdsjkds}.\ref{hdshjdskjdsjkds1} shows
\begin{align*}
	(\psi^{-1}\starm \psi)(t)&= \psi^{-1}(t) + \Add^+_{\psi^{-1}}[t](\psi(t))= \psi^{-1}(t) + \Add^-_{\psi}[t](\psi(t))= \psi^{-1}(t)-\psi^{-1}(t)=0
\end{align*}
for each $t\in [a,b]$. Finally, we obtain from Lemma \ref{hdshjdskjdsjkds}.\ref{hdshjdskjdsjkds2} that 
\begin{align*}
	((\phi\starm \psi)\starm \chi)(t)&= (\phi\starm \psi)(t) + \Add^+_{\phi\he\starm\he \psi}[t](\chi(t))\\
	&=\phi(t) + \Add^+_\phi[t](\psi(t)) + (\Add^+_{\phi}[t]\cp \Add^+_\psi[t])(\chi(t))\\
	&= \phi(t) + \Add^+_\phi[t]((\psi\star\chi)(t))\\
	&= (\phi\star (\psi\star\chi))(t) 
\end{align*}
holds for each $t\in [a,b]$, which establishes the proof.
\end{proof}

\begin{corollary}
\label{dlkfdlkjfdlkfd}
For $\phi,\psi\in  C^0([a,b],\mq)$ and $t\in [a,b]$, we have 
\begin{align*}
	\Add^+_{\phi+\psi}[t]&=\Add^+_{\phi}[t]\cp\Add^+_{\Add^-_{\phi}[\cdot](\psi(\cdot))}[t].
\end{align*}
\end{corollary}
\begin{proof}
	By Lemma \ref{lkjkjsjlksdlkjdskj}, we have 
	\begin{align*}
		(\phi^{-1}\starm  (\phi + \psi))(t)&= \phi^{-1}(t) + \Add^+_{\phi^{-1}}[t](\phi(t)) + \Add^+_{\phi^{-1}}[t](\psi(t))\\
		&=(\phi^{-1}\starm \phi)(t) + \Add^+_{\phi^{-1}}[t](\phi(t))\\
		&=\Add^+_{\phi^{-1}}[t](\phi(t))
	\end{align*}
	for each $t\in [a,b]$. 
	 Lemma \ref{hdshjdskjdsjkds}.\ref{hdshjdskjdsjkds2} (first step) and Lemma \ref{hdshjdskjdsjkds}.\ref{hdshjdskjdsjkds1}  (third step) yield
	\begin{align*}
		\Add^+_{\phi^{-1}}[t]\cp \Add^+_{\phi + \psi}[t]=\Add^+_{\phi^{-1}\he\starm \he (\phi + \psi)}[t]=\Add^+_{\Add^+_{\phi^{-1}}[\cdot](\psi(\cdot))}[t]=\Add^+_{\Add^-_{\phi}[\cdot](\psi(\cdot))}[t]
	\end{align*}
	for each $t\in [a,b]$. Applying $\Add^+_\phi[t]$ to both hand sides, the claim follows from Lemma \ref{hdshjdskjdsjkds}.\ref{hdshjdskjdsjkds2}, Lemma \ref{lkjkjsjlksdlkjdskj}, as well as $\Add^+_0[t]=\id_\mq$ for each $t\in [a,b]$.
\end{proof}

\subsubsection{An Integral Transformation}
\label{oidsoidsoidsoidsds}
In analogy to \eqref{dfdfdfdfdfd}, for $a<b$ and $\phi\in C^0([a,b],\mq)$, we set
\begin{align}
\label{kjdkjdkjkjdfkjfdkjfdjkfd}
	\textstyle \TMAP(\phi)\colon [0,1]\ni t\mapsto \int_a^b \Add^+_{t\cdot \phi|_{[s,b]}}(\phi(s)) \:\dd s\in \mq
\end{align}
as well as $\TMAP(\phi|_{[a,a]})\colon [0,1]\ni t \mapsto 0\in \mq$. Observe that \eqref{kjdkjdkjkjdfkjfdkjfdjkfd} is defined  as the integrand is continuous; because, by Lemma \ref{jdjdfkjfdkjfd} and  Proposition \ref{lkjfdldlkkjfdYYX} we have
\begin{align*}
	\Add^+_{t\cdot \phi|_{[s,b]}}(\phi(s))=\Add^+_{t\cdot \phi}[b](\Add^-_{t\cdot \phi}[s](\phi(s)))\qquad\quad\forall\: s\in [a,b]. 
\end{align*}
Continuity of the integrand also implies the following statement.
\begin{remark}
\label{pofdpofdpofdpofdfdfdcvcvvvcjk}
	Let $a<b$ and $\psi\in C^0([a,b],\mq)$. Let $\varrho\colon [a',b']\rightarrow [a,b]$ ($a'<b'$) be of class $C^1$,  with $\dot\varrho|_{(a',b')}>0$ as well as $\varrho(a')=a$ and $\varrho(b')=b$. Then, \eqref{kjfdkjfdjkfdkjfdkjkjfdkjfdkjfdkjfd} in Remark \ref{iufdicxcxcxcxufdiufdiufddd} (second step) shows
	\begin{align*}
		\textstyle\TMAP(\dot\varrho\cdot (\psi\cp\varrho))(t)
		&\textstyle=\int_{a'}^{b'} \Add^+_{ (\dot\varrho\he\cdot\he ((t\cdot \psi)\he\cp\he\varrho))|_{[s,b']}}(\dot\varrho(s)\cdot (\psi\cp\varrho)(s)) \:\dd s\\
		&\textstyle=\int_{a'}^{b'} \dot\varrho(s)\cdot  \Add^+_{t\cdot \psi|_{[\varrho(s),b]}}( \psi(\varrho(s))) \:\dd s\\
		&\textstyle=\int_{a}^{b} \Add^+_{t\cdot \psi|_{[s,b]}}( \psi(s)) \:\dd s\\
		&=\TMAP(\psi)
	\end{align*}
	where we have applied the substitution formula \eqref{substitRI} in the third step.
\hspace*{\fill}\qed
\end{remark}
\noindent
In this section, we investigate the elementary properties of this map, and provide some statements that we shall need in Sect.\ \ref{mncxnmnxcmnxmnmnxciooioiweioewwe}. We start with some technical properties:
\begin{lemma}
\label{kjfdkjfdjkfdjkfd}
Let $a<b$, $\psi\in C^0([a,b],\mq)$, and $X\in\mq$ be given. Then, 
$\eta_{\psi,X}\colon \RR\ni t \mapsto \Add^+_{t\cdot \psi}[b](X)\in \mq$ 
is smooth, with
\begin{align*}
	\textstyle\eta_{\psi,X}^{(\ell)}\colon \RR\ni t\mapsto\sum_{p=\ell}^\infty  \frac{p!}{(p - \ell)!}\cdot t^{p-\ell}\cdot \APOL^+_{p,\psi}[b](X)\qquad\quad\forall\: n\in \NN.
\end{align*}
\end{lemma}
\begin{proof} 
It follows from the definitions that $\eta_{\psi,X}(t)=\sum_{p=0}^\infty t^p\cdot\APOL^+_{p,\psi}[b](X)$ holds for each $t\in \RR$. Now, for $\vv\leq \ww$ as in \eqref{assaaass}, we have
\begin{align*}
	\textstyle\vv(\APOL^+_{p,\psi}[b](X))\leq \ww(X)\cdot \frac{(b-a)^p}{p!}\cdot \ww_\infty(\psi)^p\qquad\quad\forall\: p\in \NN.
\end{align*} 	 
 Consequently,  
	$\textstyle \gamma[\qq]\colon \RR\ni t\mapsto\sum_{p=0}^\infty t^p\cdot \qq(\APOL^+_{p,\psi}[b](X))\in [0,\infty)$ 
	is defined for each $\qq\in \Sem{\mq}$, so that the claim is clear  from Corollary \ref{dsdsdsdsdsdsdsds}. 
\end{proof}
\noindent
Next, let $a<b$ and $\psi\in  C^0([a,b],\mq)$ be given: 
\begingroup
\setlength{\leftmargini}{12pt}{
\begin{itemize}
\item
For $a\leq  s< z\leq b$, we define $\Yps_0(\psi,s,z):=\psi(s)$, as well as
	\begin{align*}
		\textstyle\Yps_p(\psi,s,z):= \int_s^{z}\dd s_1\int_s^{s_1}\dd s_2 \: {\dots} \int_s^{s_{p-1}}\dd s_p \: (\bilbr{\psi(s_1)}\cp \dots \cp \bilbr{\psi(s_{p})})(\psi(s))\qquad\quad\forall\: p \geq 1.
	\end{align*}
	For each $t\in [0,1]$ and $a<z\leq b$, we have
\begin{align}
\label{slksdklklsdsd}
	\TMAP(\psi|_{[a,z]})(t)= \textstyle\int_a^z \sum_{p=0}^\infty  \APOL^+_{p,\he t\cdot \psi|_{[s,z]}}(\psi(s))\: \dd s=\textstyle\int_a^z \sum_{p=0}^\infty  t^p \cdot \Yps_p(\psi,s,z)\: \dd s.
\end{align}
\vspace{-23pt}
\item
We define
\begin{align}
\label{nmvcnmvckjfdkjfdkjriuiureiure}
\textstyle \Alp_p(\psi,z,t):= \textstyle t^p\cdot \int_a^z \Yps_p(\psi,s,z)\: \dd s
\qquad\quad\forall\: p\in \NN,\:\: t\in[0,1],\:\: z\in[a,b].
\end{align}
\end{itemize}}
\endgroup
\noindent
We have the following statement. 
\begin{lemma}
\label{klsdjklsdkjlsddjklsjklsd}
For each $a<b$, we have $\TMAP\colon C^0([a,b],\mq)\rightarrow C^\infty([0,1],\mq)$. Moreover, 
\label{klsdjklsdkjlsddjklsjklsd1}
\begin{align*}
	\textstyle\TMAP(\psi|_{[a,z]})=  \sum_{p=0}^\infty  \Alp_p(\psi,z,\cdot)
\end{align*}
converges w.r.t.\ the $C^\infty$-topology for each $\psi\in C^0([a,b],\mq)$ and $a<z\leq b$; and we have
\begin{align}
\label{kjdskjdskjds}
	\sup\{a\leq z\leq b\:|\: \qq_\infty^\dind(\TMAP(\psi|_{[a,z]}))\}<\infty\qquad\quad\forall\: \psi\in C^0([a,b],\mq),\: \qq\in \Sem{\mq},\:\dind\in \NN.
\end{align}
\end{lemma} 
\begin{proof}
	Let $\psi\in C^0([a,b],\mq)$ be given, and let $\vv\leq \ww$ be as in \eqref{assaaass}. Then, we have
\begin{align*}
	\textstyle\vv_\infty( \Yps_p(\psi,\cdot,z))\leq \ww_\infty(\psi)\cdot \frac{(b-a)^p}{p!}\cdot \ww_\infty(\psi)^p\qquad\quad\forall\: p\in \NN,\: a< z\leq b.
\end{align*}
It follows that $\textstyle\{\sum_{p=0}^\ell  t^p\cdot \Yps_p(\psi,\cdot,z)\}_{\ell\in \NN}\rightarrow \sum_{p=0}^\infty t^p\cdot \Yps_p(\psi,\cdot,z)$ converges uniformly on $[a,b]$, for each fixed $t\in [0,1]$ and $a< z\leq b$. Since the Riemann integral is $C^0$-continuous, we obtain
\begin{align*}
	\TMAP(\psi|_{[a,z]})(t)\stackrel{\eqref{slksdklklsdsd}}{=} \textstyle 
	\textstyle \sum_{p=0}^\infty \overbracket{t^p\cdot  \underbracket{\textstyle\int_a^z \Yps_{p}(\psi,s,z)\: \dd s}_{=:\:X_p(z)}}^{=\: \Alp_p(\psi,z,t)}\qquad\quad\forall\: t\in [0,1],\: a<z\leq b.
\end{align*}
Now, for each $p\in \NN$ and $a<z\leq b$, we have 
\begin{align*}
	\textstyle\vv(X_p(z))\leq (b-a)\cdot \ww_\infty(\psi)\cdot \frac{(b-a)^p}{p!}\cdot \ww_\infty(\psi)^p. 
\end{align*}
	 Consequently,  
	$\textstyle \gamma[\qq]\colon \RR\ni t\mapsto\sum_{p=0}^\infty t^p\cdot \qq(X_p(z))\in [0,\infty)$ 
	is defined for each $\qq\in \Sem{\mq}$ and $a<z\leq b$, so that the claim is clear from Corollary \ref{dsdsdsdsdsdsdsds}. 
\end{proof}

\begin{lemma}
\label{kjfdkjfdjkfdjkoioioiouiuiuuzud}
For $\psi\in C^0([a,b],\mq)$, we have
\begin{align*}
	\Add^+_{t\cdot \psi}[b]=\Add^+_{\TMAP(\psi)}[t]\qquad\quad\forall\:t\in [0,1].
\end{align*}
\end{lemma}
\begin{proof}
Let $X\in \mq$ be given. Lemma \ref {kjfdkjfdjkfdjkfd} yields the following: 
\begingroup
\setlength{\leftmargini}{12pt}{
\begin{itemize}
\item
We have $C^\infty(\RR,\mq)\ni \eta_{\psi,X}\colon \RR\ni t \mapsto \Add^+_{t\cdot \psi}[b](X)\in \mq$.
\item
By Lemma \ref{knkjfdsjkfasdkjdsjkdsa}, we have 
\begin{align*}
	\chi_t:=\Add^-_{t\cdot\psi}[\cdot](\psi(\cdot))\in C^0([a,b],\mq)\qquad\quad\forall\: t\in \RR. 
\end{align*}
Lemma \ref{kjfdkjfdjkfdjkfd} shows that $\eta_{\chi_t,X}\colon \RR\ni h\mapsto \Add^+_{h\cdot \chi_t}[b](X)\in \mq$ is smooth for each $t\in \RR$, with
\begin{align}
\label{aoiaoisoiosasa}
	\dot\eta_{\chi_t,X}(0)\textstyle = \APOL^+_{1,\chi_t}[b](X)\stackrel{\eqref{pofdpofdpofdsddsdsfd}}{=}\bil{\int_a^b \chi_t(s)\:\dd s}{X}.
\end{align}
\end{itemize}}
\endgroup
\noindent
By Lemma \ref{jdjdfkjfdkjfd} (second step) and Proposition \ref{lkjfdldlkkjfdYYX}.\ref{lkjfdldlkkjfdYYX1} (third step), we have 
\begin{align*}
	\Add^+_{t\cdot\psi}[b]\cp\Add^-_{t\cdot\psi}[s]=(\Add^+_{t\cdot \psi|_{[s,b]}}\cp \Add^+_{t\cdot \psi|_{[a,s]}})\cp \Add^-_{t\cdot \psi|_{[a,s]}}=\Add^+_{t\cdot \psi|_{[s,b]}}\qquad\quad\forall\: t\in \RR,\:\: a\leq s\leq b.
\end{align*}
This yields (second step), together with Corollary \ref{fdjkfdlkjfdkjfd} and \eqref{aoiaoisoiosasa} (first step) that
\begin{align}
\label{sdlkkdslkdsds}
\begin{split}
	\Add^+_{t\cdot\psi}[b](\dot\eta_{\chi_t,X}(0))&\textstyle =  \bil{\int_a^b\Add^+_{t\cdot\psi}[b](\chi_t(s))\:\dd s}{\Add^+_{t\cdot\psi}[b](X)}) \\
	&\textstyle= \bil{\int_a^b \Add^+_{t\cdot\psi|_{[s,b]}}[b](\psi(s))\:\dd s}{\Add^+_{t\cdot\psi}[b](X)}) \\
	&\textstyle= \bil{\TMAP(\psi)(t)}{\eta_{\psi,X}(t)})
\end{split}
\end{align}
holds for each $t\in [0,1]$. 
Now, Corollary \ref{dlkfdlkjfdlkfd} (second step) shows 
 \begin{align*}
	\eta_{\psi,X}(t+h)\textstyle&=\Add^+_{t\cdot \psi + h\cdot \psi}[b](X)\\[5pt]
	&\textstyle=\Add^+_{t\cdot\psi}[b]\big( \Add^+_{h\cdot\Add^-_{t\cdot\psi}[\cdot](\psi(\cdot))}[b](X)\big)\\
	&=\Add^+_{t\cdot\psi}[b](\eta_{\chi_t,X}(h))
\end{align*}
for $t,h\in \RR$. Hence, \eqref{sdlkkdslkdsds}, together with Lemma \ref{knkjfdsjkfasdkjdsjkdsa} and Part \ref{linear} of Proposition \ref{iuiuiuiuuzuzuztztttrtrtr} yields
\begin{align*}
	\dot\eta_{\psi,X}(t)= \bil{\TMAP(\psi)(t)}{\eta_{\psi,X}})\qquad\quad\forall\: t\in [0,1],
\end{align*}
with $\eta_{\psi,X}(0)=X$. 
Then, Proposition \ref{lkjfdldlkkjfdYYX}.\ref{lkjfdldlkkjfdYYX2} ( second step) shows
\begin{align*}
	 \Add^+_{t\cdot \psi}[b](X)=\eta_{\psi,X}(t)=\Add^+_{\TMAP(\psi)}[t](X)\qquad\quad \forall\: t\in [0,1].
\end{align*}
 Since $X\in \mq$ was arbitrary, the claim follows.
\end{proof}
\noindent
We obtain the following statement.
\begin{lemma}
\label{ksdkskjdsdssd}
Let $\psi\in C^0([a,b+\varepsilon],\mq)$ with $a< b$ and $\varepsilon>0$ be given. Then, 
we have
\begin{align*}
	\TMAP(\psi|_{[a,z+h]})=\TMAP(\psi|_{[z,z+h]})\starm \TMAP(\psi|_{[a,z]})\qquad\quad\forall\: a< z\leq b,\:\: 0< h\leq \varepsilon.
\end{align*}
\end{lemma}
\begin{proof}
Let $a< z\leq b$ and $0< h\leq \varepsilon$ be fixed. 
For $t\in [0,1]$, we have
\begin{align*}
	\textstyle \TMAP(\psi|_{[a,z+h]})(t)&\textstyle=\int_a^{z+h} \Add^+_{\he t\cdot \psi|_{[s,z+h]}}[\psi(s)]	\:\dd s\\[1pt]
	&\textstyle= \underbracket{\textstyle\int_z^{z+h} \Add^+_{\he t\cdot \psi|_{[s,z+h]}}[\psi(s)]	\:\dd s}_{=\:\TMAP(\psi|_{[z,z+h]})}\he +\he \underbracket{\textstyle\int_a^{z} \Add^+_{\he t\cdot \psi|_{[s,z+h]}}[\psi(s)]	\:\dd s}_{=:\: \alpha(t)}.
\end{align*} 
Now, Lemma \ref{jdjdfkjfdkjfd} and \eqref{pofdpofdpofdsddsdsfd} (first step), as well as Lemma \ref{kjfdkjfdjkfdjkoioioiouiuiuuzud} (third step) yield
\begin{align*}
	\alpha(t)&\textstyle= \Add^+_{\he t\cdot \psi|_{[z,z+h]}}[\int_a^{z} \Add^+_{\he t\cdot \psi|_{[s,z]}}[\psi(s)]	\:\dd s]\\
	&\textstyle= \Add^+_{\he t\cdot \psi|_{[z,z+h]}}[z+h](\TMAP(\psi|_{[a,z]})(t))\\
	&\textstyle=\Add^+_{\TMAP(\psi|_{[z,z+h]})}[t](\TMAP(\psi|_{[a,z]})(t))
\end{align*}
for each $t\in [0,1]$. We obtain 
\begin{align*}
	\TMAP(\psi|_{[a,z+h]})(t)= \TMAP(\psi|_{[z,z+h]}) + \Add^+_{\TMAP(\psi|_{[z,z+h]})}[t](\TMAP(\psi|_{[a,z]})(t))= \big(\TMAP(\psi|_{[z,z+h]})\starm \TMAP(\psi|_{[a,z]})\big)(t)
\end{align*}
for each $t\in [0,1]$, which proves the claim.
\end{proof}
\noindent
We obtain the following statement (in analogy to Remark \ref{oidsoidsoidsiods}).
\begin{proposition}
\label{oidsoidsoidsiodsdffdd}
Let $a<c<b$, $\psi\in C^k([a,b],\mq)$, and $\varrho\colon [a',b']\rightarrow [a,b]$ ($a'<b'$) be of class $C^1$ with $\dot\varrho|_{(a',b')}>0$, $\varrho(a')=a$, $\varrho(b')=b$. Then, we have
\begin{align}
\label{pdpofdpofdpofds1}
\TMAP(\psi)^{-1}&=\TMAP(\inverse{\psi})\\[2pt]
\label{pdpofdpofdpofds2}
	\TMAP(\psi)&=\TMAP(\psi|_{[c,b]})\star \TMAP(\psi|_{[a,c]})\\
\label{pdpofdpofdpofds3}
	\TMAP(\psi)&=\TMAP(\dot\varrho\cdot (\psi\cp\varrho)).
\end{align}
\end{proposition}
\begin{proof}
Lemma  \ref{ksdkskjdsdssd} shows \eqref{pdpofdpofdpofds2}, and Remark \ref{pofdpofdpofdpofdfdfdcvcvvvcjk} shows \eqref{pdpofdpofdpofds3}. To prove \eqref{pdpofdpofdpofds1}, let $\varrho_1\colon [a,b]\rightarrow [a,b]$ with $\varrho_1(a)=a$, $\varrho_1(b)=b$, as well as $\varrho_2\colon [b,2b+a]\rightarrow [a,b]$ with $\varrho_2(b)=a$, $\varrho_2(2b+a)=b$ both be of class $C^1$, such that $\dot\varrho_1|_{(a,b)}>0$ as well as $\dot\varrho_2|_{(b,2b+a)}>0$ holds with
\begin{align*}
	\dot\varrho_1(a),\he\dot\varrho_1(b),\he \dot\varrho_2(b),\he\dot\varrho_2(2b+a)=0.
\end{align*}
We set $\chi_1:=\dot\varrho_1\cdot (\psi\cp\varrho_1)\in C^0([a,b],\mq)$ as well as $\chi_2:=\dot\varrho_2\cdot (\inverse{\psi}\cp\varrho_2)\in C^0([b,2b+a],\mq)$, and define $\phi\in C^0([a,2b+a],\mq)$ by 
\begin{align*}
\phi(t):= 
\begin{cases}
\chi_1(t) &\text{for}\quad t\in [a,b]\\
\chi_2(t) &\text{for}\quad t\in (b,2b+a] 
\end{cases}
\end{align*}	
for each $t\in [a,2b+a]$. Then, we have
\begin{align}
\label{dflfdlkfdkfdlkkfdfd}
	\TMAP(\psi)\starm\TMAP(\inverse{\psi})\stackrel{\eqref{pdpofdpofdpofds3}}{=}\TMAP(\chi_1) \starm \TMAP(\chi_2)\stackrel{\eqref{pdpofdpofdpofds2}}{=} \TMAP(\phi).
\end{align}
Let now $t\in [0,1]$ and $X\in \mq$ be given (and observe $t\cdot \inverse{\psi}=\inverse{t\cdot \psi}$).
\begingroup
\setlength{\leftmargini}{12pt}{
\begin{itemize}
\item
	For $s\in [a,b]$, we have by Lemma \ref{jdjdfkjfdkjfd} (first step),    \eqref{kjfdkjfdjkfdkjfdkjkjfdkjfdkjfdkjfd} in Remark \ref{iufdicxcxcxcxufdiufdiufddd} (third step),  as well as Lemma \ref{jdjdfkjfdkjfd} and Proposition \ref{lkjfdldlkkjfdYYX}.\ref{lkjfdldlkkjfdYYX1} (fourth step) that
\begin{align}
\label{oifdoioioifdoifd1}
\begin{split}
	\Add^+_{t\cdot \phi|_{[s,2b+a]}}(X)&=\Add^+_{t\cdot \phi|_{[b,2b+a]}}\big(\Add^+_{t\cdot \phi|_{[s,b]}}(X)\big)\\
	&=\Add^+_{t\cdot \chi_2}\big(\Add^+_{t\cdot \chi_1|_{[s,b]}}(X)\big)\\
	&=\Add^+_{ \inverse{t\cdot\psi}}\big(\Add^+_{t\cdot \psi|_{[\varrho_1(s),b]}}(X)\big)\\
	&=\Add^-_{t\cdot \psi|_{[a,\varrho_1(s)]}}(X).
\end{split}
\end{align}
\item
For $s\in [b,2b+a]$, we define 
\begin{align*}
	\chi\colon [\varrho_2(s),b]\ni z \mapsto  \inverse{t\cdot \psi|_{[a,a+b-\varrho_2(s)]}}(z-(\varrho_2(s)-a))  \in \mq.
\end{align*}
We observe the following:
\begingroup
\setlength{\leftmarginii}{17pt}
{
\renewcommand{\theenumi}{\roman{enumi})} 
\renewcommand{\labelenumi}{\theenumi}
\begin{enumerate}
\item 
\label{oizzuzu1}
We have $\chi=\inverse{t\cdot \psi}|_{[\varrho_2(s),b]}$ as
\begin{align*}
	\chi(z)=t\cdot \psi(2a+b-\varrho_2(s)- (z-(\varrho_2(s)-a)))=t\cdot \psi(a+b-z)=\inverse{t\cdot \psi}|_{[\varrho_2(s),b]}(z)
\end{align*}
holds for each $z\in [\varrho_2(s),b]$.
\item
\label{oizzuzu2}
By the second point in Remark \ref{iufdicxcxcxcxufdiufdiufddd} (second step), we have
\begin{align*}
	\Add^+_{\chi}&=\Add^+_{\chi}[b]\\
	&=	\Add^+_{\chi(\cdot-(a-\varrho_2(s)))}[b+(a-\varrho_2(s))]\\
	&=\Add^+_{\inverse{t\cdot\psi|_{[a,a+b-\varrho_2(s)]}}}[b+(a-\varrho_2(s))]\\
	&=\Add^+_{\inverse{t\cdot\psi|_{[a,a+b-\varrho_2(s)]}}}.
\end{align*}
\end{enumerate}}
\endgroup
\noindent
We obtain from \eqref{kjfdkjfdjkfdkjfdkjkjfdkjfdkjfdkjfd} in Remark \ref{iufdicxcxcxcxufdiufdiufddd} (second step), Point \ref{oizzuzu1} (third step), Point \ref{oizzuzu2} (fourth step), as well as  Proposition \ref{lkjfdldlkkjfdYYX}.\ref{lkjfdldlkkjfdYYX1} (fifth  step) that
\begin{align}
\label{oifdoioioifdoifd2}
\begin{split}
	\Add^+_{t\cdot \phi|_{[s,2b+a]}}(X)&=\Add^+_{t\cdot \chi_2|_{[s,2b+a]}}(X)\\
	&=\Add^+_{\inverse{t\cdot\psi}|_{[\varrho_2(s),b]}}(X)\\
	&=\Add^+_{\chi}(X)\\[2pt]
		&=\Add^+_{\inverse{t\cdot\psi|_{[a,a+b-\varrho_2(s)]}}}(X)\\
	&=\Add^-_{t\cdot \psi|_{[a,a+b-\varrho_2(s)]}}(X).
\end{split}
\end{align}
\end{itemize}}
\endgroup
\noindent 
We obtain from  
\eqref{oifdoioioifdoifd1} and \eqref{oifdoioioifdoifd2} (third step), \eqref{substitRI} (fourth step), as well as the second point in Remark \ref{seqcomp} (sixth step) that
\begin{align*}
	\TMAP(\phi)(t)&\textstyle= \int_a^{2b + a} \Add^+_{t\cdot \phi|_{[s,2b+a]}}(\phi(s)) \:\dd s\\
	&\textstyle =  \int_a^{b} \Add^+_{t\cdot \phi|_{[s,2b+a]}}(\phi(s)) \:\dd s +  \int_b^{2b + a} \Add^+_{t\cdot \phi|_{[s,2b+a]}}(\phi(s)) \:\dd s\\
	&\textstyle =  \int_a^{b} \dot\varrho_1(s)\cdot\Add^-_{t\cdot \psi|_{[a,\varrho_1(s)]}}(\psi(\varrho_1(s))) \:\dd s 
	+  \int_b^{2b + a}\dot\varrho_2(s)\cdot \Add^-_{t\cdot \psi|_{[a, a+b-\varrho_2(s)]}}(\inverse{\psi}(\varrho_2(s))) \:\dd s\\
	&\textstyle =  \int_a^{b} \Add^-_{t\cdot \psi|_{[a,s]}}(\psi(s)) \:\dd s +  \int_a^{b}\Add^-_{t\cdot \psi|_{[a,a+b-s]}}(\inverse{\psi}(s)) \:\dd s\\
	&\textstyle =  \int_a^{b} \Add^-_{t\cdot \psi|_{[a,s]}}(\psi(s)) \:\dd s -  \int_a^{b}\Add^-_{t\cdot \psi|_{[a,a+b-s]}}(\psi(a+b-s)) \:\dd s\\
	&=0
\end{align*}
holds for each $t\in [0,1]$. Together with \eqref{dflfdlkfdkfdlkkfdfd}, this proves \eqref{pdpofdpofdpofds3}.
\end{proof}
\noindent
Moreover, applying Lemma \ref{ksdkskjdsdssd} to the situation where $(\mq,\bil{\cdot}{\cdot})\equiv (\mg,\bil{\cdot}{\cdot})$ is an asymptotic estimate and sequentially complete Lie algebra that is inherited by a Lie group $G$, we obtain the following statement.
\begin{corollary}
\label{kjfdkjfdkjfd}
Let $G$ be weakly $C^\infty$-regular, with sequentially complete and asymptotic estimate Lie algebra $(\mg,\bil{\cdot}{\cdot})$. Then, for $\psi\in C^0([a,b+\varepsilon])$ with $a< b$ and $\varepsilon>0$, we have
\begin{align*}
	\textstyle\innt_0^1 \TMAP(\psi|_{[a,z+h]})=\innt_0^1 \TMAP(\psi|_{[z,z+h]})\cdot \innt_0^1 \TMAP(\psi|_{[a,z]})\qquad\quad\forall\: a< z\leq b,\: 0< h\leq \varepsilon.
\end{align*}
\begin{proof}
	Let $a< z\leq b$ and $0< h\leq \varepsilon$ be fixed.  
	Corollary \ref{ljkdskjdslkjsda} shows\footnote{Recall that in Sect.\ \ref{kfdkjfdkjfdfd}, $\Add^+_\psi[X]$ had been defined in Equation \eqref{kldslkdslkdslcxcxcxcx} by $\Ad_{\innt_0^\bullet\psi}(X)$, for $\psi\in \DIDE_{[0,1]}\supseteq C^\infty([0,1],\mg)$ and $X\in \mg$. Furthermore, observe that the right side of \eqref{sdssddsds} in Corollary \ref{ljkdskjdslkjsda} obviously equals the definition of $\Add^\pm_\psi[X]$ in the beginning of Sect.\ \ref{kdsjlkdslkjsdjkds}.}
	\begin{align*}
		\Ad_{\innt_0^t \TMAP(\psi|_{[z,z+h]})}(X)=\Add^+_{\TMAP(\psi|_{[z,z+h]})}[t](X)\qquad\quad\forall\: t\in [0,1],\: X\in \mg.
	\end{align*}
	We thus obtain from \ref{kdsasaasassaas} (first step) as well as Lemma \ref{ksdkskjdsdssd} (last step) that
	\begin{align*}
		\textstyle\innt_0^1 \TMAP(\psi|_{[z,z+h]})\cdot \innt_0^1 \TMAP(\psi|_{[a,z]})&\textstyle= \innt_0^1 \TMAP(\psi|_{[z,z+h]}) + \Ad_{\innt_0^\bullet \TMAP(\psi|_{[z,z+h]})}(\TMAP(\psi|_{[a,z]}))\\
		&\textstyle= \innt_0^1 \TMAP(\psi|_{[z,z+h]}) + \Add^+_{\TMAP(\psi|_{[z,z+h]})}[\cdot](\TMAP(\psi|_{[a,z]})(\cdot))\\
		&=\textstyle\innt_0^1 \TMAP(\psi|_{[z,z+h]}) \starm  \TMAP(\psi|_{[a,z]})\\
		&=\textstyle\innt_0^1 \TMAP(\psi|_{[a,z+h]})
	\end{align*}
	holds, which proves the claim.
\end{proof}
\end{corollary}
\subsection{Weak Regularity}
\label{mncxnmnxcmnxmnmnxciooioiweioewwe}
In this section, we prove Theorem \ref{dkjkjfdkjfdjfdkfdfdfd}. 
In the following, let thus $G$ be weakly $C^\infty$-regular, with asymptotic estimate and sequentially complete Lie algebra   $(\mg,\bil{\cdot}{\cdot})$. 
\vspace{6pt}

\noindent
Let $a\leq b$, $0<\varepsilon<1$, and $\psi\in C^0([a,b+\varepsilon],\mg)$ be given. 
\begingroup
\setlength{\leftmargini}{12pt}{
\begin{itemize}
\item
We define $\Phi_\psi\colon (a-1,b+\varepsilon)\times [0,1]\rightarrow \mg$ by
\begin{align}
\label{nmdfnmfdnmfd}
\begin{split}
\Phi_\psi(z,t) := \begin{cases}
(z-a)\cdot \psi(a) &\text{for}\quad z\in (a-1,a)\\
 \TMAP(\psi|_{[a,z]})(t) &\text{for}\quad z\in [a,b+\varepsilon) 
\end{cases}
\end{split}
\end{align}
for each $t\in [0,1]$. Lemma \ref{klsdjklsdkjlsddjklsjklsd} yields the following:
\begingroup
\setlength{\leftmarginii}{12pt}{
\begin{itemize}
\item
For $z\in (a-1,b+\varepsilon)$, we have $\Phi_\psi(z,\cdot)\in C^\infty([0,1],\mg)$. 
\item
For $z\in [a,b+\varepsilon)$, we have
\begin{align*}
	\textstyle\Phi_\psi(z,t)=\int_a^z \Add^+_{\he t\cdot \psi|_{[s,z]}}[\psi(s)] \:\dd s
	=\sum_{\ell=0}^\infty \Alp_\ell(\psi,z,t)\qquad\quad\forall\: t\in [0,1],
\end{align*} 
where $\Phi_\psi(z,\cdot)
	=\sum_{\ell=0}^\infty \Alp_\ell(\psi,z,\cdot)$ converges w.r.t.\ the $C^\infty$-topology. 
\item
	For $\pp\in \Sem{E}$ and $\dind\in \NN$, we have
\begin{align}
\label{cxcxcxcxcxcxcxoidsoidsoidsdsdsccv}
	\sup\{a-1\leq z\leq b+\varepsilon\:|\: \pp_\infty^\dind(\Phi_\psi(z,\cdot))\}<\infty.
\end{align}
\vspace{-27pt}
\end{itemize}}
\endgroup
\noindent
\item
For each $t\in [0,1]$, we define the map 
\begin{align}
\label{skjskjsjksdcxcx}
	\textstyle\mu_{t,\psi}\colon (a-1,b+\varepsilon)\ni z\mapsto \innt_0^t \Phi_\psi(z,\cdot)\in G.
\end{align}
We observe the following:
\begingroup
\setlength{\leftmarginii}{12pt}{
\begin{itemize}
\item
We have 
\vspace{-5pt}
\begin{align}
\label{lkfdlkkdllkfdlkfdkfd}
	\textstyle\mu_{t,\psi}(z)=\innt_0^t\TMAP(\psi|_{[a,z]})\qquad\quad\forall\: z\in [a,b+\varepsilon).
\end{align}
\item
We have
\begin{align}
\label{kjdskjdskjdsjkjdskkjdsds}
	 \Phi_\psi(a,\cdot)=\TMAP(\psi|_{[a,a]})(\cdot)=0 \qquad\text{hence}\qquad \mu_{t,\psi}(a)=e\qquad\text{for each}\qquad t\in [0,1].
\end{align} 
\end{itemize}}
\endgroup
\end{itemize}}
\endgroup
\noindent 
We obtain the following lemma.
\begin{lemma}
\label{kjdskjdskjdskjsdxyxyyxy}
Let $a\leq b$, $0<\varepsilon<1$, and $\psi\in C^0([a,b+\varepsilon],\mg)$. Define $\ppsi\in C^0((a-1,b+\varepsilon),\mg)$ by
$$
\ppsi(z) := \begin{cases}
\psi(a) &\text{for}\quad z\in (a-1,a)\\
 \psi(z) &\text{for}\quad z\in [a,b+\varepsilon), 
\end{cases}
$$ 
\vspace{-15pt}

\noindent
and set
\begin{align*}
	\DDelta(z,t):=\ppsi(z)+t\cdot \bil{\ppsi(z)}{\Phi_\psi(z,t)}\qquad\quad\forall\: z\in (a-1,b+\varepsilon),\: t\in [0,1].
\end{align*}
Then, the following assertions hold:
\begingroup
\setlength{\leftmargini}{16pt}
{
\renewcommand{\theenumi}{{\arabic{enumi})}} 
\renewcommand{\labelenumi}{\theenumi}
\begin{enumerate}
\item
\label{kjdskjdskjdskjsdxyxyyxy0}
For $z\in (a-1,b+\varepsilon)$, $\pp\in \Sem{E}$, and $\dind\in \NN$, we have 
\begin{align}
\label{uiuieiuewewewewpopoewew}
	\textstyle\lim_{ h\rightarrow 0}\frac{1}{|h|}\cdot\pp^\dind_\infty\big(\Phi_\psi(z+h,\cdot)-\Phi_\psi(z,\cdot)- h\cdot\DDelta(z,\cdot)\big)=0.
	\end{align}
	\item
\label{kjdskjdskjdskjsdxyxyyxy1}
	The following assertions hold:
	\begingroup
\setlength{\leftmarginii}{20pt}
{
\renewcommand{\theenumi}{{\alph{enumi})}} 
\renewcommand{\labelenumi}{\theenumi}
\begin{enumerate}
\item
\label{subbbbb0}
	We have $\partial_1\Phi_\psi(z,\cdot)=\DDelta(z,\cdot)\in C^\infty([0,1],\mg)$ for each $z\in (a-1,b+\varepsilon)$. 
\item
\label{subbbbb2}
	To each $\pp\in \Sem{E}$, $\dind\llleq k$, and $z\in (a-1,b+\varepsilon)$, there exists $L_{\pp,\dind}\geq 0$ as well as $I_{\pp,\dind}\subseteq (a-1,b+\varepsilon)$ open with $z\in I_{\pp,\dind}$, such that
\begin{align*}
	\textstyle\pp^\dind_\infty(\Phi(z+h,\cdot)-\Phi(z,\cdot))\leq |h|\cdot L_{\pp,\dind}\qquad\quad \forall\: h\in \RR_{\neq 0}\:\text{ with }\: z+h\in I_{\pp,\dind}.
	\end{align*}
\item
\label{subbbbb1}
	$\Phi_\psi$ is continuous.
\end{enumerate}}
\endgroup
\item
\label{kjdskjdskjdskjsdxyxyyxy2}
	We have $\mu_{t,\psi}\in C^1((a-1,b+\varepsilon),G)$ for each $t\in [0,1]$, with
	\begin{align*}
	\textstyle\dot\mu_{t,\psi}(z)\textstyle=\dd_e\LT_{\mu_{t,\psi}(z)}\big(\int_0^t \Ad_{\mu_{s,\psi}(z)^{-1}}\big(\ppsi(z)+s\cdot \bil{\ppsi(z)}{\Phi_\psi(z,s)}\big)\:\dd s\he\big)\qquad\forall\:z\in (a-1,b+\varepsilon). 
\end{align*}
In particular, for each $t\in [0,1]$ and $z\in [a,b+\varepsilon)$, we have 
\begin{align}
\label{jkkjkjfdjfdkjkjdfjkfkjfdkjdf}
\textstyle\dot\mu_{t,\psi}(z)\textstyle=\dd_e\LT_{\mu_{t,\psi}(z)}\big(\int_0^t \Ad_{\mu_{s,\psi}(z)^{-1}}(\psi(z)+s\cdot \bil{\psi(z)}{\TMAP(\psi|_{[a,z]})(s)})\:\dd s\he\big).
\end{align}
\end{enumerate}}
\endgroup
\end{lemma}
\begin{proof} 
\begingroup
\setlength{\leftmargini}{16pt}
{
\renewcommand{\theenumi}{{\arabic{enumi})}} 
\renewcommand{\labelenumi}{\theenumi}
\begin{enumerate}
\item
For $t\in [0,1]$, $z\in (a-1,a]$ and $h\neq 0$ with $z+h\in (a-1,a]$, we have
\begin{align*}
	\Phi_\psi(z+h,t)\:&-\Phi_\psi(z,t)- h\cdot\DDelta(z,t)\\
	&\textstyle=(z+h-a)\cdot \psi(a)- (z-a)\cdot \psi(a) -h\cdot(\psi(a) + t\cdot \bil{\psi(a)}{(z-a)\cdot \psi(a)})\\
	&=0.
\end{align*}
Hence,  
\eqref{uiuieiuewewewewpopoewew} holds for each $z\in (a-1,a)$, and we have 
\begin{align}
	\label{yxxyxyyxyxxyuiuieiuewewewewpopoewew}
	\textstyle\lim_{0> h\rightarrow 0}\frac{1}{|h|}\cdot\pp^\dind_\infty\big(\Phi_\psi(a+h,\cdot)-\Phi_\psi(a,\cdot)- h\cdot\DDelta(a,\cdot)\big)=0\qquad\quad\forall\: \pp\in \Sem{E},\: \dind\in \NN.
	\end{align}
	For the remaining cases, we need the following observations:
\begingroup
\setlength{\leftmarginii}{12pt}{
\begin{itemize}
\item
	For $t\in [0,1]$ and $z\in [a,b+\varepsilon)$, we define $L_0(z,t):=\psi(z)$ as well as (recall \eqref{nmvcnmvckjfdkjfdkjriuiureiure}) 
\begin{align*}
	\textstyle L_\ell(z,t):=&\:\textstyle t\cdot \bil{\psi(z)}{\Alp_{\ell-1}(\psi,z,t)}\\
	=&\:\textstyle t\cdot \bil{\psi(z)}{t^{\ell-1}\cdot \int_a^z \Yps_{\ell-1}(\psi,s,z)\: \dd s} \\
	=&\:\textstyle t^\ell\cdot  \int_a^{z}\dd s  \int_s^{z}\dd s_2 \: {\dots} \int_s^{s_{\ell-1}}\dd s_\ell \: (\bilbr{\psi(z)}\cp \bilbr{\psi(s_2)}\cp \dots \cp \bilbr{\psi(s_{\ell})})(\psi(s))
\end{align*}
for $\ell\geq 1$. 
Since $(\mg,\bil{\cdot}{\cdot})$ is asymptotic estimate and sequentially complete, Corollary  \ref{dsdsdsdsdsdsdsds} shows that the map
\begin{align}
\label{posdpodsopdsods}
	\textstyle  L\colon [a,b+\varepsilon)\times[0,1]\rightarrow \mg,\qquad (z,t)\mapsto \sum_{\ell=0}^\infty L_\ell(z,t)
\end{align}
is defined, such that $C^\infty([0,1],\mg)\ni L(z,\cdot)=\sum_{\ell=0}^\infty L_\ell(z,t)$ converges w.r.t.\ the $C^\infty$-topology for each $z\in [a,b+\varepsilon)$. 
\item
Let $t\in [0,1]$, $z\in [a,b+\varepsilon)$, and $h>0$ be fixed.
\begingroup
\setlength{\leftmarginiii}{12pt}{
\begin{itemize}
\item
	If $z+h<b$ holds, we have $\Alp_0(\psi,z+h,t)-\Alp_0(\psi,z,t)=\int_z^{z+h} \psi(s)\: \dd s$, as well as for $\ell\geq 1$
	\begin{align*}
		\Alp_\ell(\psi,z+h,t)&\textstyle
		-\Alp_\ell(\psi,z,t)\\[1pt]
		&=\textstyle t^\ell\cdot\int_a^{z} \dd s \textstyle\int_z^{z+h}\dd s_1\int_s^{s_1}\dd s_2 \: {\dots} \int_s^{s_{\ell-1}}\dd s_\ell \: (\bilbr{\psi(s_1)}\cp \dots \cp \bilbr{\psi(s_{\ell})})(\psi(s))\\
	&\quad\textstyle + t^\ell\cdot\int_z^{z+h}\dd s  \textstyle\int_s^{z+h}\dd s_1\int_s^{s_1}\dd s_2 \: {\dots} \int_s^{s_{\ell-1}}\dd s_\ell \: (\bilbr{\psi(s_1)}\cp \dots \cp \bilbr{\psi(s_{\ell})})(\psi(s))\\[1pt]
		&=\textstyle t^\ell\cdot\int_a^{z} \dd s \textstyle\int_z^{z+h}\dd s_1\int_s^{z}\dd s_2 \: {\dots} \int_s^{s_{\ell-1}}\dd s_\ell \: (\bilbr{\psi(s_1)}\cp \dots \cp \bilbr{\psi(s_{\ell})})(\psi(s))\\
		&\quad\textstyle +t^\ell\cdot\int_a^{z} \dd s \textstyle\int_z^{z+h}\dd s_1\int_z^{s_1}\dd s_2 \: {\dots} \int_s^{s_{\ell-1}}\dd s_\ell \: (\bilbr{\psi(s_1)}\cp \dots \cp \bilbr{\psi(s_{\ell})})(\psi(s))\\
	&\quad\textstyle + t^\ell\cdot\int_z^{z+h}\dd s  \textstyle\int_s^{z+h}\dd s_1\int_s^{s_1}\dd s_2 \: {\dots} \int_s^{s_{\ell-1}}\dd s_\ell \: (\bilbr{\psi(s_1)}\cp \dots \cp \bilbr{\psi(s_{\ell})})(\psi(s)).
	\end{align*} 
	(In the second step, we have split the third integral in the first summand at $z$.)
	\vspace{5pt}
\item 
	If $a<z-h$ holds, we have $\Alp_0(\psi,z,t)-\Alp_0(\psi,z-h,t)=\int_{z-h}^{z} \psi(s)\: \dd s$, as well as for $\ell\geq 1$
	\begin{align*}
		\Alp_\ell(\psi,z,t)&\textstyle-\Alp_\ell(\psi,z-h,t)\\[1pt]
		&=\textstyle t^\ell\cdot\int_a^{z-h} \dd s \textstyle\int_{z-h}^{z}\dd s_1\int_s^{s_1}\dd s_2 \: {\dots} \int_s^{s_{\ell-1}}\dd s_\ell \: (\bilbr{\psi(s_1)}\cp \dots \cp \bilbr{\psi(s_{\ell})})(\psi(s))\\
	&\quad\textstyle + t^\ell\cdot\int_{z-h}^{z}\dd s  \textstyle\int_s^{z}\dd s_1\int_s^{s_1}\dd s_2 \: {\dots} \int_s^{s_{\ell-1}}\dd s_\ell \: (\bilbr{\psi(s_1)}\cp \dots \cp \bilbr{\psi(s_{\ell})})(\psi(s))\\[1pt]
	&=\textstyle t^\ell\cdot\int_a^{z-h} \dd s \textstyle\int_{z-h}^{z}\dd s_1\int_s^{z-h}\dd s_2 \: {\dots} \int_s^{s_{\ell-1}}\dd s_\ell \: (\bilbr{\psi(s_1)}\cp \dots \cp \bilbr{\psi(s_{\ell})})(\psi(s))\\
	&\quad\textstyle + t^\ell\cdot\int_a^{z-h} \dd s \textstyle\int_{z-h}^{z}\dd s_1\int_{z-h}^{s_1}\dd s_2 \: {\dots} \int_s^{s_{\ell-1}}\dd s_\ell \: (\bilbr{\psi(s_1)}\cp \dots \cp \bilbr{\psi(s_{\ell})})(\psi(s))\\
	&\quad\textstyle + t^\ell\cdot\int_{z-h}^{z}\dd s  \textstyle\int_s^{z}\dd s_1\int_s^{s_1}\dd s_2 \: {\dots} \int_s^{s_{\ell-1}}\dd s_\ell \: (\bilbr{\psi(s_1)}\cp \dots \cp \bilbr{\psi(s_{\ell})})(\psi(s)).\hspace{33pt}
	\end{align*}
	(In the second step, we have split the third integral in the first summand at $z-h$.)
\end{itemize}}
\endgroup
\end{itemize}}
\endgroup
\noindent 
Let $\vv\leq \ww$ be as in \eqref{assaaass}, and set
\begin{align*}
	\textstyle\Delta(\ww,z,h):=\ww_\infty\big(\psi(z)-\psi|_{[z-|h|,z+|h|]\cap [a,b+\varepsilon]}\big)\qquad\quad\forall\:z\in [a,b+\varepsilon),\: h\neq 0. 
\end{align*}
Let $z\in [a,b+\varepsilon)$ be fixed, and set
\begin{align*}
	D:=\{x-z\:|\: x\in [a,b+\varepsilon)\}. 
\end{align*}
For $\dind\in \NN$, $\ell\geq 1$, and $h\in D_z$, we obtain  
\begin{align*}
	\vv^\dind_\infty(\Alp_0(\psi,z+h,\cdot)&\textstyle-\Alp_0(\psi,z,\cdot)-h\cdot L_0(z,\cdot))\\
	&\textstyle=
	\vv\big(\!\int_z^{z+h} \psi(s)\:\dd s-\psi(z)\big)\\
	&\textstyle\leq |h|\cdot \Delta(\ww,z,h),\\[8pt]
	\textstyle\vv^\dind_\infty(\Alp_\ell(\psi,z+h,\cdot)&-\Alp_\ell(\psi,z,\cdot)-h\cdot L_\ell(z,\cdot))\\
&\textstyle\leq \dind!\cdot |h|\cdot\Delta(\ww,z,h)\cdot  (b+\varepsilon-a)\cdot \ww_\infty(\psi) \cdot\frac{(b+\varepsilon-a)^{\ell-1}}{(\ell-1)!}\cdot
 \ww_\infty(\psi)^{\ell-1}\\
&\textstyle\quad\he + \dind!\cdot(b+\varepsilon-a)\cdot \ww_\infty(\psi) \cdot|h|^2\cdot \ww_\infty(\psi)^2\cdot \frac{(b+\varepsilon-a)^{\ell-2}}{(\ell-2)!}\cdot \ww_\infty(\psi)^{\ell-2} \\
&\textstyle\quad\he + \dind!\cdot|h|^2\cdot \ww_\infty(\psi)^2\cdot \frac{(b+\varepsilon-a)^{\ell-1}}{(\ell-1)!}\cdot \ww_\infty(\psi)^{\ell-1}.
\end{align*}
Now, the maps (with $h\in D$)
\begin{align*}
	\textstyle\Phi_\psi(z+h,\cdot)=\sum_{\ell=0}^\infty\Alp_\ell(\psi,z+h,\cdot),\quad\:\:\:
	\Phi_\psi(z,\cdot)=\sum_{\ell=0}^\infty\Alp_\ell(\psi,z,\cdot),\quad\:\:\:
    L(z,\cdot)=\sum_{\ell=0}^\infty L_\ell(z,\cdot)
\end{align*} 
converge w.r.t.\ the $C^\infty$-topology; and continuity of $\psi$ implies $\lim_{h\rightarrow 0}\Delta(\ww,z,h)=0$. The triangle inequality thus yields 
\begin{align*}
	\textstyle\lim_{D\ni h\rightarrow 0}\frac{1}{|h|}\cdot\vv^\dind_\infty(\Phi_\psi(z+h,\cdot)-\Phi_\psi(z,\cdot)- h\cdot L(z,\cdot))=0.
\end{align*}
Together with \eqref{yxxyxyyxyxxyuiuieiuewewewewpopoewew} this proves the claim, because for $z\in [a,b+\varepsilon)$ and $t\in [0,1]$, we have
\begin{align*}
	\textstyle	L(z,t)\stackrel{\eqref{posdpodsopdsods}}{=}\psi(z) + t\cdot \sum_{\ell=1}^\infty \bil{\psi(z)}{\Alp_{\ell-1}(\psi,z,t)}=\psi(z) + t\cdot \bil{\psi(z)}{\Phi_\psi(z,t)}=\DDelta(z,t)
\end{align*} 
by continuity of $\bil{\cdot}{\cdot}$.
\item
Point (a) is clear from Part \ref{kjdskjdskjdskjsdxyxyyxy0}. For the points (b) and (c), we observe the following: 
\begingroup
\setlength{\leftmarginii}{12pt}{
\begin{itemize}
\item
For $\pp\in \Sem{E}$, $\dind\in \NN$, $z\in (a-1,b+\varepsilon)$, and $h\in \RR$ with $z+h\in (a-1,b+\varepsilon)$, we have 	
\begin{align}
\label{kjdskjdskjsd}
\begin{split}
		\pp^\dind_\infty(\Phi_\psi(z+h,\cdot)&-\Phi_\psi(z,\cdot))\\[2pt]
		&\leq \pp^\dind_\infty(\Phi_\psi(z+h,\cdot)-\Phi_\psi(z,\cdot)-h\cdot \DDelta(z,\cdot)) + |h|\cdot\pp^\dind_\infty( \DDelta(z,\cdot)).
\end{split}
	\end{align}
\item
Since $\bil{\cdot}{\cdot}$ is continuous and bilinear, there exists $\qq\in \Sem{E}$ with
\begin{align*}
	\pp(\bil{X}{Y})\leq \qq(X)\cdot \qq(Y) \qquad\quad\forall\: X,Y\in \mg. 
\end{align*}
Consequently, \eqref{cxcxcxcxcxcxcxoidsoidsoidsdsdsccv} implies
\begin{align}
\label{cxcxcxcxcxcxcxoidsoidsoidsdsdsccvxc}
	\sup\{a-1\leq z\leq b+\varepsilon\:|\: \pp_\infty^\dind(\DDelta(z,\cdot))\}<\infty\qquad\quad\forall\: \pp\in \Sem{E}.
\end{align}
\end{itemize}}
\endgroup
\noindent 
Point (b) is now clear from Part \ref{kjdskjdskjdskjsdxyxyyxy0}, \eqref{kjdskjdskjsd}, and \eqref{cxcxcxcxcxcxcxoidsoidsoidsdsdsccvxc}. 
For point (c), we set $\dind:=0$. Then, by Part \ref{kjdskjdskjdskjsdxyxyyxy0} and  \eqref{cxcxcxcxcxcxcxoidsoidsoidsdsdsccvxc}, both summands in \eqref{kjdskjdskjsd} tend to zero if $h$ tends to zero. Point (c) now follows from the triangle inequality, as well as continuity of $\Phi_\psi(z,\cdot)$ for each $z\in (a-1,b+\varepsilon)$.   
\item
Let $z\in (a-1,b+\varepsilon)$ and 
$\RR_{\neq 0}\supseteq \{h_n\}_{n\in \NN}\rightarrow 0$ be given. Then, Part \ref{kjdskjdskjdskjsdxyxyyxy1}.(b) implies that
	$\{\Phi_\psi(z+h_n,\cdot)\}_{n\in \NN}\mackarr{\infty} \Phi_\psi(z,\cdot)$ holds (recall \eqref{iudsiudsudspodsisdpoids}).
Theorem \ref{weakdiffdf} and Lemma \ref{fdhjfdkjj} yield
\begin{align*}
	\textstyle\limin \innt_0^\bullet \Phi_\psi(z+h_n,\cdot) = \innt_0^\bullet \Phi_\psi(z,\cdot).
\end{align*}
It follows that the map
\begin{align}
\label{mncxnmcxnmcxnmdhjdshjdshjdszuew}
	\textstyle(a-1,b+\varepsilon)\times [0,1]\ni (z,s)\mapsto \innt_0^s \Phi_\psi(z,\cdot) =\mu_{s,\psi}(z)\in G
\end{align}
is continuous. In particular, $\mu_{t,\psi}$ is continuous for each $t\in [0,1]$.

Let now $t\in [0,1]$  be fixed. The points (a) and (b) in Part  \ref{kjdskjdskjdskjsdxyxyyxy1} show that the map $\Phi_\psi|_{(a-1,b+\varepsilon)\times [0,t]}$ fulfills the assumptions in Theorem \ref{ofdpofdpofdpofdpofd}. We obtain (apply Part \ref{kjdskjdskjdskjsdxyxyyxy1}.(a) in the third step) 
\begin{align*}
	\textstyle\dot\mu_{t,\psi}(z)&\textstyle=\frac{\dd}{\dd h}\big|_{h=0} \he \innt_0^t\Phi_\psi(z+h,\cdot)\\
	&=\textstyle \dd_e\LT_{\innt_0^t \Phi_\psi(z,\cdot)}\big(\int_0^t \Ad_{[\innt_0^s\Phi_\psi(z,\cdot)]^{-1}}(\partial_1\Phi_\psi(z,s))\:\dd s\he\big)\\
	&\textstyle=\dd_e\LT_{\mu_{t,\psi}(z)}\big(\int_0^t \Ad_{\mu_{s,\psi}(z)^{-1}}\big(\ppsi(z)+s\cdot \bil{\ppsi(z)}{\Phi_\psi(z,s)}\big)\:\dd s\he\big)
\end{align*}
for each $z\in (a-1,b+\varepsilon)$.  
It follows from smoothness of the group operations, continuity of \eqref{mncxnmcxnmcxnmdhjdshjdshjdszuew}, and continuity of $\Phi_\psi$   (by Part \ref{kjdskjdskjdskjsdxyxyyxy1}.(c)) that $\dot\mu_{t,\psi}$ is continuous. This shows that $\mu_{t,\psi}$ is of class $C^1$, which proves the claim.
\qedhere
\end{enumerate}}
\endgroup
\end{proof}
\noindent
We are ready for the proof of Theorem \ref{dkjkjfdkjfdjfdkfdfdfd}.
\begin{proof}[Proof of Theorem \ref{dkjkjfdkjfdjfdkfdfdfd}]
	Since $\mg$ is sequentially complete, $\mg$ is both Mackey complete and integral complete. To prove the claim, it thus suffices to show that $G$ is $C^0$-semiregular. For this, let $\phi\in C^0([0,1],\mg)$ be given, and fix a continuous extension $\psi\colon [0,1+\varepsilon]\rightarrow \mg$ for some $\varepsilon>0$. We define $\mu_{1,\psi}\colon (-1,1+\varepsilon)\rightarrow G$ as in \eqref{skjskjsjksdcxcx} (for $a\equiv 0$ and $b\equiv 1$ there). Lemma \ref{kjdskjdskjdskjsdxyxyyxy}.\ref{kjdskjdskjdskjsdxyxyyxy2} shows that $\mu_{1,\psi}$ is of class $C^1$, so that it remains to verify $\Der(\mu_{1,\psi}|_{[0,1]})=\phi$. 
	For this, let $z\in [0,1]$ be fixed. Corollary \ref{kjfdkjfdkjfd} shows ($\Phi_\psi$ is defined as in \eqref{nmdfnmfdnmfd})
\vspace{-10pt}
\begin{align}
\label{oidoidoidoidxxxxx}
	\textstyle\mu_{1,\psi}(z+h)= \innt_0^1 \Phi_\psi(z+h,\cdot)\stackrel{\eqref{lkfdlkkdllkfdlkfdkfd}}{=}\innt_0^1 \TMAP(\psi|_{[0,z+h]}) =\innt_0^1 \TMAP(\psi|_{[z,z+h]})\cdot \overbracket{\textstyle\innt_0^1 \TMAP(\psi|_{[0,z]})}^{=\: \mu_{1,\psi}(z)} 
\end{align}	 
for each $h\in [0,\varepsilon)$. 
Set $\wh{\psi}:=\psi|_{[z,z+\varepsilon]}$. Then, 
\begin{align*}
	\textstyle\mu_{1,\wh{\psi}}(z+h)=\innt_0^1 \TMAP(\wh{\psi}|_{[z,z+h]})=\innt_0^1\TMAP(\psi|_{[z,z+h]})\qquad\quad\forall\: h\in [0,\varepsilon)
\end{align*}
holds by definition, so that we have
\begin{align}
\label{oidsoisdoids}
	\mu_{1,\psi}(z+h)\cdot \mu_{1,\psi}(z)^{-1}\stackrel{\eqref{oidoidoidoidxxxxx}}{=} \mu_{1,\wh{\psi}}(z+h)\qquad\quad\forall\: h\in [0,\varepsilon).
\end{align}
Lemma \ref{kjdskjdskjdskjsdxyxyyxy}.\ref{kjdskjdskjdskjsdxyxyyxy2}  shows $\mu_{1,\wh{\psi}}\in C^1((z-1,z+\varepsilon),G)$, and Equation \eqref{jkkjkjfdjfdkjkjdfjkfkjfdkjdf} in Lemma \ref{kjdskjdskjdskjsdxyxyyxy}.\ref{kjdskjdskjdskjsdxyxyyxy2} together with both sides of \eqref{kjdskjdskjdsjkjdskkjdsds} (first step)  yields
\begin{align*}
	\dot\mu_{1,\wh{\psi}}(z)=\wh{\psi}(z)=\psi(z)=\phi(z).
\end{align*}
Then, \eqref{oidsoisdoids} shows   
$\Der(\mu_{1,\psi})(z)=\phi(z)$, which proves the claim. 
\end{proof}

\section*{Acknowledgements}   
The author thanks Stefan Waldmann for drawing his attention to the construction made by Duistermaat and Kolk \cite{DUIS} to prove Lie's third theorem in the finite-dimensional context.    
This research was supported by the Deutsche Forschungsgemeinschaft, DFG, Project Number HA 8616/1-1.

\addtocontents{toc}{\protect\setcounter{tocdepth}{0}}
\appendix

\section*{APPENDIX}

\section{Appendix}
\subsection{Bastiani's Differential Calculus}
\label{Diffcalc}
In this appendix, we recall Bastiani's differential calculus, confer also  \cite{HA,HG,MIL,KHN,KHN2,KHNM}.  
Let $E,F\in \HLCV$ be given.  
A map $f\colon U\rightarrow F$, with $U\subseteq E$ open, is said to be  
differentiable if
\begin{align*}
	\textstyle(D_v f)(x):=\lim_{t\rightarrow 0}1/t\cdot (f(x+t\cdot v)-f(x))\in F
\end{align*} 
exists for each $x\in U$ and $v\in E$. The map $f$ is said to be $k$-times differentiable for $k\geq 1$ if 
	\begin{align*}
	D_{v_k,\dots,v_1}f:= D_{v_k}(D_{v_{k-1}}( {\dots} (D_{v_1}(f))\dots))\colon U\rightarrow F
\end{align*}
is defined for all $v_1,\dots,v_k\in E$. Implicitly, this means that $f$ is $p$-times differentiable for each $1\leq p\leq k$, and we set
\begin{align*}
	\dd^p_xf(v_1,\dots,v_p)\equiv \dd^p f(x,v_1,\dots,v_p):=D_{v_p,\dots,v_1}f(x)\qquad\quad\forall\: x\in U,\:v_1,\dots,v_p\in E
\end{align*} 	
for $p=1,\dots,k$. We furthermore define $\dd f:= \dd^1 f$, as well as $\dd_x f:= \dd^1_x f$ for each $x\in U$.  
The map $f\colon U\rightarrow F$ is said to be  
\begingroup
\setlength{\leftmargini}{12pt}
\begin{itemize}
\item
of class $C^0$ if it is continuous. In this case, we define $\dd^0 f:= f$.
\item
of class $C^k$ for $k\geq 1$ if it is $k$-times differentiable, such that 
\begin{align*}
	\dd^p f\colon U\times E^p\rightarrow F,\qquad (x,v_1,\dots,v_p)\mapsto D_{v_p,\dots,v_1}f(x)
\end{align*} 
is continuous for $p=0,\dots,k$.  
In this case, $\dd^p_x f$ is symmetric and $p$-multilinear for each $x\in U$ and $p=1,\dots,k$, cf.\ \cite{HG}.  
\item
of class $C^\infty$ if it is of class $C^k$ for each $k\in \NN$. 
\end{itemize}
\endgroup
\begin{rem}
Assume that $E,F$ are normed spaces. Let $U\subseteq E$ be non-empty open, as well as $k\in \NN\cup\{\infty\}$. Let $\mathcal{F}C^k(U,F)$ denote the set of all $k$-times Fr\'{e}chet differentiable maps $U\rightarrow F$. Then, $C^{k+1}(U,F)\subseteq \mathcal{F}C^k(U,F)\subseteq C^k(U,F)$ holds \cite{KHN,WALTER}, hence $C^\infty(U,F)=\mathcal{F}C^\infty(U,F)$.
\hspace*{\fill}\qed
\end{rem}
\noindent
We have the following differentiation rules \cite{HG}.
\begin{custompr}{A.1}
\label{iuiuiuiuuzuzuztztttrtrtr}
\noindent

\vspace{-6pt}
\begingroup
\setlength{\leftmargini}{17pt}
{
\renewcommand{\theenumi}{{\alph{enumi}})} 
\renewcommand{\labelenumi}{\theenumi}
\begin{enumerate}
\item
\label{iterated}
A map $f\colon E\supseteq U\rightarrow F$ is of class $C^k$ for $k\geq 1$ if and only if $\dd f$ is of class $C^{k-1}$ when considered as a map $E \times E \supseteq U\times E \rightarrow F$.
\item
\label{linear}
Let $f\colon E\rightarrow F$ be linear and continuous. Then, $f$ is smooth, with $\dd^1_xf=f$ for each $x\in E$, as well as $\dd^pf=0$ for each $p\geq 2$.  
\item
\label{speccombo}
Let $F_1,\dots,F_m$ be Hausdorff locally convex vector spaces, and 
$f_q\colon E\supseteq U\rightarrow F_q$ be of class $C^k$ for $k\geq 1$ and $q=1,\dots,m$. Then, 
\begin{align*}
	f:=f_1\times{\dots}\times f_m \colon U\rightarrow F_1\times{\dots}\times F_m,\qquad x\mapsto (f_1(x),\dots,f_m(x))
\end{align*}
if of class $C^k$, with $\dd^p f=\dd^pf_1{\times} \dots\times \dd^pf_m$ for $p=1,\dots,k$.
\item
\label{chainrule}
	Let $F,\bar{F},\bar{\bar{F}}\in \HLCV$, $1\leq k\leq \infty$, as well as  
	 $f\colon F\supseteq U\rightarrow \bar{U}\subseteq \bar{F}$ and $\bar{f}\colon \bar{F}\supseteq \bar{U}\rightarrow \bar{\bar{U}}\subseteq \bar{\bar{F}}$ be of class $C^k$. Then, $\bar{f}\cp f\colon U\rightarrow \bar{\bar{F}}$ is of class $C^k$, with 
	\begin{align*}
		\dd_x(\bar{f}\cp f)=\dd_{f(x)}\bar{f}\cp \dd_x f\qquad\quad \forall\: x\in U.
	\end{align*}
	\vspace{-18pt}
\item
\label{productrule}
	Let $F_1,\dots,F_m,E\in \HLCV$, and $f\colon F_1\times {\dots} \times F_m\supseteq U\rightarrow E$ be of class $C^0$. Then, $f$ is of class $C^1$ if and only if for $p=1,\dots,m$, the partial derivatives
	\begin{align*}
		\partial_p f \colon U\times F_p\ni((x_1,\dots,x_m),v_p)&\textstyle\mapsto \lim_{t\rightarrow 0} 1/t\cdot (f(x_1,\dots, x_p+t\cdot v_p,\dots,x_m)-f(x_1,\dots,x_m))
	\end{align*}
	exist in $E$ and are continuous. In this case, we have
	\begin{align*}
		\textstyle\dd f((x_1,\dots,x_m),v_1,\dots,v_m)&\textstyle=\sum_{p=1}^m\partial_p f((x_1,\dots,x_m),v_p)\\
		\big(\!&\textstyle= \sum_{p=1}^m \hspace{4.5pt}\dd f((x_1,\dots,x_m),(0,\dots,0, v_p,0,\dots,0))\he\big)
	\end{align*}
	for each $(x_1,\dots,x_m)\in U$, and $v_p\in F_p$ for $p=1,\dots,m$.
\end{enumerate}}
\endgroup
\end{custompr}

\subsection{Proof of Equation \eqref{kjsdkjdsjdsa}}
\label{asassadsdsdsdsdsdsdsaaa}
The claim is clear for $m=1$ and $n\geq 1$. Assume now that \eqref{kjsdkjdsjdsa} holds for some $m\geq 1$, as well as each $n\geq 1$. For $X\in \SP_1(S)$, $\alpha\in \SP_m(S)$, and $\beta\in \SP_n(S)$ with $n\geq 1$, we obtain from \eqref{nmvcnmvcnmnkjsakjsakjsaa} and antysimmetry of $\bil{\cdot}{\cdot}$ that 
\begin{align*}
	\bil{\bil{X}{\alpha}}{\beta}=\bil{X}{\bil{\alpha}{\beta}}-\bil{\alpha}{\bil{X}{\beta}}\in \SP(S)_{(m+1)+n}
\end{align*}
holds. Since $\bil{\cdot}{\cdot}$ is bilinear, this implies $\SPAN{\bil{\mathcal{V}_{m+1}(S)}{\mathcal{V}_n(S)}}\subseteq \mathcal{V}_{(m+1)+n}(S)$ for each $n\geq 1$, so that the claim follows inductively.

\subsection{Proof of Lemma \ref{hjfdkjfdkjfd}}
\label{asassadsdsdsdsdsdsds}
\begin{customlem}{\ref{hjfdkjfdkjfd}}
Let $V\subseteq F$ be open with $0\in V$. Let furthermore  
$\Psi\colon V\times E\rightarrow E$ 
be smooth with $\Psi(0,\cdot)=\id_{E}$, such that $\Psi(x,\cdot)$ is linear for each $x \in V$. Then, to each $\pp\in \Sem{E}$, there exist  $\qq\in \Sem{F}$ and $\ww\in \Sem{E}$ with
\begin{align*}
	\pp(\Psi(x,Y)-Y)\leq \qq(x)\cdot \ww(Y)\qquad\quad\forall\: x\in \B_{\qq,1}\subseteq V,\: Y\in E.
\end{align*}
\end{customlem}
\begin{proof}[Proof of Lemma \ref{hjfdkjfdkjfd}]
We consider the continuous map
	\begin{align*}
		\Phi\colon V \times E \times F \rightarrow E,\qquad (x,Y,Z)\mapsto \partial_1 \Psi((x,Y),Z)
	\end{align*}
	that is linear in the last two arguments. Then, for $\pp\in \Sem{E}$ fixed, by Lemma \ref{alalskkskaskaskas} there exist $\qq\in \Sem{F}$  and $\pp\leq \ww\in \Sem{E}$ with
	\begin{align}
	\label{jdkjkjdkjdkjfd}
		(\pp\cp\Phi)((x,Y),Z)&\leq \ww(Y)\cdot \qq(Z)\qquad\quad\forall\: x\in \B_{\qq,1}\subseteq V,\: Y\in E,\: Z\in F.
	\end{align}
For $x\in \B_{\qq,1}$, we let $\gamma\colon [0,1]\ni t\mapsto t\cdot x \in \B_{\qq,1}$, and define 
\begin{align*}
	C^\infty([0,1],E)\ni\alpha_Y\colon [0,1]\ni t\mapsto \Psi(\gamma(t),Y)\qquad\quad\forall\: Y\in E.
\end{align*}
Then, the parts \ref{chainrule} and \ref{productrule} of Proposition \ref{iuiuiuiuuzuzuztztttrtrtr} yield $\dot\alpha_Y=\Phi((\gamma,Y),\dot\gamma)=\Phi((\gamma,Y),x)$, and we obtain 	
\begin{align*}
	\pp(\Psi(x,Y)-Y)&=\pp(\alpha_Y(1)-\alpha_Y(0))
	\textstyle \stackrel{\eqref{isdsdoisdiosd1}}{\leq} \int_0^1 \pp(\dot\alpha_Y(s)) \:\dd s\\[-5pt]
	&\textstyle\leq \int_0^1 \pp(\Phi((\gamma(s),Y),x)) \:\dd s
	\textstyle\leq 
	 \pp_\infty(\Phi((\gamma,Y),x))
	 \textstyle \stackrel{\eqref{jdkjkjdkjdkjfd}}{\leq} \ww(Y)\cdot\qq(x)
\end{align*}	
for all $x\in \B_{\qq,1}$ and $Y\in E$.
\end{proof}

\subsection{Appendix to Sect.\ \ref{ouwuiewuiewuoew}}
\label{asassadsdsdsdfdfdsdsdsds}
For combinatorial reasons, there exist $c[p]_0,\dots,c[p]_{p^2}\in \korp$ with 
\begin{align*}
	f_p(g_p(z))&\textstyle=(f_p\diamond g_p)(z)=\sum_{n=0}^{p^2} c[p]_{n}\cdot z^n\\[4pt]
	f_p(g_p(\model))&\textstyle=(f_p\diamond g_p)(\model)=\sum_{n=0}^q c[p]_{n}\cdot \model^n\\
	f_p(g_p(\banel))&\textstyle=(f_p\diamond g_p)(\banel)=\sum_{n=0}^{p^2} c[p]_{n}\cdot \banel^n.
\end{align*}
Set $c[p]_n:=0$ for $n\geq p^2+1$. Since $\{f_p\cp g_p\}_{p\in \NN}\rightarrow f\cp g$ converges compactly on $\mU_{S}(0)$, the Weierstrass convergence theorem yields 
\begin{align}
\label{kjfdkjfkjfdf}
	\textstyle \lim_{p\rightarrow\infty} c[p]_{n}= c_n\qquad\quad\forall\: n\in \NN.
\end{align}
\begin{proof}[Proof of Equation \eqref{ljksdkljsdklsdnil}]
Since $b_0=0$ holds, we have $c[q+\ell]_{n}=c[q]_{n}$ for $n=0,\dots,q$ and $\ell\in \NN$, hence $c[q]_n=c_n$ for $n=0,\dots,q$ by \eqref{kjfdkjfkjfdf}.  
We obtain  
\begin{align*}
	\textstyle f(g(\model))=\sum_{n=0}^q a_n\cdot (\sum_{\ell=1}^q b_\ell\cdot \model^\ell)^n=f_{q}(g_{q}(\model))=\sum_{n=0}^q c[q]_n\cdot \model^n=\sum_{n=0}^q c_n\cdot \model^n=(f\diamond g)(\model)
\end{align*}
which proves \eqref{ljksdkljsdklsdnil}. 
\end{proof}
\begin{proof}[Proof of Equation \eqref{ljksdkljsdklsd}]
Fix $\bannorm{\banel}< s <S$, and set $M:=|f|_{|g|_s}$. Then,  
$|(f_p\cp g_p)(z)|\leq M$ holds for all $p\in \NN$ and $z\in \CC$ with $|z|\leq s$. 
Cauchy's estimate yields $\textstyle|c[p]_n|\leq \frac{M}{s^n}$ for all $p,n\in \NN$, hence
\begin{align*}
	\textstyle \sum_{n=N+1}^\infty |c[p]_n|\cdot (\bannorm{\banel})^n= \sum_{n=N+1}^\infty |c[p]_n|\cdot s^n\cdot \frac{(\bannorm{\banel})^n}{s^n}\leq M\cdot \sum_{n=N+1}^\infty  \frac{(\bannorm{\banel})^n}{s^n}
\end{align*} 
for all $N,p\in \NN$. Given $\varepsilon>0$, there thus exists $N_\varepsilon\in \NN$ with 
\begin{align}
\label{ddffdfdiofdiofdoi1}
	\textstyle\bannorm{\sum_{n=N_\varepsilon+1}^\infty c_n\cdot \banel^n}<\frac{\varepsilon}{4}\qquad\quad\text{as well as}\qquad\quad \bannorm{\sum_{n=N_\varepsilon+1}^\infty c[p]_n\cdot \banel^n}<\frac{\varepsilon}{4} \qquad\forall\: p\in \NN. 
\end{align}
By \eqref{kjfdkjfkjfdf}, there exists $m\in \NN$ with 
\begin{align}
\label{ddffdfdiofdiofdoi2}
	\textstyle|c[p]_n-c_n|<\frac{\varepsilon}{4(N_\varepsilon+1)\cdot\max(\bannorm{\banel^q}\:|\: 0\leq q\leq N_\varepsilon)}\qquad\quad\forall\: 0\leq n\leq N_\varepsilon,\: p\geq m. 
\end{align}
Increasing $m$ if necessary, we can additionally assume\footnote{Apply the triangle inequality in the form $\bannorm{f(g(\banel))-f_p(g_p(\banel))}\leq \bannorm{f(g(\banel))-f(g_p(\banel))}+\bannorm{f(g_p(\banel))-f_p(g_p(\banel))}$.} 
\begin{align*}
	\textstyle\bannorm{f(g(\banel))-f_p(g_p(\banel))}<\frac{\varepsilon}{4}\qquad\quad\forall\: p\geq m. 
\end{align*}
Set $p:= m$. Then, we have
\begin{align*}
	\textstyle\bannorm{f(g(\banel))-(f\diamond g)(\banel)}
	&\textstyle\leq \underbracket{\bannorm{f(g(\banel))-f_p(g_p(\banel))}}_{<\frac{\varepsilon}{4}}+ \underbracket{\textstyle\bannorm{f_p(g_p(\banel)) - \sum_{n=0}^\infty c_n\cdot \banel^n }}_{=:\: \mathrm{A}}
\end{align*}
\vspace{-20pt}

\noindent
with
\begin{align*}
	\mathrm{A}&\textstyle \leq \bannorm{\sum_{n=0}^{N_\varepsilon} c[p]_n\cdot \banel^n - \sum_{n=0}^{N_\varepsilon} c_n\cdot \banel^n }
	\textstyle+ \bannorm{\sum_{n=N_\varepsilon+1}^\infty c[p]_n\cdot \banel^n - \sum_{n=N_\varepsilon+1}^\infty c_n\cdot \banel^n}\\[2pt]
	&\textstyle \leq \sum_{n=0}^{N_\varepsilon} |c[p]_n-c_n|\cdot \bannorm{\banel^n}\textstyle+ \bannorm{\sum_{n=N_\varepsilon+1}^\infty c[p]_n\cdot \banel^n} + \bannorm{\sum_{n=N_\varepsilon+1}^\infty c_n\cdot \banel^n} \\[2pt]
	&\textstyle\leq \frac{3}{4}\cdot \varepsilon 
\end{align*}
by \eqref{ddffdfdiofdiofdoi1} and \eqref{ddffdfdiofdiofdoi2}.
\end{proof}

\subsection{Proof of Lemma \ref{lkjflkjfdjkdflkjfd}}
\label{asassadsdsdsdsdsdcxcxsds}

\begin{customlem}{\ref{lkjflkjfdjkdflkjfd}}
Assume that $G$ is weakly $C^k$-regular for $k\in \NN\cup\{\lip,\infty\}$, and let $\phi\in C^k([a,b])$ ($a<b$) be given. Then, 
$\kappa\colon \RR\times [a,b]\ni (t,x)\mapsto \innt_x^bt\cdot \phi\in G$ is continuous.  
\end{customlem}
\begin{proof}
Let $z\in [a,b]$, $t\in \RR$, and $U\subseteq G$ be an open neighbourhood of $e$. We fix $W\subseteq G$ open with $e\in W$ and $W\cdot W\subseteq U$, and observe that 
\begin{align*}
	\textstyle \kappa(t,\cdot)\colon [a,b]\ni x\mapsto [\innt_a^{b} t\cdot \phi] \cdot [\innt_a^x t\cdot \phi]^{-1}\in G
\end{align*} 
is continuous. Since $[a,b]$ is compact, there exists $V\subseteq G$ open with $e\in V=V^{-1}$, such that 
\begin{align*}
	\textstyle (\innt_x^b t\cdot \phi)^{-1} \cdot V\cdot (\innt_x^b t\cdot \phi) \cdot V \subseteq W\qquad\quad\forall\: x\in [a,b].
\end{align*}
Theorem \ref{weakdiffdf} together with Lemma 15 in \cite{MDM} implies that there exists $\delta>0$ with
\begin{align*}
	\textstyle\innt_a^\bullet (t+h)\cdot \phi \in 
	V\cdot [\innt_a^\bullet t\cdot \phi]
	\qquad\quad\forall\: |h|<\delta.  
\end{align*}
For 
$a\leq x\leq b$ and
 $|h|<\delta$, we obtain 
\begin{align*}
	\textstyle\innt_x^b(t+h)\cdot \phi &\textstyle = [\innt_a^b (t+h)\cdot \phi]\cdot[\innt_a^x(t+h)\cdot \phi]^{-1}\\
	&\textstyle\in V\cdot ([\innt_a^b t\cdot \phi]\cdot[\innt_a^x t\cdot \phi]^{-1}) \cdot V^{-1}\\
	&\textstyle=  V\cdot (\innt_x^b t\cdot \phi) \cdot V\\
	&\textstyle\subseteq W\cdot \innt_x^b t\cdot \phi.
\end{align*}
Since $\kappa(t,\cdot)$ is continuous,  
we can shrink $\delta>0$ such that for each $x\in (z-\delta,z+\delta) \cap [a,b]$, we have $[\innt_x^b t\cdot \phi]\cdot[\innt_z^b t\cdot\phi]^{-1}\in W$. 
We obtain 
\begin{align*}
	\textstyle[\innt_x^b (t+h)\cdot \phi]\cdot [\innt_z^b t\cdot\phi]^{-1}= ([\innt_x^b (t+h)\cdot \phi]\cdot [\innt_x^b t\cdot \phi]^{-1})\cdot ([\innt_x^b t\cdot \phi]\cdot[\innt_z^b t\cdot\phi]^{-1})\in W\cdot W\subseteq U
\end{align*}
for $|h|<\delta$ and $x\in (z-\delta,z+\delta) \cap [a,b]$. 
\end{proof}

\subsection{Proof of Equation \eqref{iureiureiueriureiure}}
\label{asassadsdsdsdsdsdcxcxsdscvvcvcvcvcvcc}

\begin{proof}[Proof of Equation \eqref{iureiureiueriureiure}]
By definition, we have $\Der(\mu)|_{[a,b]}=\psi$ for certain $\mu\in C^1((a-\varepsilon,b+\varepsilon),G)$ and $\varepsilon>0$. Replacing $\mu$ by $\mu\cdot \mu(a)^{-1}$ if necessary, we can assume that $\mu(a)=e$ holds, hence 
\begin{align*}
	\textstyle\mu|_{[a,b]}=\innt_a^\bullet \phi=\innt_a^\bullet \psi\qquad\text{for}\qquad \phi:=\Der(\mu).
\end{align*}
Then, $\alpha^\pm:= \Ad_{\mu^\pm}(X)\in C^1((a-\varepsilon,b+\varepsilon))$ holds by Lemma \ref{Adlip}, with $\alpha^\pm|_{[a,b]}=\Add^\pm_\psi[X]$. This shows $\Add_\psi^\pm[X]\in C^1([a,b],\mg)$, with $\Add^\pm_\psi[X](a)=\Ad_e(X)=X$. Let now $t\in [a,b]$ and $0<h<\varepsilon$ be given. 
We observe 
\begin{align*}
	\textstyle[\innt_t^{t+h}\phi]^{-1}\stackrel{\ref{oitoioiztoiztoiztoizt}}{=}\innt_t^{t+h} -\Ad_{[\innt_t^\bullet\phi]^{-1}}(\phi),
\end{align*} 
and obtain from \ref{pogfpogf} that 
\begin{align*}
	\textstyle\alpha^+(t+h)&\textstyle=\Ad_{\innt_t^{t+h}\phi}(\Add^+_\psi[X](t))\\
	\textstyle\alpha^-(t+h)&\textstyle=\Add_\psi^-[\Ad_{[\innt_t^{t+h}\phi]^{-1}}(X)](t)\\
	&\textstyle= \Add_\psi^-[\Ad_{\innt_t^{t+h} -\Ad_{[\innt_t^\bullet\phi]^{-1}}(\phi)}(X)](t).
\end{align*}
Differentiating both hand sides and then applying \eqref{iufiugfiugfiugfiugfgf}, we obtain \eqref{iureiureiueriureiure}.
\end{proof}

\subsection{The Proof of Lemma \ref{knkjfdsjkfasdkjdsjkdsa}}
\label{asassadsdsdsdsdsdcxcxsdssdssddssd}
\begin{customlem}{\ref{knkjfdsjkfasdkjdsjkdsa}}
	Let $\psi \in C^0([a,b],\mq)$ be given. Then, the following assertions hold:
\begingroup
\setlength{\leftmargini}{17pt}
{
\renewcommand{\theenumi}{\emph{\arabic{enumi})}} 
\renewcommand{\labelenumi}{\theenumi}
\begin{enumerate}
\item
\label{knkjfdsjkfasdkjdsjkdsa1}
	The maps $\etam_\psi^\pm$ 	
	 are defined and of class $C^1$, and we have 
\begin{align*}
	\partial_1 \etam_\psi^+(t,X)&\textstyle=\partial_t\Add_\psi^+[X](t)=\bil{\psi(t)}{\Add^+_\psi[X](t)}\\
	\partial_1 \etam^-_\psi(t,X)&\textstyle=\partial_t\Add_\psi^-[X](t)=-\Add^-_\psi[\bil{\psi(t)}{X}](t)
\end{align*}
for all $t\in [a,b]$ and $X\in \mq$.
\item
\label{knkjfdsjkfasdkjdsjkdsa2}
	$\etam^\pm_\psi(t,\cdot)=\Add_\psi^\pm[t]$ is linear and continuous for each $t\in [a,b]$.
\item
\label{knkjfdsjkfasdkjdsjkdsa3}
	If $\psi\in C^k([a,b],\mq)$ holds for $k\in \NN\cup\{\infty\}$, then $\etam_\psi^\pm$ is of class $C^{k+1}$.    	
\end{enumerate}}
\endgroup   	   	
\end{customlem}
\begin{proof}
Let $X\in \mq$ be fixed. Then, $\APOL^\pm_{\ell,\psi}[X]$ is of class $C^1$ for each $\ell\in \NN$ by \eqref{oidsoidoisdoidsoioidsiods}; with $\partial_t\APOL^\pm_{0,\psi}[X]=0$ as well as
\begin{align*}
	\partial_t\APOL^+_{\ell,\psi}[X](t)= \bil{\psi(t)}{\APOL^+_{{\ell-1},\psi}[X](t)}\qquad\quad\text{and}\qquad\quad
	\partial_t\APOL^-_{\ell,\psi}[X](t)= -\APOL^-_{{\ell-1},\psi}[\bil{\psi(t)}{X}](t)
\end{align*}	
for each $\ell \geq 1$ and $t\in [0,1]$. 
	For $\vv\leq \ww$ as in \eqref{assaaass}, we have 
\begin{align}
\label{lkdsklslkdsds}
	\textstyle\sum_{\ell=0}^\infty\vv_\infty(\APOL^\pm_{\ell,\psi}[X])\textstyle\leq \ww(X)\cdot \e^{(b-a)\cdot \ww_\infty(\psi)}<\infty.
\end{align}	
Corollary \ref{dsdsdsdsdsdsdsdsasas} thus implies $\etam^\pm(\cdot,X)=\Add_\psi^\pm[X]\in C^1([a,b],\mq)$, with 
\begin{align*}
	\textstyle\partial_t\Add_\psi^+[X](t)&\textstyle=\hspace{10.4pt}\sum_{\ell=0}^\infty \bil{\psi(t)}{\APOL_\psi^+[X](t)}=\bil{\psi(t)}{\Add^+_\psi[X](t)}\\
	\textstyle\partial_t\Add_\psi^-[X](t)&\textstyle=-\sum_{\ell=0}^\infty \APOL_\psi^-[\bil{\psi(t)}{X}](t)=-\Add^-_\psi[\bil{\psi(t)}{X}](t)
\end{align*}
for each $t\in [a,b]$; which establishes Part \ref{knkjfdsjkfasdkjdsjkdsa1}. Now, 
it is clear from the definitions that $\etam^\pm_\psi(t,\cdot)$ is linear for each $t\in [a,b]$; and we furthermore observe that
\begin{align}
\label{lkdsklslkdsdsd}
\begin{split}
	\textstyle\vv\big(\sum_{\ell=0}^\infty\APOL^\pm_{\ell,\psi}[X](t')&\textstyle-\sum_{\ell=0}^\infty\APOL^\pm_{\ell,\psi}[X](t)\big)\\
		&\textstyle \leq |t-t'| \cdot \ww(X)\cdot\ww_\infty(\psi)\cdot  \sum_{\ell=1}^\infty  \frac{(b-a)^{\ell-1}}{(\ell-1)!}\cdot \ww_\infty(\psi)^{\ell-1}\\
		&\textstyle = |t-t'| \cdot \ww(X)\cdot\ww_\infty(\psi)\cdot  \e^{(b-a)\cdot \ww_\infty(\psi)}
\end{split}
\end{align}
holds for all $t,t'\in [a,b]$ and $X\in \mq$. We obtain from \eqref{lkdsklslkdsds} and \eqref{lkdsklslkdsdsd} that
\begin{align*}
	\vv(\etam_\psi^\pm(t',X')-\etam_\psi^\pm(t,X))&\textstyle \leq \vv(\etam_\psi^\pm(t',X')-\etam_\psi^\pm(t',X))+\vv(\etam_\psi^\pm(t',X)-\etam_\psi^\pm(t,X))\\
	&\leq \textstyle (\ww(X'-X) + |t-t'| \cdot \ww(X)\cdot\ww_\infty(\psi) )\cdot \e^{(b-a)\cdot \ww_\infty(\psi)}	
\end{align*}
for all $t,t'\in [a,b]$ and $X,X'\in \mq$. This shows that $\etam_\psi^\pm$ is continuous, which establishes Part \ref{knkjfdsjkfasdkjdsjkdsa2}.  We furthermore conclude that the partial derivatives 
\begin{align*}
	\partial_1\etam_\psi^+((t,X),\lambda)\:&= \lambda\cdot \bil{\psi(t)}{\etam_\psi^+(t,X)}\qquad\quad\hspace{8.5pt}\forall\: t\in[a,b],\:\lambda\in \RR,\: X\in \mq \\
	\partial_1\etam_\psi^-((t,X),\lambda)\:&= -\lambda\cdot  \etam_\psi^-(t,\bil{\psi(t)}{X})\qquad\quad\forall\: t\in[a,b],\:\lambda\in \RR,\: X\in \mq \\
    \partial_2\etam_\psi^\pm((t,X),Y)&=   \etam_\psi^\pm(t,Y) \qquad\qquad\qquad\qquad\hspace{1pt}\forall\: t\in [a,b],\: X,Y\in \mq
\end{align*}
are continuous. Part \ref{productrule} of Proposition \ref{iuiuiuiuuzuzuztztttrtrtr} thus shows that the maps $\etam_\psi^\pm$ are of class $C^1$, with
\begin{align}
\label{hjdshjhjdshjdshjdshjds}
\begin{split}
	\dd\etam_\psi^+((t,X),(\lambda,Y))&=  \hspace{8.2pt}\lambda\cdot \bil{\psi(t)}{\etam_\psi^+(t,X)} + \etam_\psi^+(t,Y)\\
	\dd\etam_\psi^-((t,X),(\lambda,Y))&=  -\lambda\cdot \etam_\psi^-(t,\bil{\psi(t)}{X}) + \etam_\psi^-(t,Y).
\end{split}
\end{align}  
Finally, assume that $\psi\in C^k([a,b],\mq)$ holds for $k \geq 1$. Then, \eqref{hjdshjhjdshjdshjdshjds} together with the parts \ref{iterated}, \ref{speccombo}, \ref{chainrule} of Proposition \ref{iuiuiuiuuzuzuztztttrtrtr} imply that given $1\leq \ell\leq k$, then
$\etam_\psi^\pm$ is of class $C^{\ell+1}$ if $\etam_\psi^\pm$ is of class $C^\ell$, which establishes Part \ref{knkjfdsjkfasdkjdsjkdsa3}. 
It thus follows by induction that $\etam_\psi^\pm$ is of class $C^{k+1}$. 
\end{proof}


\end{document}